\documentclass[11pt]{article}
\usepackage{amsmath,amsthm,amssymb}
\usepackage[usenames,dvipsnames]{xcolor}
\usepackage{enumerate}
\usepackage{graphicx}
\usepackage{cite}
\usepackage{comment}
\usepackage{oands}
\usepackage{tikz}
\usepackage{changepage}
\usepackage{bbm}
\usepackage{mathtools}
\usepackage[margin=1in]{geometry}
\usepackage[pagewise,mathlines]{lineno}
\usepackage{appendix}
\usepackage{stmaryrd}
\usepackage{multicol}
\usepackage{microtype}
\usepackage[colorinlistoftodos]{todonotes}

\usepackage[pdftitle={Random walk on mated-CRT planar maps and Liouville Brownian motion},
  pdfauthor={Nathanael Berestycki and Ewain Gwynne},
colorlinks=true,linkcolor=NavyBlue,urlcolor=RoyalBlue,citecolor=PineGreen,bookmarks=true,bookmarksopen=true,bookmarksopenlevel=2,unicode=true,linktocpage]{hyperref}

\theoremstyle{plain}
\newtheorem{thm}{Theorem}[section]
\newtheorem{cor}[thm]{Corollary}
\newtheorem{lem}[thm]{Lemma}
\newtheorem{prop}[thm]{Proposition}

\newtheorem{notation}[thm]{Notation}

\def\@rst #1 #2other{#1}
\newcommand\MR[1]{\relax\ifhmode\unskip\spacefactor3000 \space\fi
  \MRhref{\expandafter\@rst #1 other}{#1}}
\newcommand{\MRhref}[2]{\href{http://www.ams.org/mathscinet-getitem?mr=#1}{MR#2}}

\theoremstyle{definition}
\newtheorem{defn}[thm]{Definition}
\newtheorem{remark}[thm]{Remark}

\numberwithin{equation}{section}

\newcommand{\dsb}{\begin{adjustwidth}{2.5em}{0pt}
\begin{footnotesize}}
\newcommand{\dse}{\end{footnotesize}
\end{adjustwidth}}

\newcommand{\ssb}{\begin{adjustwidth}{2.5em}{0pt}}
\newcommand{\sse}{\end{adjustwidth}}

\newcommand{\aryb}{\begin{eqnarray*}}
\newcommand{\arye}{\end{eqnarray*}}
\def\alb#1\ale{\begin{align*}#1\end{align*}}
\def\allb#1\alle{\begin{align}#1\end{align}}
\newcommand{\eqb}{\begin{equation}}
\newcommand{\eqe}{\end{equation}}
\newcommand{\eqbn}{\begin{equation*}}
\newcommand{\eqen}{\end{equation*}}

\newcommand{\BB}{\mathbbm}
\newcommand{\ol}{\overline}
\newcommand{\ul}{\underline}
\newcommand{\op}{\operatorname}

\newcommand{\frk}{\mathfrak}
\newcommand{\eqD}{\overset{d}{=}}
\newcommand{\ep}{\varepsilon}
\newcommand{\rta}{\rightarrow}

\newcommand{\wt}{\widetilde}
\newcommand{\wh}{\widehat}
\newcommand{\mcl}{\mathcal}

\newcommand{\bdy}{\partial}

\newcommand{\srta}{\shortrightarrow}

\newcommand{\eps}{\varepsilon}

\let\originalleft\left
\let\originalright\right
\renewcommand{\left}{\mathopen{}\mathclose\bgroup\originalleft}
\renewcommand{\right}{\aftergroup\egroup\originalright}

\title{Random walks on mated-CRT planar maps\\ and Liouville Brownian motion}
 \date{ }
 \author{
\begin{tabular}{c} Nathana\"el Berestycki\\[-5pt]\small University of Vienna \end{tabular}
\begin{tabular}{c} Ewain Gwynne\\[-5pt]\small University of Chicago \end{tabular}
}

\begin{document}

\maketitle

\begin{abstract}
We prove a scaling limit result for random walk on certain random planar maps with its natural time parametrization.
In particular, we show that for $\gamma \in (0,2)$, the random walk on the \emph{mated-CRT map} with parameter $\gamma$ converges to \textit{$\gamma$-Liouville Brownian motion}, the natural quantum time parametrization of Brownian motion on a $\gamma$-Liouville quantum gravity (LQG) surface.
Our result applies if the mated-CRT map is embedded into the plane via the embedding which comes from SLE / LQG theory or via the Tutte embedding (a.k.a.\ the harmonic or barycentric embedding). In both cases, the convergence is with respect to the local uniform topology on curves and it holds in the quenched sense, i.e., the conditional law of the walk given the map converges.

Previous work by Gwynne, Miller, and Sheffield (2017) showed that the random walk on the mated-CRT map converges to Brownian motion modulo time parametrization.
This is the first work to show the convergence of the parametrized walk.
As an intermediate result of independent interest, we derive an axiomatic characterisation of Liouville Brownian motion, for which the notion of  Revuz measure of a Markov process plays a crucial role.
\end{abstract}


\tableofcontents

\section{Introduction}
\label{sec-intro}

\subsection{Background}

The mathematical study of random walks on random planar maps goes back to the seminal paper of Benjamini and Schramm \cite{benjamini-schramm-topology}.
This paper was motivated by contemporary works by Ambj{\o}rn et al. \cite{Ambjorn_diffusion,Ambjorn_spectraldimension} which introduced and studied (in a nonrigorous way) a notion of diffusion in the geometry of Liouville quantum gravity. As we will discuss in more detail below (and partly make rigourous in this paper), this diffusion can heuristically be thought of as a continuum limit of random walk on random planar maps.
Benjamini and Schramm proved recurrence of \emph{any} local limit of a sequence of randomly rooted planar graphs, subject to a bounded degree assumption, and initiated a remarkably fruitful program of research aiming to describe properties of random walks on random planar maps, especially random walks on the so-called Uniform Infinite Planar Triangulation (UIPT) as well as other models of random planar maps in the same universality class.

In the years since then, this program of research has blossomed into a very rich and quickly expanding area of probability theory.
We will not attempt to give a complete review of this literature here. Instead, we mention a few highlights and refer to~\cite{curien-peeling-notes,ghs-mating-survey} as well as the excellent lecture notes by Nachmias \cite{Nachmias} for relevant expository articles.
The UIPT was constructed by Angel and Schramm in \cite{angel-schramm-uipt}. Gurel-Gurevich and Nachmias \cite{gn-recurrence} removed the bounded degree assumption of \cite{benjamini-schramm-topology} and replaced it by an exponential tail on the degree of the root vertex of the map (in particular applying to the UIPT). The random walk {on the UIPT} was proved to be subdiffusive by Benjamini and Curien in \cite{benjamini-curien-uipq-walk} through the consideration of the so-called pioneer points which were analyzed in great detail. More recently, the exact value of the diffusivity exponent on the UIPT (equal to $1/4$) was obtained through the combination of two works: a paper by Gwynne and Miller \cite{gm-spec-dim} which proves the lower bound for the diffusivity exponent and also computes the spectral dimension, and a paper by Gwynne and Hutchcroft \cite{gh-displacement} which proves the upper bound for the diffusivity exponent.

 Naturally, these developments cannot be dissociated from the remarkable progress in the understanding of the purely geometric features of these random maps, culminating for instance in the convergence of random planar maps (viewed as random metric spaces) to the Brownian map which was proven by Le Gall \cite{legall-uniqueness} and Miermont \cite{miermont-brownian-map}, see for instance the survey~\cite{legall-sphere-survey}.

On the continuum side, the paper by Duplantier and Sheffield \cite{shef-kpz} provided a framework (corresponding to the \emph{DDK ansatz}~\cite{david-conformal-gauge,dk-qg}) for building a continuum theory of Liouville quantum gravity (LQG) based on a rigorous construction of the volume form associated with the formal metric tensor
\eqb \label{eqn-lqg-tensor}
e^{\gamma h} \, (dx^2 + dy^2)
\eqe
where $h$ is a variant of the Gaussian free field on a planar domain, $dx^2 + dy^2$ is the Euclidean metric tensor on that domain, and $\gamma \in (0,2)$ is a constant (which is related to the type of random planar map being considered). In~\cite{shef-kpz}, the volume for associated with \eqref{eqn-lqg-tensor} is understood as a random measure (in fact, this random measure is a special case of the theory of \emph{Gaussian multiplicative chaos}, which was initiated in~\cite{kahane}; see~\cite{rhodes-vargas-log-kpz} and \cite{bp-lqg-notes} for additional context and a modern approach).
A rigorous construction of a continuum diffusion in the geometry associated with~\eqref{eqn-lqg-tensor} was proposed in \cite{grv-lbm, berestycki-lbm}, for every parameter $\gamma \in (0,2)$. This diffusion, called \textbf{Liouville Brownian motion} (LBM) and described in more details below, is widely conjectured to be the scaling limit of simple random walks on many models of random planar maps. Much is now known about the behaviour of LBM, including a rigorous verification in \cite{rhodes-vargas-spec-dim} of the predictions made in \cite{Ambjorn_spectraldimension} about the spectral dimension; as well as several works considerably refining our understanding of its heat kernel: see,  e.g.,~\cite{grv-heat-kernel,dzz-heat-kernel,andres-heat-kernel,jackson-lbm-thick-pts,mrvz-heat-kernel,grv-kpz}.
There are also known scaling limit results for discrete approximations of (analogues of) Liouville Brownian motion on tree-like spaces and low-dimensional fractals: see, e.g.~\cite{chk-time-changes}, which is in the spirit of earlier works on the Bouchaud trap model and its scaling limit to the so-called FIN diffusion in dimension one~\cite{FIN} and to fractional kinetics in higher dimensions~\cite{BC07}.

However, prior to the present paper there were no rigorous results relating the behaviour of simple random walk on natural models of random planar maps to Liouville Brownian motion. The main purpose of this paper is to establish the first such result, in the case of random walk on the so-called \textbf{mated-CRT random planar maps}. These random planar maps are in some sense more directly connected to Liouville quantum gravity than other random planar map models (thanks to the results of~\cite{wedges}) and will be described in more detail below.  Interestingly, mated-CRT maps also provide a coarse-grained approximation to many other models of random planar maps (see Remark~\ref{remark-bijections}). The main result of this paper is that for any parameter $\gamma \in (0,2)$, the scaling limit of simple random walk on the mated-CRT planar map with parameter $\gamma$ is Liouville Brownian motion with the same parameter. See Theorems~\ref{thm-lbm-conv} and~\ref{thm-lbm-conv_disk} for precise statements, including the topology of convergence. Our results build on the earlier work~\cite{gms-tutte}, which shows that the random walk on the mated-CRT map converges to Brownian motion modulo time parametrization.

\subsection{Setup}
\label{sec-setup}

\paragraph{Infinite-volume mated-CRT maps.}

Mated-CRT maps are a one-parameter family of random planar maps, indexed by $\gamma \in (0,2)$, which were first used implicitly in~\cite{wedges} and studied more explicitly in~\cite{ghs-dist-exponent,gms-tutte}.

We start with a brief description of infinite-volume mated-CRT maps, with the topology of the plane.
In this case, the basic data is provided by a pair of correlated real-valued two-sided Brownian motions $(L_t, R_t)_{t \in \BB{R}}$ such that $L_0 = R_0 = 0$ and with correlation coefficient given by $ - \cos ( \pi \gamma^2/ 4)$, i.e.,
\eqbn
\op{cov} (L_t, R_t) = - \cos \left( \frac{\pi \gamma^2}4 \right) |t|.
\eqen

The mated-CRT map is the map obtained by gluing discretized versions of the the (non compact) Continuum Random Trees defined by $L$ and $R$. More precisely, for a given $\eps  > 0$, the mated-CRT map with {scale} $\eps$ is the random graph $\mcl G^\ep$ whose vertex set is $\mcl{VG}^\ep = \eps \BB{Z} $ and where there is an edge between two vertices $x, y \in \eps \BB{Z} $ (with $x <y$) if and only if:
\begin{equation} \label{edgemap}
\left( \inf_{t \in [x - \eps,x]} X_t \right) \vee \left( \inf_{t \in [y - \eps, y]} X_t \right)\le \inf_{t \in [x, y - \eps]} X_t
\end{equation}
where $X$ can be either $L$ or $R$. Note that in this definition $x$ is always connected to $x+\eps$ via both $L$ and $R$, but we only include one such edge in this case (let us call such an edge \emph{trivial}).
If $y> x+\eps$ it is possible that \eqref{edgemap} is satisfied for {neither, one, or both of} $L$ and $R$: in the case when~\eqref{edgemap} is satisfied for both $L$ and $R$ there are two edges joining $x$ and $y$.
When $y>x+\eps$ and \eqref{edgemap} is satisfied we call the corresponding edge nontrivial.
A nontrivial edge can be of two types: type $L$ or type $R$ depending on whether \eqref{edgemap} is satisfied with $X = L$ or $X = R$.
{By Brownian scaling, it is clear that the law of $\mcl G^\ep$ viewed as a graph does not depend on $\ep$, but for reasons we will explain just below it is convenient to consider the whole family of graphs $\{\mcl G^\ep\}_{\ep > 0}$ constructed from the same pair $(L,R)$. }

Note that the condition \eqref{edgemap} says that there are times $s \in [x- \eps, x]$ and $t \in [y- \eps, y]$ such that $s$ and $t$ are identified {in the equivalence relation used to construct} the CRT associated with $L$ or $R$. The {definition of the mated-CRT map} can therefore indeed be thought of as a gluing of discretized versions of the CRT's associated to $L$ and $R$.

\begin{figure}[ht!]\begin{center}
\includegraphics[scale=.8]{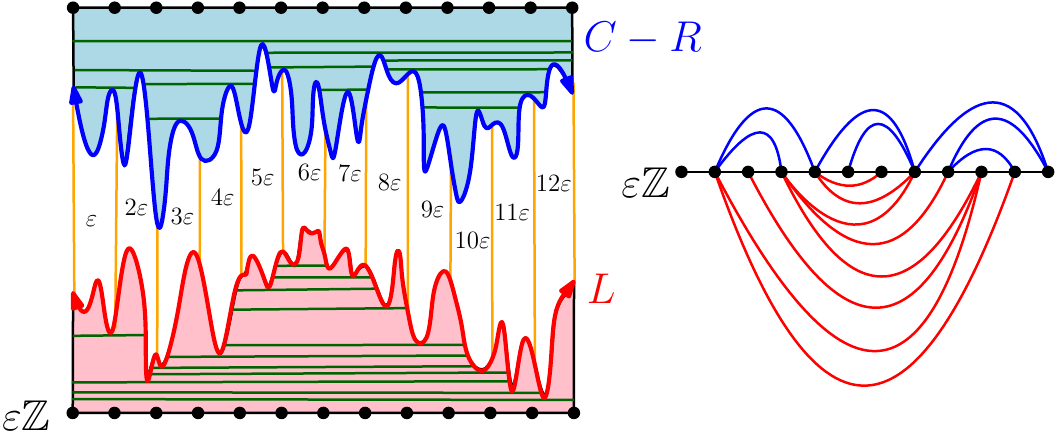}
\caption{{\textbf{Left:} A geometric description of the definition of the mated-CRT map. We consider the restrictions of $L$ and $R$ to some interval (in this case, $[0,12\ep]$). We then draw the graphs of the restrictions of $L$ and $C-R$ to this interval in the same rectangle in the plane, where $C$ is a large constant chosen so that the graphs do not intersect. The adjacency condition~\eqref{edgemap} for $L$ (resp.\ $R$) is equivalent to the condition that there is a horizontal line segment under the graph of $L$ (resp.\ above the graph of $C-R$) which intersects the graph only in the vertical strips $[x-\ep, x] \times [0,C]$ and $[y-\ep,y]\times [0,C]$. For each pair $(x,y)$ for which this adjacency condition holds, we have drawn the lowest (resp.\ highest) such horizontal line segment in green.
\textbf{Right:} Illustration of a proper planar embedding of the given portion of the mated-CRT map which realizes its planar map structure. Trivial edges (resp.\ $L$-edges, $R$-edges) are shown in black (resp.\ red, blue). A similar illustration appeared as~\cite[Figure 1]{gms-harmonic}. }
}
\label{F:map}
\end{center}
\end{figure}

In order to turn the graph $\mcl G^\ep$ into a (rooted) planar map we {first} specify the root vertex to be the trivial edge from $0$ to $\ep$. In a planar map there is also a notion of counterclockwise cyclical order on the edges surrounding a given vertex. For the mated-CRT map, this ordering is obtained as follows. We order the $L$-edges emanating from a given vertex $x \in \eps \BB{Z}$ by giving them the order they inherit from their endpoints on $\eps\BB{Z}$. We order the $R$-edges in the reverse way, and declare that the $L$-edges come before the $R$-edges. The $L$-edges and the $R$-edges are separated by the two trivial edges emanating from $x$. With this order, it is easy to see why a proper planar embedding of this graph exists (i.e., an embedding in which the edges do not overlap). Simply draw the $R$-edges on one side of $\eps\BB{Z}$ in the plane and the $L$-edges on the other, and observe that the $R$-edges (and equivalently the $L$-edges) can be drawn without overlapping: if $s_1<s_2$ and $t_1< t_2$ are identified in $R$, then either $s_1<s_2 < t_1< t_2$ or $s_1< t_1< t_2 < t_2$. See Figure \ref{F:map} for an illustration. {We note that with this planar map structure, the mated-CRT map is in fact a triangulation; see the caption of~\cite[Figure 1]{gms-harmonic} for an explanation. }

\begin{remark}[Connection to other random planar map models] \label{remark-bijections}
The above construction of the mated-CRT map is motivated by a class of combinatorial bijections called \emph{mating of trees bijections}. Basically, such bijections tell us that certain natural random planar maps decorated by statistical mechanics models can be constructed via discrete analogs of the above construction of the mated-CRT map, with the two coordinates of a random walk on $\BB Z^2$ (with an increment distribution depending on the map) used in place of $(L,R)$.
Examples of planar maps which can be encoded in this way include uniform triangulations decorated by site percolation configurations~\cite{bernardi-dfs-bijection,bhs-site-perc} as well as planar maps decorated by spanning trees~\cite{mullin-maps,bernardi-maps}, the critical Fortuin-Kasteleyn cluster model~\cite{shef-kpz,bernardi-maps}, or bipolar orientations~\cite{kmsw-bipolar}.
Due to the convergence of random walk to Brownian motion, the mated-CRT map can be viewed as a coarse-grained approximation of these other random planar maps.
This coarse-graining can sometimes be used to transfer results from the mated-CRT map to other random planar maps, as is done, e.g., in~\cite{ghs-map-dist,gm-spec-dim,gh-displacement}. Currently, however, the estimates comparing mated-CRT maps to other maps are not sufficiently precise to transfer scaling limit results like the one proven in this paper.
\end{remark}

\paragraph{Finite volume mated-CRT maps.} Mated-CRT planar maps can also be defined in finite volume with other topologies such as the disk.
In this version, $L$ and $R$ are two correlated Brownian motions over the time-interval $[0,1]$ (instead of the entire real line $\BB{R}$), start from $L_0 = R_0 = 0$ and
are conditioned so that $(L_t,R_t)_{0\le t \le 1}$ remains in the positive quadrant $[0,\infty)^2$ until time 1 and ends up in the position $(1,0)$ at time one. (This is a degenerate conditioning; see~\cite{sphere-constructions} for a construction, building on the cone excursions studied by Shimura in~\cite{shimura-cone}.) We call such a pair a Brownian excursion in the quadrant. The vertex set of the mated-CRT map with the disk topology is then taken to be $ \mcl{VG}^\eps = \eps \BB{Z} \cap [0,1]$ instead of $\eps \BB{Z} $, but apart from that the definitions above stay the same.

Note that the tree encoded by $R$ is a standard (compact) Continuum Random Tree, whereas the one encoded by $L$ also comes with a natural boundary. More precisely, we define boundary vertices $\partial \mcl G^\eps$ as the set of vertices $x \in \eps \BB{Z} \cap [0,1]$ such that
\begin{equation}\label{boundary}
\inf_{t \in [x- \eps, x]} L_t \le \inf_{t \in [x, 1]} L_t.
\end{equation}
See Figure~\ref{fig-mated-crt-map-disk} for an illustration.

\begin{figure}[ht!]\begin{center}
\includegraphics[scale=.8]{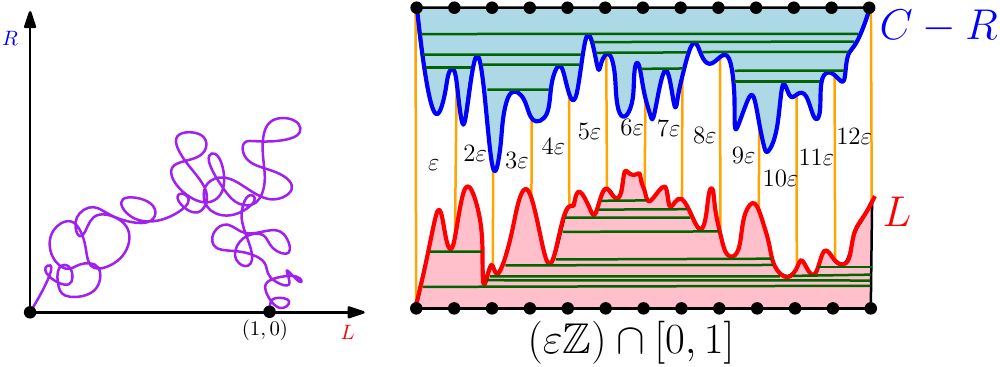}
\caption{{\textbf{Left:} The conditioned correlated Brownian motion $(L,R)$ used to construct the mated-CRT map with the disk topology.
\textbf{Right:} The analog of Figure~\ref{F:map}, left, for the disk topology, with $\ep = 1/12$. We have $x \in \bdy\mcl G^\ep$ if and only if there is a horizontal line segment under the graph of $L$ which intersects the graph of $L$ only in $[x-\ep , x] \cap [0,C]$ and also intersects $\{1\} \times [0,C]$. Here the boundary vertices are $\ep , 2\ep , 10\ep, 11\ep,$ and $12\ep$.}  }
\label{fig-mated-crt-map-disk}
\end{center}
\end{figure}

\paragraph{Tutte embedding.} The advantage of working in the finite volume setup discussed above (as opposed to say the earlier whole plane setup) is that the Tutte embedding can be defined in a straightforward way. The Tutte embedding of a graph is a planar embedding which has the property that each interior vertex (i.e., a vertex not on the boundary) is equal to the average of its neighbours. Equivalently, the simple random walk on the graph is a martingale. A concrete construction in the case of a finite volume mated-CRT map is as follows. {We first choose a marked root vertex of the mated-CRT map by sampling $\BB t$ uniformly from $[0,1]$ and letting $\BB x^\ep \in (\ep\BB Z)\cap [0,1]$ be chosen so that $\BB t \in [\BB x^\ep - \ep , \BB x^\ep]$.}
Let $x_1 < \ldots < x_k$ denote the vertices of the boundary $\partial \mcl{G}^\eps$, in numerical order.
 Let $\mathfrak{p}(x_k)$ denote the probability that simple random walk on $\mcl{G}^\eps$, {started from $\BB x^\ep$}, first hits $\partial \mcl{G}^\eps$ {in the arc $\{x_1,\dots,x_k\}$.} Then the boundary vertices $x_1,\ldots, x_k$ are mapped (counterclockwise) respectively to $z_1 = e^{2i \pi \mathfrak{p}(x_1)}, \ldots, z_k = e^{2i \pi \mathfrak{p} (x_k)}$. {This makes it so that the harmonic measure from $\BB x^\ep$ approximates Lebesgue measure on $\bdy\BB D$.} As for the interior vertices, if $x \in \mcl{VG}^\eps$, let $\psi^\ep : \mcl{VG}^\eps \to \BB{C}$ denote the unique function which is discrete harmonic on $\mcl{VG}^\eps \setminus \partial \mcl{G}^\eps$ and whose boundary values are given by $\psi^\ep(x_i) = z_i$. This gives an embedding of the vertices of the graph (into the unit disk $\BB{D}$, by an argument similar to the maximum principle). A theorem of Tutte \cite{tutte-embedding} then guarantees that if the edges of the graph are drawn as straight lines then the edges can overlap but not cross.

\paragraph{SLE/LQG embedding.} One of the reasons the mated-CRT maps (either in finite or infinite volume) are convenient is that, due to the main result of~\cite{wedges}, they admit an elegant alternative description given by a certain type of $\gamma$-LQG surface (represented by a random distribution $h$ in a portion $D \subset \BB{C}$ of the complex plane), decorated by an independent space-filling version of a Schramm-Loewner evolution (SLE) curve $\eta$ with parameter $\kappa = 16/\gamma^2 \in (4, \infty)$. The notion of LQG surfaces will be described in greater details in Section \ref{sec-sle-lqg}, but let us nevertheless describe the embedding as succinctly as possibly, deferring the definitions to that section. {See Figure~\ref{fig-sg-def-disk} for an illustration. }

\begin{figure}[ht!]
\begin{center}
\includegraphics[scale=.6]{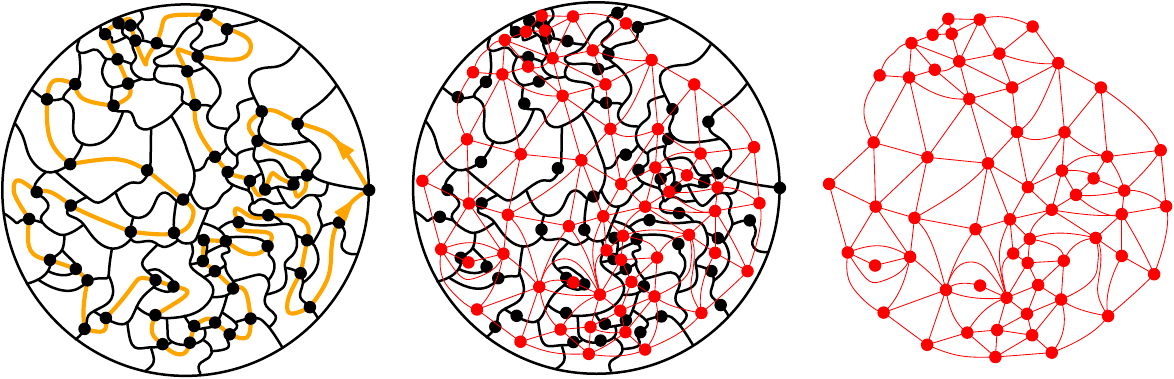}
\caption{\label{fig-sg-def-disk} {Illustration of the LQG / SLE embeddding of the mated-CRT map in the disk case (a similar figure appears in~\cite{gms-tutte}). \textbf{Left:} A segment of a space-filling curve $\eta : [0,1] \rta \ol{\BB D}$, divided into cells $  \eta([x-\ep ,x])$ for $x\in (\ep \BB Z) \cap (0,1]$. The order in which the cells are hit is shown by the orange path. This figure looks like what we would expect to see for $\kappa \geq 8$ ($\gamma \leq \sqrt 2$), since the cells are simply connected. \textbf{Middle:} To get an embedding of the mated-CRT map, we draw a red point in each cell and connect the points whose cells share a corresponding boundary arc. \textbf{Right:} Same as the middle picture but without the original cells, so that only the embedded mated-CRT map is visible. }
}
\end{center}
\end{figure}

We start with the infinite volume case, which is slightly easier to explain. In that case, $D=\BB C$ and $h$ is the random distribution corresponding to a so-called \textbf{$\gamma$-quantum cone}.
Since the law of $h$ is locally absolutely continuous with respect to {that of a} Gaussian free field one can define a volume measure $\mu_h$ which is a version of the LQG measure (or Gaussian multiplicative chaos) associated to $h$.
{One way to define this measure is as follows.
For $z\in\BB C$ and $\ep > 0$, let $h_\ep(z)$ be the average of $h$ over the circle of radius $\ep$ centered at $z$.}
Then for any open set $A\subset \BB C$,
\begin{equation}\label{E:gmc}
\mu_h(A) = \lim_{\eps \to 0} \int_A \eps^{\gamma^2/2} e^{\gamma h_\eps (z)} dz
\end{equation}
in probability (see, e.g., \cite{shef-kpz,berestycki-gmt-elementary}). In fact the convergence was shown to be a.s.\ in~\cite{shef-wang-lqg-coord}.

The curve $\eta$ mentioned above is in the whole-plane case a \textbf{whole-plane space-filling SLE$_\kappa$} from $\infty$ to $\infty$, sampled independently from $h$. When $\kappa \ge 8$, i.e. when $\gamma \le \sqrt{2}$, whole-plane SLE$_{\kappa}$ is already space-filling and $\eta$ is then nothing but an ordinary whole-plane SLE {from $\infty$ to $\infty$}. However, when $\kappa \in (4,8)$ ordinary SLE is not space-filling and the construction is more complicated, see~\cite[Section 1.2.3]{ig4} and Section~\ref{sec-wpsf-prelim}. This choice of $(h, \eta)$ corresponds to the setting of the {whole-plane} \textbf{mating of trees theorem}~\cite[Theorem 1.9]{wedges}.

Let us now describe the disk case, which is slightly more involved. In this case, $D=\BB D$ is the unit disk and $h$ is the random distribution corresponding to a type of quantum surface known as a \textbf{quantum disk} with unit area and unit boundary length. Such a surface does not normally come with a specified marked point either on the boundary or in the bulk, but we add a marked point on the boundary by sampling from the boundary length measure (the appropriate analogue of \eqref{E:gmc} restricted to the boundary).

The curve $\eta$ is a \textbf{space-filling SLE$_\kappa$ loop} from the marked boundary point to itself (with loops being filled in clockwise). Such a curve is obtained as the limit of a chordal space-filling SLE$_\kappa$ between two boundary points $x$ and $y$, as $y \to x$; see Section~\ref{sec-wpsf-prelim} for more details.
This setup corresponds essentially to the (finite volume) mating of trees theorem proved by Ang and Gwynne~\cite[Theorem 1.1]{ag-disk}, with two unessential differences: one is that the loops of $\eta$ are filled counterclockwise in~\cite{ag-disk} (corresponding to the Brownian pair $(L,R)$ ending on the $R$ axis instead of the $L$ axis as is the case here). The second is
that the paper \cite{ag-disk} describes the situation corresponding to unit boundary length (but random area). The corresponding mating of trees theorem for the case we need here is naturally obtained by conditioning on the total area, whose law is described in~\cite[Theorem 1.2]{ag-disk}.

In both setups, since the curve $\eta$ is space-filling, it can be reparametrized by $\mu_h$-area: in any time interval of length $t$, the curve covers an area of mass $t$. The \textbf{LQG embedding of mated-CRT planar maps} is then as follows. Each vertex $x \in \mcl{VG}^\eps$ {corresponds to} a \emph{cell} $\eta([x- \eps, x])$ with appropriate modifications for the last cell if $\ep$ is not of the form $1/n$ for some $n \ge 1$. Two cells are declared \emph{adjacent} if their intersection is nonempty and contains a nontrivial curve\footnote{When $\kappa \ge 8$, this second condition is in fact not necessary.
On the other hand, when $\kappa \in (4,8)$, an SLE$_\kappa$ curve is self-touching, resulting in pinch points for the cells formed by the space-filling curve.
On either side of such a pinch point could be two different cells, but our definition ensures that these are nevertheless not declared adjacent, even though they intersect at the pinch point.}. It is not hard to see via the mating of trees theorems mentioned earlier (Theorem 1.9 in \cite{wedges} and Theorem 1.1 in \cite{ag-disk}) that this gives an alternative and equivalent construction of the mated-CRT planar maps in the respective setups. Concretely, this gives an \emph{a priori} embedding of the mated-CRT planar map into the domain $D$ (either the whole plane or the disk) obtained by sending each vertex $x$ to a point of the corresponding cell. 

This embedding is extremely useful to carry out concrete computations on the mated-CRT planar maps. Furthermore, it is closely related to the Tutte embedding: indeed, one of the main results of \cite{gms-tutte} implies that the Tutte embedding converges uniformly to the SLE embedding. More precisely, if $  \psi^\ep$ denotes the Tutte embedding of the mated-CRT map with the disk topology, as above, then
\begin{equation}\label{E:TutteLQG}
\max_{x \in \mcl{VG}^\eps} | \psi^\eps(x) - \eta(x)| \to 0
\end{equation}
in {probability}; see (3.3) in Section 3.4 of \cite{gms-tutte}.

\paragraph{Convergence of the random walk modulo time parametrization.}
Most of the work in~\cite{gms-tutte} goes into proving that that the random walk on the mated-CRT map converges to Brownian motion modulo time parametrization under the LQG embedding discussed above. 
More precisely, it is shown in~\cite[Theorem 3.4]{gms-tutte} that for each compact set $K\subset \BB C$, the Prokhorov distance between the following two probability measures converges to zero in probability as $\ep\rta 0$, uniformly over all $z\in K$:
\begin{itemize}
\item The law of standard planar Brownian motion started from $z$, viewed modulo time parametrization;
\item The conditional law given $(h,\eta)$ of the random walk on the mated-CRT map under the LQG embedding with cell size $\ep$, started from the (a.s.\ unique) vertex $x\in\mcl V\mcl G^\ep$ such that $z \in \eta([x-\ep,x])$, extended to be a continuous curve by piecewise linear interpolation, and viewed modulo time parametrization. 
\end{itemize}
See~\cite[Section 2.1.2]{gms-tutte} for a discussion of the topology on curves modulo time parametrization. The proof of the above convergence result is based on a general scaling limit result for random walks in planar random environments which are in some sense ``translation invariant modulo scaling"~\cite[Theorem 3.10]{gms-random-walk}. 

Due to the relationship between random walk and discrete harmonic functions, the convergence of random walk on the mated-CRT map to Brownian motion implies that, under the LQG embedding, discrete harmonic functions on the mated-CRT map approximate their continuum counterparts when $\ep$ is small. This, in turn, allows one to show that the Tutte embedding is close to the LQG embedding, in the sense of~\eqref{E:TutteLQG}. 

This paper will build on the result of~\cite{gms-tutte} to get convergence of the \emph{parametrized} walk on the mated-CRT map. Let us now discuss the limiting object.
 
\paragraph{Liouville Brownian motion.} Let $h$ be {either field considered} in the preceding section, {i.e., a $\gamma$-quantum cone with circle average embedding or a unit quantum disk.} Along with an area measure $\mu_h$ on $D$, it is possible to associate to the field $h$ a diffusion $X$, called Liouville Brownian motion. We give a few more details now. Let $z \in D$, let $B^z$ denote a standard planar Brownian motion in $D$ started from $z$ and let $\tau$ denote its exit time from $D$. By definition\footnote{In fact, the works \cite{berestycki-lbm,grv-lbm} focus on slightly different versions of the field $h$ where instead of a $\gamma$-quantum cone or quantum disk, $h$ is a Dirichlet GFF. Extending this definition to the case of a $\gamma$-quantum cone or to a quantum disk is relatively straightforward: in particular, in the case of the quantum cone one needs to check that the process does not stay stuck at zero (due to the $\gamma$-log singularity). This is easily confirmed since the expected time for the LBM started from 0 to reach the boundary of the unit disk, say, given the quantum cone $h$, is given by $\int_{\BB D} \op{Gr}_{\BB D}(0,y)\mu_h (dy)< \infty$ a.s. (where $\op{Gr}_{\BB D}$ is the Green's function on the disk). The finiteness of this last integral can be checked by splitting the disk into dyadic annuli and considering the contribution of each annulus.} (see, e.g., \cite{berestycki-lbm,grv-lbm}), Liouville Brownian motion is defined as the time-change $B^z_{\phi^{-1}(t)}$, where the Liouville clock $\phi$  satisfies
\eqbn
\phi(t) = \lim_{\eps \to 0}\int_0^{t \wedge \tau} \eps^{\gamma^2/2} e^{\gamma h_\eps (B^z_s)}ds.
\eqen
In fact, for technical reasons our results will require applying a fixed global rescaling of time, chosen so that the median of the time needed by Liouville Brownian motion on a $\gamma$-quantum cone to leave a fixed Euclidean ball, say the ball $B_{1/2}$ of radius $1/2$ centered at zero, is one. In other words, if $h$ is the circle average embedding of a $\gamma$-quantum cone,
\eqbn
\tau_{1/2} = \inf \{ t>0: |B_{\phi^{-1}(t)}| \ge 1/2 \},
\eqen
and if $m_0$ is its (annealed) median, then our scaled Liouville Brownian motion is by definition
\begin{equation}\label{D:LBMmed}
X^z_t = B^z_{\phi^{-1} (tm_0)}.
\end{equation}
Note that $m_0$ is defined w.r.t.\ the circle average embedding of a $\gamma$-quantum cone, regardless of what field $h$ we are using to define the Liouville Brownian motion.

\subsection{Main results}
\label{sec-main-results}

Our main result will come in two versions, corresponding respectively to the whole plane and disk setups. We start with the whole-plane case.
In this case, we work only with the SLE/LQG embedding since the Tutte embedding is harder to define if we do not have a boundary.
 Let $(h, \eta)$ denote the $\gamma$-quantum cone decorated by an independent space-filling SLE describing the SLE/LQG embedding of the whole plane mated-CRT planar maps  $(\mcl G^\ep)_{\ep>0}$ with parameter $\gamma$, as described above. {We assume that $h$ is the circle-average embedding of the $\gamma$-quantum cone.}

For $z\in \BB C$ and $\ep > 0$, let $X^{z,\ep} : \BB N_0 \rta \mcl V\mcl G^\ep$ be the simple random walk on $\mcl G^\ep$ started from $x$, where $x\in\mcl V\mcl G^\ep = \ep\BB Z$ is chosen so that $z$ belongs to the cell $\eta([x-\ep, x])$ {(there is a.s.\ a unique such $x$ for each fixed $z\in\BB C$; for the atypical points for which there are multiple possibilities, we arbitrarily choose the smallest possible value of $x$)}.
We extend the domain of definition of the embedded walk $\BB N_0 \ni j \mapsto \eta(X^{z,\ep}_j)$ from $\BB N_0$ to $[0,\infty)$ by piecewise linear interpolation.
Also let
\eqb \label{eqn-median-def}
m_\ep := \left( \text{median exit time of $\eta(X^{0,\ep})$ from $B_{1/2}$}\right)
\eqe
where $B_{1/2}$ is the Euclidean ball of radius $1/2$ centered at 0.
Note that this is an \emph{annealed} median, i.e., we are not conditioning on $(h,\eta)$ and so $m_\ep$ is deterministic.
In this whole plane setup, our main result is the following theorem.

\begin{thm} \label{thm-lbm-conv}
For each $z\in\BB C$, the conditional law of the embedded, linearly interpolated walk $(\eta(X^{z,\ep}_{m_\ep t}))_{t\geq 0}$ given $(L_t,R_t)_{t \in \BB R}$ (equivalently, given $(h,\eta)$) converges in probability to the rescaled law of $\gamma$-Liouville Brownian motion started from $z$ associated with\ $h$, defined in \eqref{D:LBMmed}, with respect to the Prokhorov topology induced by the local uniform metric on curves $[0,\infty)\rta \BB C$.
In fact, the convergence occurs uniformly over all points $z$ in any compact subset of $\BB C$.
\end{thm}

\begin{remark} \label{remark-scaling-constants}
Our proof shows that there is a constant $C>1$ depending only on $\gamma$ such that the time scaling constants $m_\ep$ of~\eqref{eqn-median-def} satisfy
\eqb \label{eqn-scaling-constants}
C^{-1} \ep^{-1} \leq m_\ep \leq C \ep^{-1} ,\quad \text{for all sufficiently small $\ep > 0$.}
\eqe
We do not know that $\ep m_\ep$ converges, so we do not get convergence in Theorems~\ref{thm-lbm-conv} and \ref{thm-lbm-conv_disk} when we scale time by $1/\ep$ instead of by $ m_\ep$.
See, however, Section~\ref{sec-outline}.
\end{remark}

We now give our theorem statement in the disk case.
The theorem is similar to the whole-plane case except that we state it for the {more intrinsic} Tutte embedding (rather than the a priori LQG embedding, although the result is also valid for this latter embedding).
Let $(h, \eta)$ denote the quantum disk decorated by an independent space-filling SLE describing the SLE/LQG embedding of the mated-CRT maps $(\mcl G^\ep)_{\ep>0}$ with the disk topology and parameter $\gamma$, as described above.
We assume that the marked point for $h$ (equivalently, the starting point for $\eta$) is 1.
Let $\psi^\ep :\mcl{VG}^\ep \to \BB{C}$ be the Tutte embedding of $\mcl G^\ep$, as defined in Section~\ref{sec-setup}.
For $z \in \BB{D}$ and $\ep > 0$, let $X^{z,\ep} : \BB N_0 \rta \mcl V\mcl G^\ep$ be the simple random walk on $\mcl G^\ep$ started from $x$, where $x = x(z,\ep)\in\mcl V\mcl G^\ep = \ep\BB Z \cap [0,1]$ is chosen to be a vertex such that $z$ is closest to $\psi^\ep(x)$. In case of ties, pick $x$ arbitrarily among the possible choices.
We extend the domain of definition of the embedded walk $\BB N_0 \ni j \mapsto \psi(X^{z,\ep}_j)$ from $\BB N_0$ to $[0,\infty)$ by piecewise linear interpolation.
{Let $m_\ep$ be as in~\eqref{eqn-median-def} (note that we define $m_\ep$ in terms of the whole-plane mated-CRT map in both the whole-plane and disk settings). }

\begin{thm} \label{thm-lbm-conv_disk}
For each $z\in\BB D$, consider the conditional law of the embedded, linearly interpolated walk $(\psi^\ep(X^{z,\ep}_{m_\ep t}))_{t\geq 0}$, given $(L_t,R_t)_{t \in [0,1]}$ (equivalently, given $(h,\eta)$). As $\ep\rta 0$, these laws converge in probability to the rescaled law of $\gamma$-Liouville Brownian motion started from $z$ {associated with }\ $h$, defined in \eqref{D:LBMmed}, stopped upon leaving the unit disk $\BB{D}$, with respect to the Prokhorov topology induced by the local uniform metric on curves $[0,\infty)\rta \BB C$.
In fact, the convergence occurs uniformly over all points $z$ in any compact subset of $\ol{\BB D}\setminus \{1\}$.
\end{thm}

Due to~\eqref{E:TutteLQG}, Theorem~\ref{thm-lbm-conv_disk} is equivalent to the analogous statement with the SLE/LQG embedding in place of the Tutte embedding.
Note that we only claim uniform convergence on compact subsets of $\ol{\BB D}\setminus \{1\}$ in Theorem~\ref{thm-lbm-conv_disk}, not uniform convergence on all of $\ol{\BB D}$. We know that the random walk on $\mcl G^\ep$ converges to Brownian motion \emph{modulo time parametrization} uniformly over all starting points in $\ol{\BB D}$~\cite[Theorem 3.4]{gms-tutte} (in the quenched sense). We expect that the result of Theorem~\ref{thm-lbm-conv_disk} also holds uniformly over all $z\in\ol{\BB D}$, but in order to prove this one would need some arguments to prevent the walk started near the marked point $1$ from getting ``stuck" and staying in a small neighborhood of $1$ for an unusually long time without hitting $\bdy\BB D$.

As we explain below, our proof of Theorem~\ref{thm-lbm-conv} relies on the one hand on proving a tightness result for the conditional law of $\{\eta(X^{z,\ep}_{m_\ep t})\}_{t\geq 0}$ given $(h,\eta)$, and on the other hand on a characterization statement for the subsequential limits, which is stated below as Theorem \ref{T:LBMchar_intro}.  
Our proof of convergence to Liouville Brownian motion is restricted to mated-CRT maps and does not apply in full to other natural models of random planar maps such as the UIPT.
Nevertheless, several aspects of our proof might be useful in this more general setting. 
In particular, the following characterization of Liouville Brownian motion is of independent interest and may be useful for identifying (subsequential) limits of random walk on other models of conformally embedded random planar maps. 
We state this result somewhat informally, deferring the actual necessary definitions and full statement to Section \ref{sec-characterization} and in particular Proposition \ref{prop-bm-unique}.

\begin{thm}
  \label{T:LBMchar_intro}
Let $h$ be the random distribution associated with the circle average embedding of a $\gamma$-quantum cone.
Suppose that we are given a coupling of $h$ and a random continuous function $P  = \{P_z : z\in\BB C\}$ which takes each $z\in\BB C$ to a (random) element $P_z$ of the space $\op{Prob}(C([0,\infty),\BB C))$ of probability measures on random continuous paths in $\BB C$ started from $z$.
Suppose that the following conditions hold:
\begin{itemize}
\item Conditional on $(h,P)$, a.s.\ for each $z\in\BB C$ the process with the law $P_z$ is Markovian;
\item For each $z \in \BB C$, $P_z$ has the law of a (random) time-change of a standard Brownian motion starting from $z$;
\item $P$ leaves the Liouville measure $\mu = \mu_h$ associated to $h$ invariant.
\end{itemize}
Then there is a (possibly random) constant $c$ such that for each $z \in\BB C$, $P_z$ coincides with the law of $(X^z_{ct}, t \ge 0)$, where $X^z$ is a Liouville Brownian motion {associated with $h$}, starting from $z$.
\end{thm}

\subsection{Outline}
\label{sec-outline}

Most of the paper is devoted to the proof of Theorem~\ref{thm-lbm-conv}. See Section~\ref{sec-plane-to-disk} for an explanation of how one deduces Theorem~\ref{thm-lbm-conv_disk} from Theorem~\ref{thm-lbm-conv} using~\eqref{E:TutteLQG} and a local absolute continuity argument.
It is shown in~\cite[Theorem 3.4]{gms-tutte} that, in the setting of Theorem~\ref{thm-lbm-conv}, the conditional law of $\{\eta(X^{z,\ep}_{m_\ep t} )\}_{t\geq 0}$ given $(h,\eta)$ converges in probability to the law of Brownian motion started from $z$ with respect to the Prokhorov topology induced by the metric on curves in $\BB C$ modulo time parametrization. It is also shown there that the convergence is uniform over $z$ in any compact subset of $\BB C$.
We need to show that the walk in fact converges uniformly to $\gamma$-LBM.

Most of the work in the proof goes into proving an appropriate tightness result, which says that the conditional law of $\{\eta(X^{z,\ep}_{m_\ep t})\}_{t\geq 0}$ given $(h,\eta)$ admits subsequential limits as $\ep\rta 0$ (Proposition~\ref{prop-tight}).
This is done in Sections~\ref{sec-harnack} through~\ref{sec-tight} and requires us to establish many new estimates for the random walk on $\mcl G^\ep$, building on the results of~\cite{gms-tutte,gms-harmonic}.
The identification of the subsequential limit, {summarized in Theorem \ref{T:LBMchar_intro}}, is carried out in Section~\ref{sec-characterization}. {A crucial role is played in particular by the notion of Revuz measure associated to a Positive Additive Continuous Functional (PCAF) for general Markov processes.}

Before outlining the proof, we make some general comments about our arguments.
\begin{itemize}
\item Only a few main results from each section (usually stated at the beginning of the section) are used in subsequent sections. So, the different sections can to a great extent be read independently of each other.
\item Our proofs give quantitative estimates for random walk on the mated-CRT map which are stronger than what is strictly necessary to prove convergence to Liouville Brownian motion. Such estimates include up-to-constants bounds for exit times from Euclidean balls (Propositions~\ref{prop-exit-lower} and~\ref{prop-exit-moment}), bounds for the Green's function (Lemmas~\ref{lem-gr-bound-lower}, \ref{lem-gr-bound}, and \ref{lem-eff-res-diag}), and modulus of continuity estimates for random walk paths (Proposition~\ref{prop-walk-cont}) and for the law of the walk as a function of its starting point (Proposition~\ref{prop-walk-law-cont}).
\item For most of the proof of Theorem~\ref{thm-lbm-conv}, we will scale time by $1/\ep$ instead of by $m_\ep$, i.e., we will work with $(\eta(X^{z,\ep}_{t/\ep }))_{t\geq 0}$.
We switch to scaling time by $m_\ep$ at the very end of the proof, in Section~\ref{sec-lbm-conv}, for reasons which will be explained at the end of this outline.
\end{itemize}

In order to establish tightness of  $(\eta(X^{z,\ep}_{t/\ep }))_{t\geq 0}$, we need up-to-constants estimates for the conditional law given $(h,\eta)$ of the exit time of the (embedded) random walk on $\mcl G^\ep$ from a Euclidean ball.
Moreover, these estimates need to hold uniformly over all Euclidean balls contained in any given compact subset of $\BB C$.
Moments of exit times from Euclidean balls can be expressed as sums involving the discrete Green's function for random walk on $\mcl G^\ep$ stopped upon exiting the ball.
Hence we need to prove up-to-constants estimates for the discrete Green's function.

The results of~\cite{gms-harmonic} allow us to bound the discrete Dirichlet energies of discrete harmonic functions on $\mcl G^\ep$, which leads to bounds for effective resistances (equivalently, for the values of the discrete Green's function on the diagonal).
The main task in the proof of tightness is to transfer from these estimates to bounds for the behavior of the Green's function \emph{off} the diagonal.
\textbf{Section~\ref{sec-harnack}} contains the two key estimates which allow us to do this.
\begin{itemize}
\item Lemma~\ref{lem-eff-res-green} says that for \emph{any} finite graph $G$ and any vertex sets $A\subset B\subset\mcl V(G)$, the following is true.
The maximum (resp.\ minimum) value on the boundary $\bdy A$ of the Green's function for random walk killed upon exiting $B$ is bounded below (resp.\ above) by the effective resistance from $A$ to $B$.
\item Proposition~\ref{prop-harnack} is a Harnack-type inequality which says that the maximum and minimum values of the discrete Green's function on $\mcl G^\ep$ on the boundary of a Euclidean ball differ by at most a constant factor.
\end{itemize}
Combined, these two estimates reduce the problem of estimating the Green's function to the problem of estimating effective resistances, which we can do (with a non-trivial amount of work) using tools from~\cite{gms-harmonic}.

In \textbf{Sections~\ref{sec-exit-lower} and~\ref{sec-exit-upper}}, we use the ideas discussed above to establish a lower (resp.\ upper) bound for the Green's function, which in turn leads lower (resp.\ upper) bounds for the exit times of the random walk on $\mcl G^\ep$ from Euclidean balls.
In \textbf{Section~\ref{sec-tight}}, we use these estimates to establish the tightness of the conditional law of $\{\eta(X^{z,\ep}_{t/\ep})\}_{t\geq 0}$ given $(h,\eta)$.
The arguments involved in these steps are non-trivial, but we do not outline them here; see the beginnings of the individual sections and subsections for outlines.

In \textbf{Section~\ref{sec-characterization}}, we show that the subsequential limit must be LBM as follows.
Since we already know the convergence of the random walk on $\mcl G^\ep$ to Brownian motion modulo time parametrization, every subsequential limit of the conditional law of $\{\eta(X^{z,\ep}_{t/\ep})\}_{t\geq 0}$ given $(h,\eta)$ is a probability measure on curves $\wh X^z$ in $\BB C$ which are time-changed Brownian motions, with a continuous time change function.
Our estimates show that the time change function is in fact strictly increasing.
Furthermore, it can be checked using the Markov property for random walk on $\mcl G^\ep$ that $\wh X^z$ is Markovian (Lemma~\ref{lem-ssl-markov}).
Finally, since the counting measure on vertices of $\mcl G^\ep$ weighted by their degree is reversible for the random walk on $\mcl G^\ep$, it is easily seen that the $\gamma$-LQG measure $\mu_h$ is invariant (in fact, reversible) for $\wh X^z$.

The properties of the subsequential limit process described in the preceding paragraph are all also known to hold for Liouville Brownian motion~\cite{grv-lbm,grv-heat-kernel,berestycki-lbm}.
It turns out that these properties are in fact sufficient to uniquely characterize LBM up to a time change of the form $t\mapsto c t$ for a deterministic $c>0$ which does not depend on the starting point.
This is a consequence of a general proposition (Proposition~\ref{prop-bm-unique}, which is a generalization of Theorem \ref{T:LBMchar_intro}), which says that two Markovian time changed Brownian motions with a common invariant measure agree in law modulo such a time change. In our setting, the two Markovian time changed Brownian motions are the subsequential limit process and the LBM and the common invariant measure is the LQG measure.

As explained in Section~\ref{sec-bm-unique}, Proposition~\ref{prop-bm-unique} is a {(in our view surprisingly simple)} consequence of known results in general Markov process theory.
In particular, the time change function for a time changed Brownian motion is a positive continuous additive functionals (PCAF) of the Brownian motion with the standard parametrization. The Revuz measure  associated with the PCAF is an invariant measure for the time changed Brownian motion, and is in fact the unique such invariant measure up to multiplication by a deterministic constant.
Since the Revuz measure uniquely determines the PCAF, this shows that two different time changed Brownian motions with the same invariant measure must agree in law modulo a linear time change, as required.

The above argument shows that for any sequence of positive $\ep$'s tending to zero, there is a subsequence $\mcl E$ and a deterministic\footnote{We are glossing over a technicality here --- as explained in Section~\ref{sec-lbm-conv}, some argument is needed to show that the constant $c$ is deterministic rather than just a random variable which is determined by $(h,\eta)$ and the limiting law on random paths.} constant $c> 0$ \emph{which may depend on $\mcl E$} along which $\{\eta(X^{z,\ep}_{t/\ep})\}_{t\geq 0}$ converges in law to LBM started from $z$ pre-composed with the linear time change $t\mapsto c t$ (in the quenched sense).
This does not yet give Theorem~\ref{thm-lbm-conv} since $c$ can depend on the subsequence.
To get around this, we need to scale time by $m_\ep$ instead of by $1/\ep$.
Indeed, by the definition~\eqref{eqn-median-def}, the median exit time of the process  $\{\eta(X^{z,\ep}_{m_\ep t})\}_{t\geq 0}$ from $B_{1/2}$ is equal to 1, so all of the possible subsequential limits of the laws of this process must be LBM with the \emph{same} linear time change.

\textbf{Appendix~\ref{sec-appendix}} contains some basic estimates for the $\gamma$-LQG measure and for space-filling SLE cells which are used in our proofs.
The proofs in this appendix are routine and do not use any of the other results in the paper, so are collected here to avoid distracting from the main ideas of the argument.

\subsubsection*{Acknowledgments}
We thank an anonymous referee for helpful comments on an earlier version of the paper. 
We thank Sebastian Andres, Zhen-Qing Chen, Takashi Kumagai, Jason Miller, and Scott Sheffield for helpful discussions.
N.B.'s work was partly supported by EPSRC grant EP/L018896/1 and FWF grant P 33083 on ``Scaling limits in random conformal geometry". E.G.\ was partially supported by a Trinity College junior research fellowship, a Herchel Smith fellowship, and a Clay Research Fellowship.
Part of this work was conducted during a visit by E.G.\ to the University of Vienna in March 2019 and a visit by N.B.\ to the University of Cambridge in May 2019.
We thank the two institutions for their hospitality.

\section{Preliminaries}
\label{sec-prelim}

\subsection{Basic notation}
\label{sec-basic}

\noindent
We write $\BB N = \{1,2,3,\dots\}$ and $\BB N_0 = \BB N \cup \{0\}$.
\medskip

\noindent
For $a < b$, we define the discrete interval $[a,b]_{\BB Z}:= [a,b]\cap\BB Z$.
\medskip

\noindent
For a graph $G$, we write $\mcl V(G)$ and $\mcl E(G)$ for its vertex and edge sets, respectively.
\medskip

\noindent
For $z\in\BB C$ and $r>0$, we write $B_r(z)$ for the Euclidean ball of radius $r$ centered at $z$.
We abbreviate $B_r = B_r(0)$.
\medskip

\noindent
If $f  :(0,\infty) \rta \BB R$ and $g : (0,\infty) \rta (0,\infty)$, we say that $f(\ep) = O_\ep(g(\ep))$ (resp.\ $f(\ep) = o_\ep(g(\ep))$) as $\ep\rta 0$ if $f(\ep)/g(\ep)$ remains bounded (resp.\ tends to zero) as $\ep\rta 0$. We similarly define $O(\cdot)$ and $o(\cdot)$ errors as a parameter goes to infinity.
\medskip

\noindent
If $f,g : (0,\infty) \rta [0,\infty)$, we say that $f(\ep) \preceq g(\ep)$ if there is a constant $C>0$ (independent from $\ep$ and possibly from other parameters of interest) such that $f(\ep) \leq  C g(\ep)$. We write $f(\ep) \asymp g(\ep)$ if $f(\ep) \preceq g(\ep)$ and $g(\ep) \preceq f(\ep)$.
\medskip

\noindent
Let $\{E^\ep\}_{\ep>0}$ be a one-parameter family of events. We say that $E^\ep$ occurs with
\begin{itemize}
\item \emph{polynomially high probability} as $\ep\rta 0$ if there is a $p > 0$ (independent from $\ep$ and possibly from other parameters of interest) such that  $\BB P[E^\ep] \geq 1 - O_\ep(\ep^p)$.
\item \emph{superpolynomially high probability} as $\ep\rta 0$ if $\BB P[E^\ep] \geq 1 - O_\ep(\ep^p)$ for every $p>0$.
\end{itemize}
We similarly define events which occur with polynomially, superpolynomially, and exponentially high probability as a parameter tends to $\infty$.
\medskip

\noindent
We will often specify any requirements on the dependencies on rates of convergence in $O(\cdot)$ and $o(\cdot)$ errors, implicit constants in $\preceq$, etc., in the statements of lemmas/propositions/theorems, in which case we implicitly require that errors, implicit constants, etc., appearing in the proof satisfy the same dependencies.

\subsection{Background on Liouville quantum gravity and SLE}
\label{sec-sle-lqg}

Throughout this paper we fix the LQG parameter $\gamma \in (0,2)$ and the corresponding SLE parameter $\kappa = 16/\gamma^2  > 4$.

\subsubsection{Liouville quantum gravity}
\label{sec-lqg-prelim}

We will give some relatively brief background on LQG. We will in particular focus on quantum cones (which are the main types of quantum surfaces considered in this paper) and also briefly discuss quantum disks.
The following definition is taken from~\cite{shef-kpz,shef-zipper,wedges}.

\begin{defn} \label{def-lqg-surface}
A \emph{$\gamma$-Liouville quantum gravity (LQG) surface} is an equivalence class of pairs $(D,h)$, where $D\subset \BB C$ is an open set and $h$ is a distribution on $D$ (which will always be taken to be a realization of a random distribution which locally looks like the Gaussian free field), with two such pairs $(D,h)$ and $(\wt D , \wt h)$ declared to be equivalent if there is a conformal map $f : \wt D \rta D$ such that
\eqb \label{eqn-lqg-coord}
\wt h = h\circ f + Q\log |f'| \quad \text{for} \quad Q = \frac{2}{\gamma}  + \frac{\gamma}{2} .
\eqe
More generally, for $k\in \BB N$ a \emph{$\gamma$-LQG surface with $k$ marked points} is an equivalence class of $k+2$-tuples $(D,h,x_1,\dots,x_k)$ where $x_1,\dots,x_k \in D\cup \bdy D$, with the equivalence relation defined as in~\eqref{eqn-lqg-coord} except that the map $f$ is required to map the marked points of one surface to the corresponding marked points of the other.
\end{defn}

We think of different elements of the same equivalence class in Definition~\ref{def-lqg-surface} as representing different parametrizations of the same surface.
If $(D,h,x_1,\dots,x_k)$ is a particular equivalence class representative of a quantum surface, we call $h$ an \emph{embedding} of the surface into $(D,x_1,\dots,x_k)$.

Suppose now that $h$ is a random distribution on $D$ which can be coupled with a GFF on $D$ in such a way that their difference is a.s.\ a continuous function.
Following~\cite[Section 3.1]{shef-kpz}, we can then define for each $z\in\BB C$ and $\ep > 0$ the \emph{circle average} $h_\ep(z)$ of $h$ over $\bdy B_\ep(z)$.
It is shown in~\cite[Proposition 3.1]{shef-kpz} that $(z,\ep) \mapsto h_\ep(z)$ a.s.\ admits a continuous modification.
We will always assume that this process has been replaced by such a modification.

One can define the \emph{$\gamma$-LQG area measure} $\mu_h$ on $D$, which is defined to be the a.s.\ limit
\eqbn
\mu_h = \lim_{\ep \rta 0} \ep^{\gamma^2/2} e^{\gamma h_\ep(z)} \, dz
\eqen
with respect to the local Prokhorov distance on $D$ as $\ep\rta 0$ along powers of 2~\cite{shef-kpz}.
One can similarly define a boundary length measure $\nu_h$ on certain curves in $D$, including $\bdy D$~\cite{shef-kpz} and SLE$_{\ul\kappa}$-type curves for $\ul\kappa =\gamma^2$ which are independent from $h$~\cite{shef-zipper}. If $h$ and $\wt h$ are related by a conformal map as in~\eqref{eqn-lqg-coord}, then $f_* \mu_{\wt h} = \mu_h$ and $f_* \nu_{\wt h} = \nu_h$.
Hence $\mu_h$ and $\nu_h$ are intrinsic to the LQG surface --- they do not depend on the choice of parametrization.
The measures $\mu_h$ and $\nu_h$ are a special case of a more general theory of regularized random measures called \emph{Gaussian multiplicative chaos} which originates in work of Kahane~\cite{kahane}. See~\cite{berestycki-gmt-elementary} for an elementary proof of convergence (and independence of the limit with respect to the regularization procedure) and \cite{rhodes-vargas-review} for a survey of this theory.

The main type of LQG surface which we will be interested in in this paper is the \emph{$\gamma$-quantum cone}, which is the surface appearing in the SLE/LQG embedding of the mated-CRT map.
The $\gamma$-quantum cone is a doubly marked LQG surface $(\BB C ,h , 0, \infty)$ introduced in~\cite[Definition~4.10]{wedges}.
Roughly speaking, the $\gamma$-quantum cone is the surface obtained by starting with a general $\gamma$-LQG surface, sampling a point $z$ uniformly from the $\gamma$-LQG area measure, then ``zooming in" near the marked point and re-scaling so that the $\mu_h$-mass of the unit disk remains of constant order~\cite[Proposition~4.13(ii) and Lemma~A.10]{wedges}.
This surface can be described explicitly as follows.

\begin{defn}[Quantum cone] \label{def-quantum-cone}
The \emph{$\gamma$-quantum cone} is the LQG surface $(\BB C , h , 0 , \infty)$ with the distribution $h$ defined as follows.
Let $B$ be a standard linear Brownian motion and let $\wh B$ be a standard linear Brownian motion conditioned so that $\wh B_t + (Q-\gamma) t > 0$ for all $t >0$ (here $Q = 2/\gamma + \gamma/2$, as in~\eqref{eqn-lqg-coord}).
Let $A_t = B_t - \alpha t$ for $t\geq 0$ and let $A_t = \wh B_{-t} + \gamma t$ for $t < 0$.
The circle average process of $h$ centered at the origin satisfies $h_{e^{-t}}(0) = A_t$ for each $t\in\BB R$.
The ``lateral part" $h - h_{|\cdot|}(0)$ is independent from $\{h_{e^{-t}}(0)\}_{t\in\BB R}$ and has the same law as $\wt h - \wt h_{|\cdot|}(0)$, where $\wt h$ is a whole-plane GFF.
\end{defn}

By Definition~\ref{def-lqg-surface}, one can get another embedding of the $\gamma$-quantum cone by replacing $h$ by $h(a\cdot) + Q\log |a|$ for any $a\in\BB C\setminus \{0\}$.
The particular embedding $h$ appearing in Definition~\ref{def-quantum-cone} is called the \emph{circle average embedding} and is characterized by the condition that $\sup\{r > 0 : h_r(0) +Q \log r = 0\} = 1$.
The circle average embedding is especially convenient to work with since for this embedding, $h|_{\BB D}$ agrees in law with the corresponding restriction of a whole-plane GFF plus $-\gamma \log |\cdot|$, normalized so that its circle average over $\bdy \BB D$ is 0.
Indeed, this is essentially immediate from Definition~\ref{def-quantum-cone} since the process $A_t$ has no conditioning for $t > 0$.

Since we can only compare $h$ to the whole-plane GFF  on the unit disk $\BB D$, for many of our estimates we will restrict attention to a Euclidean ball of the form $B_\rho$ for $\rho \in (0,1)$.
It is possible to transfer these estimates to larger Euclidean balls using the scale invariance property of the $\gamma$-quantum cone, which is proven in~\cite[Proposition 4.13(i)]{wedges}.

\begin{lem} \label{lem-cone-scale}
Let $h$ be the circle average embedding of a $\gamma$-quantum cone.
For $b >0$, define
\eqb \label{eqn-cone-scale-radius}
R_b := \sup\left\{ r > 0  : h_r(0) + Q\log r =  \frac{1}{\gamma} \log b \right\} .
\eqe
Then $h^b := h(R_b\cdot) + Q \log R_b -\frac{1}{\gamma} \log b \eqD h$.
\end{lem}

Lemma~\ref{lem-cone-scale} says that the law of $h$ is invariant under the operation of scaling areas by a constant (i.e., adding $\frac{1}{\gamma} \log b$ to $h$), then re-scaling space and applying the LQG coordinate change formula~\eqref{eqn-lqg-coord} so that the new field is embedded in the same way as $h$.

In addition to the $\gamma$-quantum wedge we will also have occasion to consider the quantum disk, which appears in Theorem~\ref{thm-lbm-conv_disk}.
To define the quantum disk, one first defines an infinite measure $\mcl M^{\op{disk}}$ on doubly quantum surfaces $(\BB D , h , -1,1)$ using the Bessel excursion measure.
We will not need the precise definition here, so we refer to~\cite[Section 4.5]{wedges} for details.
The infinite measure $\mcl M^{\op{disk}}$ assigns finite mass to quantum surfaces with boundary length $\nu_h(\bdy\BB D) \geq L$ for any $L > 0$.
By considering the regular conditional law of $\mcl M^{\op{disk}}$ given $\{\nu_h(\bdy \BB D) = L , \mu_h(\BB D ) = A\}$ for any $L , A  > 0$, one defines the \emph{doubly marked quantum disk with boundary length $L$ and area $A$}. We define the \emph{singly marked quantum disk} or the \emph{unmarked quantum disk} by forgetting one or both of the marked boundary points. By Proposition~\cite[Proposition A.8]{wedges}, the marked points of a quantum disk are uniform samples from the LQG boundary length measure. That is, if $(\BB D , h)$ is an unmarked quantum disk (with some specified area and boundary length) and conditional on $h$ we sample $x,y$ independently from the probability measure $\nu_h / L$, then $(\BB D , h , x)$ is a single marked quantum disk and $(\BB D , h , x,y)$ is a doubly marked quantum disk.

\subsubsection{Space-filling SLE$_{\kappa}$}
\label{sec-wpsf-prelim}

The Schramm-Loewner evolution (SLE$_\kappa$) for $\kappa > 0$ is a one-parameter family of random fractal curves originally defined by Schramm in~\cite{schramm0}.
SLE$_\kappa$ curves are simple for $\kappa \in (0,4]$, self-touching, but not space-filling or self-crossing, for $\kappa \in (4,8)$, and space-filling but still not self-crossing for $\kappa \geq 8$~\cite{schramm-sle}.
One can consider SLE$_\kappa$  curves between two marked boundary points of a simply connected domain (chordal), from a boundary point to an interior point (radial), or between two points in $\BB C\cup\{\infty\}$ (whole-plane). We refer to~\cite{lawler-book} or~\cite{werner-notes} for an introduction to SLE.

We first consider the \textbf{whole-plane space-filling SLE$_\kappa$ from $\infty$ to $\infty$}, which was introduced in~\cite[Sections 1.2.3 and 4.3]{ig4}.
This is a random space-filling curve $\eta$ in $\BB C$ which travels from $\infty$ to $\infty$ which fills all of $\BB C$ and never enters the interior of its past.
Moreover, its law is invariant under spatial scaling: for any $r>0$, $r\eta$ agrees in law with $ \eta$ viewed as curves modulo time parametrization.
This is essentially the only information about space-filling SLE$_\kappa$ which is needed to understand this paper --- we do not use any detailed information about the geometry of the curve.
However, we briefly describe how space-filling SLE$_\kappa$ is defined (without details) to give the reader some more intuition about what this curve is.
See~\cite[Section 3.6]{ghs-mating-survey} for a detailed review of space-filling SLE$_\kappa$.

When $\kappa \geq 8$, in which case ordinary SLE$_\kappa$ is already space-filling, whole-plane SLE$_\kappa$ from $\infty$ to $\infty$ describes the local behavior of an ordinary chordal SLE$_\kappa$ curve near a typical interior point.
For any $a < b$, the set $\eta([a,b])$ has the topology of a closed disk and its boundary is the union of four SLE$_{\ul\kappa}$-type curves for $\ul\kappa = 16/\kappa$ which intersect only at their endpoints.

When $\kappa \in (4,8)$, the definition of space-filling SLE$_\kappa$ is more involved.
Roughly speaking, chordal space-filling SLE$_\kappa$ can be obtained by starting with an ordinary chordal SLE$_\kappa$ curve and iteratively ``filling in" the bubbles which it disconnects from its target point by SLE$_\kappa$-type curves.
It is shown in~\cite{ig4} that the curve one obtains after countably many iterations is space-filling and continuous when parameterized so that it traverses one unit of Lebesgue measure in one unit of time.
The whole-plane space-filling SLE$_\kappa$ from $\infty$ to $\infty$ for $\kappa \in (4,8)$ describes the local behavior of the chordal version near a typical interior point, as in the case $\kappa \geq 8$.

The topology of the curve is much more complicated for $\kappa \in (4,8)$ than for $\kappa \geq 8$.
For $a < b$, neither the interior of the set $\eta([a,b])$ nor its complement $\BB C\setminus \eta([a,b])$ is connected. Rather, each of these sets consists of a countable union of domains with the topology of the open disk (or the punctured plane in the case of the unbounded connected component of $\BB C\setminus \eta([a,b])$.
The boundary of $\eta([a,b])$ is the union of four SLE$_{\ul\kappa}$-type curves for $\ul\kappa = 16/\kappa$, but these curves typically intersect each other.

If $D\subset \BB C$  is a simply connected domain and $a\in\bdy D$, we can similarly define the clockwise \textbf{space-filling SLE$_\kappa$ loop in $D$ based at $a$}. This is a random curve in $D$ from $a$ to $a$ which is the limit of chordal space-filling SLE from $a$ to $b$ in $D$ as $a\rta b$ from the counterclockwise direction. See Section \ref{sec-space-filling-sle} for details.

\subsection{Mated-CRT map setup}
\label{sec-mated-crt-setup}

Throughout most of this paper we will work with the following setup.
Let $(\BB C , h , 0 , \infty)$ be a $\gamma$-quantum cone with the circle average embedding.
Let $\eta$ be a whole-plane space-filling SLE$_\kappa$ with $\kappa = 16/\gamma^2$ sampled independently from $h$ and then parametrized so that $\eta(0) = 0$ and $\mu_h(\eta([a,b])) = b-a$ for every $a < b$. We define the \emph{cells}
\eqb \label{eqn-cell-def}
H_x^\ep := \eta([x-\ep, x])  ,\quad \forall\ep> 0,\quad \forall x\in\ep\BB Z.
\eqe

For $\ep > 0$, we let $\mcl G^\ep$ be the mated-CRT map associated with $(h,\eta)$, i.e., $\mcl V\mcl G^\ep = \ep\BB Z$ and two distinct vertices $x,y\in \mcl V\mcl G^\ep$ are connected by one (resp.\ two) edges if and only if $H_x^\ep \cap H_y^\ep$ has one (resp.\ two) connected component which are not singletons.
For $z\in\BB C$, we define
\eqb \label{eqn-pt-def}
x_z^\ep := \left(\text{smallest $x\in \ep\BB Z$ such that $z\in H_x^\ep$}\right)
\eqe
and we note that for a fixed $z\in\BB C\setminus \{0\}$, $x_z^\ep$ is in fact the only $x\in\ep\BB Z$ for which $z\in H_x^\ep$ (this is related to the fact that the boundaries of the cells of $\mcl G^\ep$ have zero Lebesgue measure).
For a set $D\subset\BB C$, we define
\eqb \label{eqn-subgraph}
\mcl G^\ep(D) :=\left( \text{subgraph of $\mcl G^\ep$ induced by $\{x \in \ep \BB Z: H_x^\ep \cap D\not=\emptyset\}$}\right) .
\eqe

We write $X^\ep$ for the random walk on $\mcl G^\ep$.
For $x\in\mcl V\mcl G^\ep$, we define
\eqb \label{eqn-cond-law}
\ol{\BB P}_x^\ep := \left(\text{conditional law given $(h,\eta)$ of $X^\ep$ started from $X_0^\ep = x$} \right)
\eqe
and we write $\ol{\BB E}_x^\ep$ for the corresponding expectation.
For $D\subset \BB C$, we define
\eqb \label{eqn-exit-time}
\tau_D^\ep := \left(\text{first exit time of $X^\ep$ from $\mcl G^\ep(D)$} \right).
\eqe

Since $h|_{B_1 }$ agrees in law with the corresponding restriction of a whole-plane GFF plus $\gamma\log|\cdot|^{-1}$, it will be convenient in most of our arguments to restrict attention to a proper subdomain of $B_1$.
Consequently, throughout the paper we fix $\rho \in (0,1)$ and work primarily in $B_\rho$.

We will frequently need the following estimate for the Euclidean diameters of space-filling SLE cells.

\begin{lem}[Cell diameter estimates] \label{lem-cell-diam}
Fix a small parameter $\zeta\in (0,1)$. With polynomially high probability as $\ep\rta 0$, the Euclidean diameters of the cells of $\mcl G^\ep$ which intersect $B_\rho$ satisfy
\eqb \label{eqn-cell-diam}
\ep^{\frac{2}{(2-\gamma)^2} + \zeta} \leq  \op{diam} H_x^\ep \leq \ep^{\frac{2}{(2+\gamma)^2} - \zeta} ,\quad \forall x \in \mcl V\mcl G^\ep(B_\rho) .
\eqe
\end{lem}
\begin{proof}
The upper bound for cell diameters is proven in~\cite[Lemma 2.7]{gms-harmonic}.
To get the lower bound, use Lemma~\ref{lem-ball-mass} with $\delta = \ep^{\frac{2}{(2-\gamma)^2} + \zeta}$, which implies that with polynomially high probability as $\ep\rta 0$, each Euclidean ball of radius $\ep^{\frac{2}{(2-\gamma)^2} + \zeta}$ contained in $\BB D$ has $\mu_h$-mass at most $\ep$.
If this is the case then no such Euclidean ball can contain one of the cells $H_x^\ep$ (since each such cell has $\mu_h$-mass $\ep$).
\end{proof}

\section{Green's function, effective resistance, and Harnack inequalities}
\label{sec-harnack}

\subsection{Definitions of Dirichlet energy, Green's function, and effective resistance}

In this subsection we review some mostly standard notions in the theory of electrical networks and fix some notation.

\begin{defn} \label{def-discrete-dirichlet}
For a graph $G$ and a function $f : \mcl V(G) \to \BB R$, we define its \emph{Dirichlet energy} to be the sum over unoriented edges
\eqbn
\op{Energy}(f; G) := \sum_{\{x,y\} \in \mcl E(G)} (f(x) - f(y))^2 ,
\eqen
with edges of multiplicity $m$ counted $m$ times.
\end{defn}

\begin{defn} \label{def-green}
For a graph $G$, a set of vertices $V\subset \mcl V(G)$, and two vertices $x,y\in G$, we define the \emph{Green's function} for random walk on $G$ killed upon exiting $V$ by
\eqb \label{eqn-green}
\op{Gr}_V(x,y) := \BB E_x\left[ \text{number of times that $X$ hits $y$ before leaving $V$} \right]
\eqe
where $\BB E_x$ denotes the expectation for random walk on $G$ started from $x$.
Note that $\op{Gr}_V(x,y) = 0$ if either $x$ or $y$ is not in $V$.
We also define the \emph{normalized Green's function}
\eqb \label{eqn-green-deg}
\op{gr}_V (x,y) := \frac{\op{Gr}_V^G(x,y)}{\op{deg}(y)} .
\eqe
\end{defn}

By, e.g.,~\cite[Proposition 2.1]{lyons-peres}, the function $\op{gr}_V(x,\cdot)$ is the voltage function when a unit of current flows from $x$ to $V$. In particular, $\op{gr}_V(x,\cdot)$ is discrete harmonic on $\mcl V(G) \setminus (V\cup \{x\})$.

To define effective resistance, we view a graph $G$ as an electrical network where each edge has unit resistance.
For a vertex $x \in \mcl V(G)$ and a set $Z \subset \mcl V(G)$ with $x\notin Z$, the \emph{effective resistance} from $x$ to $Z$ in $G$ is defined (using Definition~\ref{def-green}) by
\allb \label{eqn-eff-res-def}
\mcl R \left( x \leftrightarrow Z \right) := \op{gr}_{\mcl V(G)\setminus Z}(x,x) ,
\alle
i.e., $\mcl R \left( x \leftrightarrow Z \right)$ is the expected number of times that random walk started at $x$ returns to $x$ before hitting $Z$, divided by the degree of $x$.
For two disjoint vertex sets $W,Z\subset \mcl V(G)$, we define $\mcl R(W\leftrightarrow Z)$ to be the effective resistance from $\ol w$ to $Z$ in the graph $\ol G$ obtained from $G$ by identifying all of the vertices of $W$ to a single vertex $\ol w$. 

There are several equivalent definitions of effective resistance which will be useful for our purposes.
\begin{enumerate}
\item \textbf{Dirichlet's principle.} Let $\frk f = \frk f_{Z,x} : \mcl V(G) \rta [0,1]$ be the function such that $\frk f (x ) = 1$, $\frk f |_Z \equiv 0$, and $\frk f$ is discrete harmonic on $\mcl V(G) \setminus (Z\cup \{x\})$. Then (see, e.g.,~\cite[Exercise 2.13]{lyons-peres})
\eqb \label{eqn-eff-res-dirichlet}
\mcl R^G\left( x \leftrightarrow Z \right) = \frac{1}{\op{Energy}(\frk f ; G) } .
\eqe
\item \textbf{Thomson's principle.} A \emph{unit flow} from $x$ to $Z$ in $G$ is a function $\theta$ from oriented edges $e = (y,z)$ of $G$ to $\BB R$ such that $\theta(y,z) = -\theta(z,y)$ for each oriented edge $(y,z)$ of $G$ and
\eqbn
\sum_{\substack{z \in \mcl V(G) \\ z\sim y}} \theta(y,z) = 0 \quad \forall y\in \mcl V(G)\setminus  ( \{x\} \cup Z )   \quad\op{and} \quad
\sum_{\substack{z \in \mcl V(G) \\ z\sim x}} \theta(x,z) =1 .
\eqen
One has (see, e.g.,~\cite[Theorem 9.10]{markov-mixing} or \cite[Page 35]{lyons-peres})
\eqb \label{eqn-thomson}
\mcl R^G(x\leftrightarrow V) = \inf\left\{ \sum_{e\in\mcl E(G)} [\theta(e)]^2 : \text{$\theta$ is a unit flow from $x$ to $Z$} \right\} .
\eqe
\end{enumerate}

\begin{notation} \label{notation-eff-res-mated-crt}
In the case when $G = \mcl G^\ep$ is the mated-CRT map, we will typically include a superscript $\ep$ in the above notations, so, e.g., $\mcl R^\ep$ denotes effective resistance on $\mcl G^\ep$.
We will commonly apply the above definitions in the case when the vertex sets $V,Z,W$ are of the form $\mcl V\mcl G^\ep(A)$ for some set $A\subset \BB C$.
To lighten notation we will simply write $A$ instead of $\mcl V\mcl G^\ep(A)$.
So, for example, $\op{Gr}_A^\ep$ denotes the Green's function for the random walk on $\mcl G^\ep$ killed upon exiting $\mcl G^\ep(A)$.
\end{notation}

\subsection{Comparison of Green's function and effective resistance}
\label{sec-eff-res-green}

The results of~\cite{gms-harmonic,gm-spec-dim} give us good control on effective resistances in the mated-CRT map.
In this subsection we will prove a general lemma which allows us to convert estimates for effective resistance to estimates for the Green's function.

To state the lemma, we first introduce some notation.
For a graph $G$ and a set $A\subset\mcl V(G)$, we write
\eqb \label{eqn-graph-bdy}
\bdy A := \left\{ x \in A  : \text{$\exists y\in \mcl V(G)\setminus A$ with $x\sim y$}\right\} .
\eqe
We also write
\eqb
\ol A := \left(\text{subgraph of $G$ induced by $A\cup \bdy(\mcl V(G)\setminus A)$}\right) .
\eqe

The following lemma bounds the maximum and minimum of the Green's function on $\bdy A$ in terms of effective resistances.
It will be used in conjunction with Lemma~\ref{lem-disconnect-coupling} to get a uniform bound for the Green's function on $\bdy A$.

\begin{lem} \label{lem-eff-res-green}
Let $G$ be a finite graph and let $A\subset B\subset\mcl V(G)$ be a finite sets of vertices.
Let $x\in A$ and define (using the notation of Definition~\ref{def-green})
\eqb
a := \min_{y\in \bdy A} \op{gr}_B(x,y), \quad
b := \max_{y\in \bdy A} \op{gr}_B(x,y) ,\quad
\delta := \max\left\{ |\op{gr}_B(x,u) - \op{gr}_B(x,v)|  : u,v \in B , u\sim v \right\} .
\eqe
Then for any $x\in A$,
\eqb \label{eqn-eff-res-green}
\frac{a^2}{ a+\delta } \leq \mcl R\left( A\leftrightarrow \mcl V(G)\setminus B \right) \leq b+\delta .
\eqe
\end{lem}

Lemma~\ref{lem-eff-res-green} is a discrete version of~\cite[Proposition 4.1]{grigoryan-background}, and is proven in a similar manner.
It can be checked that the error term $\delta$ in Lemma~\ref{lem-eff-res-green} is at most 1; see Lemma~\ref{lem-eff-res-grad}.
This is the only bound for $\delta$ which is needed for our purposes.
To prove Lemma~\ref{lem-eff-res-green}, we need some preparatory lemmas.

\begin{lem} \label{lem-unit-flow}
Let $G$ be a graph, let $x\in G$, and let $A\subset\mcl V(G)$ be a finite set of vertices with $x\in A$.
If $\theta$ is a unit flow from $x$ to $Z\subset\mcl V(G)\setminus A$, then
\eqbn
\sum\left\{\theta(u,v) : u\in A, v\in \mcl V(G)\setminus A , u\sim v \right\}
= 1 .
\eqen
\end{lem}
\begin{proof}
Since $\theta$ is a unit flow, we have $\sum_{v\sim x} \theta(x,v) = 1$ and for any $y\in A\setminus \{x\}$, $\sum_{v\sim y} \theta(y,v) = 0$.
Therefore,
\eqbn
  \sum_{u\in A} \sum_{\substack{v\in \mcl V(G) \\ v\sim u} } \theta(u,v) = 1.
\eqen
We can break up the double sum on the left into two parts: the sum over the pairs $(u,v) \in A\times A$ with $u\sim v$ and the sum over the pairs $(u,v) \in A\times (\mcl V(G)\setminus A)$ with $u\sim v$. The first sum is zero since $\theta(u,v) = -\theta(v,u)$ for every $u,v\in A$.
Therefore the second sum is equal to 1.
\end{proof}

\begin{lem}[Discrete Green's identity] \label{lem-green-identity}
Let $G$ be a graph, let $A\subset\mcl V(G)$ be a finite set of vertices, let $ f , \frk g : \mcl V(G) \rta \BB R$ and suppose that $\frk g$ is discrete harmonic on $\mcl V(G)\setminus A$. Then
\eqb
\sum_{\{x,y\} \in \mcl E(G)} (f(x) - f(y)) (\frk g(x) -\frk g(y))
=  -  \sum_{x\in A} f(x) \Delta \frk g(x) ,
\eqe
where $\Delta$ is the discrete Laplacian and the sum is over all unoriented edges of $G$.
\end{lem}
\begin{proof}
Write $\mcl E^\srta$ for the set of oriented edges of $G$. We have
\alb
&\sum_{\{x,y\} \in \mcl E (G)} (  f(x) - f(y)) (\frk g(x) -\frk g(y))  \notag\\
&\qquad = \frac12 \sum_{(x,y) \in \mcl E^\srta(G)} (  f(x) - f(y)) (\frk g(x) -\frk g(y)) \notag\\
&\qquad = \frac12\sum_{(x,y) \in \mcl E^\srta(G)}  f(x)   (\frk g(x) -\frk g(y))  +  \frac12\sum_{(x,y) \in \mcl E^\srta(G)}   f(y)   (\frk g(y) -\frk g(x)) \notag \\
&\qquad=  \sum_{x \in \mcl V(G)} f(x) \sum_{y \sim x}     (\frk g(x) -\frk g(y)) \notag \\
&\qquad = -  \sum_{x\in \mcl V(G)} f(x) \Delta \frk g(x) = -  \sum_{x\in A} f(x) \Delta \frk g(x) ,
\ale
where in the last inequality we used that $\Delta\frk g = 0$ on $\mcl V(G)\setminus A$.
\end{proof}

\begin{proof}[Proof of Lemma~\ref{lem-eff-res-green}]
For $c > 0$, let
\eqb \label{eqn-gr-level-set}
F_c := \left\{ y\in B : \op{gr}_B(x,y) \geq c\right\} , \quad\forall c \in \left[ 0 ,  \op{gr}_B(x,x)  \right] .
\eqe
Note that each $F_c$ contains $x$, $F_c$ is non-increasing in $c$, and $F_0 = B$.

We claim that with $a,b$ as in the lemma statement, we have $F_{b'} \subset A \subset F_a$ for each $b' > b$.
Indeed, since $\op{gr}_B(x,\cdot)$ is discrete harmonic on $B\setminus A$ and vanishes outside of $B$, its maximum value on $B\setminus A$ is by the maximum principle at most its maximum value on $\bdy A$, which is $b$. Hence $F_{b'} \cap (B\setminus A) = \emptyset$ for $b' > b$, i.e., $F_{b'} \subset A$.
Similarly, since $\op{gr}_B(x,\cdot)$ is discrete harmonic on $A \setminus \{x\}$ and attains its maximum value at $x$ (again by the maximum principle, on all of $B$), its minimum value on $A$ is the same as its minimum value on $\bdy A$, which is $a$. Hence $A\subset F_a$.
In particular, by Rayleigh's monotonicity principle (see Chapter 2.4 in \cite{lyons-peres}),
\eqb \label{eqn-eff-res-compare}
 \mcl R\left(F_a \leftrightarrow \mcl V(G)\setminus B \right)  \leq  \mcl R\left(A\leftrightarrow \mcl V(G)\setminus B \right) \leq  \mcl R\left(F_{b'} \leftrightarrow \mcl V(G)\setminus B \right) .
\eqe

We claim that
\eqb \label{eqn-eff-res-bound}
\frac{c^2}{ c+\delta } \leq \mcl R\left(F_c \leftrightarrow \mcl V(G)\setminus B \right) \leq c+\delta   ,\quad\forall c \in \left( 0 , \op{gr}_B(x,x) \right] .
\eqe
Once~\eqref{eqn-eff-res-bound} is established, \eqref{eqn-eff-res-green} follows from~\eqref{eqn-eff-res-compare} upon letting $b'$ decrease to $b$.

It remains to prove~\eqref{eqn-eff-res-bound}.
To this end, consider the function $\frk f(y) := c^{-1} \op{gr}_B(x,y)$. Then $\frk f$ is discrete harmonic on $B\setminus \{x\}$, it vanishes outside of $B$, its values on $F_c$ are at least 1, and its values on $\bdy F_c$ are at most $(c+\delta)/c$.
We first argue that
\eqb \label{eqn-f-energy}
\op{Energy}\left( \frk f ; \ol{\mcl V(G) \setminus F_c} \right)^{-1}
\leq \mcl R\left(F_c \leftrightarrow \mcl V(G)\setminus B \right)
\leq c^2 \op{Energy}\left( \frk f ; \ol{\mcl V(G) \setminus F_c} \right) .
\eqe

Indeed, the function $\frk f \wedge 1$ vanishes outside of $B$ and is identically equal to 1 on $F_c$.
By Dirichlet's principle~\eqref{eqn-eff-res-dirichlet},
\eqb \label{eqn-use-dirichlet}
\op{Energy}\left( \frk f ; \ol{\mcl V(G) \setminus F_c} \right)
\geq \op{Energy}\left( \frk f \wedge 1 ; \ol{\mcl V(G)\setminus F_c} \right)
\geq \mcl R\left( F_c \leftrightarrow \mcl V(G)\setminus B \right)^{-1} .
\eqe

Recalling that $\frk f(y) := c^{-1} \op{gr}_B(x,y)$, we have that $\theta(u,v) := c(\frk f(u) - \frk f(v))$ is a unit flow from $x$ to $B\setminus\mcl V(G)$ (here we use that the gradient of $\op{gr}(x,\cdot)$ is a unit flow~\cite[Proposition 2.1]{lyons-peres}).
By Lemma~\ref{lem-unit-flow},
\eqb \label{eqn-green-unit-flow}
\sum \left\{ \theta(u,v) : u\in F_c , v\in B\setminus F_c, u\sim v \right\} = 1.
\eqe
Hence $\theta|_{\mcl E(\ol{\mcl V(G) \setminus F_c})}$ is a unit flow from $\bdy F_c$ to $\mcl V(G)\setminus B$.
By Thomson's principle~\eqref{eqn-thomson},
\eqb \label{eqn-use-thomson}
c^2 \op{Energy}\left(\frk f ; \ol{\mcl V(G)\setminus F_c} \right)
= \sum_{\{u,v\} \in \mcl E\left(  \ol{\mcl V(G)\setminus F_c} \right)} [\theta(u,v)]^2
\geq \mcl R\left( F_c \leftrightarrow \mcl V(G)\setminus B \right) .
\eqe

From~\eqref{eqn-use-dirichlet} and~\eqref{eqn-use-thomson}, we obtain~\eqref{eqn-f-energy}.
We now need to bound $\op{Energy}\left( \frk f ; \ol{\mcl V(G) \setminus F_c} \right) $.
By Lemma~\ref{lem-green-identity} applied to the graph $\ol{\mcl V(G)\setminus F_c} $, the vertex set $  \bdy F_c$, and the functions $f = \frk g = \frk f$,
\allb \label{eqn-energy-to-flow}
\op{Energy}\left( \frk f ; \ol{\mcl V(G)\setminus F_c}  \right)
&=   \sum_{u\in \bdy F_c} \frk f(u) \sum_{v\in \mcl V(G) \setminus F_c , u\sim v}  (  \frk f(u) - \frk f(v) ) \notag \\
&\leq   \frac{c+\delta}{c} \sum \left\{\frk f(u) - \frk f(v) : u\in F_c , v\in \mcl V(G)\setminus F_c, u\sim v \right\} \quad \text{(recall $\frk f \leq \frac{c+\delta}{c}$ on $\bdy F_c$)} \notag\\
&\leq \frac{c+\delta}{c^2} ,\quad\text{by~\eqref{eqn-green-unit-flow}}.
\alle
Plugging~\eqref{eqn-energy-to-flow} into~\eqref{eqn-f-energy} gives~\eqref{eqn-eff-res-bound}.
\end{proof}

The following crude estimate is useful when applying Lemma~\ref{lem-eff-res-green}.

\begin{lem} \label{lem-eff-res-grad}
In the setting of Lemma~\ref{lem-eff-res-grad}, we have $\delta \leq 1$.
\end{lem}
\begin{proof}
Consider an edge $\{u_0 ,v_0\} \in \mcl E(G)$ with $\op{gr}_B(x,u_0) > \op{gr}_B(x,v_0)$.
We need to show that $\op{gr}_B(x,u_0) - \op{gr}_B(x,v_0) \leq 1$.
To this end, let $c = \op{gr}_B(x,u_0)$ and let $F_c = \left\{ y\in B : \op{gr}_B(x,y) \geq c\right\} $, as in~\eqref{eqn-gr-level-set}.
Then $u_0\in F_c$ and $v_0\notin F_c$.
Since $\theta(u,v) := \op{gr}_B(x,u) - \op{gr}_B(x,v)$ is a unit flow from $x$ to $\mcl V(G)\setminus B$ and $x\in F_c$, we can apply Lemma~\ref{lem-unit-flow} to get that
\eqb \label{eqn-eff-res-grad-sum}
\sum\left\{\op{gr}_B(x,u) - \op{gr}_B(x,v)  : u\in F_c, v\in \mcl V(G)\setminus F_c , u\sim v \right\}
= 1 .
\eqe
For each $u\in F_c$ and each $v\in \mcl V(G)\setminus F_c$, we have $\op{gr}_B(x,u) - \op{gr}_B(x,v) \geq 0$.
Therefore, each term in the sum~\eqref{eqn-eff-res-grad-sum} is at most 1.
In particular, $\op{gr}_B(x,u_0) - \op{gr}_B(x,v_0) \leq 1$.
\end{proof}

\subsection{Harnack inequality for the mated-CRT map}
\label{sec-harnack-inequality}

Lemma~\ref{lem-eff-res-green} gives an upper (resp.\ lower) bound for the minimum (resp.\ maximum) value of the Green's function on $\bdy A$ in terms of effective resistances.
It is of course much more useful to have an upper (resp.\ lower) bound for the maximum (resp.\ minimum) value of the Green's function on $\bdy A$.
Hence we need a Harnack-type estimate which allows us to compare the maximum and minimum values of the Green's function on $\bdy A$.
In this subsection we will prove such an estimate for the mated-CRT map $\mcl G^\ep$ in the case when $A$ is the set of vertices which intersect a Euclidean ball.

Throughout this subsection we assume the setup and notation of Section~\ref{sec-mated-crt-setup}.
In particular we recall that for $z\in\BB C$, $x_z^\ep \in \mcl V\mcl G^\ep$ is chosen so that $z$ is contained in the cell $H_{x_z^\ep}^\ep$ and for $D\subset\BB C$, $\mcl G^\ep(D)$ is the subgraph of $\mcl G^\ep$ induced by the vertex set $\{x_z^\ep : z\in D\}$.
We also recall that $\rho \in (0,1)$ is fixed.

\begin{prop}[Harnack inequality on a circle]  \label{prop-harnack}
There exists $\beta = \beta(\gamma ) > 0$ and $C = C(\rho , \gamma) > 0$ such that the following holds with polynomially high probability as $\ep\rta 0$.
Simultaneously for every $z \in B_\rho $, every $s \in \left[  \ep^\beta, \frac13 \op{dist}(z ,\bdy B_\rho) \right]$, and every function $\frk f : \mcl V\mcl G^\ep \rta [0,\infty)$ which is discrete harmonic on $\mcl V\mcl G^\ep(B_{3s}(z) ) \setminus \{x_z^\ep\}$,
\eqb \label{eqn-harnack}
\max_{y \in \mcl V\mcl G^\ep(\bdy B_s(z))} \frk f(y) \leq C \min_{y\in \mcl V\mcl G^\ep(\bdy B_s(z))} \frk f(y) .
\eqe
\end{prop}

The reason why we do not require that $\frk f$ is discrete harmonic at the center point $x_z^\ep$ is that we eventually want to apply Proposition~\ref{prop-harnack} with $\frk f$ equal to the normalized Green's function $\op{gr}_U^\ep(x_z^\ep , \cdot)$ for $B_{3s}(z) \subset U \subset B_\rho$ (recall Notation~\ref{notation-eff-res-mated-crt}).
The key input in the proof of Proposition~\ref{prop-harnack} is the following coupling lemma for random walks on a graph with different starting points, which follows from Wilson's algorithm and is a variant of~\cite[Lemma 3.12]{gms-harmonic}.

\begin{lem} \label{lem-disconnect-coupling}
Let $G$ be a connected graph and let $A\subset \mcl V(G)$ be a set such that the simple random walk started from any vertex of $G$ a.s.\ hits $A$ in finite time.
For $x\in \mcl V(G)$, let $X^x$ be the simple random walk started from $x$ and let $\tau^x$ be the first time $X^x$ hits $A$.
For $x,y\in \mcl V(G) \setminus A$, there is a coupling of $X^x$ and $X^y$ such that a.s.\
\eqb \label{eqn-disconnect-coupling}
\BB P\left[ X^x_{\tau^x} = X^y_{\tau^y} \,\Big|\, X^y|_{[0,\tau^y]} \right] \geq  \BB P\left[ \text{$X^x$ disconnects $y$ from $A$ before time $\tau^x$} \right]  .
\eqe
\end{lem}
\begin{proof}
The lemma is a consequence of Wilson's algorithm.
Let $T$ be the uniform spanning tree of $G$ with all of the vertices of $A$ wired to a single point. For $x\in \mcl V(G)$, let $L^x$ be the unique path in $T$ from $x$ to $A$.
For a path $P$ in $G$, write $\op{LE}(P) $ for its chronological loop erasure.
By Wilson's algorithm~\cite{wilson-algorithm}, we can generate the union $L^x\cup L^y$ in the following manner.
\begin{enumerate}
\item Run $X^y$ until time $\tau^y$ and generate the loop erasure $ \op{LE}(X^y|_{[0,\tau^y]})$.
\item Conditional on $X^y|_{[0,\tau^y]}$, run $X^x$ until the first time $\wt\tau^x$ that it hits either $\op{LE}(X^y|_{[0,\tau^y]})$ or $A$.
\item Set $L^x \cup L^y = \op{LE}(X^y|_{[0,\tau^y]}) \cup \op{LE}(X^x|_{[0,\wt\tau^x]})$.
\end{enumerate}
Note that $L^y =  \op{LE}(X^y|_{[0,\tau^y]})$ in the above procedure.
Applying the above procedure with the roles of $x$ and $y$ interchanged shows that $L^x \eqD \op{LE}(X^x|_{[0,\tau^x]})$.
When we construct $L^x\cup L^y$ as above, the points where $L^x$ and $L^y$ hit $A$ coincide provided $X^x$ hits $ \op{LE}(X^y|_{[0,\tau^y]})$ before $A$.
The conditional probability given $X^y|_{[0,\tau^y]}$ that $X^x$ hits $ \op{LE}(X^y|_{[0,\tau^y]})$ before $A$ is at least the unconditional probability that $X^x$ disconnects $y$ from $A$ before time $\tau^x$. We thus obtain a coupling of $\op{LE}(X^x|_{[0,\tau^x]})$ and $\op{LE}(X^y|_{[0,\tau^y]})$ such that the conditional probability given $X^y|_{[0,\tau^y]}$ that these two loop erasures hit $A$ at the same point is at least $ \BB P\left[ \text{$X^x$ disconnects $y$ from $A$ before time $\tau^x$} \right]$. We now obtain~\eqref{eqn-disconnect-coupling} by observing that $X^x_{\tau^x}$ is the same as the point where $\op{LE}(X^x|_{[0,\tau^x]})$ first hits $A$, and similarly with $y$ in place of $x$.
\end{proof}

In light of Lemma~\ref{lem-disconnect-coupling}, to prove Proposition~\ref{prop-harnack} we need a lower bound for the probability that the random walk on $\mcl G^\ep$ disconnects a circle from $\bdy B_\rho$ before exiting $B_\rho$.

\begin{lem}  \label{lem-walk-dc}
There exists $\beta = \beta(\gamma ) > 0$ and $p = p(\rho , \gamma) \in (0,1)$ such that the following is true.
Let $z\in B_\rho$ and let $s\in \left[ \ep^\beta, \frac12 \op{dist}\left(z , \bdy B_\rho\right) \right] $.
With polynomially high probability as $\ep\rta 0$, uniformly in the choices of $z$ and $s$,
\allb \label{eqn-walk-dc}
&\min_{x\in  \mcl V\mcl G^\ep(B_s \setminus B_{s/2})}
\ol{\BB P}_x^\ep\bigg[ \text{$X $ disconnects $\mcl V\mcl G^\ep(\BB A_{s/2,s}(z))$ from $\mcl V\mcl G^\ep(\bdy \BB A_{s/4, 2s}(z)) $} \notag\\
&\qquad\qquad\qquad\qquad\qquad \text{before hitting $\mcl V\mcl G^\ep(\bdy \BB A_{s/4, 2s}(z)) $} \bigg] \geq p.
\alle
\end{lem}
\begin{proof}
This follows from~\cite[Proposition 3.6]{gms-harmonic}.
\end{proof}

We now upgrade Lemma~\ref{lem-walk-dc} to a statement which holds \emph{simultaneously} for all choices of $z$ and $s$ via a union bound argument.

\begin{lem} \label{lem-walk-dc-uniform}
There exists $\beta = \beta(\gamma ) > 0$ and $p = p(\rho , \gamma) \in (0,1)$ such that the following is true.
With polynomially high probability as $\ep\rta 0$, it holds simultaneously for every $z\in B_\rho$ and every $s\in \left[ \ep^\beta, \frac13 \op{dist}\left(z , \bdy B_\rho\right) \right] $ that
\allb \label{eqn-walk-dc-uniform}
&\min_{x\in   \mcl V\mcl G^\ep(\bdy B_s(z))}
\ol{\BB P}_x^\ep\bigg[ \text{$X $ disconnects $\mcl V\mcl G^\ep(\bdy B_s(z))$ from $\left(\mcl V\mcl G^\ep \setminus \mcl V\mcl G^\ep(\bdy B_{3s}(z) ) \right) \cup \{x_z^\ep\}$} \notag\\
&\qquad\qquad\qquad\qquad\qquad \text{before hitting $\left(\mcl V\mcl G^\ep \setminus \mcl V\mcl G^\ep(\bdy B_{3s}(z) ) \right) \cup \{x_z^\ep\}$} \bigg] \geq p.
\alle
\end{lem}
\begin{proof}
By Lemma~\ref{lem-walk-dc}, there exists $\beta_0 = \beta_0(\gamma)  >0$, $\alpha_0 = \alpha_0(\gamma) > 0$, and $p = p(\rho,\gamma) > 0$ such that for each $z\in B_\rho$ and each $s\in \left[ \ep^{\beta_0} , \frac12 \op{dist}\left(z , \bdy B_{\rho}\right) \right] $, the estimate~\eqref{eqn-walk-dc} holds with probability $1-O_\ep(\ep^{\alpha_0})$, uniformly in the choices of $z$ and $s$.
Let $\beta := \frac{1}{100} \min\{\beta_0 , \alpha_0\}$.
We can find a collection $\mcl Z$ of at most $O_\ep(\ep^{-4\beta})$ points $z\in B_\rho$ such that each point of $B_\rho$ lies within Euclidean distance $\ep^{2\beta}$ of some $z \in\mcl Z$.

Due to our choice of $\beta$, we can take a union bound over all $z\in\mcl Z$ and all $s\in \left[\ep^\beta , \frac12 \op{dist}(z, \bdy B_{\rho}) \right] \cap \{2^{-n/2}\}$ to find that with polynomially high probability as $\ep\rta 0$, \eqref{eqn-walk-dc} holds simultaneously for every such $z$ and $s$.
Due to our choice of $\mcl Z$, if $z\in B_\rho$ and $s\in [\ep^\beta,\rho/2]$ with $B_{3s}(z) \subset B_\rho$, then there exists $z'\in \mcl Z$ and $s' \in \left[\ep^\beta , \frac12 \op{dist}(z, \bdy B_\rho) \right] \cap \{2^{-n/2}\}$ such that
\eqbn
\bdy B_s(z) \subset \BB A_{s'/2,s'}(z') \quad\text{and} \quad  \BB A_{s'/4, 2s'}(z) \subset B_{3s}(z) \setminus \{z\} .
\eqen
The estimate~\eqref{eqn-walk-dc-uniform} for $(z,s)$ therefore follows from~\eqref{eqn-walk-dc} for $(z',s')$.
\end{proof}

\begin{proof}[Proof of Proposition~\ref{prop-harnack}]
By Lemma~\ref{lem-walk-dc-uniform}, it holds with polynomially high probability as $\ep\rta 0$ that~\eqref{eqn-walk-dc-uniform} holds simultaneously for every $z \in B_\rho$ and $ s\in \left[ \ep^\beta, \frac13\op{dist}\left(z ,\bdy B_\rho \right) \right] $.
We henceforth assume that this is the case and work under the conditional law given $\mcl G^\ep$.

Fix $z \in B_\rho$ and $ s\in \left[ \ep^\beta, \frac13\op{dist}\left(z ,\bdy B_{\rho} \right) \right] $.
For $x \in\mcl V\mcl G^\ep$, we let $X^x$ be the simple random walk on $\mcl G^\ep$ and we let $\tau^x$ be the first time that $X^x$ either exits $\mcl V\mcl G^\ep(B_{3s}(z))$ or hits $x_z^\ep$.
It follows from~\eqref{eqn-walk-dc-uniform} that the conditional probability given $\mcl G^\ep$ that $X^x$ disconnects $\mcl V\mcl G^\ep(\bdy B_s(z))$ from $\left(\mcl V\mcl G^\ep \setminus  \mcl V\mcl G^\ep(B_{3s}(z) ) \right) \cup \{x_z^\ep\}$ before time $\tau^x$ is at least $p$.
By this and Lemma~\ref{lem-disconnect-coupling}, if $x,y \in \mcl V\mcl G^\ep(\bdy B_s(z) )$ we can couple the random walks $X^x$ and $X^y$ in such a way that a.s.
\eqb \label{eqn-use-coupling}
\BB P\left[ X_{\tau^x}^x = X^y_{\tau^y} \,|\, X^y|_{[0,\tau^y]} , \mcl G^\ep \right] \geq p  .
\eqe

If $\frk f : \mcl V\mcl G^\ep \rta [0,\infty)$ is discrete harmonic on $\mcl V\mcl G^\ep(B_{3s}(z) ) \setminus \{x_z^\ep\}$, then
\allb
\frk f(x)
= \BB E\left[ \frk f(X_{\tau^x}) \,|\, \mcl G^\ep  \right]
= \BB E\left[ \BB E\left[ \frk f(X_{\tau^x}) \,|\, X^y|_{[0,\tau^y]} , \mcl G^\ep \right]   \,|\, \mcl G^\ep  \right]
\geq p \BB E\left[ \frk f(X_{\tau^y}) \,|\, \mcl G^\ep  \right]
= p\frk f(y) .
\alle
Choosing $x$ (resp.\ $y$) so that $\frk f$ attains its minimum (resp.\ maximum) value on $\mcl V\mcl G^\ep(\bdy B_s(z))$ at $x$ (resp.\ $y$) now gives~\eqref{eqn-harnack} with $C =1/p$.
\end{proof}

\section{Lower bound for exit times}
\label{sec-exit-lower}

In this subsection we prove a lower bound for the quenched expected exit time of the random walk on $\mcl G^\ep$ from a Euclidean ball.
We continue to work in the setting of Section~\ref{sec-mated-crt-setup}.
We recall in particular the notation $\tau_D^\ep$ from~\eqref{eqn-exit-time} for (roughly speaking) the exit time of random walk on $\mcl G^\ep$ from $D\subset\BB C$.

\begin{prop} \label{prop-exit-lower}
There exists $\omega=\omega(\gamma) > 0$ and $c = c(\rho,\gamma) > 0$ such that with polynomially high probability as $\ep\rta 0$ (at a rate depending on $\rho,\gamma$), the following is true.
Simultaneously for each $x\in\mcl V\mcl G^\ep(B_{\rho-\ep^\omega} )$ and each $r \in \left[\ep^\omega , \op{dist}\left( \eta(x) , \bdy B_\rho \right) \right]$,
\eqb \label{eqn-exit-lower}
\ol{\BB E}_x^\ep\left[  \tau_{B_r(\eta(x)) }^\ep \right] \geq    \ep^{-1} \mu_h\left( B_{c r} (\eta(x)) \right)  .
\eqe
\end{prop}

Note that
\eqb
\ol{\BB E}_x^\ep\left[  \tau_{B_r(\eta(x)) }^\ep \right]
= \sum_{y\in \mcl V\mcl G^\ep(B_r(\eta(x))} \op{Gr}_{B_r(\eta(x))}^\ep(x,y)
\eqe
where $ \op{Gr}_{B_r(\eta(x))}^\ep$ is the Green's function as in Notation~\ref{notation-eff-res-mated-crt}.
Hence to prove Proposition~\ref{prop-exit-lower} we need a lower bound for this Green's function.
We will obtain such a lower bound using Lemma~\ref{lem-eff-res-green} and Proposition~\ref{prop-harnack}.

We first establish a lower bound for the effective resistance in $\mcl G^\ep$ across a Euclidean annulus in Section~\ref{sec-eff-res-lower} using Dirichlet's principle~\eqref{eqn-eff-res-dirichlet} and bounds for the Dirichlet energies of discrete harmonic functions on $\mcl G^\ep$ from~\cite{gms-harmonic}.
Due to Lemma~\ref{lem-eff-res-green}, this gives us a lower bound for the maximum value of $\op{Gr}_{B_r(\eta(x))}(x,\cdot)$ over all vertices $y\in\mcl V\mcl G^\ep$ whose corresponding cells intersects a specified Euclidean circle centered at $\eta(x)$.
We will then use Proposition~\ref{prop-harnack} to upgrade this to a uniform lower bound for $\op{Gr}_{B_r(\eta(x))}(x,\cdot)$.

\subsection{Lower bound for effective resistance}
\label{sec-eff-res-lower}

Let $\omega > 0$ be an exponent to be chosen later, in a manner depending only on $\gamma$. To lighten notation, for $\ep > 0$ we let $\mcl Z^\ep  = \mcl Z^\ep( \omega,\rho)$ be the set of triples $(z,r,s) \in\BB C\times [0,\infty)\times [0,\infty)$ such that
\eqb \label{eqn-eff-res-lower-param}
z\in B_{\rho-\ep^\omega} ,\quad r\in \left[\ep^\omega , \op{dist}(z,\bdy B_\rho)\right] ,\quad \text{and} \quad s\in [0,  r/3] .
\eqe
We also let $  \mcl R^\ep$ denote the effective resistance on $\mcl G^\ep$.

\begin{lem} \label{lem-eff-res-lower}
There exists $\omega = \omega(\gamma ) > 0$ and $C = C( \rho,\gamma) > 0$ such that with polynomially high probability as $\ep\rta 0$ (at a rate depending on $  \rho,\gamma $),
\eqb \label{eqn-eff-res-lower}
   \mcl R^\ep\left( B_{s}(z) \leftrightarrow  \bdy B_r(z)  \right)  \geq \frac{1}{C} \log\left(   (r/s) \wedge \ep^{-1} \right) ,
   \quad \forall (z,r,s) \in \mcl Z^\ep .
\eqe
\end{lem}

Due to Dirichlet's principle, Lemma~\ref{lem-eff-res-lower} will follow from the following upper bound for the Dirichlet energy of certain harmonic functions on $\mcl G^\ep$.

\begin{lem} \label{lem-annulus-energy}
There exists $\omega = \omega(\gamma) > 0$ and $C = C( \rho,\gamma) > 0$ such that it holds with polynomially high probability as $\ep\rta 0$ (at a rate depending on $ \rho,\gamma $) that the following is true.
For each $(z,r,s) \in \mcl Z^\ep $,
let $\frk f^\ep_{z,r,s} : \mcl V\mcl G^\ep  \rta [0,1]$ be the function which is equal to 1 on $\mcl V\mcl G^\ep(B_{s }(z))$, is equal to zero outside of $\mcl V\mcl G^\ep(B_r(z) )$, and is discrete harmonic on $\mcl V\mcl G^\ep(B_r(z) ) \setminus \mcl V\mcl G^\ep(B_{s }(z))$.
Then (in the notation of Definition~\ref{def-discrete-dirichlet})
\eqb \label{eqn-annulus-energy}
\op{Energy}\left(\frk f^\ep_{z,r,s} ; \mcl G^\ep \right) \leq \frac{C}{\log\left(   (r/s) \wedge \ep^{-1} \right)} ,
\quad \forall (z,r,s) \in \mcl Z^\ep  .
\eqe
\end{lem}
\begin{proof}
The main input in the proof is~\cite[Proposition 4.4]{gms-tutte}, which allows us to upper-bound the discrete Dirichlet energy of a discrete harmonic function on $\mcl G^\ep$ in terms of the Dirichlet energy of the continuum harmonic function on $\BB C$ with the same boundary data.
This estimate is applied in Step 1. The rest of the proof consists of union bound arguments which are needed to make~\eqref{eqn-annulus-energy} hold for all $(z,r,s) \in \mcl Z^\ep$ simultaneously.
\medskip

\noindent\textit{Step 1: comparing discrete and continuum Dirichlet energy.}
Let $\wt{\mcl Z}^\ep$ be defined as in~\eqref{eqn-eff-res-lower-param} except with $\frac12\ep^\omega$ in place of $\ep^\omega$ and $r/2$ in place of $r/3$.
We first prove an estimate for a fixed choice of $(z,r,s) \in \wt{\mcl Z}^\ep $.
The reason for working with $\wt{\mcl Z}^\ep $ instead of $\mcl Z^\ep$ initially is that we will have to increase the parameters slightly in the union bound argument.

Let $  \frk g_{z,r,s} : \BB C\rta [0,1]$ be the function which is equal to 1 on $B_{ s}(z)$, is equal to 0 outside of $B_r(z)$, and is harmonic on the interior of $B_r(z) \setminus B_{s}(z)$.
That is,
\eqb
\frk g_{z,r,s}(w) = \max\left\{ \frac{\log ( |w-z|  / r) }{\log(s/r)} , 0 \right\} .
\eqe
A direct calculation shows that the continuum Dirichlet energy of $\frk g_{z,r,s}$
is
\eqbn
\op{Energy}\left(\frk g_{z,r,s} ; B_r(z) \setminus B_s(z) \right)  
=  \int_{B_r(z) \setminus B_s(z)} |\nabla  \frk g_{z,r,s} |^2
 =  \frac{4\pi}{\log(r/s)} .
\eqen
Furthermore, the modulus of each of $\frk g_{z,r,s}$ and each of its first and second order partial derivatives is bounded above by a universal constant times $s^{-2} (\log(r/s))^{-1} \leq s^{-2}(\log 2)^{-1}$ on $B_r(z) \setminus B_s(z)$.

Consequently, we can apply~\cite[Theorem 3.2]{gms-harmonic} to find that there exists $\alpha=\alpha(\gamma) > 0$, $\beta=\beta(\gamma)> 0$, and $C_1 = C_1( \rho,\gamma) > 0$ such that the following is true. If we take $\omega \leq \beta/2$, say, then for each \emph{fixed} choice of $(z,r,s) \in \wt{\mcl Z}^\ep$ such that $s\geq \ep^\beta $, it holds with probability at least $1-O_\ep(\ep^\alpha)$ that
\eqb \label{eqn-annulus-energy0}
\op{Energy}\left( \frk f_{z,r,s}^\ep ; \mcl G^\ep \right) \leq \frac{C_1}{\log\left( r/s \right)}   .
\eqe
Note that we absorbed the $\ep^\alpha$ error in~\cite[Theorem 3.2]{gms-harmonic} into the main term $C_1/\log(r/s)$, which we can do since $r/s \leq \ep^{-\beta}$ and hence $1/\log(r/s) \geq \beta/(\log\ep^{-1})$.
\medskip

\noindent\textit{Step 2: transferring to a bound for all choices of $z,r,s$ simultaneously.}
Let $\omega := \frac{1}{100} \min\{\alpha,\beta\}$. By a union bound over $O_\ep(\ep^{8\omega} (\log\ep^{-1})^2)$ choices of $(z,r,s)$, we get that with polynomially high probability as $\ep \rta 0$, the estimate~\eqref{eqn-annulus-energy0} holds simultaneously for each 
\eqb  \label{eqn-annulus-energy-union}
(z,r,s) \in \wt{\mcl Z}^\ep \quad \text{such that} \quad z \in B_\rho \cap (\ep^{4\omega}\BB Z^2) , \quad r,s \in \{2^{-n}\}_{n\in\BB N} ,\quad s \geq \ep^{2\omega}/2 .  
\eqe
Henceforth assume that this is the case.

Let $(z,r,s) \in \mcl Z^\ep$ with $s\geq \ep^{2\omega}$. Since $r\geq \ep^\omega$ by~\eqref{eqn-eff-res-lower-param}, if $\ep \in (0,1)$ is small enough (how small depends only on $  \rho,\gamma$) we can find $(z',r',s') $ satisfying the conditions in~\eqref{eqn-annulus-energy-union} such that
\eqb \label{eqn-annulus-energy-balls}
B_{r'}(z') \subset B_r(z) , \quad
B_{s }(z) \subset B_{s' }(z') ,
\quad r' \in [r/2,2r],
\quad \text{and} \quad
s' \in [s/2,2s] .
\eqe
By~\eqref{eqn-annulus-energy-balls}, the function $\frk f^\ep_{z',r',s'}$ is equal to zero on $\mcl V\mcl G^\ep( B_{s }(z) )$ and vanishes outside of $\mcl V\mcl G^\ep(B_r(z))$.
Since the discrete harmonic function $\frk f_{z,r,s}^\ep$ minimizes discrete Dirichlet energy subject to these conditions, it follows from~\eqref{eqn-annulus-energy0} for $z',r',s'$ that
\eqb \label{eqn-annulus-energy1}
\op{Energy}\left( \frk f_{z,r,s}^\ep ; \mcl G^\ep \right)
\leq \op{Energy}\left( \frk f^\ep_{z',r',s'} ; \mcl G^\ep \right)
\leq \frac{C_1}{\log\left( r'/s' \right)} 
\leq \frac{ C_2}{\log\left(r/s \right) } ,
\eqe
where $C_2 = 100 C_1$, say.
\medskip

\noindent\textit{Step 3: the case when $s$ is small.}
We have only proven~\eqref{eqn-annulus-energy1} in the case when $s\geq \ep^{2\omega}$. We now need to remove this constraint.
To this end, we observe that for each $(z,r,s ) \in\mcl Z^\ep $ with $s\leq \ep^{2\omega}$, the function $\frk f_{z,r,\ep^{2\omega}}^\ep$ is equal to 1 on $\mcl V\mcl G^\ep(B_s(z))$ and is equal to zero outside of $\mcl V\mcl G^\ep(B_r(z))$.
Since $\frk f_{z,r,s}^\ep$ minimizes discrete Dirichlet energy subject to these conditions, we infer from~\eqref{eqn-annulus-energy1} with $\ep^{2\omega}$ in place of $s$ that
\eqb \label{eqn-annulus-energy2}
\op{Energy}\left( \frk f_{z,r,s}^\ep ; \mcl G^\ep \right)
\leq \op{Energy}\left( \frk f_{z,r,\ep^{2\omega}} ; \mcl G^\ep \right)
\leq \frac{ (2\omega)^{-1} C_2}{\log\left( \ep^{-1} \right) } .
\eqe
Combining~\eqref{eqn-annulus-energy1} and~\eqref{eqn-annulus-energy2} gives~\eqref{eqn-annulus-energy} with $C = (2\omega)^{-1} C_2$.
\end{proof}

\begin{proof}[Proof of Lemma~\ref{lem-eff-res-lower}]
By Dirichlet's principle~\eqref{eqn-eff-res-dirichlet}, the discrete Dirichlet energy of the function $\frk f_{z,r,s}^\ep$ of Lemma~\ref{lem-annulus-energy} equals $\mcl R^\ep\left(B_s(z) \leftrightarrow  \bdy B_r(z) \right)^{-1}$. Hence Lemma~\ref{lem-eff-res-lower} follows from Lemma~\ref{lem-annulus-energy}.
\end{proof}

\subsection{Lower bound for the Green's function}
\label{sec-gr-bound-lower}

We now use Lemma~\ref{lem-annulus-energy} together with the results of Section~\ref{sec-harnack} to prove a lower bound for the Green's function on $\mcl G^\ep$.

\begin{lem} \label{lem-gr-bound-lower}
There exists $\omega=\omega(\gamma) >0$ and $C = C( \rho,\gamma ) > 0$ such that the following holds with polynomially high probability as $\ep\rta 0$ (at a rate depending on $\rho,\gamma$).
For each $x \in\mcl V\mcl G^\ep(B_{\rho-\ep^\omega})$ and each $r \in \left[\ep^\omega , \op{dist}\left( \eta(x) , \bdy B_\rho \right) \right]$,
\eqb \label{eqn-gr-bound0-lower}
\op{gr}_{B_r(\eta(x) )}^\ep(x , y) \geq C^{-1} \log\left( \frac{r}{|\eta(x) - \eta(y)|} \wedge \ep^{-1} \right) - 1  ,
\quad \forall y \in \mcl V \mcl G^\ep(B_{ r  /3}(\eta(x))   .
\eqe
\end{lem} 
\begin{proof} 
Let $\omega_0 = \omega_0(\gamma ) > 0$ be the exponent $\omega$ from Lemma~\ref{lem-eff-res-lower}, let $\beta$ be as in Proposition~\ref{prop-harnack}, and let $\omega := \frac12\min\{\omega_0,\beta\}$. 
The function $y\mapsto \op{gr}_{B_r(z)}(x_z^\ep ,y)$ is discrete harmonic on $\mcl V\mcl G^\ep(B_r(z) ) \setminus \{x_z^\ep\}$.
We can therefore apply Proposition~\ref{prop-harnack} to find that there exists $C_1 = C_1(\rho,\gamma)>1$ such that with polynomially high probability as $\ep\rta 0$, it holds simultaneously for each $z \in B_{\rho-\ep^\omega}$, each $r \in \left[ \ep^\omega ,\op{dist}(z,\bdy B_\rho) \right]$, and each $s\in \left[  \ep^\beta , r/3 \right]$ that
\eqb \label{eqn-use-harnack-lower}
 \max_{y\in  \mcl V\mcl G^\ep(\bdy  B_s(z))} \op{gr}_{B_r(z) }^\ep(x_z^\ep ,y)  \leq C_1  \min_{y\in  \mcl V\mcl G^\ep(\bdy B_s(z))} \op{gr}_{B_r(z) }^\ep(x_z^\ep ,y)    .
\eqe

By Lemma~\ref{lem-eff-res-green} (applied with $A = \mcl V\mcl G^\ep(B_s(z))  $ and $B = \mcl V\mcl G^\ep(B_r(z) ) $) and Lemma~\ref{lem-eff-res-grad} (to say that the error term $\delta$ from Lemma~\ref{lem-eff-res-green} is at most 1), a.s.\ for every such $z$, $r$, and $s$,\footnote{We note that $\bdy\mcl V\mcl G^\ep(B_s(z))$, in the notation~\eqref{eqn-graph-bdy}, is contained in $\mcl V\mcl G^\ep(\bdy B_s(z))$ but the inclusion could be strict since there could be cells of $\mcl G^\ep$ which intersect $\bdy B_s(z)$ but which do not intersect any cells which are contained in $\BB C\setminus B_s(z)$. However, all we need is that the maximum of $\op{gr}_{B_r(z) }^\ep(x_z^\ep ,y) $ over $y\in \mcl V\mcl G^\ep(\bdy B_s(z))$ is bounded below by the maximum over $\bdy\mcl V\mcl G^\ep(B_s(z))$.}
\allb \label{eqn-use-eff-res-green-lower}
\max_{y\in   \mcl V\mcl G^\ep(\bdy B_s(z))} \op{gr}_{B_r(z) }^\ep(x_z^\ep ,y) \geq \mcl R^\ep\left( B_s(z) \leftrightarrow \bdy B_r(z)  \right) - 1  .
\alle

By Lemma~\ref{lem-eff-res-lower}, there exists $C_2 = C_2(\rho,\gamma) > 1$ such that with polynomially high probability as $\ep\rta 0$, it holds for each $z$, $r$, and $s$ as above that
\eqb \label{eqn-use-eff-res-lower}
 \mcl R^\ep\left( B_s(z) \leftrightarrow \bdy B_r(z)   \right)  \geq  C_2^{-1} \log\left( r/s \right)  .
\eqe
We now apply~\eqref{eqn-use-eff-res-lower} to lower-bound the right side of~\eqref{eqn-use-eff-res-green-lower}, then use this to lower-bound the left side of~\eqref{eqn-use-harnack-lower}. This shows that with polynomially high probability as $\ep\rta 0$, it holds for each $z \in B_{\rho-\ep^\omega}$, each $r \in \left[ \ep^\omega ,\op{dist}(z,\bdy B_\rho) \right]$, and each $s\in \left[  \ep^\beta , r/3 \right]$ that
\eqb \label{eqn-gr-bound-lower-z}
 \min_{y\in  \bdy \mcl V\mcl G^\ep(B_s(z))} \op{gr}_{B_r(z) }^\ep(x_z^\ep ,y)  \geq C_1^{-1} \left(  C_2^{-1}  \log(r/s) - 1 \right) \geq (C_1C_2)^{-1} \log(r/s) - 1 .
\eqe
Setting $z = \eta(x)$ and $s = |\eta(x) - \eta(y)|$ in~\eqref{eqn-gr-bound-lower-z} gives~\eqref{eqn-gr-bound0-lower} with $C =  C_1C_2 $ in the special case when $|\eta(x) - \eta(y)| \geq \ep^\beta$. 

In the case when $|\eta(x) - \eta(y)| < \ep^\beta$, we use the maximum principle for the discrete harmonic function $\op{gr}_{B_r(\eta(x))}^\ep(x,\cdot)$, followed by~\eqref{eqn-gr-bound-lower-z}, to get
\eqb \label{eqn-gr-bound-lower-small}
\op{gr}_{B_r(\eta(x) )}^\ep(x , y) \geq \min_{y'  : |\eta(x) - \eta(y')| = \ep^\beta} \op{gr}_{B_r(\eta(x) )}^\ep(x , y' ) \geq (C_1C_2)^{-1} \log \ep^{\omega-\beta} -1 .
\eqe 
We thus obtain~\eqref{eqn-gr-bound0-lower} with $C = (C_1 C_2)^{-1} (\beta - \omega )$. 
\end{proof}

\begin{proof}[Proof of Proposition~\ref{prop-exit-lower}]
By Lemma~\ref{lem-gr-bound-lower}, there exists $\omega=\omega(\gamma ) > 0$ and $C_0 = C_0(\rho,\gamma ) > 0$ such that with polynomially high probability as $\ep\rta 0$, it holds for each $x$ and $r$ as in the proposition statement that
\eqb \label{eqn-use-gr-bound0-lower}
\op{gr}_{B_r(\eta(x) )}^\ep(x , y) \geq C_0^{-1} \log\left( \frac{r}{|\eta(x) - \eta(y)|} \wedge \ep^{-1} \right) - 1  ,
\quad \forall y \in \mcl V \mcl G^\ep(B_{r/3}(\eta(x)))   .
\eqe
We now choose $c := \min\left\{   e^{-2 C_0 },   1/3  \right\}$.
Then if $\ep <e^{-2C_0}$, the first term on the right of the inequality in~\eqref{eqn-use-gr-bound0-lower} is at least 2 whenever $|\eta(x) - \eta(y)| \leq c r$.
Hence~\eqref{eqn-gr-bound0-lower} implies that for each $x$ and $r$ as in the lemma statement,
\eqb \label{eqn-gr-lower-const}
\op{gr}_{B_r(\eta(x) )}^\ep(x , y) \geq 1
\quad \forall y \in \mcl V \mcl G^\ep(B_{c r}(\eta(x))  .
\eqe

Recall from Definition~\ref{def-green} that
\eqbn
\op{Gr}_{B_r(\eta(x) )}^\ep(x,y) = \op{deg}^\ep(y) \op{gr}_{B_r(\eta(x) )}^\ep(x,y) \geq \op{gr}_{B_r(\eta(x) )}^\ep(x,y) .
\eqen
If~\eqref{eqn-gr-lower-const} holds, then for any $x$ and $r$ as in the lemma statement,
\eqb \label{eqn-gr-lower-sum}
\ol{\BB E}_x^\ep\left[  \tau_{B_r(\eta(x)) }^\ep \right]
= \sum_{y\in \mcl V\mcl G^\ep(B_r(\eta(x) )} \op{Gr}_{B_\rho}^\ep(x,y)
\geq  \# \mcl V \mcl G^\ep(B_{c r}(\eta(x)))  .
\eqe
Since each cell of $\mcl V\mcl G^\ep$ has $\mu_h$-mass $\ep$, we have $\# \mcl V \mcl G^\ep(B_{c r}(\eta(x)) \geq \ep^{-1} \mu_h(B_{c r}(\eta(x))$.
Plugging this into~\eqref{eqn-gr-lower-sum} gives~\eqref{eqn-exit-lower}.
\end{proof}

\section{Upper bound for exit times}
\label{sec-exit-upper}

In this section we will prove an upper bound for exit times from Euclidean balls for the random walk on $\mcl G^\ep$.
We will eventually use the $L^q$ version of the Payley-Zygmund inequality for $q  = q(\gamma) \rta \infty$ as $\gamma\rta 2$, so we need a bound for moments of all orders.

\begin{prop} \label{prop-exit-moment}
There exists $\alpha=\alpha(\gamma ) >0$ and $C = C(\rho,\gamma) > 0$ such that with polynomially high probability as $\ep\rta 0$, it holds simultaneously for every Borel set $D\subset B_\rho$ with $\mu_h(D) \geq \ep^{ \alpha}$ and every $x\in \mcl V\mcl G^\ep(D)$ that
\eqb \label{eqn-exit-moment}
 \ol{\BB E}_x^\ep \left[ (\tau_{D}^\ep)^N  \right] \leq  N! C^N \ep^{-N} \left( \sup_{z\in D} \int_{B_{\ep^\alpha}(D)} \log\left(\frac{1}{| z - w|} \right) + 1 \, d\mu_h(w) \right)^N , \quad\forall N\in\BB N.
\eqe
\end{prop}

We note that Proposition~\ref{prop-exit-moment} together with a basic estimate for the $\gamma$-LQG measure (deferred to the appendix) implies the following simpler but less precise estimate.

\begin{cor} \label{cor-exit-moment-eucl}
For each $q\in \left(0 , (2-\gamma)^2/2 \right)$, it holds with probability tending to 1 as $\ep\rta 0$ and then $\delta \rta 0$ that for every $N\in\BB N$,
\eqb
\sup_{z\in B_\rho} \max_{x\in\mcl V\mcl G^\ep(B_\delta(z))} \ol{\BB E}_x^\ep \left[ (\tau_{B_\delta(z)}^\ep)^N  \right]
\leq N! \ep^{-N} \delta^{N q } .
\eqe
\end{cor}
\begin{proof}
By Proposition~\ref{prop-exit-moment}, it holds with probability tending to 1 as $\ep\rta 0$, at a rate depending on $\delta$, that~\eqref{eqn-exit-moment} holds simultaneously with $D = B_\delta(z)$ for each $z\in B_\rho$ and each $x\in \mcl V\mcl G^\ep(B_\delta(z))$.  Note that we can apply Proposition~\ref{prop-exit-moment} with $\rho$ replaced by $\rho' \in (\rho,1)$ to deal with the possibility that $B_\delta(z) \not\subset B_\rho$.

For small enough $\ep$ (depending on $\delta$) we have $\delta + \ep^\alpha \leq 2\delta$, so we can bound the integral over $B_{\ep^\alpha}(B_\delta(z)) = B_{\delta+\ep^\alpha}(z)$ in~\eqref{eqn-exit-moment} by the integral over $B_{2\delta}(z)$.
Upon absorbing the factor of $C^N$ in~\eqref{eqn-exit-moment} into a small power of $\delta$, we have reduced our problem to showing that with polynomially high probability as $\delta\rta 0$,
\eqb \label{eqn-exit-eucl-show}
\sup_{z\in B_\rho} \sup_{u\in B_\delta(z) } \int_{B_{2\delta}(z) } \log\left(\frac{1}{| u - w|} \right) + 1 \, d\mu_h(w) \leq \delta^q .
\eqe
By Lemma~\ref{lem-log-int} (applied with $A  = (q/c_1) \log\delta^{-1}$) it holds with polynomially high probability as $\delta \rta 0$ that the left side of~\eqref{eqn-exit-eucl-show} is bounded above by a $q,\rho,\gamma$-dependent constant times $(\log\delta^{-1}) \sup_{z\in B_\rho} \mu_h(B_{2\delta}(z)) + \delta^q$. We now conclude~\eqref{eqn-exit-eucl-show} by applying Lemma~\ref{lem-ball-mass}.
\end{proof}

The proof of Proposition~\ref{prop-exit-moment} uses the same basic ideas as the proof of Proposition~\ref{prop-exit-lower} except that all of the bounds go in the opposite direction.
Since $\ol{\BB E}_x^\ep \left[ (\tau_{D}^\ep)^N  \right]$ can be expressed in terms of the Green's function $\op{Gr}_D^\ep $ (see Lemma~\ref{lem-second-moment}), we need to establish an upper bound for $\op{Gr}_D^\ep$. In fact, we have $\op{Gr}_D^\ep \leq \op{Gr}_{B_\rho}^\ep$ so we only need an upper bound for $\op{Gr}_{B_\rho}^\ep$.

To prove such an upper bound, we first establish an upper bound for the effective resistance in $\mcl G^\ep$ from a Euclidean ball $B_s(z) \subset B_\rho$ to $\bdy B_\rho$ in Section~\ref{sec-eff-res-upper}.
This is done by constructing a unit flow (via the method of random paths) and applying Thomson's principle~\eqref{eqn-thomson}.
As explained in Section~\ref{sec-off-diag}, this estimate together with Lemma~\ref{lem-eff-res-upper} and Proposition~\ref{prop-harnack} leads to an upper bound for $\op{Gr}_{B_\rho}^\ep(x,y)$ in the case when $|\eta(x) - \eta(y)|$ is not too small (at least some fixed small positive power of $\ep$).

In Section~\ref{sec-on-diag}, we will establish a crude upper bound for $\op{Gr}_{B_\rho}^\ep(x,y)$ which holds uniformly over all $x,y\in \mcl V\mcl G^\ep(B_\rho)$, including pairs for which $|\eta(x) - \eta(y)|$ is small or even $x=y$. This bound will be sufficient for our purposes since we will eventually sum over all $x,y$ so the pairs for which $|\eta(x) - \eta(y)|$ is small will not contribute significantly to the sum.
The proof is again based on Thomson's principle but different estimates are involved.
In Section~\ref{sec-exit-moment} we use our upper bounds for the Green's function to establish Proposition~\ref{prop-exit-moment}.

\subsection{Upper bound for effective resistance across an annulus}
\label{sec-eff-res-upper}

We have the following upper bound for the effective resistance from $\bdy B_s(z)$ to $\bdy B_\rho$.
Together with Lemma~\ref{lem-eff-res-green} and Proposition~\ref{prop-harnack} this will lead to an upper bound for the off-diagonal Green's function on $\mcl G^\ep$.

\begin{lem} \label{lem-eff-res-upper}
There exists $\beta = \beta(\gamma)$ and $C = C(\rho,\gamma) > 0$ such that with polynomially high probability as $\ep\rta 0$, it holds simultaneously for every $z\in B_\rho$ and every $s\in [\ep^\beta,\op{dist}(z,\bdy B_\rho)]$ that
\eqb \label{eqn-eff-res-upper}
\mcl R^\ep\left( B_s(z) \leftrightarrow \bdy B_\rho  \right) \leq C \log s^{-1} .
\eqe
\end{lem}
\begin{proof}
We will prove the lemma by constructing a unit flow $\theta^\ep$ from the vertex $x_z^\ep$ to $\mcl V\mcl G^\ep(\bdy B_\rho)$ and applying Thompson's principle to the restriction of $\theta^\ep$ to $\ol{\mcl V\mcl G^\ep \setminus \mcl V\mcl G^\ep(B_s(z))}$.
The proof is similar to arguments in~\cite[Section 3.3]{gm-spec-dim}.
We will first prove an estimate for a fixed choice of $z\in B_\rho$ and $s\in [0,\op{dist}(z,\bdy B_\rho)]$, then transfer to an estimate which holds for all such $(z,s)$ simultaneously via a union bound argument in Step 4 at the end of the proof.
\medskip

\noindent\textit{Step 1: defining a unit flow.}
We use the method of random paths; see, e.g.,~\cite[Section 2.5, page 42]{lyons-peres} for a general discussion of this method.
For a fixed choice of $z$, let $\BB u$ be sampled uniformly from Lebesgue measure on $\bdy B_1(z)$, independently from everything else.
Consider the infinite ray from $z$ which passes though $\BB u$ and let $S$ be the segment of this ray from $z$ to a point of $\bdy B_\rho$. For $\ep\in (0,1)$, choose (in some measurable manner) a simple path $P^\ep$ in $\mcl G^\ep(S)$ from $x_z^\ep$ to a vertex whose corresponding cell contains the single point in $S\cap \bdy B_\rho$. For an oriented edge $e  $ of $\mcl G^\ep(B_\rho)$, let $\theta^\ep(e)$ be the probability that the path $P^\ep$ traverses $e$ in the forward direction, minus the probability that $P^\ep$ traverses $e$ in the reverse direction.
It is easily seen that $\theta^\ep$ is a unit flow from $x_z^\ep$ to $\mcl V\mcl G^\ep(\bdy B_\rho)$ (this follows from general theory, and is checked carefully in our particular case in~\cite[Lemma 3.5]{gm-spec-dim}).

By Lemma~\ref{lem-unit-flow}, for each $s \in [0,\op{dist}(z,\bdy B_\rho)]$, the restriction of $\theta^\ep $ to the set of oriented edges of  $\mcl G^\ep$ with at least one endpoint in $  \mcl V\mcl G^\ep \setminus \mcl V\mcl G^\ep(B_s(z)) $ is a unit flow from $ \bdy \mcl V\mcl G^\ep( B_s(z)) $ to $\mcl V\mcl G^\ep(\bdy B_\rho)$ (here the boundary of a vertex set is defined as in~\eqref{eqn-graph-bdy}).
Note that each edge in this set also belongs to $\mcl E\mcl G^\ep(\BB C\setminus B_s(z))$.
By Thomson's principle~\eqref{eqn-thomson}, we therefore have
\eqb \label{eqn-eff-res-upper-thomson}
\mcl R^\ep\left( B_s(z) \leftrightarrow \bdy B_\rho  \right) \leq \sum_{e \in \mcl E\mcl G^\ep(B_\rho\setminus B_s(z) )}[ \theta^\ep(e)]^2 ,
\quad \forall s \in [0,\op{dist}(z,\bdy B_\rho)] .
\eqe
\medskip

\noindent\textit{Step 2: bounding the energy in terms of a sum over cells.}
We now bound the right side of~\eqref{eqn-eff-res-upper-thomson}.
The argument is essentially the same as that of~\cite[Lemma 3.6]{gm-spec-dim}, but we give the details for the sake of completeness.
We will bound the energy of $\theta^\ep$ in terms of a sum over cells of $\mcl G^\ep$ which intersect $B_\rho\setminus B_s(z)$, which can in turn be bounded using~\cite[Lemma~3.1]{gms-harmonic}.
Throughout the proof, we require all implicit constants in the symbol $\preceq$ and the rates of all ``polynomially high probability" statements to be deterministic and depend only on $\rho$ and $\gamma$ (not on $z$ or $s$).

By the definition of $S$, for each $x\in \mcl V\mcl G^\ep(B_\rho) \setminus \{x_z^\ep\}$, the conditional probability given $(h,\eta)$ that the segment $S$ intersects the cell $H_x^\ep$ is at most a universal constant times $\op{diam} \left(H_x^\ep \right) \times \op{dist}\left(z , H_x^\ep \right)^{-1}$, where here $\op{diam}$ and $\op{dist}$ denote Euclidean diameter and distance, respectively. By the definition of $\theta^\ep$, this implies that for any edge $\{x,y\} \in \mcl V\mcl G^\ep(B_\rho)$ with $x\not=x_z^\ep$,
\eqb \label{eqn-flow-prob}
|\theta^\ep(x,y)| \leq  \min\left\{ \frac{ \op{diam} \left(H_x^\ep \right) }{ \op{dist}\left(z , H_x^\ep \right) } ,   \frac{ \op{diam} \left(H_y^\ep \right) }{ \op{dist}\left(z , H_y^\ep \right) } \right\}.
\eqe

Now fix $q\in \left(0,\tfrac{2}{(2+\gamma)^2}\right)$, chosen in a manner depending only on $\gamma$.
By Lemma~\ref{lem-cell-diam}, it holds with polynomially high probability as $\ep\rta 0$ that each cell $H_x^\ep$ for $x\in\mcl V\mcl G^\ep(B_\rho)$ has diameter at most $\ep^q$, so
\eqbn
\op{dist}\left(H_x^\ep , z\right)^{-1} \leq \left( |\eta(x) - z| -  \ep^q \right)^{-1} \preceq    |\eta(x) - z|^{-1}  ,\quad \forall x \in \mcl V\mcl G^\ep\left(B_\rho\setminus B_{2\ep^q}(z) \right)  .
\eqen
Plugging this into~\eqref{eqn-flow-prob} shows that
\eqbn
|\theta^\ep(x,y)| \preceq |\eta(x) - z|^{-1} \op{diam} \left( H_x^\ep \right) , \quad \forall \{x,y\} \in \mcl E\mcl G^\ep\left(B_\rho\setminus B_{2\ep^q}(z) \right)
\eqen
By summing this last estimate, we get that if $\beta \in (0,q)$ and $s\in [\ep^\beta , \op{dist}(z, \bdy B_\rho)]$ then whenever~\eqref{eqn-flow-prob} holds and $\ep$ is sufficiently small (depending on $\beta$ and $q$),
\allb \label{eqn-flow-path-sum}
\sum_{e \in \mcl E\mcl G^\ep(B_\rho \setminus B_s(z) )}[ \theta^\ep(e)]^2
 \preceq \sum_{x \in \mcl V\mcl G^\ep(B_\rho \setminus B_s(z) )}  |\eta(x) - z|^{-2}   \op{diam} \left( H_x^\ep \right)^2 \op{deg}^\ep(x) .
\alle
\medskip

\noindent\textit{Step 3: bounding the sum over cells.}
By~\cite[Lemma~3.1]{gms-harmonic}, applied with $f(w) = |w-z|^{-2}$, there exists $\alpha = \alpha(\gamma) > 0$ and $C_0 = C_0(\rho,\gamma) > 0$ such that if $\beta \in (0,q)$ is chosen sufficiently small, in a manner depending only on $\gamma$, then the following is true. For each $s\in [\ep^\beta,\op{dist}(z,\bdy B_\rho)]$, it holds with polynomially high probability as $\ep\rta 0$ that the right side of~\eqref{eqn-flow-path-sum} is bounded above by
\eqb \label{eqn-flow-path-log}
C_0 \int_{B_\rho \setminus B_s(z)} |w-z|^{-2} \, dw    + \ep^\alpha \preceq  \log s^{-1}  .
\eqe
Since~\eqref{eqn-flow-path-sum} holds simultaneously for all $s\in [\ep^\beta, \op{dist}(z,\bdy B_\rho)]$ with polynomially high probability as $\ep\rta 0$, by plugging~\eqref{eqn-flow-path-log} into~\eqref{eqn-flow-path-sum} and then into~\eqref{eqn-eff-res-upper-thomson}, we get that the following is true. There exists $C = C(\rho,\gamma) > 0$ such that for each \emph{fixed} $z\in B_\rho$ and $s\in [\ep^\beta , \op{dist}(z,\bdy B_\rho)]$, the bound~\eqref{eqn-eff-res-upper} holds with polynomially high probability as $\ep\rta 0$.
\medskip

\noindent\textit{Step 4: transferring to a bound for all $(z,s)$ simultaneously.}
We observe that if $B_s(z) \subset B_{s'}(z') \subset B_\rho$, then
\eqb \label{eqn-eff-res-mono}
\mcl R^\ep\left( B_{s'}(z') \leftrightarrow \bdy B_\rho  \right)  \leq \mcl R^\ep\left( B_s(z) \leftrightarrow \bdy B_\rho  \right) .
\eqe
After possibly shrinking $\beta$ so that it is much smaller than the exponent in the ``polynomially high probability" statement at the end of the last step, we can take a union bound over all $z\in B_\rho \cap (\ep^{2\beta}\BB Z^2)$ and all $s\in [\ep^\beta  ,\op{dist}(z,B_\rho)] \cap \{2^{-n}\}_{n\in\BB Z}$ to get that the following is true. With polynomially high probability as $\ep\rta 0$, the bound~\eqref{eqn-eff-res-upper} holds simultaneously for each $z\in B_\rho \cap (\ep^{2\beta}\BB Z^2)$ and all $s\in [\ep^\beta  ,\op{dist}(z,B_\rho)] \cap \{2^{-n}\}_{n\in\BB Z}$.
Due to~\eqref{eqn-eff-res-mono}, this implies that~\eqref{eqn-eff-res-upper} holds simultaneously for all $z\in B_\rho$ and $s\in [0,\op{dist}(z,\bdy B_\rho)]$ with $2C$, say, in place of $C$.
\end{proof}

\subsection{Upper bound for off-diagonal Green's function}
\label{sec-off-diag}

We now use Lemma~\ref{lem-eff-res-green}, Proposition~\ref{prop-harnack}, and Lemma~\ref{lem-eff-res-upper} to prove the following upper bound for the Green's function.

\begin{lem} \label{lem-gr-bound}
There exists $\beta = \beta(\gamma) > 0$ and $C = C(\rho,\gamma ) > 0$ such that with polynomially high probability as $\ep\rta 0$,
\eqb \label{eqn-gr-bound-upper}
   \op{gr}_{B_\rho}^\ep(x,y) \leq C   \log \left( \frac{1}{|\eta(x) - \eta(y)|}  \right) + C
\eqe
simultaneously for all $x,y\in \mcl V\mcl G^\ep(B_\rho)$ with $|\eta(x) - \eta(y)| \geq \ep^\beta$.
\end{lem}

We first prove a variant of Lemma~\ref{lem-gr-bound} where we also impose an upper bound for $|\eta(x) - \eta(y)|$.
The proof is basically the same as that of Lemma~\ref{lem-gr-bound-lower}, but with the inequalities going in the opposite direction.

\begin{lem} \label{lem-gr-bound0}
There exists $\beta = \beta(\gamma) > 0$ and $C = C(\rho,\gamma ) > 0$ such that with polynomially high probability as $\ep\rta 0$,
\eqb \label{eqn-gr-bound0-upper}
 \op{gr}_{B_\rho}^\ep(x,y) \leq C \log \left( \frac{1}{|\eta(x) - \eta(y)|} \right) .
\eqe
simultaneously for all $x,y\in \mcl V\mcl G^\ep(B_\rho)$ with $|\eta(x) - \eta(y)| \in \left[  \ep^\beta , \frac13 \op{dist}(\eta(x) , \bdy B_\rho) \right]$.
\end{lem}
\begin{proof}
Since the function $y\mapsto \op{gr}_{B_\rho}^\ep(x,y)$ is discrete harmonic on $\mcl V\mcl G^\ep(B_\rho) \setminus (\mcl V\mcl G^\ep(\bdy B_\rho)\cup \{x\})$, we can apply Proposition~\ref{prop-harnack} to find that there exists $\beta = \beta(\gamma) > 0$ and $C_1 = C_1(\rho,\gamma) > 1$ such that with polynomially high probability as $\ep\rta 0$, it holds simultaneously for each $z\in B_\rho$ and each $s\in \left[  \ep^\beta , \frac13 \op{dist}(z , \bdy B_\rho) \right]$ that
\eqb \label{eqn-use-harnack}
 \max_{y\in  \bdy \mcl V\mcl G^\ep(B_s(z))} \op{gr}_{B_\rho}^\ep(x_z^\ep ,y)  \leq C_1 \min_{y\in  \bdy \mcl V\mcl G^\ep(B_s(z))} \op{gr}_{B_\rho}^\ep(x_z^\ep ,y)  .
\eqe
By Lemma~\ref{lem-eff-res-green} (applied with $A = \mcl V\mcl G^\ep(B_s(z))  $ and $B = \mcl V\mcl G^\ep(B_\rho) $) and Lemma~\ref{lem-eff-res-grad} (to say that the error term $\delta$ from Lemma~\ref{lem-eff-res-green} is at most 1), a.s.\ for every such $z$ and $s$,
\allb \label{eqn-use-eff-res-green-upper0}
\frac{a}{1 + 1/a} =  \frac{a^2}{a+1}    \leq   \mcl R^\ep\left( B_s(z) \leftrightarrow \bdy B_\rho   \right) ,\quad \text{for} \quad a =   \min_{y\in  \bdy \mcl V\mcl G^\ep(B_s(z))} \op{gr}_{B_\rho}^\ep(x_z^\ep ,y)  .
\alle
Re-arranging the inequality~\eqref{eqn-use-eff-res-green-upper0} gives
\allb \label{eqn-use-eff-res-green-upper}
  \min_{y\in  \bdy \mcl V\mcl G^\ep(B_s(z))} \op{gr}_{B_\rho}^\ep(x_z^\ep ,y)     \leq  \left( 1 +  \frac{1}{  \min_{y\in  \bdy \mcl V\mcl G^\ep(B_s(z))} \op{gr}_{B_\rho}^\ep(x_z^\ep ,y)  } \right)  \mcl R^\ep\left( B_s(z) \leftrightarrow \bdy B_\rho   \right) .
\alle
\medskip

By Lemma~\ref{lem-eff-res-upper}, after possibly shrinking $\beta$ we can find $C_2 = C_2(\rho,\gamma)> 1$ such that with polynomially high probability as $\ep\rta 0$, it holds simultaneously for each $z\in B_\rho$ and each $s\in \left[\ep^\beta , \frac13\op{dist}(\eta(x) , \bdy B_\rho)\right]$ that
\eqb \label{eqn-use-eff-res-upper}
 \mcl R^\ep\left( B_s(z) \leftrightarrow \bdy B_\rho  \right) \leq C_2 \log s^{-1}  .
\eqe
We now use~\eqref{eqn-use-eff-res-upper} to upper-bound the left side of~\eqref{eqn-use-eff-res-green-upper}, but some care is needed due to the term $ \min_{y\in  \bdy \mcl V\mcl G^\ep(B_s)} \op{gr}_{B_\rho}^\ep(x_z^\ep ,y)$ on the right side.
This is dealt with as follows.
If $\min_{y\in  \bdy \mcl V\mcl G^\ep(B_s(z))} \op{gr}_{B_\rho}^\ep(x_z^\ep ,y) $ were larger than $2C_2 \log s^{-1} \ge 1$, then the right side of~\eqref{eqn-use-eff-res-green-upper} would be at most $ C_2 \log s^{-1} +1 \le 2 C_2 \log (s^{-1})$, so either way we must in fact have
\eqb \label{eqn-green-function-upper}
 \min_{y\in  \mcl V\mcl G^\ep(B_s)} \op{gr}_{B_\rho}^\ep(x_z^\ep ,y)   \leq 2C_2   \log s^{-1} .
\eqe
Combining~\eqref{eqn-use-harnack} and~\eqref{eqn-green-function-upper}, each applied with $z =\eta(x)$ and $s = |\eta(x) -\eta(y)|$, now gives~\eqref{eqn-gr-bound0-upper} with $C = 2C_1C_2$.
\end{proof}

\begin{proof}[Proof of Lemma~\ref{lem-gr-bound}]
We need to remove the constraint that $|\eta(x) - \eta(y)| \leq \frac13 \op{dist}(\eta(x) , \bdy B_\rho)$.
Set $\rho' := (1+\rho)/2$.
Obviously, $\op{gr}_{B_{\rho'}}^\ep(x,y) \geq \op{gr}_{B_\rho}^\ep(x,y)$ for each $x,y\in \mcl V\mcl G^\ep(B_\rho)$, so it suffices to prove~\eqref{eqn-gr-bound-upper} with $\op{gr}_{B_{\rho'}}^\ep$ in place of $\op{gr}_{B_\rho}^\ep$.
This will allow us to avoid dealing with boundary effects (which are manifested in the requirement that $|\eta(x) - \eta(y)| \leq \frac13 \op{dist}(\eta(x) , \bdy B_\rho)$ in Lemma~\ref{lem-gr-bound0}).

By Lemma~\ref{lem-gr-bound0} with $\rho'$ in place of $\rho$, there exists $\beta$ and $C$ as in the lemma statement such that with polynomially high probability as $\ep\rta 0$,
\eqb \label{eqn-use-gr-bound0}
 \op{gr}_{B_{\rho'}}^\ep(x,y) \leq C \log \left( \frac{1}{|\eta(x) - \eta(y)|} \right)
\eqe
for all $x,y\in \mcl V\mcl G^\ep(B_{\rho'})$ with $|\eta(x) - \eta(y)| \in \left[  \ep^\beta , \frac13 \op{dist}(\eta(x) , \bdy B_{\rho'}) \right]$.
By Lemma~\ref{lem-cell-diam}, there exists $q = q(\gamma) > 0$ such that with polynomially high probability as $\ep\rta 0$, we have $\eta(x) \in B_{\rho + \ep^q}$ for each $x\in \mcl V\mcl G^\ep(B_{\rho })$.
If this is the case, then for small enough $\ep > 0$ (how small depends only on $\rho,\gamma$),
\eqb
 \frac13 \op{dist}(\eta(x) , \bdy B_{\rho'}) \geq \frac{\rho'  -\rho - \ep^q}{3}   \geq \frac{\rho'-\rho}{6} .
\eqe
Therefore, with polynomially high probability as $\ep\rta 0$, the bound~\eqref{eqn-use-gr-bound0} holds for all $x,y\in \mcl V\mcl G^\ep(B_{\rho'})$ with $|\eta(x) - \eta(y)| \in \left[  \ep^\beta , (\rho'-\rho)/6 \right]$.
Henceforth assume that this is the case.

Since $\op{gr}_{B_{\rho'}}^\ep(x,\cdot)$ is discrete harmonic on $\mcl V\mcl G^\ep(B_{\rho'}) \setminus \{x\}$ and vanishes on $\mcl V\mcl G^\ep\setminus \mcl V\mcl G^\ep( B_{\rho'})$, we can apply~\eqref{eqn-use-gr-bound0} and the maximum principle to get that for $x,y\in\mcl V\mcl G^\ep(B_\rho)$ with $|\eta(x) - \eta(y)|  > (\rho'-\rho)/6 $,
\eqb
\op{gr}_{B_\rho}^\ep(x,y)
\leq C \log \left( \frac{6}{\rho'-\rho} \right)   .
\eqe
Combining this with~\eqref{eqn-use-gr-bound0} gives~\eqref{eqn-gr-bound-upper} with $C \log \left( \frac{6}{\rho'-\rho} \right)$ in place of $C$.
\end{proof}

\subsection{Upper bound for on-diagonal Green's function}
\label{sec-on-diag}

Lemma~\ref{lem-gr-bound} only provides an upper bound for $\op{Gr}_{B_\rho}^\ep(x,y)$ in the case when $|\eta(x) - \eta(y)| \geq \ep^\beta$.
The purpose of this subsection is to prove a crude upper bound for $\op{Gr}_{B_\rho}^\ep(x,y)$ which will deal with the case when $|\eta(x) - \eta(y)| \leq \ep^\beta$.

\begin{lem} \label{lem-eff-res-diag}
For each $\zeta \in (0,1)$, it holds polynomially high probability as $\ep\rta 0$ that
\eqb
\max_{x,y\in \mcl V\mcl G^\ep(B_\rho)} \op{Gr}_{B_\rho}^\ep(x,y) \leq \ep^{-\zeta}   .
\eqe
\end{lem}

Since $ \op{Gr}_{B_\rho}^\ep(x,y) $ attains its maximum value when $x=y$, to prove Lemma~\ref{lem-eff-res-diag}, we only need to prove an upper bound for $\op{Gr}_{B_\rho}^\ep(x,x)$. Since degrees in $\mcl G^\ep$ are easy to control, this amounts to proving an upper bound for the effective resistance from $x\in\mcl V\mcl G^\ep(B_\rho)$ to $ \mcl V\mcl G^\ep(\bdy B_\rho)$.
The proof of this upper bound is based on Thomson's principle applied to the same unit flow used in the proof of Lemma~\ref{lem-eff-res-upper}.
However, the estimates involved are somewhat more delicate.
The following lemma is one of the key ingredients of the proof.

\begin{lem} \label{lem-cell-comparable}
For each $\zeta \in (0,1)$, it holds with polynomially high probability as $\ep\rta 0$ that the following is true.
Simultaneously for each $z\in B_\rho$, the number of vertices $x \in \mcl V\mcl G^\ep$ whose corresponding cell satisfies
\eqb \label{eqn-cell-comparable}
\op{dist}\left(z,H_x^\ep \right) \leq \op{diam}\left( H_x^\ep \right)
\eqe
is at most $\ep^{- \zeta} $.
\end{lem}
\begin{proof}
By Lemma~\ref{lem-cell-diam}, there are constants $0 < q < p < \infty$ depending only on $\gamma$ such that with polynomially high probability as $\ep\rta 0$,
\eqb \label{eqn-comp-max-cell-diam}
\ep^p \leq \op{diam}(H_x^\ep) \leq \ep^q ,\quad\forall x\in\mcl V\mcl G^\ep(B_\rho) .
\eqe
By~\eqref{eqn-comp-max-cell-diam} together with~\cite[Proposition 3.4 and Remark 3.9]{ghm-kpz}, it holds with polynomially high probability as $\ep\rta 0$ that
\eqb  \label{eqn-comp-inrad}
\text{$H_x^\ep$ contains a Euclidean ball of radius at least $\ep^{\zeta/4} \op{diam}\left( H_x^\ep \right)$} ,
\quad \forall x\in \mcl V\mcl G^\ep(B_\rho) .
\eqe
Henceforth assume that the estimates~\eqref{eqn-comp-max-cell-diam} and~\eqref{eqn-comp-inrad} hold.
We will show that the event in the lemma statement occurs.

To this end, fix $z\in B_\rho$ and for $k\in\BB N$ let $\mcl C_k$ be the set of $x\in\mcl V\mcl G^\ep$ such that $\op{dist}\left( z , H_x^\ep \right) \leq \op{diam}\left( H_x^\ep \right)$ and $\op{diam}(H_x^\ep) \in [2^{-k - 1}, 2^{-k }]$.
By~\eqref{eqn-comp-max-cell-diam}, $\mcl C_k = \emptyset$ for $k\notin [\log_2 \ep^{-q} , \log_2 \ep^{-p}]_{\BB Z}$.
Therefore, we only need to show that
\eqb \label{eqn-cell-comparable-decomp}
\sum_{k= \lfloor \log_2 \ep^{-q} \rfloor  }^{ \lceil \log_2 \ep^{-p} \rceil } \#\mcl C_k \leq \ep^{-\zeta } .
\eqe

For each $x\in \mcl C_k$, the cell $H_x^\ep$ is contained in $B_{2^{-k+1}}(z)$.
Since the cells $x\in\mcl V\mcl G^\ep$ intersect only along their boundaries, the Euclidean balls of~\eqref{eqn-comp-inrad} for different choices of $x$ are disjoint.
Therefore,
\eqb \label{eqn-cell-comparable-sum}
\pi 2^{-2k+2}
= \op{area} B_{2^{-k+1}}(z)
\geq  \pi \sum_{x\in \mcl C_k} \left[ \ep^{\zeta/4} \op{diam}\left( H_x^\ep \right)\right]^2
\geq \pi \ep^{\zeta/2} 2^{-2k-2} \#\mcl C_k ,
\eqe
where in the last inequality we used that  $\op{diam}(H_x^\ep)^2  \geq 2^{-k-1}$ for each $x\in\mcl C_k$.
By re-arranging~\eqref{eqn-cell-comparable-sum} we get $\#\mcl C_k \leq 16 \ep^{- \zeta/2}$.
Summing this estimate over the logarithmically many values of $k\in [\log_2 \ep^{-q} , \log_2 \ep^{-p}]_{\BB Z}$ now shows that~\eqref{eqn-cell-comparable-decomp} holds provided $\ep$ is at most some $\zeta$-dependent constant.
\end{proof}

\begin{proof}[Proof of Lemma~\ref{lem-eff-res-diag}]
\noindent\textit{Step 1: regularity events.}
We first define a high-probability regularity event which we will truncate on throughout the proof.
After this step, all of the arguments are entirely deterministic.
As explained at the beginning of the proof of Lemma~\ref{lem-cell-comparable}, it holds with polynomially high probability as $\ep\rta 0$ that the bounds~\eqref{eqn-comp-max-cell-diam} and~\eqref{eqn-comp-inrad} both hold.
By~\cite[Lemma 2.6]{gms-harmonic}, it holds with superpolynomially high probability as $\ep\rta 0$ that
\eqb \label{eqn-diag-max-deg}
\max_{x\in\mcl V\mcl G^\ep(B_\rho)} \op{deg}^\ep(x) \leq \ep^{-\zeta} ,\quad\forall x \in \mcl V\mcl G^\ep(B_\rho) .
\eqe
By Lemma~\ref{lem-cell-comparable}, it holds with polynomially high probability as $\ep\rta 0$ that
\eqb \label{eqn-diag-comparable}
\# \left\{x\in\mcl V\mcl G^\ep(B_\rho) : \op{dist}\left(z,H_x^\ep \right) \leq \op{diam}\left( H_x^\ep \right) \right\} \leq \ep^{-\zeta} ,\quad\forall z\in B_\rho .
\eqe
Henceforth assume that~\eqref{eqn-comp-max-cell-diam}, \eqref{eqn-comp-inrad}, \eqref{eqn-diag-max-deg}, and~\eqref{eqn-diag-comparable} all hold, which happens with polynomially high probability as $\ep\rta 0$.
\medskip

\noindent\textit{Step 2: bounding the effective resistance in terms of a unit flow.}
For $x\in\mcl V\mcl G^\ep$, we let $\frk u^\ep(x)$ be a point of the cell $H_x^\ep$ chosen in such a way that $\frk u^\ep(x)$ lies at maximal Euclidean distance from $\bdy H_x^\ep$, so that by~\eqref{eqn-comp-inrad},
\eqb \label{eqn-inrad-pt}
\op{dist}\left( \frk u^\ep(x) , \bdy H_x^\ep \right)  \geq       \ep^{\zeta/4} \op{diam}\left( H_x^\ep \right)  .
\eqe
For $x\in\mcl V\mcl G^\ep(B_\rho)$, let $\theta_x^\ep$ be the flow defined in the proof of Lemma~\ref{lem-eff-res-upper} with $z = \frk u^\ep(x)$.
As explained in that lemma, $\theta_x^\ep$ is a unit flow from $x$ to $ \mcl V\mcl G^\ep(\bdy B_\rho)$ and hence by Thomson's principle~\eqref{eqn-thomson},
\eqb \label{eqn-diag-thomson}
\op{gr}_{B_\rho}^\ep(x,x)
=  \mcl R^\ep\left(x \leftrightarrow   \mcl V\mcl G^\ep(\bdy B_\rho) \right)
\leq \sum_{e\in \mcl E\mcl G^\ep(B_\rho)} [\theta^\ep(e)]^2 .
\eqe
Furthermore, by~\eqref{eqn-flow-prob}, for each oriented edge $(y,y')$ of $\mcl G^\ep(B_\rho)$ with $y\not=x$, we have
\eqb \label{eqn-diag-flow-prob}
|\theta_x^\ep(y,y')|
\preceq  \frac{ \op{diam} \left( H_y^\ep \right)}{ \op{dist}\left( \frk u^\ep(x) , H_y^\ep \right) },
\eqe
with a universal implicit constant.

We want to bound the sum on the right side of~\eqref{eqn-diag-thomson} using~\eqref{eqn-diag-flow-prob}.
We define the set of ``good" vertices $\mcl S_x^\ep$ to be the set of $y\in\mcl V\mcl G^\ep(B_\rho)$ such that $\op{dist}\left(   \frk u^\ep(x) , H_y^\ep \right) > \op{diam}\left( H_y^\ep \right)$.
We write $\mcl E\mcl S_x^\ep$ to be the set of edges $e\in\mcl E\mcl G^\ep(B_\rho)$ which have at least one endpoint in $\mcl S_x^\ep$.
We will bound the sum of $[\theta_x^\ep(e)]^2$ over the set of ``bad" edges $\mcl E\mcl G^\ep(B_\rho)\setminus \mcl E\mcl S_x^\ep$ and the set of ``good" edges $   \mcl E\mcl S_x^\ep$ separately.
The idea is that~\eqref{eqn-diag-flow-prob} gives a useful bound for the sum over good edges, whereas~\eqref{eqn-diag-comparable} tells us that there cannot be too many bad edges.
\medskip

\noindent\textit{Step 3: bounding the sum over the bad edges.}
By~\eqref{eqn-diag-comparable} with $z = \frk u^\ep(x)$, the total number of ``bad" vertices satisfies $\#(\mcl V\mcl G^\ep(B_\rho) \setminus \mcl S_x^\ep ) \leq \ep^{-\zeta}$.
By combining this with~\eqref{eqn-diag-max-deg}, we have $\#(\mcl E\mcl G^\ep(B_\rho)\setminus \mcl E\mcl S_x^\ep) \leq \ep^{-2\zeta}$.
By the definition of $\theta_x^\ep$, we have $|\theta_x^\ep(e)| \leq 1$ for each $e\in\mcl E\mcl G^\ep(B_\rho)$.
Therefore,
\eqb \label{eqn-diag-bad}
\sum_{e\in \mcl E\mcl G^\ep(B_\rho)\setminus \mcl E\mcl S_x^\ep} [\theta_x^\ep(e)]^2 \leq \ep^{-2\zeta} .
\eqe
\medskip

\noindent\textit{Step 4: bounding the sum over the good edges.}
We eventually want to compare $\sum_{e\in \mcl E\mcl S_x^\ep} [\theta_x^\ep(e)]^2$ to a deterministic integral using~\eqref{eqn-diag-flow-prob}.
We start by summing~\eqref{eqn-diag-flow-prob} over all edges in $\mcl E\mcl S_x^\ep$ to get
\allb \label{eqn-eff-res-diag-sum}
\sum_{e\in\mcl E\mcl S_x^\ep} [\theta_x^\ep(e)]^2
&\leq   \sum_{y \in \mcl S_x^\ep} \frac{ \op{diam} \left( H_y^\ep \right)^2 \op{deg}^\ep(y) }{ \op{dist}\left(  \frk u^\ep(x) , H_y^\ep\right)^2 }  \notag  \quad \text{by~\eqref{eqn-diag-flow-prob}} \\
&\leq  \ep^{-  3\zeta  /2 }  \sum_{y \in \mcl S_x^\ep} \frac{ \op{dist}\left( \frk u^\ep(y) , \bdy H_y^\ep) \right)^2 }{ \op{dist}\left(  \frk u^\ep(x) , H_y^\ep \right)^2   } \quad \text{by \eqref{eqn-comp-inrad} and~\eqref{eqn-diag-max-deg}} .
\alle

If $y\in \mcl S_x^\ep$, then by definition $\op{dist}\left(   \frk u^\ep(x) , H_y^\ep \right) > \op{diam}\left( H_y^\ep \right)$, so for each $w\in H_y^\ep$,
\eqb \label{eqn-center-to-generic}
|w - \frk u^\ep(x)|
\leq  \op{dist}\left(   \frk u^\ep(x) , H_y^\ep \right)  +  \op{diam}\left( H_y^\ep \right)
\leq 2 \op{dist}\left(   \frk u^\ep(x) , H_y^\ep \right) .
\eqe
By applying~\eqref{eqn-center-to-generic} to lower-bound the denominator, then averaging over all $w\in H_y^\ep$, we see that each term in the sum on the right side of~\eqref{eqn-eff-res-diag-sum} satisfies
\allb \label{eqn-eff-res-diag-term}
\frac{ \op{dist}\left( \frk u^\ep(y) , \bdy H_y^\ep) \right)^2 }{ \op{dist}\left(  \frk u^\ep(x) , H_y^\ep \right)^2   }
\leq 4 \int_{H_y^\ep} \frac{ \op{dist}\left( \frk u^\ep(y) , \bdy H_y^\ep) \right)^2 }{ \op{area}(H_y^\ep) |w - \frk u^\ep(x)|^2    } \,dw
\leq  \frac{4}{\pi} \int_{H_y^\ep} \frac{ 1 }{   |w - \frk u^\ep(x)|^2    } \,dw
\alle
where in the last inequality we used that $\op{area}(H_y^\ep)  \geq \pi \op{dist}\left( \frk u^\ep(y) , \bdy H_y^\ep) \right)^2 $.
By plugging the inequality~\eqref{eqn-eff-res-diag-term} into the right side of~\eqref{eqn-eff-res-diag-sum}, we get
\allb \label{eqn-eff-res-diag-int}
\sum_{e\in\mcl E\mcl S_x^\ep} [\theta_x^\ep(e)]^2
\preceq  \ep^{-  3\zeta  /2 } \int_{\BB D\setminus H_x^\ep} \frac{1}{ |w- \frk u^\ep(x)|^2    } \,dw   ,
\alle
with a universal implicit constant.

By the lower bound in~\eqref{eqn-comp-max-cell-diam} and the definition of $\frk u^\ep(x)$,
\eqb \label{eqn-min-inrad}
\op{dist}\left( \frk u^\ep(x) , \bdy H_x^\ep) \right) \geq \ep^{p+\zeta/4}
\quad \text{which implies} \quad
B_{\ep^{p+\zeta/4}}(\frk u^\ep(x)) \subset H_x^\ep
,\quad\forall x\in \mcl V\mcl G^\ep(B_\rho) .
\eqe
Combining~\eqref{eqn-eff-res-diag-int} and~\eqref{eqn-min-inrad} shows that
\eqb \label{eqn-diag-good}
\sum_{e\in\mcl E\mcl S_x^\ep} [\theta_x^\ep(e)]^2
\preceq  \ep^{-  3\zeta  /2 } \int_{\BB D\setminus B_{\ep^{p+\zeta/4}}(\frk u^\ep(x) ) } \frac{1}{ |w- \frk u^\ep(x)|^2    } \,dw
\preceq  \ep^{-  3\zeta  /2 }  \log \ep^{-1}    ,
\eqe
with the implicit constant depending only on $\rho , \gamma$.
\medskip

\noindent\textit{Step 5: conclusion.}
Plugging~\eqref{eqn-diag-bad} and~\eqref{eqn-diag-good} into~\eqref{eqn-diag-thomson} gives $\op{gr}_{B_\rho}^\ep(x,x) \preceq \ep^{-3\zeta/2}  \log \ep^{-1} $ with the implicit constant depending only on $\gamma$.
Since $\op{gr}_{B_\rho}^\ep(x,\cdot)$ is discrete harmonic on $\mcl V\mcl G^\ep(B_\rho) \setminus (\mcl V\mcl G^\ep(\bdy B_\rho) \cup \{x\})$ and vanishes on $\mcl V\mcl G^\ep(\bdy B_\rho)$, the maximum principle implies that for any $x,y\in \mcl V\mcl G^\ep(x)$, we have $\op{gr}_{B_\rho}^\ep(x,y) \leq \op{gr}_{B_\rho}^\ep(x,x) \preceq \ep^{-3\zeta/2}  \log \ep^{-1}   $.
Since $\op{Gr}_{B_\rho}^\ep(x,y) = \op{deg}^\ep(x) \op{gr}^\ep_{B_\rho}(x,y)$, we see from this and and~\eqref{eqn-diag-max-deg} that $\op{Gr}_{B_\rho}^\ep(x,y) \preceq \ep^{-5\zeta/2} \log\ep^{-1}$ for all $x,y\in\mcl V\mcl G^\ep(B_\rho)$.
Since $\zeta \in (0,1)$ is arbitrary, this concludes the proof.
\end{proof}

\subsection{Proof of Proposition~\ref{prop-exit-moment}}
\label{sec-exit-moment}

The following estimate is the main input in the proof of Proposition~\ref{prop-exit-moment}.

\begin{lem} \label{lem-gr-sum}
There are constants $\alpha=\alpha(\gamma) > 0$ and $C = C(\rho,\gamma)  >0$ such that with polynomially high probability as $\ep\rta 0$, it holds simultaneously for every Borel set $D\subset B_\rho$ and every $x\in \mcl V\mcl G^\ep(D)$ that
\eqb \label{eqn-gr-sum}
  \sum_{y \in \mcl V\mcl G^\ep(D)} \op{Gr}_D^\ep(x , y)
\leq C \ep^{-1} \int_{B_{\ep^\alpha}(D) } \log\left(\frac{1}{| \eta(x) - w|} \right) + 1 \, d\mu_h(w) + \ep^{-1+\alpha}  .
\eqe
\end{lem}
\begin{proof}
Note that the Green's function is increasing in the domain, so for any $D\subset B_\rho$,
\eqb \label{eqn-green-mono}
\op{Gr}_D^\ep(x,y) \leq \op{Gr}_{B_\rho}^\ep(x,y),\quad\forall x,y\in \mcl V\mcl G^\ep(D) .
\eqe
Hence our upper bounds for $\op{Gr}_{B_\rho}^\ep$ automatically give upper bounds for $\op{Gr}_D^\ep$.

Fix $\beta   ,\zeta  > 0$, to be chosen later in a manner depending only on $\gamma$.
For a Borel set $D\subset B_\rho$ and $x\in\mcl V\mcl G^\ep(D)$, we break the sum on the left side of~\eqref{eqn-gr-sum} into two pieces:
\alb
Y_{D,x}^{\op{near}} &:= \left\{ y\in \mcl V\mcl G^\ep(D) : |\eta(x ) -\eta(y)| <  \ep^\beta \right\} \notag \\
Y_{D,x}^{\op{typical}} &:= \mcl V\mcl G^\ep(D) \setminus  Y_{D,x}^{\op{near}}  .
\ale

Let us first bound the sum over $Y_{D,x}^{\op{near}}$.
By Lemma~\ref{lem-eff-res-diag}, it holds with polynomially high probability as $\ep\rta 0$ that
\eqb \label{eqn-max-gr}
\max_{x,y\in \mcl V\mcl G^\ep(B_\rho)} \op{Gr}_{B_\rho}^\ep(x,y) \leq \ep^{-\zeta}   .
\eqe
Since the cells of $\mcl G^\ep$ have $\mu_h$-mass $\ep$, it is easily seen from Lemmas~\ref{lem-ball-mass} and~\ref{lem-cell-diam} that there exists $\alpha_0 = \alpha_0( \beta , \gamma) > 0$ such that with polynomially high probability as $\ep\rta 0$, we have $\sup_{z\in B_\rho} \#\mcl V\mcl G^\ep(B_{\ep^\beta}(z)) \leq \ep^{-1+\alpha_0}$, which implies that
\eqb \label{eqn-near-count}
\# Y_{D,x}^{\op{near}}  \leq \ep^{-1+\alpha_0} ,\quad \forall D \subset B_\rho , \quad \forall x \in \mcl V\mcl G^\ep(B_\rho) .
\eqe
By~\eqref{eqn-max-gr} and~\eqref{eqn-near-count} we get that with polynomially high probability as $\ep\rta 0$,
\eqb \label{eqn-gr-sum-near}
\sum_{y \in Y_{D,x}^{\op{near}} } \op{Gr}_{B_\rho}(x , y) \leq \ep^{- 1+\alpha_0-\zeta } ,  \quad\forall D \subset B_\rho , \quad \forall x\in\mcl V\mcl G^\ep(D) .
\eqe
Henceforth assume that $\zeta  < \alpha_0$.

We now turn our attention to $Y_{D,x}^{\op{typical}}$.
By Lemma~\ref{lem-gr-bound}, there exists $C_0 = C_0(\rho,\gamma) > 0$ such that if $\beta  = \beta(\gamma ) >0$ is chosen to be sufficiently small, then with polynomially high probability as $\ep\rta 0$,
\allb \label{eqn-gr-bound}
 \op{gr}_{B_\rho}^\ep(x,y)
 \leq C_0  \log \left( \frac{1}{|\eta(x) - \eta(y)|}  \right)  +C_0 ,
 \quad\forall x,y\in \mcl V\mcl G^\ep( B_\rho ) \quad \text{with} \quad |\eta(x) - \eta(y)| \geq \ep^\beta .
\alle
For $y\in Y_{D,x}^{\op{typical}}$, we have $|\eta(x) - \eta(y)| \geq \ep^\beta$ by definition, so~\eqref{eqn-gr-bound} implies that
\eqb \label{eqn-gr-upper-typical}
  \op{gr}_{B_\rho}^\ep(x,y)  \leq   f_x(\eta(y) ) ,\quad
  \forall y\in Y_{D,x}^{\op{typical}} ,
   \quad \text{where} \quad
  f_x(w) :=  \frac{C_0}{ \log\left( |\eta(x) - w|^{-1} \wedge \ep^{-\beta} \right) } + C_0 .
\eqe

The absolute value of $f_x$ is bounded above by $C\log(\ep^{-\beta})$ on $B_\rho$.
Furthermore, it is easily seen by differentiating that $f_x$ is Lipschitz continuous on $B_\rho$ with Lipschitz constant at most a $C$-dependent constant times $\ep^{-\beta}$.
Therefore, Lemma~\ref{lem-cell-sum-deg} implies that there exists $\alpha_1 = \alpha_1(\gamma)  >0$ such that if $\beta = \beta(\gamma)  >0$ is chosen to be sufficiently small, then with polynomially high probability as $\ep\rta 0$, it holds simultaneously for every Borel set $D\subset B_\rho$ and every $x\in \mcl V\mcl G^\ep(D)$ that
\allb \label{eqn-gr-sum-typical}
\sum_{y \in Y_{D,x}^{\op{typical}}} \op{Gr}_{B_\rho}^\ep(x , y)
&\leq \sum_{y\in \mcl V\mcl G^\ep(D)} \op{deg}^\ep(y) f_x(\eta(y)) \quad \text{by~\eqref{eqn-gr-upper-typical} and~\eqref{eqn-green-deg}} \notag \\
&\leq 6 \ep^{-1} \int_{B_{\ep^{\alpha_1}}(D) } f_x(w) \,d\mu_h( w) + \ep^{-1+\alpha_1} \quad \text{by Lemma~\ref{lem-cell-sum-deg}} \notag \\
&\leq 6 C_0 \ep^{-1} \int_{B_{\ep^{\alpha_1}}(D) } \log\left(\frac{1}{|\eta(x) - w|} \right) + 1 \,d\mu_h( w) + \ep^{-1+\alpha_1} \quad \text{by~\eqref{eqn-gr-upper-typical}}.
\alle

Combining~\eqref{eqn-gr-sum-near} and~\eqref{eqn-gr-sum-typical} and recalling~\eqref{eqn-green-mono} gives the statement of the lemma with $C = 6C_0$ and $\alpha := \min\{\alpha_0-\zeta,\alpha_1\}$.
\end{proof}

We now prove an elementary lemma for general graphs which will allow us to deduce moment bounds for exit times of the random walk on $\mcl G^\ep$ from Lemma~\ref{lem-gr-sum}.

\begin{lem} \label{lem-second-moment}
Let $G$ be a graph, let $B\subset\mcl V(G)$, and let $x\in B$.
Let $\tau_B$ be the exit time from $B$ of the simple random walk on $G$ started from $x$.
Then for each $N \in \BB N$,
\eqb \label{eqn-second-moment}
\BB E\left[  \tau_B^N \right] \leq   N!  \sum_{y_1,\dots,y_N \in B} \op{Gr}_B(x,y_1) \prod_{n=2}^N \op{Gr}_B(y_{n-1} , y_n ).
\eqe
\end{lem}
\begin{proof}
We have
\eqbn
\tau_B =  \sum_{y \in B} \sum_{j=0}^\infty \BB 1_{(j  < \tau_B , \, X_j = y)}  .
\eqen
By expanding out the sum, we get
\allb \label{eqn-second-moment-expand}
\BB E\left[ \tau_B^N \right]
&= \sum_{y_1,\dots,y_N \in B}  \sum_{ j_1 , \dots , j_N \in \BB N_0}  \BB P\left[ j_N < \tau_B ,\, X_{j_1} = y_1 , \dots , X_{j_N} = y_N \right] \notag\\
&\leq  N! \sum_{y_1,\dots,y_N \in B}  \sum_{0 \leq j_1 \leq \dots \leq j_N <\infty}  \BB P\left[ j_N < \tau_B ,\, X_{j_1} = y_1 , \dots , X_{j_N} = y_N \right] .
\alle
For $y \in B$, let $X^y$ denote the simple random walk on $G$ started from $y$ and let $\tau_B^y$ be its exit time from $B$ (so that $X = X^x$ and $\tau_B = \tau_B^x$).
By $N$ applications of the strong Markov property, we obtain
\eqb \label{eqn-second-moment-markov}
 \BB P\left[ j_N < \tau_B ,\, X_{j_1} = y_1 , \dots , X_{j_N} = y_N \right]
= \BB P\left[ j_1 < \tau_B ,\,  X_{j_1} = y_1 \right] \prod_{n=2}^N \BB P\left[ j_n - j_{n-1} < \tau_B^{y_{n-1}} , X^{y_{n-1}}_{j_n-j_{n-1}} = y_j \right] .
\eqe
Note that we allow $j_n = j_{n-1}$ for some $n\in[1,N]_{\BB Z}$, in which case the $n$th factor in the product on the right side of~\eqref{eqn-second-moment-markov} is either zero or one according to whether $y_n = y_{n-1}$.
By~\eqref{eqn-second-moment-markov}, the right side of~\eqref{eqn-second-moment-expand} equals
\allb
N! \sum_{y_1,\dots,y_N \in B}  \sum_{0 \leq j_1 \leq \dots \leq j_N <\infty}   \BB P\left[ j_1 < \tau_B ,\,  X_{j_1} = y_1 \right] \prod_{n=2}^N \BB P\left[ j_n - j_{n-1} < \tau_B^{y_{n-1}} , X^{y_{n-1}}_{j_n-j_{n-1}} = y_n \right]  .
\alle
Summing over $j_N$, then over $j_{N-1}$, $\dots$, then over $j_1$ gives~\eqref{eqn-second-moment}.
\end{proof}

\begin{proof}[Proof of Proposition~\ref{prop-exit-moment}]
Let $\alpha$ and $C$ be as in Lemma~\ref{lem-gr-sum}.
That lemma tells us that with polynomially high probability as $\ep\rta 0$, it holds simultaneously for every Borel set $D\subset B_\rho$ with $\mu_h(D) \geq \ep^\alpha$ and every $x\in\mcl V\mcl G^\ep(B_\rho)$ that
\allb \label{eqn-use-gr-sum}
\sum_{y \in \mcl V\mcl G^\ep(D)} \op{Gr}_D^\ep(x , y)
&\leq C \ep^{-1} \int_{B_{\ep^\alpha}(D) } \log\left(\frac{1}{| \eta(x) - w|} \right) + 1 \, d\mu_h(w) + \ep^{-1+\alpha} \notag\\
&\leq (C+1) \ep^{-1} \int_{B_{\ep^\alpha}(D) }  \log\left(\frac{1}{| \eta(x) - w|} \right) + 1 \, d\mu_h(w) .
\alle
Note that we have used the hypothesis that $\mu_h(D)\geq \ep^\alpha$ to absorb the $\ep^{-1+\alpha}$ error into the first term.
By Lemma~\ref{lem-second-moment}, for each such $D$ and $x$,
\eqb \label{eqn-use-second-moment}
\ol{\BB E}_x^\ep \left[  (\tau_D^\ep)^N  \right]
\leq   N! \sum_{y_1,\dots,y_N \in \mcl V\mcl G^\ep(D) } \op{Gr}_D^\ep(x,y_1) \prod_{n=2}^N \op{Gr}_D^\ep(y_{n-1} , y_n ).
\eqe
The estimate~\eqref{eqn-exit-moment} (with $C+1$ in place of $C$) follows by applying~\eqref{eqn-use-gr-sum} to bound the sum over $y_N$, then the sum over $y_{N-1}$, $\dots$, then the sum over $y_1$ on the right side of~\eqref{eqn-use-second-moment}.
\end{proof}

\section{Tightness}
\label{sec-tight}

For $z\in \BB C$, let $X^{z,\ep} : \BB N_0 \rta \mcl V\mcl G^\ep $ denote the simple random walk on $ \mcl G^\ep$ started from the vertex $x_z^\ep \in\mcl V\mcl G^\ep$.
Let $\wh X^{z,\ep} : [0,\infty) \rta \BB C$ be the process $t \mapsto \eta(X_{t/\ep }^{z,\ep})$, extended from $\ep\BB N_0$ to $[0,\infty)$ by piecewise linear interpolation.

We want to establish a tightness result for the random functions
\eqb \label{eqn-tight-laws}
P^\ep : z\mapsto   \left\{\text{Conditional law of $t\mapsto \wh X^{z,\ep}_t$ given $(h,\eta)$} \right\}
\eqe
from $\BB C$ to the space of probability measures on continuous paths $[0,\infty)\rta \BB C$.
To state our tightness result we need to specify metrics on various spaces.
\begin{itemize}
\item The space $C([0,\infty) , \BB C)$ of continuous paths $X : [0,\infty)\rta \ol B_\rho$ is equipped with the local uniform metric
\eqb
\BB d_{\op{Unif}}(X,X') = \int_0^\infty e^{-T} \wedge\left\{\sup_{t\in [0,T]} |X_t - X_t'| \right\} \, dT,
\eqe
 so that $X^n \rta X$ if and only if $X^n\rta X$ uniformly on each compact subset of $[0,\infty)$.
\item The space $\op{Prob}\left( C([0,\infty) , \BB C ) \right)$ of probability measures on $C([0,\infty) , \BB C )$ is equipped with the Prokhorov metric $\BB d_{\op{Prok}}$ induced by the metric $\BB d_{\op{Unif}}$ on $C([0,\infty) , \BB C)$.
\item The space of functions $P : \BB C \rta  \op{Prob}\left( C([0,\infty) , \BB C) \right)$ is equipped with the $L^\infty$ metric
\eqb
\BB d_\infty(P , P') = \sup_{z\in \BB C} \BB d_{\op{Prok}} \left( P_z , P'_z \right) ,
\eqe
so that $P^n\rta P$ if and only if $P^n_z \rta P_z$ uniformly on each compact subset of $\BB C$ with respect to the above Prokhorov topology on $\op{Prob}\left( C([0,\infty) ,  \BB C ) \right)$.
\end{itemize}

The main result of this subsection is the following proposition.

\begin{prop} \label{prop-tight}
The laws of the random functions~\eqref{eqn-tight-laws} for $\ep \in (0,1)$ are tight with respect to the topology specified just above.
If $P$ is any subsequential limit in law of these functions, then a.s.\ $P$ is continuous and the following is true.
If $z\in\BB C$ and $\wh X^z$ is a path with the law $P_z$, then a.s.\ $\wh X^z$ does not stay at a single point for a positive interval of time.
\end{prop}

Fix $\rho \in (0,1)$.
Since most of our estimates only apply for cells of $\mcl G^\ep$ which intersect $B_\rho$, we will first prove a tightness statement when we stop our walks upon exiting $B_\rho$.
For $z\in\BB C$ and $\ep > 0$, let $\wh\tau^{z,\ep}$ be the first time that $\wh X^{z,\ep}$ hits a point $\eta(x) \in \eta(\ep\BB Z)$ whose corresponding cell $\eta([x-\ep,x])$ is disjoint from $B_\rho$.
We define
\eqb \label{eqn-tight-laws-rho}
P_{\rho,\cdot}^\ep  : z\mapsto   \left\{\text{Conditional law of $t\mapsto \wh X^{z,\ep}_{t\wedge \wh\tau^{z,\ep}}$ given $(h,\eta)$} \right\}  .
\eqe
Note that $P_{\rho,z}^\ep$ assigns probability 1 to the constant path at $x_z^\ep$ if $z$ is not contained in a cell of $\mcl G^\ep$ which intersects $B_\rho$.

\begin{prop} \label{prop-tight-rho}
The laws of the random functions~\eqref{eqn-tight-laws-rho} for $\ep \in (0,1)$ are tight with respect to the topology specified just above Proposition~\ref{prop-tight}.
If $P_{\rho,\cdot}$ is any subsequential limit in law of these functions, then a.s.\ $P$ is continuous and the following is true.
If $z\in\BB C$ and $\wh X^z$ is a path with the law $P_{\rho,z}$, then a.s.\ $\wh X^z$ does not stay at a single point for a positive interval of time before it exits $B_\rho$.
\end{prop}

Most of this section is devoted to the proof of Proposition~\ref{prop-tight-rho}. We will deduce Proposition~\ref{prop-tight} from Proposition~\ref{prop-tight-rho} and a scaling argument in Section~\ref{sec-tight-proof}.

For proving tightness results, it is convenient to work with the space
\eqb \label{eqn-cont-space-def}
\mcl C := C\left[ \BB C , \op{Prob}\left( C([0,\infty) , \BB C) \right) \right]
\eqe
of \emph{continuous} functions $\BB C \rta  \op{Prob}\left( C([0,\infty) , \BB C) \right)$, rather than the space of \emph{all} functions $ \BB C \rta  \op{Prob}\left( C([0,\infty) , \BB C) \right)$ (e.g., so that we can use the Arz\'ela-Ascoli theorem).
So, we will approximate the function $P_{\rho,\cdot}^\ep$ of~\eqref{eqn-tight-laws-rho} by a continuous function $\wt P_{\rho,\cdot}^\ep \in \mcl C$ which satisfies $\wt P_{\rho,\eta(x)}^\ep = P_{\rho,\eta(x)}^\ep$ for each $x\in\ep\BB Z$. One way to do this is as follows.
For each $x\in \ep\BB Z$, choose $r_x^\ep > 0$ such that the Euclidean neighborhood of the cell $B_{r_x^\ep}(H_x^\ep)$ does not contain $\eta(y)$ for any $y\in \ep\BB Z\setminus \{x\}$.
Choose a partition of unity $\{\phi_x^\ep \}_{x\in\ep\BB Z}$ subordinate to the open cover $\{B_{r_x^\ep}(H_x^\ep)\}_{x\in\ep\BB Z}$ of $\BB C$.
Then set $\wt P_{\rho,z}^\ep = \sum_{x\in \ep\BB Z} P_{\rho,\eta(x)}^\ep \phi_x^\ep(z)$.

To prove Proposition~\ref{prop-tight-rho}, we need a compactness criterion, which follows easily from the Prokhorov theorem and the Arz\'ela-Ascoli theorem.

\begin{lem} \label{lem-compact}
Let $\mcl K   \subset \mcl C$.
Then $\mcl K$ is pre-compact (i.e., the closure of $\mcl K$ is compact) if the following conditions are satisfied.
\begin{enumerate}
\item (Equicontinuity of laws) For each $L >1$, there exists $\delta = \delta(L) > 0$ such that for each $P  = (P_z)_{z\in\BB C} \in \mcl K$ and each $z,w\in B_L(0)$ with $|z-w|\leq \delta$, the Prokhorov distance between the laws $P_z$ and $P_w$ is at most $1/L$. \label{item-compact-laws}
\item (Equicontinuity of paths) For each $L > 0$ and each $z\in \BB C$, there exists $\delta = \delta(z,L) > 0$ such that
for each $P   = (P_w)_{ w \in \BB C}  \in \mcl K$, the following is true.
If $X : [0,\infty) \rta \BB C$ is sampled from the law $P_z$, then with probability at least $1- 1/L$,  \label{item-compact-paths}
\eqbn
  |X_s - X_t| \leq 1/L , \quad\forall s,t\in [0,  L] \quad\text{with} \quad   |s-t| \leq  \delta   .
\eqen
\end{enumerate}
\end{lem}
\begin{proof}
By the Arz\'ela-Ascoli theorem for continuous functions into the complete metric space
$\op{Prob}(C([0,\infty) , \BB C))$, $\mcl K$ is pre-compact provided condition~\ref{item-compact-laws} holds and for
each fixed $z\in \BB C$, the set $\{P_z\}_{P\in \mcl K}$ is a pre-compact subset of $ \op{Prob}(C([0,\infty) , \BB C))$.

By the Prokhorov theorem, $\{P_z\}_{P\in \mcl K}$  is pre-compact provided that for each
$\zeta \in (0,1)$, there is a pre-compact subset $\mcl K'_\zeta = \mcl K'_\zeta(z)$ of $C([0,\infty) , \BB C)$
such that for each $P\in\mcl K$, we have $P_z[ \mcl K'_\zeta] \geq 1-\zeta$.
By the Arz\'ela-Ascoli theorem, $\mcl K'_\zeta$ is pre-compact provided there is a uniform modulus of continuity for all of the functions in $\mcl K'_\zeta$.

We now assume condition~\ref{item-compact-paths} and check that $\{P_z\}_{P\in \mcl K }$ is pre-compact using the criterion of the preceding paragraph.
To this end, fix $\zeta \in (0,1)$ and for $k\in\BB N$ let $\delta_k$ be as in condition~\ref{item-compact-paths} for $L =   2^{ k} / \zeta $.
By a union bound over all $k\in\BB N$, we get that if $P\in\ \mcl K$ and $X$ is sampled from $P_z$ then it holds with probability at least $1-\zeta$ that
\eqb \label{eqn-compact-equicont}
|X_s -X_t| \leq  \zeta/2^k ,\quad \forall s,t\in [0,   2^k  / \zeta] \quad\text{with} \quad   |s-t| \leq  \delta_k  , \quad \forall k\in\BB N .
\eqe
The set $\mcl K'_\zeta$ of $X \in C([0,\infty),\BB C)$ for which~\eqref{eqn-compact-equicont} holds and $X_0 = z$ is a pre-compact subset of $ C([0,\infty),\BB C)$, which verifies the condition of the preceding paragraph.
\end{proof}

The proof of the tightness part of Proposition~\ref{prop-tight-rho} is divided into two steps, based on the two conditions of Lemma~\ref{lem-compact}.
First we prove equicontinuity of paths in Section~\ref{sec-equicont-path}, then we prove equicontinuity of laws in Section~\ref{sec-equicont-law}.
In Section~\ref{sec-no-stuck}, we prove an estimate for the random walk on $\mcl G^\ep$ which will help us verify the last part of Proposition~\ref{prop-tight} (i.e., that the subsequential limiting path does not stay at a point for a positive time interval).
In Section~\ref{sec-tight-proof} we deduce Proposition~\ref{prop-tight} from Proposition~\ref{prop-tight-rho} via a scaling argument.
Throughout this section, we will use the following terminology.

\begin{defn} \label{def-superpoly}
We say that $\psi : [0,\infty) \rta [0,\infty)$ is a \emph{superpolynomial function} if $\psi(0) = 0$ and $\lim_{\delta\rta 0} \delta^{-p} \psi(\delta) = 0$ for every $p > 0$.
\end{defn}

We do not require the superpolynomial functions appearing in different proposition / lemma statements to be the same.

\subsection{Equicontinuity of paths}
\label{sec-equicont-path}

The main result of this subsection is the following estimate.

\begin{prop}[Equicontinuity of paths] \label{prop-walk-cont}
Let $\chi >   \frac{2(2+\gamma)}{2-\gamma}$.
There is a superpolynomial function $\psi$, depending only on $\rho,\chi,\gamma$, such that probability tending to 1 as $\ep\rta 0$, it holds simultaneously for each $z \in B_\rho$ that the following is true.
With conditional probability at least $1-\psi(\delta)$ given $(h,\eta)$,
\eqb \label{eqn-walk-cont}
|\wh X^{z,\ep}_s - \wh X^{z,\ep}_t| \leq 2\delta,\quad \forall s,t\in [0,\wh\tau^{z,\ep}] \quad \text{with} \quad |s-t| \leq \delta^\chi .
\eqe
\end{prop}

Once Theorem~\ref{thm-lbm-conv} is established, Proposition~\ref{prop-walk-cont} implies that a.s.\ Liouville Brownian motion is locally H\"older continuous with any exponent less than $\frac{2-\gamma}{2(2+\gamma)}$. This bound is not optimal: indeed, it follows from~\cite[Corollary 2.14]{grv-lbm} that Liouville Brownian motion is H\"older continuous with any exponent less than $\frac{1}{2(1+\gamma/2)^2}$.
Note in particular that this exponent does not degenerate as $\gamma$ approaches 2. The proof in \cite{grv-lbm}
relies on the study of negative moments for the exit time from a ball, which in turn relies on the \emph{a priori} knowledge of
the multifractal spectrum of Liouville Brownian motion. Such tools are not directly available here and we instead use a ``worst-case scenario" approach
which leads to a suboptimal (but still finite) exponent.


To prove Proposition~\ref{prop-walk-cont} we want to use our lower bound for the exit times of the walk on $\mcl G^\ep$ from Euclidean balls (Proposition~\ref{prop-exit-lower}) to say that $\wh X^{z,\ep}$ is very unlikely to travel a long Euclidean distance in a short amount of time.
The main difficulty is that Proposition~\ref{prop-walk-cont} only gives a lower bound for the conditional \emph{expectation} of the exit time given $(h,\eta)$.
This by itself is far from good enough to get a modulus of continuity estimate for $\wh X^{z,\ep}$, so we need to use additional arguments to get a better lower bound for exit times.

The first step of the proof, carried out in Lemma~\ref{lem-exit-pos}, is to use Proposition~\ref{prop-exit-lower} together with our upper bound for moments of exit times from Euclidean balls (Proposition~\ref{prop-exit-moment}) and the Payley-Zygmund inequality to say the following.
Let $c >0$ be the constant from Proposition~\ref{prop-exit-lower}.
If $\delta\in (0,1)$ is fixed and $\ep > 0$ is small, then with high probability it holds uniformly over all $x \in \mcl V\mcl G^\ep( B_\rho)$ and all  $r \in \left[\delta , \op{dist}\left( \eta(x) , \bdy B_\rho \right) \right]$ that
\eqb \label{eqn-exit-pos0}
\BB P\left[ \text{Exit time of $\wh X^{\eta(x) ,\ep} $ from $B_r(\eta(x))$} \geq \frac12  \mu_h\left( B_{c r} (\eta(x)) \right) \,|\,  h,\eta   \right]
\eqe
is bounded below by an explicit quantity depending on $\mu_h|_{B_r(\eta(x))}$.

Using basic ``worst-case" estimates for the LQG measure, the above bound can be made more explicit.
In particular, fix a large constant $\ell > 1$ to be chosen later. Then when $\ep$ and $\delta$ are small, it holds with high probability
over $(h, \eta)$ that
\eqb \label{eqn-min-exit-time-pos0}
\BB P\left[ \text{Exit time of $\wh X^{z,\ep} $ from $B_{ \delta^\ell }(z)$} \gtrsim \delta^{ (2+\gamma)^2 \ell/2 } \,|\, h,\eta   \right]
\gtrsim  \delta^{\ell   \gamma^2/2   }
\eqe
uniformly over all $x\in \mcl V\mcl G^\ep(B_{\rho-\delta^\ell})$, where here $\gtrsim$ denotes inequality up to a multiplicative error of order $\delta^\zeta$, where $\zeta > 0$ can be made as small as we like (see Lemma~\ref{lem-min-exit-time-pos}).

At first glance, the bound~\eqref{eqn-min-exit-time-pos0} still seems to be far from a modulus of continuity estimate for $\wh X^{z,\ep}$ since it only gives a lower bound for exit times which holds with probability at least $ \delta^{\ell   \gamma^2/2   }  $, not with probability close to 1.
To get around this difficulty, we exploit the (quenched) convergence of $\wh X^{z,\ep}$ to Brownian motion modulo time parametrization, which was proven in~\cite{gms-tutte}.
Indeed, using a basic Brownian motion estimate (Lemma~\ref{lem-brownian-exit}) and the aforementioned convergence, we can show that when $\ep$ and $\delta$ are small the following is true with extremely high conditional probability given $(h,\eta)$.
If $z\in B_\rho$, then the walk $\wh X^{z,\ep}$ has to travel Euclidean distance $\delta^\ell$ at least $N\approx \delta^{-2(\ell-1)}$ times before exiting the ball of radius $\delta$ centered at its starting point.
More precisely, if $T$ is the exit time of $\wh X^{z,\delta}$ from $B_\delta(z)$, then typically there are of order $N\approx \delta^{-2(\ell-1) }$ times $0  = \tau_0 < \tau_1 < \dots < \tau_N < T$ such that $\wh X^{z,\ep}_{\tau_n}  -\wh X^{z,\ep}_{\tau_{n-1}} \geq \delta^\ell$ for each $n=1,\dots,N$ (Lemma~\ref{lem-walk-increment}).
This is related to the fact that Brownian motion has linear quadratic variation.

The estimate~\eqref{eqn-min-exit-time-pos0} tells us that (ignoring small multiplicative errors) under the conditional law given $(h,\eta)$, each of the time intervals $[\tau_{n-1} , \tau_n]$ has probability at least $\delta^{\ell\gamma^2/2  }$ to have length at least $\delta^{ (2+\gamma)^2 \ell / 2  }$.
Since there are of order $\delta^{-2(\ell-1)}$ such intervals, and using the Markov property of the random walk on $\mcl G^\ep$, we get that the probability that \emph{at least one} of these intervals has length at least $\delta^{(2+\gamma)^2\ell/2  }$ is at least
\eqb \label{eqn-exit-conc0}
1 - \left( 1 - \delta^{\ell\gamma^2/2}  \right)^{\delta^{-2(\ell-1)}} .
\eqe

If we choose $\ell > 4/(4-\gamma^2)$, so that $2(\ell-1) > \ell \gamma^2/2$, then~\eqref{eqn-exit-conc0} is of the form $1-\psi(\delta)$ for a superpolynomial function $\psi$.
We therefore get that when $\ep> 0$ is small, it holds uniformly over all $x\in \mcl V\mcl G^\ep(B_{\rho-\delta})$ that
\eqb \label{eqn-exit-conc1}
\BB P\left[ \text{Exit time of $\wh X^{z,\ep}$ from $B_\delta(z)$} \gtrsim    \delta^{ (2+\gamma)^2 \ell/2  }  \,|\, h,\eta \right]
\geq  1-\psi(\delta) ;
\eqe
see Lemma~\ref{lem-exit-lower-conc}.
Once~\eqref{eqn-exit-conc1} is established, it is a simple matter to conclude the proof of Proposition~\ref{prop-walk-cont} using the Markov property and a union bound.
We remind the reader that this only holds with high probability over $(h, \eta)$ because the bound \eqref{eqn-min-exit-time-pos0} only holds with high probability.

We now proceed with the details of the above argument.
For convenience we will mostly phrase our estimates in terms of the random walk on $\mcl G^\ep$ instead of in terms of the embedded, linearly interpolated walk $\wh X^{z,\ep}$ as in the outline above.
We start with the Payley-Zygmund argument mentioned above.

\begin{lem} \label{lem-exit-pos}
Fix $p  > 1$ and let $\alpha=\alpha(\gamma)$ be as in Proposition~\ref{prop-exit-moment}.
There are constants $a = a(p,\rho,\gamma) > 0$ and $c = c(  \rho,\gamma) > 0$ such that for each $\delta \in (0,1)$, it holds with polynomially high probability as $\ep\rta 0$ that simultaneously for each $x\in\mcl V\mcl G^\ep(B_{\rho-\delta} )$ and each $r \in \left[\delta , \op{dist}\left( \eta(x) , \bdy B_\rho \right) \right]$,
\eqb \label{eqn-exit-pos}
\ol{\BB P}_x^\ep\left[  \tau_{B_r(\eta(x)) }^\ep \geq  \frac12 \ep^{-1} \mu_h\left( B_{c r} (\eta(x)) \right) \right]
\geq     a\left( \frac{   \mu_h(B_{c r}(\eta(x))  }{ \sup_{z \in B_r(\eta(x))} \int_{B_{r+\ep^\alpha}(\eta(x))} \log\left(\frac{1}{| z - w|} \right) + 1 \, d\mu_h(w) } \right)^p     .
\eqe
\end{lem}
\begin{proof}
Let $q := p/(p-1)$ be the H\"older conjugate exponent to $p$ and let $N := \lceil q \rceil$.
By Proposition~\ref{prop-exit-moment}, there exists $C = C(\rho,\gamma) > 0$ such that for each $\delta \in (0,1)$, it holds with with polynomially high probability as $\ep\rta 0$ that simultaneously for each $x$ and $r$ as in the lemma statement,
\eqb \label{eqn-use-exit-moment}
 \ol{\BB E}_x^\ep \left[ (\tau_{B_r(\eta(x))}^\ep)^N  \right] \leq  N! C^N \ep^{-N} \left( \sup_{z \in B_r(\eta(x))} \int_{B_{r+\ep^\alpha}(\eta(x))} \log\left(\frac{1}{| z - w|} \right) + 1 \, d\mu_h(w) \right)^N .
\eqe
Since $N > q$, the function $x\mapsto x^{N/q}$ is convex so Jensen's inequality together with~\eqref{eqn-use-exit-moment} gives
\allb \label{eqn-use-exit-moment-p}
 \ol{\BB E}_x^\ep \left[ (\tau_{B_r(\eta(x))}^\ep)^q  \right]
&\leq \ol{\BB E}_x^\ep \left[ (\tau_{B_r(\eta(x))}^\ep)^N  \right]^{q/N} \notag\\
&\leq  (N!)^{q/N} C^q \ep^{-q} \left( \sup_{z \in B_r(\eta(x))} \int_{B_{r+\ep^\alpha}(\eta(x))} \log\left(\frac{1}{| z - w|} \right) + 1 \, d\mu_h(w) \right)^q .
\alle
By Proposition~\ref{prop-exit-lower}, there exists $c = c(\rho,\gamma) > 0$ such that with polynomially high probability as $\ep\rta 0$ that for each $x$ and $r$ as in the lemma statement,
\eqb \label{eqn-use-exit-lower}
\ol{\BB E}_x^\ep\left[  \tau_{B_r(\eta(x)) }^\ep \right] \geq    \ep^{-1} \mu_h\left( B_{c r} (\eta(x)) \right)
\eqe
The bound~\eqref{eqn-exit-pos} for an appropriate choice of $a$ now follows from~\eqref{eqn-use-exit-moment-p}, \eqref{eqn-use-exit-moment}, and the $L^q$ variant of the Payley-Zygmund inequality.
\end{proof}

We now make the bound in Lemma~\ref{lem-exit-pos} more explicit (so that it is in terms of $\delta$ rather than in terms of $\mu_h$).
This is done using basic uniform estimates for the LQG measure from Appendix~\ref{sec-appendix}.

\begin{lem} \label{lem-min-exit-time-pos}
Let $\ell > 1$ and $\zeta  > 0$.
It holds with probability tending to 1 as $\ep\rta 0$ and then $\delta \rta 0$ that
\eqb \label{eqn-min-exit-time-pos}
\min_{x \in \mcl V\mcl G^\ep(B_{\rho-\delta^\ell})}
\ol{\BB P}_x^\ep\left[  \tau_{B_{\delta^\ell}(\eta(x)) }^\ep \geq    \ep^{-1} \delta^{ (2+\gamma)^2 \ell / 2    + \zeta }  \right]
\geq  \delta^{\ell   \gamma^2/2  +  o_\zeta(1)   }  ,
\eqe
where the $o_\zeta(1)$ is deterministic and depends only on $\zeta,\ell,\gamma$.
\end{lem}

As discussed in the outline above, we eventually want to choose $\ell$ to be sufficiently large depending on $\gamma$.
The parameter $\zeta$ represents a small error in the exponent and will eventually be sent to zero.

\begin{proof}[Proof of Lemma~\ref{lem-min-exit-time-pos}]
By Lemma~\ref{lem-log-int} applied with $A =  \delta^\zeta / 100$ together with Lemma~\ref{lem-ball-mass} applied with $e^{-\op{const}/\delta^{\zeta/2}}$ in place of $\delta$, it holds except on an event of probability decaying faster than any positive power of $\delta$ that
\eqb \label{eqn-use-log-int}
 \sup_{z \in B_{\delta^\ell}(z)} \int_{B_{2\delta^\ell}(z)} \log\left(\frac{1}{| z - w|} \right) + 1 \, d\mu_h(w) \leq \delta^\zeta \mu_h\left( B_{2\delta^\ell}(z) \right) ,\quad\forall z\in B_\rho .
\eqe
Note that Lemma~\ref{lem-ball-mass} is applied to absorb the additive error term $e^{ -\op{const} /\delta^\zeta}$ from Lemma~\ref{lem-log-int} into the main term $ \mu_h\left( B_{2\delta^\ell}(z) \right)$.

Let $c > 0$ be the constant from Lemma~\ref{lem-exit-pos}. We emphasize that $c$ does not depend on $p$.
By Lemma~\ref{lem-max-ball-ratio} applied with $2\delta^\ell$ in place of $\delta$, it holds with probability tending to 1 as $\delta\rta0$ that
\eqb \label{eqn-use-max-ball-ratio}
\frac{\mu_h(B_{c\delta^\ell}(z))}{\mu_h(B_{2\delta^\ell}(z))} \geq \delta^{  \ell \gamma^2/2   + \zeta } ,\quad\forall z \in B_\rho .
\eqe

We now apply~\eqref{eqn-use-log-int} followed by~\eqref{eqn-use-max-ball-ratio} to lower-bound the right side of the inequality of Lemma~\ref{lem-exit-pos} with $r = \delta^\ell$ and with $p  = 1+\zeta$.
Note that $\delta^\ell + \ep^\alpha  \leq 2\delta^\ell$ for each small enough $\ep  >0$, so we can bound the integral over $B_{\delta^\ell + \ep^\alpha}(z)$ by the integral over $B_{2\delta^\ell}(z)$.
We find that with probability tending to 1 as $\ep\rta 0$ and then $\delta \rta 0$,
\eqb \label{eqn-exit-lower-small}
\min_{x \in \mcl V\mcl G^\ep(B_{\rho-\delta^\ell})}
\ol{\BB P}_x^\ep\left[  \tau_{B_{\delta^\ell}(\eta(x)) }^\ep \geq  \frac12 \ep^{-1} \mu_h(B_{c\delta^\ell}(\eta(x)) \right]
\geq   \delta^{\ell   \gamma^2/2  +  o_\zeta(1) } .
\eqe
By Lemma~\ref{lem-ball-mass}, it holds with probability tending to 1 as $\delta\rta 0$ that
\eqb \label{eqn-use-min-ball}
\inf_{z\in B_\rho} \mu_h\left(B_{c\delta^\ell}(z) \right) \geq 2 \delta^{ (2+\gamma)^2 \ell / 2 + \zeta} .
\eqe
Combining~\eqref{eqn-exit-lower-small} and~\eqref{eqn-use-min-ball} gives~\eqref{eqn-min-exit-time-pos}.
\end{proof}

Continuing with the outline above, we now want to use the convergence of random walk on $\mcl G^\ep$ to Brownian motion modulo time parametrization to argue that the random walk on $\mcl G^\ep$ has to have at least $\delta^{-2(\ell-1)}$ increments during which it travels Euclidean distance $\delta^\ell$ before it can travel Euclidean distance $\delta$.
We start with the desired Brownian motion estimate.

\begin{lem} \label{lem-brownian-exit}
Let $\mcl B$ be a standard planar Brownian motion started from 0.
For $r  \in (0,1)$, let $\sigma_0(r) := 0$ and inductively let $\sigma_k(r)$ for $k\in\BB N$ be the first time $t\geq \sigma_{k-1}(r)$ for which $|\mcl B_t - \mcl B_{\sigma_{k-1}(r)}|\geq r$.
For each $\zeta \in (0,1)$, there is a superpolynomial function $\psi$ such that
\eqb \label{eqn-brownian-exit-lower}
\BB P\left[  \text{$\mcl B$ first exits $B_1(0)$ before time $\sigma_{\lfloor r^{-(2-\zeta) } \rfloor} (r)$} \right]  \leq  \psi(r) \quad \text{and}
\eqe
\eqb \label{eqn-brownian-exit-upper}
\BB P\left[  \text{$\mcl B$ first exits $B_1(0)$ after time $\sigma_{\lfloor r^{-(2+\zeta) } \rfloor} (r)$} \right]   \leq \psi(r) .
\eqe
\end{lem}
\begin{proof}
By Brownian scaling, the rescaled increments $r^{-2} \left( \sigma_k(r) - \sigma_{k-1}(r) \right)$ are i.i.d., and each has the same law as the exit time $\sigma_1(1)$ of $\mcl B$ from $B_1(0)$.
Using the strong Markov property of Brownian motion, it is easily seen that there are universal constants $c_0,c_1 , c_2 > 0$ such that or each $A > 1$, each $k\in\BB N$, and each $r>0$,
\eqb \label{eqn-brownian-inc-upper}
\BB P\left[ \sigma_k(r) - \sigma_{k-1}(r)  > A r^2 \right]
=\BB P\left[ \sigma_1(1) > A\right]
\leq c_0 e^{-c_1 A} \quad \text{and}
\eqe
\eqb \label{eqn-brownian-inc-lower}
\BB P\left[ \sigma_k(r) - \sigma_{k-1}(r)   \geq r^2 \right] \geq c_2 .
\eqe
\medskip

\noindent\textit{Proof of~\eqref{eqn-brownian-exit-lower}.}
By expanding $\sigma_K(r)^N$ as a sum over $K^N$ products of $N$ increments and applying~\eqref{eqn-brownian-inc-upper}, we get that for each $K \in \BB N$ and each $N\in\BB N$,
\eqb \label{eqn-brownian-exit-moment}
\BB E\left[ \sigma_K(r)^N \right] \preceq  r^{2 N} K^N ,
\eqe
with the implicit constant depending only on $N$.
By setting $K = \lfloor r^{-(2-\zeta )} \rfloor$ in~\eqref{eqn-brownian-exit-moment} and applying the Chebyshev inequality, we get
\eqb \label{eqn-brownian-exit-count}
\BB P\left[ \sigma_{\lfloor r^{-(2-\zeta) } \rfloor} (r)  > r^{ \zeta/2} \right] \preceq r^{\zeta N/2} , \quad\forall N\in\BB N.
\eqe
On the other hand, by a Gaussian tail bound, Brownian scaling, and the reflection principle,
\eqb  \label{eqn-brownian-exit-time}
\BB P\left[ \sigma_1(1)  <  r^{ \zeta/2 } \right] \preceq e^{-1/(2 r^{\zeta/2}) } ,
\eqe
with a universal implicit constant.
Combining~\eqref{eqn-brownian-exit-count} and~\eqref{eqn-brownian-exit-time} and sending $N\rta\infty$ gives~\eqref{eqn-brownian-exit-lower} for an appropriate choice of $\psi$.
\medskip

\noindent\textit{Proof of~\eqref{eqn-brownian-exit-upper}.}
By~\eqref{eqn-brownian-inc-lower}, $r^{-2} \sigma_k(r)$ stochastically dominates a binomial random variable with $k$ trials and success probability $c_2 $.
By a standard binomial concentration inequality, $\BB P[\sigma_{\lfloor r^{-(2-\zeta) } \rfloor} (r)  < r^{-\zeta/2}]$ decays faster than any positive power of $r$ as $r\rta 0$.
By the upper bound for $\sigma_1(1)$ in~\eqref{eqn-brownian-inc-upper} we also have that $\BB P[\sigma_1(1)  > r^{-\zeta/2}]$ decays faster than any positive power of $r$ as $r\rta 0$. Combining the two preceding sentences gives~\eqref{eqn-brownian-exit-upper}.
\end{proof}

The events considered in Lemma~\ref{lem-brownian-exit} do not depend on the time parametrization of the Brownian motion $\mcl B$, so we can transfer the conclusion of Lemma~\ref{lem-brownian-exit} to a statement for random walk on $\mcl G^\ep$ using the convergence modulo time parametrization established in~\cite{gms-tutte}.

\begin{lem} \label{lem-walk-increment}
Fix $\ell > 1$.
For $\delta \in (0,1)$ and a random walk $X^\ep$ on $\mcl G^\ep$ (with some choice of starting point), let $\tau_0^\ep(\delta^\ell) =0$ and for $k\in\BB N$ inductively let $\tau_k^\ep(\delta^\ell)$ be the first time $j$ after $\tau_{k-1}^\ep(\delta^\ell)$ for which $X_j^\ep \notin \mcl V\mcl G^\ep(B_{\delta^\ell}(  X^\ep_{\tau_{k-1}^\ep(\delta^\ell)}  )  )$.
For each $\zeta \in (0,1)$, there is a superpolynomial function $\psi$, depending only on $\zeta,\ell,\gamma$, such that with probability tending to 1 as $\ep\rta 0$ (at a rate which depends on $\delta$), it holds for each $x\in \mcl V\mcl G^\ep(B_{\rho-\delta})$
that
\eqb \label{eqn-walk-increment}
\ol{\BB P}_x^\ep\left[ \text{$X^\ep$ exits $\mcl V\mcl G^\ep\left( B_\delta(\eta(x))\right)$ before time $\tau_{\lfloor \delta^{-  (2-\zeta) (\ell-1)} \rfloor}^\ep(\delta^\ell)$} \right]
\leq \psi(\delta )  .
\eqe
\end{lem}
\begin{proof}
For $z\in\BB C$, let $\mcl B^z$ be a standard planar Brownian motion started from $z$.
Let $\sigma_0^z(\delta^\ell) = 0$ and inductively let $\sigma_k^z(\delta^\ell)$ for $k\in\BB N$ be the first time $t\geq \sigma_{k-1}^z(\delta^\ell)$ for which $|\mcl B^z_t - \mcl B^z_{\sigma_{k-1}^z(\delta^\ell)}| \geq \delta^\ell$.
By~\eqref{eqn-brownian-exit-lower} of Lemma~\ref{lem-brownian-exit} (applied with $r = \delta^{\ell-1}$)
and Brownian scaling, there exists a superpolynomial function $\psi$ as in the lemma statement such that for each $z\in\BB C$,
\eqb \label{eqn-use-brownian-exit}
\BB P\left[  \text{$\mcl B^z$ exits $B_\delta(z)$ before time $\sigma_{\lfloor \delta^{-  (2-\zeta)(\ell-1) } \rfloor}^z (\delta^\ell)$} \right]   \leq \psi(\delta)  .
\eqe

Now let $\wh X^{z,\ep}$ be the embedded, linearly interpolated random walk as described at the beginning of this section.
By~\cite[Theorem 3.4]{gms-tutte}, the supremum over all $z\in\BB D$ of the Prokhorov distance between the law of $\wh X^{z,\ep}$ and the law of $\mcl B^z$, each stopped at its exit time from $\BB D$, with respect to the topology on curves modulo time parametrization, tends to zero in probability as $\ep\rta 0$.
By combining this with~\eqref{eqn-use-brownian-exit} we obtain~\eqref{eqn-walk-increment} (with $2\psi$ in place of $\psi$, say).
\end{proof}

The estimate~\eqref{eqn-brownian-exit-upper} of Lemma~\ref{lem-walk-increment} also yields a useful estimate for the random walk on $\mcl G^\ep$.

\begin{lem} \label{lem-walk-increment'}
For $\delta \in (0,1)$ and a random walk $X^\ep$ on $\mcl G^\ep$ (with some choice of starting point), let $\sigma_0^\ep(\delta) =0$ and for $k\in\BB N$ inductively let $\sigma_k^\ep(\delta)$ be the first time $j$ after $\sigma_{k-1}^\ep(\delta)$ for which $X_j^\ep \notin \mcl V\mcl G^\ep(B_{\delta }(  X^\ep_{\sigma_{k-1}^\ep(\delta)}  )  )$.
For each $\zeta \in (0,1)$, there is a superpolynomial function $\psi$, depending only on $\zeta, \gamma$, such that with probability tending to 1 as $\ep\rta 0$ (at a rate which depends on $\delta$), it holds for each $x\in \mcl V\mcl G^\ep(B_\rho)$ that
\eqb \label{eqn-walk-increment'}
\ol{\BB P}_x^\ep\left[ \text{$X^\ep$ first exits $\mcl V\mcl G^\ep\left( B_\rho \right)$ before time $\sigma_{\lfloor \delta^{-  (2+\zeta) } \rfloor}^\ep(\delta)$} \right]
\leq \psi(\delta )  .
\eqe
\end{lem}
\begin{proof}
This follows from~\eqref{eqn-brownian-exit-upper} of Lemma~\ref{lem-walk-increment} and~\cite[Theorem 3.4]{gms-tutte} via exactly the same argument used in the proof of Lemma~\ref{lem-walk-increment}.
\end{proof}

We now get a lower bound for exit times which holds with extremely high probability, which is the main input in the proof of Proposition~\ref{prop-tight}.

\begin{lem} \label{lem-exit-lower-conc}
Let $\chi > \frac{2(2+\gamma)}{2-\gamma}$. There is a superpolynomial function $\psi : (0,\infty) \rta (0,\infty)$ such that with probability tending to 1 as $\ep\rta 0$ and then $\delta \rta 0$,
\eqb \label{eqn-exit-lower-conc}
\max_{x \in \mcl V\mcl G^\ep(B_{\rho-\delta})} \ol{\BB P}_x^\ep\left[  \tau_{B_\delta(\eta(x)) }^\ep  <  \ep^{-1} \delta^{\chi }   \right]  \leq   \psi(\delta)    .
\eqe
\end{lem}
\begin{proof}
Fix $\ell > 1$ and $\zeta >0$.
Define the stopping times $\{\tau_k^\ep(\delta^\ell)\}_{k\in\BB N_0} $ for $X$ as in Lemma~\ref{lem-walk-increment}.
By Lemma~\ref{lem-walk-increment}, there exists a superpolynomial function $\psi_0$, depending only on $\zeta,\ell,\gamma$, such that for each $\delta \in (0,1)$, it holds with probability tending to 1 as $\ep \rta 0$ that
\eqb \label{eqn-use-walk-increment}
\max_{ x\in \mcl V\mcl G^\ep(B_{\rho-\delta}) }
\ol{\BB P}_x^\ep\left[  \tau_{B_\delta(\eta(x))}^\ep < \tau_{\lfloor \delta^{-(2-\zeta)(\ell-1)} \rfloor}^\ep(\delta^\ell) \right]
\leq \psi_0(\delta ) .
\eqe
By the strong Markov property of the walk $X^\ep$ under $\ol{\BB P}_x^\ep$, the $\ol{\BB P}_x^\ep$-conditional law of
$X^\ep$ given $X^\ep|_{[0,\tau_{k-1}^\ep(\delta^\ell)]}$ is that of a simple random walk on $\mcl G^\ep$ started from
$X^\ep_{\tau_{k-1}^\ep(\delta^\ell)}$.
The time $\tau_k^\ep(\delta^\ell) - \tau_{k-1}^\ep(\delta^\ell)$ is the exit time of this walk from $\mcl V\mcl G^\ep( B_{\delta^\ell}(\eta( X_{\tau_{k-1}^\ep(\delta^\ell)} )))$.
By Lemma~\ref{lem-min-exit-time-pos}, we infer that with probability tending to 1 as $\ep\rta 0$ and then $\delta \rta 0$,
\eqb \label{eqn-walk-increment-pos}
\max_{ x\in \mcl V\mcl G^\ep(B_{\rho-\delta}) }
\max_{k\in\BB N}
\ol{\BB P}_x^\ep\left[ \tau_k^\ep(\delta^\ell) < \tau_{B_\delta(x)}^\ep ,\, \tau_k^\ep(\delta^\ell) - \tau_{k-1}^\ep(\delta^\ell) <    \ep^{-1} \delta^{ (2+\gamma)^2 \ell / 2 + \zeta }    \,\big|\, X|_{[0,\tau_{k-1}^\ep(\delta^\ell)]} \right]
\leq 1 - \delta^{\ell \gamma^2/2 + o_\zeta(1) } .
\eqe

By multiplying~\eqref{eqn-walk-increment-pos} over all $k\in \BB N$ and using~\eqref{eqn-use-walk-increment}, we get that with probability tending to 1 as $\ep\rta 0$ and then $\delta \rta 0$,
\allb \label{eqn-walk-increment-mult}
&\max_{ x\in \mcl V\mcl G^\ep(B_{\rho-\delta}) }
\ol{\BB P}_x^\ep\left[   \tau_k^\ep(\delta^\ell) - \tau_{k-1}^\ep(\delta^\ell)    < \ep^{-1} \delta^{(2+\gamma)^2 \ell / 2 + \zeta }    ,\,
\text{$\forall k\in\BB N$ with $\tau_k^\ep(\delta^\ell) < \tau_{B_\delta(x)}^\ep$} \right] \notag\\
&\qquad \leq \psi_0(\delta) +  \left( 1 -  \delta^{\ell  \gamma^2/2 + o_\zeta(1) } \right)^{\delta^{-(2-\zeta)(\ell-1)}} .
\alle
If we choose $\ell > 4/(4-\gamma^2)$, then for a small enough choice of $\zeta \in (0,1)$, the term $ \left( 1 -  \delta^{\ell p\gamma^2/2 + o_\zeta(1) } \right)^{\delta^{-(2-\zeta)(\ell-1)}} $ in~\eqref{eqn-walk-increment-mult} decays faster than any positive power of $\delta$ as $\delta \rta 0$.
From this, we infer that with probability tending to 1 as $\ep\rta 0$ and then $\delta \rta 0$,
\eqb \label{eqn-walk-increment-mult'}
\max_{ x\in \mcl V\mcl G^\ep(B_{\rho-\delta}) }
\ol{\BB P}_x^\ep\left[    \tau_{B_\delta(x)}^\ep <    \ep^{-1} \delta^{ (2+\gamma)^2 \ell / 2  + \zeta }      \right]
\leq \psi(\delta) ,
\eqe
for some superpolynomial function $\psi$ depending only on $\zeta,\ell,\gamma$.
If we choose $\ell$ sufficiently close to $4/(4-\gamma^2)$ and $\zeta$ sufficiently close to zero, then we can make $(2+\gamma)^2 \ell / 2  + \zeta$ as close as we like to $ \frac{2(2+\gamma)}{2-\gamma}$. This gives~\eqref{eqn-exit-lower-conc}.
\end{proof}


\begin{proof}[Proof of Proposition~\ref{prop-walk-cont}]
For $\delta > 0$, let $\wh\sigma_0^{z,\ep}(\delta) = 0$ and for $k\in\BB N$ inductively let $\wh\sigma_k^{z,\ep}(\delta)$ be the first time $t\geq \wh\sigma_{k-1}^{z,\ep}(\delta)$ for which $\wh X^{z,\ep}_t = \eta(x)$ for some $x\in \mcl V\mcl G^\ep \setminus \mcl V\mcl G^\ep( B_\delta( \wh X^{z,\ep}_{ \wh\sigma_{k-1}^{z,\ep}(\delta)}   ) )$ (basically, $\wh \sigma_k^{z,\ep}(\delta)$ is a slight modification of the exit time of $\wh X^{z,\ep}$ from $B_\delta( \wh X^{z,\ep}_{ \wh\sigma_{k-1}^{z,\ep}(\delta)}   )$).
Recall that $\wh \tau^{z,\ep}$ denotes the exit time of $\wh X^{z,\ep}$ from $B_\rho$.
By Lemma~\ref{lem-walk-increment'}, there is a superpolynomial function $\psi_0$ depending only on $\rho,\gamma$ such that with probability tending to 1 as $\ep\rta 0$,
\eqb \label{eqn-increment-count-upper}
\sup_{z\in B_\rho} \BB P\left[ \wh\tau^{z,\ep} \leq \wh\sigma_{\lfloor \delta^{-3} \rfloor}^{z,\ep}(\delta) \,|\, h,\eta \right] \geq 1-\psi_0(\delta)  .
\eqe

Each $\wh \sigma_k^{z,\ep}(\delta)$ is a stopping time for the conditional law of $\wh X^{z,\ep}$ given $(h,\eta)$ satisfying $\wh X^{z,\ep}_{\wh \sigma_k^{z,\ep}(\delta)} \in \eta(\ep\BB Z)$.
By the strong Markov property of the conditional law of $\wh X^{z,\ep}$ given $(h,\eta)$, the conditional law of $\wh\sigma_k^{z,\ep}(\delta) - \wh\sigma_{k-1}^{z,\ep}(\delta)$ given $(h,\eta)$ and $\wh X^{z,\ep}|_{[0,\wh\sigma_{k-1}^{z,\ep}(\delta)]}$ is the same as the conditional law given $(h,\eta)$ of (a slight modification of) the exit time of $\wh X^{w,\ep}$ from $B_\delta(w)$ for $w = \wh X^{z,\ep}_{\wh\sigma_{k-1}^{z,\ep}(\delta)}$.

Let $\psi_1$ be the superpolynomial function as in Lemma~\ref{lem-exit-lower-conc}.
By applying Lemma~\ref{lem-exit-lower-conc} to bound each of the exit times in the preceding paragraph, taking a union bound over all $k\in [1,\delta^{-3}]_{\BB Z}$, and recalling~\eqref{eqn-increment-count-upper}, we find that with probability tending to 1 as $\ep \rta 0$ and then $\delta\rta 0$, the following is true.
Simultaneously for each $z\in B_\rho$, it holds with conditional probability at least $1 - \delta^{-3}\psi_1(\delta) - \psi_0(\delta)$ given $(h,\eta)$ that
\eqb \label{eqn-tight-increment}
\wh\sigma_k^{z,\ep}(\delta) - \wh\sigma_{k-1}^{z,\ep}(\delta) \geq   \delta^\chi ,\quad \text{$\forall k\in\BB N$ such that $\wh\sigma_{k-1}^{z,\ep}(\delta) \leq \wh\tau^{z,\ep}$}.
\eqe
The bound~\eqref{eqn-tight-increment} implies that $\wh X^{z,\ep} $ cannot travel Euclidean distance more than $2\delta$ in less than $\delta^\chi$ units of time (recall that $\wh X^{z,\ep}$ is constant after time $\wh\tau^{z,\ep}$).
We therefore obtain~\eqref{eqn-walk-cont} with $\psi(\delta) = \psi_0(\delta)  +\delta^{-3} \psi_1(\delta)$.
\end{proof}

\subsection{Equicontinuity of laws}
\label{sec-equicont-law}

In this section we establish the other part of the compactness criterion of Lemma~\ref{lem-compact} for the function $P_{\rho,\cdot}^\ep$ of~\eqref{eqn-tight-laws-rho}.

\begin{prop}[Equicontinuity of laws] \label{prop-walk-law-cont}
There exists $\xi = \xi(\rho ,  \gamma) > 0$, and $A = A(\rho,\gamma) > 0$ such that the following is true with probability tending to 1 as $\ep\rta 0$ and then $\delta\rta 0$.
For each $z,w \in  B_\rho $ with $|z-w| \leq \delta $, the Prokhorov distance (w.r.t.\ the local Skorokhod metric) between the conditional laws given $(h,\eta)$ of the stopped processes $t\mapsto \wh X^{z,\ep}_{t\wedge \wh\tau^{z,\ep} } $ and $t\mapsto \wh X^{w,\ep}_{t \wedge \wh\tau^{w,\ep} }$ is at most $A\delta^\xi$.
\end{prop}

Unlike for the exponent $\chi$ of Proposition~\ref{prop-walk-cont}, our proof of Proposition~\ref{prop-walk-law-cont}
does not give an explicit expression for the exponent $\xi$. We note that a similar continuity result is obtained for Liouville Brownian motion in \cite[Proposition 2.19]{grv-lbm} which allows
the authors to a.s.\ start Liouville Brownian motion from \emph{any} given point.

The basic idea of the proof of Proposition~\ref{prop-walk-cont} (which differs somewhat from \cite{grv-lbm}) is that if $|z-w| < \delta$ then the embedded, linearly interpolated walk $\wh X^{z,\ep}$ is likely to disconnect $w$ from $\bdy B_\rho$ before leaving $B_{\sqrt\delta}(z)$ (similar estimates are proven in Section~\ref{sec-harnack}).
We can then use Lemma~\ref{lem-disconnect-coupling} to couple $\wh X^{z,\ep}$ and $\wh X^{w,\ep}$ in such a way that they agree after their exit times from $B_{\sqrt \delta}(z)$ with high probability (see Lemma~\ref{lem-exit-tv}).
Both of these exit times are likely to be small (Corollary~\ref{cor-exit-moment-eucl}), so the uniform modulus of continuity estimate of Proposition~\ref{prop-walk-cont} tells us that under such a coupling, the uniform distance between $t\mapsto \wh X^{z,\ep}_{t\wedge \wh\tau^{z,\ep} } $ and $t\mapsto \wh X^{w,\ep}_{t \wedge \wh\tau^{w,\ep} }$ is small.

\begin{lem} \label{lem-exit-tv}
There exists $\xi = \xi(\rho ,  \gamma) > 0$, and $A = A(\rho,\gamma) > 0$ such that for each $\delta\in (0,1)$, it holds with polynomially high probability as $\ep\rta 0$ that the following is true.
For each $z\in B_\rho$ and each $x,y\in \mcl V\mcl G^\ep(B_\delta(z)) $, the total variation distance between the conditional laws given $(h,\eta)$ of the following two random variables is at most $A \delta^\xi$:
\begin{itemize}
\item The first vertex of $\mcl G^\ep( \bdy B_{\sqrt\delta}(z) )$ hit by a random walk on $\mcl G^\ep$ started from $x$.
\item The first vertex of $\mcl G^\ep( \bdy B_{\sqrt\delta}(z) )$ hit by a random walk on $\mcl G^\ep$ started from $y$.
\end{itemize}
Furthermore, if $x \in \mcl V\mcl G^\ep(B_\rho \setminus B_{\rho-\delta})$, then it holds with probability at least
$1-A\delta^\xi$ that the random walk on $\mcl G^\ep$ started from $x$ exits $\mcl G^\ep(B_\rho)$ before hitting $\mcl V\mcl G^\ep(\bdy B_{\sqrt\delta}(\eta(x) ))$.
\end{lem}
\begin{proof}
Let $\rho' = (1+\rho)/2 \in (\rho,1)$.
By~\cite[Lemma 3.11]{gms-harmonic} (which is a variant of Lemma~\ref{lem-walk-dc-uniform}), there exists $\beta =\beta(\gamma ) >0$ and $p = p(\rho,\gamma) \in (0,1)$ such that with polynomially high probability as $\ep \rta 0$, the following is true.
For each $z \in B_{\rho'}$ and each $r\in [\ep^\beta,\rho]$ such that $B_{2r}(z)\subset B_\rho$, we have
\allb \label{eqn-pos-disconnect}
&\min_{x \in \mcl V\mcl G^\ep(B_r(z))}
\ol{\BB P}_x^\ep\big[ \text{$X^\ep$ disconnects $\mcl V\mcl G^\ep(\bdy B_{r/2}(z))$ from $\mcl V\mcl G^\ep(\bdy B_{2r}(z))$} \notag \\
& \qquad\qquad\qquad\qquad\qquad\qquad \text{before hitting $\mcl V\mcl G^\ep(\bdy B_{2r}(z))$} \big]
\geq p .
\alle
Henceforth assume that this is the case and that $\ep$ is small enough (depending on $\delta$) so that $\ep^\beta < \delta /100$.
We also assume that $\delta$ is small enough that $\rho + 2\sqrt\delta < \rho'$, so that $B_{2\sqrt\delta}(z)\subset B_{\rho'}$ for each $z\in B_\rho$ (larger values of $\delta$ can be dealt with by increasing $A$).
We will show that the condition in the lemma statement is satisfied.

We now consider $z\in B_\rho$ and a random walk $X^\ep$ on $\mcl G^\ep$ started from $x\in \mcl V\mcl G^\ep(B_\rho)$.
We apply the strong Markov property of the conditional law of $X^\ep$ given $\mcl G^\ep$ at the exit time of $X^\ep$ from $B_{2^{-k}}(\eta(x))$ for each $k\in\BB N$ such that $2^{-k} \in [2\delta  , \sqrt\delta/2]$. This together with our above constraints on $\ep$ and $\delta$ allows us to apply estimate~\eqref{eqn-pos-disconnect} with $z  =\eta(x)$ and $r = 2^{-k}$, then multiply over all $k$ for which $2^{-k} \in [2\delta  , \sqrt\delta/2]$, to get
\allb \label{eqn-pos-disconnect'}
&\min_{x \in \mcl V\mcl G^\ep(B_{2\delta }(z)}
\ol{\BB P}_x^\ep\big[ \text{$X^\ep$ disconnects $\mcl V\mcl G^\ep(\bdy B_{\delta }(z))$ from $\mcl V\mcl G^\ep(\bdy B_{\sqrt\delta}(z))$
before hitting $\mcl V\mcl G^\ep(\bdy B_{\sqrt\delta}(z ))$} \big] \notag\\
&\qquad \qquad \qquad \qquad \qquad \geq 1 - (1-p)^{\lfloor \log_2 \delta^{-1} \rfloor - 2} \geq 1 - A \delta^\xi
\alle
for constants $\xi ,A > 0$ as in the lemma statement.

If $x,y\in \mcl V\mcl G^\ep(B_\delta(z))$, then by~\eqref{eqn-pos-disconnect'}, it holds with $\ol{\BB P}_x^\ep$-probability at least $1-A\delta^\xi$ that $X^\ep$ disconnects $y$ from $\mcl V\mcl G^\ep(\bdy B_{\sqrt\delta}(z))$ before hitting $\mcl V\mcl G^\ep(\bdy B_{\sqrt \delta}(z ))$. By Lemma~\ref{lem-disconnect-coupling}, this implies the total variation condition in the lemma statement.

To get the last statement, we observe that if $x \in \mcl V\mcl G^\ep(B_\rho \setminus B_{\rho-\delta})$ and the random walk started from $x$ disconnects $\mcl V\mcl G^\ep(\bdy B_{\delta }(\eta(x) ))$ from $\mcl V\mcl G^\ep(\bdy B_{\sqrt\delta}(\eta(x)))$ before hitting $\mcl V\mcl G^\ep(\bdy B_{\sqrt\delta}(\eta(x)))$, then this walk must exit $\mcl G^\ep(B_\rho)$ before hitting $\mcl V\mcl G^\ep(\bdy B_{\sqrt\delta}(\eta(x) ))$. The last statement of the lemma therefore follows from~\eqref{eqn-pos-disconnect'}.
\end{proof}

\begin{proof}[Proof of Proposition~\ref{prop-walk-law-cont}]
\noindent \textit{Step 1: coupling walks so that they agree after exiting a ball of radius $\sqrt\delta$.}
By Lemma~\ref{lem-exit-tv}, there exists $\xi_0 , A_0 > 0$ depending only on $\rho,\gamma$ such that with polynomially high probability as $\ep\rta 0$, the following is true.
For each $z,w\in B_\rho$ with $|z-w| \leq \delta$, the total variation distance between the conditional laws given $(h,\eta)$ of the first vertex of $\mcl V\mcl G^\ep(\bdy B_{\sqrt \delta}(z))$ hit by $\wh X^{z,\ep}$ and the first vertex of $\mcl V\mcl G^\ep(\bdy B_{\sqrt\delta}(z))$ hit by $\wh X^{w,\ep}$ is at most $A_0 \delta^{\xi_0}$.
Due to the strong Markov property of the conditional law of random walk on $\mcl G^\ep$ given $(h,\eta)$, it follows that on this event we can couple the conditional laws given $(h,\eta)$ of $\wh X^{w,\ep}$ and $\wh X^{z,\ep}$ in such a way that the following is true.
If we let $T^z = T^{z,\ep}_\delta$ (resp.\ $T^w = T^{w,z,\ep}_\delta$) be the exit time of $\wh X^{z,\ep}$  (resp.\ $\wh X^{w,\ep}$) from\footnote{We use $B_{2\sqrt\delta}(z)$ instead of $B_{\sqrt \delta}(z)$ here since the points $\eta(x)$ for $x\in\mcl V\mcl G^\ep(\bdy B_{\sqrt\delta}(z))$ are not necessarily contained in $B_{\sqrt\delta}(z)$. However, the cell $\eta([x-\ep,x])$ for each such $x$ must intersect $B_{\sqrt\delta}(z)$ so by Lemma~\ref{lem-cell-diam} it holds with polynomially high probability as $\ep\rta0$ that $\eta(x)$ for each such $x$ is contained in $B_{2\sqrt\delta}(z)$.}
$B_{2\sqrt\delta}(z)$, then
\eqb \label{eqn-walk-coupling0}
\BB P\left[ \wh X^{z,\ep}_{t + T^z} = \wh X^{w,\ep}_{t+T^w},\:\forall t \geq 0 \,\big| h,\eta \right] \geq 1 - A_0 \delta^{\xi_0} .
\eqe
If $z,w \in B_{\rho-2\sqrt \delta}$, then necessarily $T^z \leq \wh\tau^{z,\ep}$ and $T^w\leq \wh\tau^{w,\ep}$ so it follows from~\eqref{eqn-walk-coupling0} that
\eqb \label{eqn-walk-coupling}
\BB P\left[ \wh X^{z,\ep}_{(t + T^z) \wedge \wh\tau^{z,\ep}  } = \wh X^{w,\ep}_{(t+T^w)\wedge \wh\tau^{z,\ep} },\quad\forall t \geq 0 \,\big| h,\eta \right] \geq 1 - A_0 \delta^{\xi_0} .
\eqe
\medskip

\noindent \textit{Step 2: the effect of the initial time increment.}
For most of the rest of the proof we assume that $z,w \in B_{\rho-2\sqrt \delta}$, hence we can couple $\wh X^{z,\ep}$ and $\wh X^{w,\ep}$ so that~\eqref{eqn-walk-coupling} holds. We will treat the case when either $z$ or $w$ belongs to $B_\rho \setminus B_{\rho-2\sqrt\delta}$ at the very end of the proof.

To deduce the proposition statement from~\eqref{eqn-walk-coupling} we need to show that the behavior of $\wh X^{z,\ep}$ and $\wh X^{w,\ep}$ before time $T^z$ and $T^w$, respectively, has a negligible effect on the uniform distance between $t\mapsto \wh X^{z,\ep}_{t\wedge \wh\tau^{z,\ep} } $ and $t\mapsto \wh X^{w,\ep}_{t \wedge \wh\tau^{w,\ep} }$ with high probability when $\delta$ is small.
We first show that the times $T^z$ and $T^w$ are typically small.
To this end, let $0 < q < q' < (2-\gamma)^2/2$.
By Corollary~\ref{cor-exit-moment-eucl} and Markov's inequality, it holds with probability tending to 1 as $\ep\rta 0$ and then $\delta \rta 0$ that the following is true.
For each $z,w\in B_\rho$ with $|z-w| \leq \delta$,
\eqb \label{eqn-walk-law-exit}
\BB P\left[ T^z \leq  \delta^{q/4} \,|\, (h,\eta) \right] \geq 1 -  \delta^{q/4} \quad\text{and} \quad
\BB P\left[ T^w \leq   \delta^{q/4} \,|\, (h,\eta) \right] \geq 1 -  \delta^{q/4} .
\eqe

We next use Proposition~\ref{prop-walk-cont} to bound the effect on the walk from shifting time by at most $\delta^{q/4}$.
Let $\chi > {\frac{2(2+\gamma)}{2-\gamma}}$, as in Proposition~\ref{prop-walk-cont}.
By Proposition~\ref{prop-walk-cont} (applied with $\delta^{q/(4\chi)}$ in place of $\delta$), there is a superpolynomial function $\psi$ depending only on $\rho,\chi$ such that the following is true with probability tending to 1 as $\ep\rta 0$ and then $\delta\rta 0$.
For each $z\in B_\rho$,
\eqb \label{eqn-walk-law-equicont}
\BB P\left[  |\wh X^{z,\ep}_{t} - \wh X^{z,\ep}_{s}| \leq 2\delta^{q/(4\chi)} ,\: \forall s,t\in [0,\wh\tau^{z,\ep}] \: \text{with} \: |s-t| \leq \delta^{q/4} \,|\,  h,\eta \right]
\geq 1 - \psi(\delta) .
\eqe

Let us now condition on $(h,\eta)$ and assume that~\eqref{eqn-walk-coupling} and~\eqref{eqn-walk-law-exit} hold for every $z,w\in B_\rho$ with $|z-w|\leq \delta$ and~\eqref{eqn-walk-law-equicont} holds for every $z\in B_\rho$, which happens with probability tending to 1 as $\ep\rta 0$ and then $\delta \rta 0$.
Fix $z,w\in B_{\rho-2\sqrt\delta}$ with $|z-w|\leq \delta$ and suppose we have coupled $\wh X^{z,\ep}$ and $\wh X^{w,\ep}$ as in~\eqref{eqn-walk-coupling}.
Except on an event whose conditional probability given $(h,\eta)$ is bounded above by $A\delta^\xi$ for $A,\xi$ as in the lemma statement, the events in~\eqref{eqn-walk-coupling} and~\eqref{eqn-walk-law-exit} occur and also the event in~\eqref{eqn-walk-law-equicont} occurs for each of $\wh X^{z,\ep}$ and $\wh X^{w,\ep}$.
If this is the case, then we can bound the uniform distance between $t\mapsto \wh X^{z,\ep}_{t\wedge \wh\tau^{z,\ep} } $ and $t\mapsto \wh X^{w,\ep}_{t \wedge \wh\tau^{w,\ep} }$ as follows.

If $s \in [0,T^z\vee T^w]$, then by the event in~\eqref{eqn-walk-law-exit} we have $s  \leq \delta^{q/4}$ so by~\eqref{eqn-walk-law-equicont} (applied once to each of $z$ and $w$),
\eqb \label{eqn-walk-close-small}
|\wh X^{z,\ep}_s - \wh X^{w,\ep}_s| \leq 4\delta^{q/(4\chi)} + |\wh X^{z,\ep}_0 - \wh X^{w,\ep}_0 | \leq 4\delta^{q/(4\chi)} + \delta \leq 5\delta^{q/(4\chi)} .
\eqe
Alternatively, if $s   >  T^z \vee T^w  $, then we can find $t^z,t^w  \geq 0$ such that $s  = t^z + T^z = t^w + T^w$.
By the event in~\eqref{eqn-walk-law-exit}, $|t^z - t^w| = |T^z - T^w| \leq   \delta^{q/4}$.
By the events in~\eqref{eqn-walk-coupling} and~\eqref{eqn-walk-law-equicont}, we therefore have
\eqb \label{eqn-walk-close-large}
|\wh X^{z,\ep}_{s \wedge \wh\tau^{z,\ep} } - \wh X^{w,\ep}_{s \wedge \wh\tau^{w,\ep} }|
\leq |\wh X^{z,\ep}_{(t^z  + T^z) \wedge \wh\tau^{z,\ep} } - \wh X^{w,\ep}_{(t^z + T^w) \wedge \wh\tau^{w,\ep}} |  +  |\wh X^{w,\ep}_{(t^z + T^w) \wedge \wh\tau^{w,\ep}} - \wh X^{w,\ep}_{(t^w + T^w) \wedge \wh\tau^{w,\ep}} |
\leq 2\delta^{q/(4\chi)} .
\eqe
By~\eqref{eqn-walk-close-small} and~\eqref{eqn-walk-close-large}, the uniform distance between $\wh X^{z,\ep}$ and $\wh X^{w,\ep} $ is at most $5\delta^{q/(4\chi)}$.
By possibly increasing $A$ and/or decreasing $\xi$, we now obtain the proposition statement in the case when $z,w\in B_{\rho-2\sqrt\delta}$.
\medskip

\noindent\textit{Step 3: points close to the boundary.}
If $|z-w|\leq \delta$ and either $z$ or $w$ belongs to $B_\rho\setminus B_{\rho-2\sqrt\delta}$, then $z,w\in B_\rho \setminus B_{\rho - 3\sqrt\delta}$.
By the last assertion of Lemma~\ref{lem-exit-tv} applied with $3\sqrt\delta$ in place of $\delta$, if $z\in B_{\rho-3\sqrt\delta}$ then holds with polynomially high probability as $\ep\rta 0$ that
\eqb \label{eqn-close-bdy-exit}
\BB P\left[ \text{$\wh X^{z,\ep}$ stays in $B_{\sqrt 3 \delta^{1/4}}(z)$ until time $\wh\tau^{z,\ep}$} \right] \geq 1 - 3^\xi A_0 \delta^{\xi_0/2} .
\eqe
If the event in~\eqref{eqn-close-bdy-exit} occurs then $|\wh X^{z,\ep}_{t \wedge\wh\tau^{z,\ep}} - z| \leq \sqrt 3\delta^{1/4}$ for all $t\geq 0$.
Hence we obtain the proposition statement in the case when either $z$ or $w$ belongs to $B_\rho\setminus B_{\rho-3\sqrt\delta}$ after possibly increasing $A$ and/or decreasing $\xi$.
\end{proof}

\subsection{The walk does not get stuck}
\label{sec-no-stuck}

In this subsection we will prove a proposition which will eventually imply the last statement of Proposition~\ref{prop-tight-rho}, which asserts that $\wh X^z$ does not remain at a single point for a positive interval of times.
We will then conclude the proof of Proposition~\ref{prop-tight-rho}.

\begin{prop} \label{prop-no-stuck}
Fix $q \in \left( 0 , (2-\gamma)^2/2\right)$.
There is a superpolynomial function $\psi$ depending only on $\zeta,\gamma$ such that with probability tending to 1 as $\ep\rta 0$ and then $\delta \rta 0$, it holds for each $x\in \mcl V\mcl G^\ep(B_\rho)$ that
\eqb \label{eqn-no-stuck}
\ol{\BB P}_x^\ep \left[ \min_{j\in [0,\tau_{B_\rho}^\ep]_{\BB Z}} \max_{i_1,i_2 \in [j , j + 2\ep^{-1} \delta^q ]_{\BB Z}} |\eta(X_{i_1}^\ep) - \eta(X_{i_2}^\ep)| \geq  \delta \right] \geq 1 - \psi(\delta) .
\eqe
\end{prop}
\begin{proof}
Fix $q' \in \left(q , (2-\gamma)^2/2\right)$.
By Corollary~\ref{cor-exit-moment-eucl} applied with $q'$ in place of $q$, it holds with probability tending to 1 as $\ep\rta 0$ and then $\delta \rta 0$ that for every $N\in\BB N$,
\eqb \label{eqn-use-exit-moment-eucl}
 \max_{x\in\mcl V\mcl G^\ep(B_\rho)} \ol{\BB E}_x^\ep \left[ (\tau_{B_\delta(\eta(x) )}^\ep)^N  \right]
\leq N! \ep^{-N} \delta^{N  q' } .
\eqe
Basically, the proposition follows by applying~\eqref{eqn-use-exit-moment-eucl}, the Chebyshev inequality, and a union bound in the appropriate manner.

To be precise, define the times $\sigma_k^\ep(\delta)$ for $k\in\BB N_0$ as in Lemma~\ref{lem-walk-increment'}: that is, let $\sigma_0^\ep(\delta) = 0$ and for $k\in \BB N$ inductively let $\sigma_k^\ep(\delta)$ be the first time $j \geq \sigma_{k-1}^\ep(\delta)$ for which $X_j^\ep \notin \mcl V\mcl G^\ep(B_\delta(\eta( X_{\sigma_{k-1}^\ep(\delta)} ) ) )$.
By the Markov property of $X^\ep$ under $\ol{\BB P}_x^\ep$, the $\ol{\BB P}_x^\ep$-conditional law of $\sigma_k^\ep(\delta) - \sigma_{k-1}^\ep(\delta)$ given $X^\ep|_{[0,\sigma_{k-1}^\ep(\delta)]}$  is the same as the $\ol{\BB P}_{X_{\sigma_{k-1}^\ep(\delta)}}$-law of the exit time of $X^\ep$ from $\mcl V\mcl G^\ep(B_\delta( \eta( X_{\sigma_{k-1}^\ep(\delta)}) ) )$.

Using~\eqref{eqn-use-exit-moment-eucl}, the Chebyshev inequality, and a union bound over all $k\in [0,\delta^{-3}]_{\BB Z}$, we find that for any fixed $\zeta \in (0,1)$ and $N\in\BB N$,
\eqb \label{eqn-no-stuck-inc}
 \min_{x\in\mcl V\mcl G^\ep(B_\rho)} \ol{\BB P}_x^\ep\left[   \sigma_k^\ep(\delta) - \sigma_{k-1}^\ep(\delta)  \leq \ep^{-1} \delta^q , \: \forall k \in [0,\delta^{-3}]_{\BB Z} \: \text{with} \: \sigma_{k-1}^\ep(\delta) \leq \tau_{B_\rho}^\ep \right]
\geq 1 - N! \delta^{N(q'-q)-3} .
\eqe
On the other hand, it follows from Lemma~\ref{lem-walk-increment'} (applied with $\zeta  = 1$) that there is a superpolynomial function $\psi_0$ such that with probability tending to 1 as $\ep\rta 0$ and then $\delta \rta 0$,
\eqb \label{eqn-no-stuck-count}
 \min_{x\in\mcl V\mcl G^\ep(B_\rho)} \ol{\BB P}_x^\ep\left[  \sigma_{B_\rho}^\ep \leq  \sigma_{\lfloor \delta^{-3} \rfloor}^\ep (\delta)  \right]
\geq 1 - \psi_0(\delta)  .
\eqe

By sending $N\rta\infty$ in~\eqref{eqn-no-stuck-inc} and combining with~\eqref{eqn-no-stuck-count}, we find that there is a superpolynomial function $\psi$ such that with probability tending to 1 as $\ep\rta 0$ and then $\delta\rta 0$,
\eqb \label{eqn-no-stuck-inc'}
 \min_{x\in\mcl V\mcl G^\ep(B_\rho)} \ol{\BB P}_x^\ep\left[   \sigma_k^\ep(\delta) - \sigma_{k-1}^\ep(\delta)  \leq \ep^{-1} \delta^q , \: \forall k \in \BB N \: \text{with} \: \sigma_{k-1}^\ep(\delta) \leq \tau_{B_\rho}^\ep \right]
\geq 1 - \psi(\delta) .
\eqe
If the event in~\eqref{eqn-no-stuck-inc'} occurs for some $x\in \mcl V\mcl G^\ep(B_\rho)$, then  every interval of times of the form $[a,b] \subset [0,\tau_{B_\rho}^\ep]$ with $b-a \geq 2\ep^{-1} \delta^q$ must contain $[\sigma_{k-1}^\ep(\delta) , \sigma_k^\ep(\delta)]$ for some $k\in \BB N$ with $\sigma_{k-1}^\ep(\delta) \leq \tau_{B_\rho}^\ep$.
By definition, $\eta(X^\ep)$ travels Euclidean distance at least $\delta$ during each such time interval, so we get~\eqref{eqn-no-stuck}.
\end{proof}

\begin{proof}[Proof of Proposition~\ref{prop-tight-rho}]
Define $P_{\rho,\cdot}^\ep : \BB C\rta \op{Prob}(C([0,\infty),\BB C))$ as in~\eqref{eqn-tight-laws-rho} and let $\wt P_{\rho,\cdot}^\ep$ be a continuous approximation of $P_{\rho,\cdot}^\ep$ as in the discussion just after~\eqref{eqn-cont-space-def}.
Since $P_{\rho,\eta(x)}^\ep =  \wt P_{\rho,\eta(x) }^\ep$ for each $x\in\ep\BB Z$ it is easily seen that if $\ep_n\rta 0$ is a sequence along which $\wt P_{\rho,\cdot}^{\ep_n}$ converges in law, then also $P_{\rho,\cdot}^{\ep_n}$ converges in law to the same limit.
Therefore, to prove tightness for $P_{\rho,\cdot}^\ep $ we only need to prove tightness for $\wt P_{\rho,\cdot}^\ep$.
The subsequential limits will be continuous since continuity is preserved under uniform convergence.

By the Prokhorov theorem, we need to show that for each $\zeta > 0$, there is a compact
subset $\mcl K_\zeta$ of the space $\mcl C$ of continuous functions $\BB C \rta \op{Prob}(C([0,\infty) , \BB C))$ such that for each small enough $\ep > 0$, we have $\BB P[\wt P_{\rho,\cdot}^\ep \in \mcl K_\zeta] \geq 1-\zeta$.
As a consequence of Lemma~\ref{lem-compact},
 it suffices to establish the following.
\begin{enumerate}
\item (Equicontinuity of laws) For each $L   >1$, there exists $\delta > 0$ such that for each small enough $\ep > 0$, it holds with (annealed) probability at least $1-1/L$ that for each $z,w\in B_\rho$ with $|z-w|\leq \delta$, the Prokhorov distance between the laws $P_{\rho,z}^\ep$ and $P_{\rho,w}^\ep$ is at most $1/L$. \label{item-tight-laws}
\item (Equicontinuity of paths) For each $z \in B_\rho$ and each $L > 1$, there exists $\delta = \delta(z,L) > 0$ such that for each small enough $\ep > 0$, the following is true. With (annealed) probability at least $1-1/L$, \label{item-tight-paths}
\eqbn
  |\wh X^{z,\ep}_s -\wh X^{z,\ep}_s| \leq 1/L , \quad\forall s,t\in [0,\wh\tau^{z,\ep}] \quad\text{with} \quad   |s-t| \leq \delta .
\eqen
\end{enumerate}
While this seems very similar to the statement of Lemma~\ref{lem-compact},
there is in fact a difference: the compactness criterion of that lemma require us to check the
equicontinuity of laws and paths conditions with \emph{quenched} probability at least $1-\zeta$ (i.e.,
conditioned on $h$ and $\eta$), whereas we claim here it suffices to check these conditions with \emph{annealed}
(i.e., unconditional) probability. That this is sufficient follows from a dyadic argument
similar to the one in the proof of Lemma~\ref{lem-compact}: in particular, we apply the conditions with $L = 2^k/\zeta$ for $k\in\BB N$, then take a union bound over $k$.

Recall that for $z\notin B_\rho$, $P_{\rho,z}^\ep$ is simply the point mass at the constant $t\mapsto x_z^\ep$, which is why we only need to consider $z\in B_\rho$.
Conditions~\ref{item-tight-laws} and~\ref{item-tight-paths} follow from Propositions~\ref{prop-walk-law-cont} and~\ref{prop-walk-cont}, respectively.
Finally, the last assertion of the proposition (that $\wh X^z$ does not stay at a single point for a positive interval of time until after exiting $B_\rho$) is immediate from Proposition~\ref{prop-no-stuck}.
\end{proof}

\subsection{Proof of Proposition~\ref{prop-tight}}
\label{sec-tight-proof}

We now deduce Proposition~\ref{prop-tight} from Proposition~\ref{prop-tight-rho} and the scale invariance property of the $\gamma$-quantum cone.
For $b > 0$, let $R_b := \sup\left\{ r > 0  : h_r(0) + Q\log r =  \frac{1}{\gamma} \log b \right\} $ and $h^b := h(R_b\cdot) + Q \log R_b -\frac{1}{\gamma} \log b$, as in Lemma~\ref{lem-cone-scale}.
Also let $\eta^b := R_b^{-1} \eta(b\cdot)$.
By Lemma~\ref{lem-cone-scale} and the scale invariance of the law of $\eta$, we have $(h^b,\eta^b) \eqD (h,\eta)$.
Consequently,
\eqb \label{eqn-scale-cells}
\{\eta([x-\ep,x]) : x\in \ep\BB Z \} \eqD \{\eta^b([ x- \ep,  x]) \}_{x\in \ep\BB Z} = \{R_b^{-1} \eta([y-b\ep ,y]) \}_{y \in b\ep\BB Z} .
\eqe
For $u\in\BB C$ and $\ep > 0$, we define the linearly interpolated walk
\eqb \label{eqn-scale-field}
\wh X^{b,u,\ep}_t : = R_b^{-1} \wh X^{R_b u , \ep}_{t/b} ,
\eqe
which is related to $(h^b,\eta^b)$ in the same manner that $\wh X^{u,\ep}$ is related to $(h,\eta)$.

For $u\in B_\rho$, let $\wh\tau^{b,u,\ep}$ be the first time that the walk $\wh X^{b,u,\ep}_t$ hits a point of the form $\eta(y)$ for $y\in \mcl V\mcl G^\ep \setminus \mcl V\mcl G^\ep( B_{ \rho} )$.
Equivalently, for $z\in B_{R_b\rho} $, we have that $b \wh\tau^{b, z / R_b ,\ep}$ is the first time that the walk $ \wh X^{z,\ep} $ hits a point of the form $\eta(x)$ for $x\in \mcl V\mcl G^\ep \setminus \mcl V\mcl G^\ep( B_{ R_b \rho} )$.

Since $(h^b,\eta^b)$ and $(h,\eta)$ generate the same $\sigma$-algebra, it follows from~\eqref{eqn-scale-cells} that the function
\eqb \label{eqn-scale-law}
P_{\rho,\cdot}^{b,\ep} : u \mapsto \left\{\text{Conditional law of $t\mapsto \wh X^{b,u,\ep}_{t\wedge \wh\tau^{b,u,\ep}}$ given $(h,\eta)$} \right\}
\eqe
has the same law as $P_{\rho,\cdot}^\ep$.
Therefore, Proposition~\ref{prop-tight-rho} implies that the laws of the functions~\eqref{eqn-scale-law} are tight and if $P_{\rho,\cdot}^{(b)}$ is a subsequential limit in law of these functions, then a.s.\ $P_{\rho,\cdot}^{(b)}$ is continuous and the following is true.
If $z\in\BB C$ and $\wh X^z$ is a path with the law $P_{\rho,z}^{(b)}$, then a.s.\ $\wh X^z$ does not stay at a single point for a positive interval of time before it exits $B_\rho$.

Unpacking the definitions, we see that for any $b >0$ the laws of the functions
\eqb
z \mapsto \left\{\text{Conditional law given $(h,\eta)$ of $\wh X^{z,\ep}$ stopped upon exiting $B_{\rho R_b}$} \right\}
\eqe
are tight. Moreover, if $P^{(b)}$ is any subsequential limit in law of these functions, then a.s.\ $P$ is continuous and the following is true.
If $z\in\BB C$ and $\wh X^z$ is a path with the law $P^{(b)}_z$, then a.s.\ $\wh X^z$ does not stay at a single point for a positive interval of time before it exits $B_{\rho R_b}$. Since $R_b\rta \infty$ a.s.\ as $b\rta\infty$, the conclusion of Proposition~\ref{prop-tight} now follows.
\qed

\section{Identifying the subsequential limit}
\label{sec-characterization}

We equip the space of distributions on $\BB C$ with the usual weak topology (whereby convergence of distributions is equivalent to convergence of the integrals against all smooth compactly supported text functions).
We also equip the space of continuous curves $\BB R\rta\BB C$ with the local uniform topology.
We know from Proposition~\ref{prop-tight} that the laws of the functions $P^\ep : \BB C\rta \op{Prob}(C([0,\infty),\BB C))$ of~\eqref{eqn-tight-laws} are tight.
By the Prokhorov theorem also the joint laws of $(h,\eta,P^\ep)$ as $\ep$ varies are tight.

Before we explain the main ideas behind the identification of $(h, \eta, P)$, we point out that \emph{a priori} we do not even know that $P$ is determined by $(h, \eta)$. To explain what is the issue, note first that if we used the known tightness of $P^\ep$ \emph{given} $(h, \eta)$, then any subsequential limit $P$ of the \emph{family of laws} $P^\ep$ could only be random insofar as $(h, \eta)$ is random: i.e., $P$ would necessarily be determined by $(h, \eta)$. However, this approach would force us to deal with possibly \emph{random} subsequences $\ep$ along which the convergence to $P$ occurs and would add a serious layer of complication to our arguments. Instead, we prefer to work with the law of the triplet $(h, \eta, P^\ep)$ which is also tight and consider a subsequential limit $(h, \eta, P)$.  The trade-off, however, is that we do not know \emph{a priori} that $P$, viewed as a random variable in the space $\mcl C$ of continuous functions from $\BB C$ to $\op{Prob}( C([0, \infty) , \BB C))$, is measurable with respect to $(h, \eta)$: indeed, it is straightforward to construct examples of random variables $(X, Y_n)$ converging to $(X,Y)$ in distribution with $Y_n$ measurable with respect to $X$, but $Y$ not measurable with respect to $X$. This lack of \emph{a priori} measurability will require specific arguments in various places (especially in Section~\ref{sec-lbm-conv}).

We want to show that if $(h,\eta,P)$ is a subsequential limit of these joint laws, then $P_z$ is the law of the $\gamma$-Liouville Brownian motion associated with $h$ started from $z$ for each $z\in\BB C$, up to a linear time-change. To do this, we first establish some basic properties (such as the Markov property) of any subsequential limit in Section~\ref{sec-ssl-properties}. These properties are easy consequences of basic properties of random walk on random planar maps.
We then prove a general uniqueness criterion for Markovian time-changed Brownian motions in Section~\ref{sec-bm-unique} (using general theory of Markov processes).
This criterion will show that $P_z$ for $z\in \BB C$ is the law as that of Liouville Brownian motion pre-composed with the linear time change $t\mapsto c t$ for some random $c>0$ which does not depend on $c$, but which may depend on $(h,\eta,P)$. In Section~\ref{sec-lbm-conv} we will argue that $c$ is a.s.\ equal to a deterministic constant.
Then, we will change the time scaling in order to eliminate the dependence of this constant on the choice of subsequence, as discussed at the end of Section~\ref{sec-outline}.
This will complete the proof of Theorem~\ref{thm-lbm-conv}.
Finally, in Section~\ref{sec-plane-to-disk} we will explain how to deduce the disk version of the theorem statement, Theorem~\ref{thm-lbm-conv_disk}, from Theorem~\ref{thm-lbm-conv}.

We now record some definitions which will play a fundamental role in this section.
In each definition, we let $P : \BB C \rta \op{Prob}\left(C([0,\infty),\BB C) \right) \in \mcl{C}$ be a function from $\BB C$ to the set of probability measures on continuous paths $[0,\infty)\rta\BB C$ such that $P_z$ for each $z\in\BB C$ is supported on paths started from $z$.
We recall that the topology on the space of such functions $P$ is defined in Section~\ref{sec-tight}.

\begin{defn} \label{def-markovian}
We say that $P$ is \emph{Markovian} if the following is true.
Let $z\in \BB C$ and let $X^z$ be sampled from $P_z$. For each $s\geq 0$, the conditional law of $X_{s+\cdot}^z$ given $X^z|_{[0,s]}$ is a.s.\ equal to $P_{X_s}$.
\end{defn}

\begin{defn} \label{def-time-change}
We say that $P$ defines a \emph{time-changed Brownian motion} if the following is true.
Let $z\in \BB C$ and let $X^z$ be sampled from $P_z$.
Then there is a standard planar Brownian motion $\mcl B$ started from $z$ and a continuous, strictly increasing process $u \mapsto \tau_u$ with $\tau_0 = 0$ such that a.s.\ $X^z_u  = \mcl B_{\tau_u}$ for each $u\geq 0$.
We say that $P$ defines a \emph{Markovian time changed Brownian motion} if also $P$ is Markovian.
\end{defn}

\begin{defn} \label{def-invariant}
We say that a measure $\mu$ on $\BB C$ is \emph{invariant} for $P$ if for every Borel set $A\subset\BB C$,
\eqb \label{eqn-invariant}
\int_{\BB C} P_z\left[ X_t \in A    \right]  d\mu (z)
=    \mu(A)
\eqe
We say that $\mu$ is \emph{reversible} for $P$ if for every pair of Borel sets $A,B\subset\BB C$,
\eqb\label{eqn-reversible}
\int_B P_z\left[ X_t \in A    \right]  d\mu (z)  = \int_A P_w\left[ X_t  \in B    \right]  d\mu (w) .
\eqe
\end{defn}

We note that if $\mu$ is reversible for $P$, then $\mu$ is stationary for $P$: just take $B = \BB C$ in~\eqref{eqn-reversible}.

\subsection{Properties of subsequential limits}
\label{sec-ssl-properties}

Let $(h,\eta,P)$ be distributed according to a subsequential scaling limit of the joint law of $(h,\eta,P^\ep)$ as $\ep\rta 0$.
Note that $P $ is a \emph{continuous} function $\BB C \rta \op{Prob}(C([0,\infty) , \BB C))$.
In this section we will establish some properties for $P$ which are known to also hold for Liouville Brownian motion due to results in~\cite{grv-lbm}. The key behind these properties is the a priori continuity in the starting point $z$ of the laws $P_z$. More precisely, we show that $P$ defines a Markovian time-changed Brownian motion and the LQG measure $\mu_h$ is reversible for $P$. The proof of the first statement uses only the Markov property for random walk on $\mcl G^\ep$ and the fact that $P^\ep \rta P$ in law w.r.t.\ the local \emph{uniform (in $z$)} topology. The proof of the second statement uses the fact that the counting measure on vertices of $\mcl G^\ep$, weighted by their degree, is reversible for the random walk on $\mcl G^\ep$; the fact that this measure converges weakly to $\mu_h$ under the SLE/LQG embedding (Corollary~\ref{cor-deg-conv}); and the aforementioned convergence $P^\ep\rta P$.

\begin{lem} \label{lem-ssl-markov}
Almost surely, the random function $P$ defines a Markovian time-changed Brownian motion in the sense of Definition~\ref{def-time-change}.
\end{lem}

To check the Markov property for $P$ in proof of Lemma~\ref{lem-ssl-markov}, we will use the following criterion for the convergence of conditional laws (see, e.g.,~\cite[Lemma 4.3]{gp-sle-bubbles}).

\begin{lem} \label{lem-cond-law-conv}
Let $(X_n,Y_n)$ be a sequence of pairs of random variables taking values in a product of separable metric spaces $\Omega_X\times\Omega_Y$ and let $(X,Y)$ be another such pair of random variables such that $(X_n,Y_n) \rta (X,Y)$ in law. Suppose further that there is a family of probability measures $\{P_y : y\in \Omega_Y\}$ on $\Omega_X$, indexed by $\Omega_Y$, such that for each bounded continuous function $f : \Omega_X\rta\BB R $, $y\mapsto \BB E_{P_y}(f)$ is a measurable function of $y\in\Omega_Y$ and
\eqb \label{eqn-cond-law-hyp}
\left( \BB E\left[f(X_n) \, |\, Y_n  \right] , Y_n \right)  \rta \left( \BB E_{P_Y}(f) , Y \right)  \quad \text{in law}.
\eqe
Then $P_Y$ is a regular conditional law of $X$ given $Y$.
\end{lem}
\begin{proof}
  This elementary lemma (stated and proved in~\cite[Lemma 4.3]{gp-sle-bubbles}) is a straightforward consequence of the uniqueness of the conditional distribution as a probability kernel (see \cite[Theorem 6.5]{kallenberg} and preceding definitions) since the assumption that $y \mapsto \BB E_{P_y} (f)$ is measurable guarantees that $(y, B) \mapsto P_y(B)$ for $y \in \Omega_Y$ and $B$ a Borel set in $\Omega_X$ is a probability kernel.
\end{proof}

The proof of Lemma~\ref{lem-ssl-markov} also involves the topology on curves viewed modulo time parametrization, which we now briefly recall.
If $\beta_1 : [0,T_{\beta_1}] \rta \BB C$ and $\beta_2 : [0,T_{\beta_2}] \rta \BB C$ are continuous curves defined on possibly different time intervals, we set
\eqb \label{eqn-cmp-metric}
\BB d^{\op{CMP}} \left( \beta_1,\beta_2 \right) := \inf_{\phi } \sup_{t\in [0,T_{\beta_1} ]} \left| \beta_1(t) - \beta_2(\phi(t)) \right|
\eqe
where the infimum is over all increasing homeomorphisms $\phi : [0,T_{\beta_1}]  \rta [0,T_{\beta_2}]$ (the CMP stands for ``curves modulo parameterization"). It is shown in~\cite[Lemma~2.1]{ab-random-curves} that $\BB d^{\op{CMP}}$ induces a complete metric on the set of curves viewed modulo time parameterization.

Since we will be working with curves defined for infinite time, we will use a local variant of the metric $\BB d^{\op{CMP}}$. Suppose $\beta_1 : [0,\infty) \rta \BB C$ and $\beta_2 : [0,\infty) \rta \BB C$ are two curves. For $r > 0$, let $T_{1,r}$ (resp.\ $T_{2,r}$) be the first exit time of $\beta_1$ (resp.\ $\beta_2$) from the ball $B_r(0)$ (or 0 if the curve starts outside $B_r(0)$).
We define
\eqb \label{eqn-cmp-metric-loc}
\BB d^{\op{CMP}}_{\op{loc}} \left( \beta_1,\beta_2 \right) := \int_1^\infty e^{-r} \left( 1 \wedge \BB d^{\op{CMP}}\left(\beta_1|_{[0,T_{1,r}]} , \beta_2|_{[0,T_{2,r}]} \right) \right) \, dr ,
\eqe
so that $\BB d^{\op{CMP}}_{\op{loc}} (\beta^n , \beta) \rta 0$ if and only if for Lebesgue a.e.\ $r > 0$, $\beta^n$ stopped at its first exit time from $B_r(0)$ converges to $\beta$ stopped at its first exit time from $B_r(0)$ with respect to the metric~\eqref{eqn-cmp-metric}.
We note that the definition~\eqref{eqn-cmp-metric} of $\BB d^{\op{CMP}}\left(\beta_1|_{[0,T_{1,r}]} , \beta_2|_{[0,T_{2,r}]} \right)$ makes sense even if one or both of $T_{1,r}$ or $T_{2,r}$ is infinite, provided we allow $\BB d^{\op{CMP}}\left(\beta_1|_{[0,T_{1,r}]} , \beta_2|_{[0,T_{2,r}]} \right) = \infty$ (this doesn't pose a problem due to the $1\wedge$ in~\eqref{eqn-cmp-metric-loc}).

\begin{proof}[Proof of Lemma~\ref{lem-ssl-markov}]
For $z\in\BB C$, let $\wh X^z$ denote a random path whose conditional law given $(h,\eta,P)$ is $P_z$.
Let $\mcl E$ be a sequence of $\ep$'s tending to zero along which $(h,\eta,P^\ep) \rta (h,\eta,P)$ in law as $\mcl E\ni \ep \rta 0$.
By~\cite[Theorem 3.4]{gms-tutte}, we know that for each fixed $z\in \BB C$, the conditional law given $(h,\eta)$ of the process $\wh X^{z,\ep}$ converges in distribution to Brownian motion w.r.t.\ the local topology on curves viewed modulo time parametrization, as defined in~\eqref{eqn-cmp-metric-loc}.
The conditional law given $(h,\eta)$ of $\wh X^{z,\ep}$, viewed modulo time parametrization, is a continuous functional of the conditional law $ P_z^\ep$ of $\wh X^{z,\ep}$ given $(h,\eta)$.
Since $P_z^\ep \rta P_z$ in law along $\mcl E$, we infer that $\wh X^z$ is a continuous time change of Brownian motion.
The last part of Proposition~\ref{prop-tight} implies that the time change function is strictly increasing.
Therefore, $P$ defines a time-changed Brownian motion in the sense of Definition~\ref{def-time-change}.

Now we need to check the Markov property for $\wh X^z$. In general, a limit of a Markov process need not be Markov, as is easily checked. Ultimately, we will obtain the Markov property as a result of
Lemma~\ref{lem-cond-law-conv} and the \emph{uniformity} (in $z$) of the assumed convergence of $P^\ep_z$ to $P_z$. We explain the details more precisely now.
By the Markov property of random walk, for each $\ep > 0$ the conditional law\footnote{We look at time $\ep \lceil t/\ep \rceil$ instead of time $t$ since by the definition of $\wh X^{z,\ep}$, $\wh X_{\ep \lceil t / \ep \rceil}^{z,\ep}$ is of the form $\eta(x)$ for some $x \in \mcl V \mcl G^\ep = \ep \BB Z$. Otherwise, since $\wh X$ is defined by linear interpolation, the Markov property is not easily expressed.}
of $s\mapsto \wh X_{s + \ep \lceil t / \ep \rceil}^{z,\ep}$ given $(h,\eta,P^\ep, \wh X^{z,\ep}|_{[0, {\ep \lceil t / \ep \rceil} ]})$ is $P^\ep_{ \wh X_{\ep \lceil t / \ep \rceil}^{z,\ep}  }$.
Along $\mcl E$, we have that the law of $(h,\eta,P^\ep , \wh X^{z,\ep})$ converges to that of  $(h,\eta , P , \wh X^z)$.

Hence, since the topology for $P$ is the local \emph{uniform} topology, we also have that
this implies the convergence in law
\eqbn
 \left( h,\eta, P^\ep, \wh X^{z,\ep}|_{[0,{\ep \lceil t / \ep \rceil} ]} , \wh X^{z,\ep}_{\cdot + \ep \lceil t / \ep \rceil}  ,  P^\ep_{ \wh X_{\ep \lceil t / \ep \rceil}^{z,\ep} }  \right)
 \rta \left( h,\eta, P, \wh X^z|_{[0,t]} ,  \wh X^z_{\cdot + t}   ,   P_{\wh X_t^z} \right) .
\eqen
We now apply Lemma~\ref{lem-cond-law-conv} with $X_n = \wh X^{z,\ep}_{\cdot + \ep \lceil t / \ep \rceil}$, $Y_n = (h,\eta, P^\ep ,\wh X^{z,\ep}|_{[0, \ep \lceil t / \ep \rceil ]})$, $X =  \wh X^z_{\cdot + t}$, $Y = (h,\eta, P ,\wh X^z|_{[0,t]})$, and $P_Y = P_{\wh X_t^z}$.
This shows that $P_{\wh X_t^z}$ is the conditional law of $\wh X^z_{\cdot + t}$ given $(h,\eta,P, \wh X^z|_{[0,t]})$, as required.
\end{proof}

\begin{lem} \label{lem-ssl-reversible}
Almost surely, the $\gamma$-LQG measure $\mu_h$ is reversible for $P$ (Definition~\ref{def-invariant}).
\end{lem}

The idea of the proof of Lemma~\ref{lem-ssl-reversible} is as follows. The measure on $\mcl G^\ep$ which assigns mass $\op{deg}^\ep(x)$ to each vertex $x$ is reversible for the simple random walk on $\mcl G^\ep$, and under the embedding $x\mapsto \eta(x)$ this measure converges in law to $\mu_h$ (see Corollary~\ref{cor-deg-conv}).
We want to pass this through to the limit to get Lemma~\ref{lem-ssl-reversible}.
The next lemma is the main technical input in the proof of Lemma~\ref{lem-ssl-reversible}. It gives us the convergence statements necessary to pass the reversibility of random walk on $\mcl G^\ep$ through to the limit.

For the statement of the lemma, we will use the following notation.
For a set $A\subset\BB C$ and $\ep > 0$, let
\eqb \label{eqn-cell-union}
A^\ep := \bigcup_{x\in \mcl V\mcl G^\ep(A)} H_x^\ep
\eqe
be the union of the space-filling SLE cells which intersect $A$.

\begin{lem} \label{lem-sum-conv}
Let $\mcl E$ be a sequence of $\ep$-values tending to zero along which $(h,\eta,P^\ep) \rta (h,\eta,P)$ in law.
There is a coupling of $(h,\eta)$ with a sequence of field / curve pairs $(h^\ep,\eta^\ep) \eqD (h,\eta)$ for $\ep\in\mcl E$ such that the following is true. Define $\mcl G^\ep$ and $\ol{\BB P}_x^\ep$ with $(h^\ep,\eta^\ep)$ in place of $(h,\eta)$.
Let $t>0$, let $U,V\subset\BB C$ be bounded open sets, and define $V^\ep$ as in~\eqref{eqn-cell-union}.
For each $\delta>0$, it holds with probability tending to 1 as $\ep\rta 0$ that
\eqb \label{eqn-sum-conv-upper}
\sum_{x\in   \mcl V\mcl G^\ep(U)  } \ol{\BB P}_x^\ep\left[ \eta^\ep( X_{ \lfloor t / \ep \rfloor}^\ep) \in V^\ep \right] \op{deg}^\ep(x)
\leq  6\int_{\ol U} P_z\left[ \wh X_t  \in \ol V   \right]  d\mu_h(z)   + \delta
\eqe
and
\eqb \label{eqn-sum-conv-lower}
 \sum_{x\in   \mcl V\mcl G^\ep(U)  } \ol{\BB P}_x^\ep\left[ \eta^\ep( X_{ \lfloor t / \ep \rfloor}^\ep) \in V^\ep \right] \op{deg}^\ep(x)
\geq 6\int_U P_z\left[ \wh X_t  \in V  \right]  d\mu_h(z) - \delta .
\eqe
\end{lem}
\begin{proof}
Let $\mu_h^\ep$ be the measure on $\BB C$ whose restriction to each cell $\eta([x-\ep,x])$ for $x\in\ep\BB Z$ is equal to $\op{deg}^\ep(x)$ times $\mu_h|_{\eta([x-\ep,x])}$.
Then by the definition of $\wh X^{z,\ep}$,
\eqb \label{eqn-walk-sum-to-int}
\sum_{x\in   \mcl V\mcl G^\ep(U)  } \ol{\BB P}_x^\ep\left[ \eta( X_{ \lfloor t / \ep \rfloor}^\ep) \in V^\ep \right] \op{deg}^\ep(x)
=    \int_{U^\ep} P_z^\ep\left[ \wh X^{z,\ep}_{\ep\lfloor t / \ep \rfloor} \in V^\ep  \right] d\mu_h^\ep(w)   , \quad\forall t \in \ep\BB N_0 .
\eqe
Here we note that $z\mapsto P_z^\ep$ is constant on the interior of each cell $H_x^\ep$ for $x\in\ep\BB Z$. Basically,~\eqref{eqn-sum-conv-upper} and~\eqref{eqn-sum-conv-lower} follow from the fact that the joint law of $(\mu_h^\ep , P^\ep)$ converges to the joint law of $(6\mu_h, P)$ w.r.t.\ the local weak topology on the first coordinate and the topology of Section~\ref{sec-tight} on the second coordinate.

Let us now make this more precise.
By Corollary~\ref{cor-deg-conv}, as $\ep\rta 0$ the measure $\mu_h^\ep$ converges in probability to $6\mu_h$ w.r.t.\ the local Prokhorov topology. (The factor 6 in front of $\mu_h$ is related to some fairly general topological properties of parabolic unimodular surfaces, see e.g. \cite{angel-hyperbolic}.)
Since $(h,\eta,P^\ep) \rta (h,\eta,P)$ in law and $\mu_h$ is a measurable function of $h$, it follows that $(h,\eta,\mu_h^\ep , P^\ep) \rta (h,\eta,6\mu_h,P)$ in law.
By the Skorokhod theorem, we can find a coupling of a sequence of field / curve pairs $(h^\ep,\eta^\ep) \eqD (h,\eta)$ for $\ep\in\mcl E$ with $(h,\eta)$ such that the following is true. If we construct $\mu_{h^\ep}^\ep$ and $P^\ep$ with $(h^\ep,\eta^\ep)$ in place of $(h,\eta)$, then a.s.\
\eqb
\mu_{h^\ep}^\ep \rta \mu_h \quad \text{and} \quad P^\ep \rta P ,\quad \text{as $\mcl E \ni \ep \rta 0$}.
\eqe
Henceforth fix such a coupling and let all objects with a superscript $\ep$ be defined with $(h^\ep ,\eta^\ep)$ in place of $(h,\eta)$.

By the definition of the topology for $P$ from Section~\ref{sec-tight}, the Prokhorov distance between $P_z^\ep$ and $P_z$ converges to zero uniformly on a neighborhood of $\ol U$. Furthermore, by Lemma~\ref{lem-cell-diam} the {maximal distance from any point in $V^\ep$ to $V$} tends to zero in probability as $\mcl E \ni \ep\rta 0$. Consequently, for each $\zeta  > 0$ it holds with probability tending to 1 as $\mcl E \ni \ep \rta 0$ that
\eqb \label{eqn-prokhorov-uniform}
P_z^\ep\left[ \wh X^{z,\ep}_{\ep\lfloor t / \ep \rfloor} \in V^\ep  \right]
\leq P_z \left[ \wh X_{\ep\lfloor t / \ep \rfloor} \in B_\zeta(V) \right] + \zeta , \quad\forall z \in U^\ep .
\eqe
By the continuity of $z\mapsto P_z$ and $t\mapsto \wh X_t$, the Prokhorov distance between the $P_z$-laws of $\wh X_{\ep \lfloor t/\ep \rfloor}$ and $\wh X_t$ converges in probability to zero as $\mcl E \ni \ep\rta 0$, uniformly over all $z$ in a neigborhood of $\ol U$. Hence,~\eqref{eqn-prokhorov-uniform} holds with $\wh X_t$ in place of $\wh X_{\ep\lfloor t/\ep \rfloor}$.

By integrating this last variant of~\eqref{eqn-prokhorov-uniform} over $U^\ep$ w.r.t.\ $\mu_{h^\ep}^\ep$ and using~\eqref{eqn-walk-sum-to-int}, we get that with probability tending to 1 as $\ep\rta 0$,
\eqb \label{eqn-prokhorov-int}
\sum_{x\in   \mcl V\mcl G^\ep(U)  } \ol{\BB P}_x^\ep\left[ \eta( X_{ \lfloor t / \ep \rfloor}^\ep) \in V^\ep \right] \op{deg}^\ep(x)
 \leq  \int_{U^\ep} P_z \left[ \wh X_t \in B_\zeta(V)  \right] d\mu_{h^\ep}^\ep(z)  + \zeta \mu_{h^\ep}^\ep(U^\ep)  .
\eqe
We know that $\mu_{h^\ep}^\ep \rta \mu_h$ a.s.\ weakly as $\ep\rta 0$ along $\mcl E$ and $z\mapsto P_z \left[ \wh X_t \in \ol{B_\zeta(V)} \right]$ is upper semicontinuous (due to the Prokhorov continuity of $z\mapsto P_z$).
It follows (using a consequence of the portmanteau theorem, see \cite[Exercice 2.6, Chapter I]{Billingsley}) that with probability tending to 1 as $\mcl E \ni \ep\rta 0$, the right side of~\eqref{eqn-prokhorov-int} is at most
\eqb \label{eqn-prokhorov-int'}
\int_{\ol U} P_z \left[ \wh X_t \in \ol{B_\zeta(V)}  \right] d\mu_h(z)  + \zeta \mu_h(\ol V) + \zeta.
\eqe
Recalling~\eqref{eqn-prokhorov-int} and choosing $\zeta$ to be sufficiently small, depending on $\delta$, gives~\eqref{eqn-sum-conv-upper}.

To obtain~\eqref{eqn-sum-conv-lower}, we first use the Prokhorov convergence $P_z^\ep \rta P_z$ as in~\eqref{eqn-prokhorov-uniform} to get that for each $\zeta  > 0$ it holds with probability tending to 1 as $\mcl E \ni \ep \rta 0$ that
\eqb \label{eqn-prokhorov-uniform''}
P_z^\ep\left[ \wh X^{z,\ep}_{\ep\lfloor t / \ep \rfloor} \in V^\ep  \right]
\geq P_z \left[ \wh X_{\ep\lfloor t / \ep \rfloor} \in V\setminus B_\zeta(\bdy V) \right] - \zeta , \quad\forall z \in U^\ep .
\eqe
We then argue exactly as in the proof of~\eqref{eqn-sum-conv-upper} but with the inequality signs reversed.
\end{proof}

We can now prove the following weaker version of Lemma~\ref{lem-ssl-reversible}.

\begin{lem} \label{lem-ssl-reversible0}
Let $U,V\subset\BB C$ be deterministic bounded open sets and let $t > 0$. Almost surely,
\eqb \label{eqn-ssl-reversible0}
   \int_{ U} P_z\left[  \wh X_t  \in  V   \right]  d\mu_h(z) \leq     \int_{\ol V} P_z\left[ \wh X_t  \in \ol U   \right]  d\mu_h(z)    .
\eqe
\end{lem}
\begin{proof}
We know that the uniform measure on vertices $x$ of $\mcl G^\ep$ weighted by $\op{deg}^\ep(x)$ is reversible for the random walk on $\mcl G^\ep$. That is, for any $x\in\mcl G^\ep$ and any $j\in\BB N$,
\eqb \label{eqn-walk-reversible0}
\ol{\BB P}_x^\ep\left[ X_j^\ep =  y \right] \op{deg}^\ep(x)   = \ol{\BB P}_y^\ep\left[ X_j^\ep =  x \right] \op{deg}^\ep(y) .
\eqe
Summing~\eqref{eqn-walk-reversible0} over all $x\in \mcl G^\ep(U)$ and $y\in \mcl G^\ep(V)$ shows that for each $j\in\BB N$,
\eqb \label{eqn-walk-reversible}
\sum_{x\in   \mcl V\mcl G^\ep(U)  } \ol{\BB P}_x^\ep\left[ \eta( X_j^\ep) \in V^\ep \right] \op{deg}^\ep(x)
= \sum_{y \in \mcl V\mcl G^\ep(V)} \ol{\BB P}_y^\ep\left[ \eta(X_j^\ep) \in U^\ep \right] \op{deg}^\ep(y) .
\eqe
We now take $j = \ep \lfloor t /\ep \rfloor$ and apply Lemma~\ref{lem-sum-conv} to get~\eqref{eqn-ssl-reversible0}.
\end{proof}

The inequality of Lemma~\ref{lem-ssl-reversible0} could be converted into an equality if we knew that $\mu_h(\bdy V ) = 0$ and $\int_{\ol V} P_z\left[ \wh X_t  \in \bdy U \right]  d\mu_h(z) = 0$. This is achieved by means of the following lemma.

\begin{lem} \label{lem-zero-mass}
Let $A ,  U \subset \BB C$ be deterministic bounded sets such that $A$ is closed with zero Lebesgue measure and $U$ is open.
For each $t>0$, a.s.\
\eqb \label{eqn-zero-mass}
\mu_h(A) = 0 \quad \text{and} \quad \int_U P_z\left[\wh X_t \in A \right] \,d\mu_h(z) = 0 .
\eqe
\end{lem}
\begin{proof}
{It is immediate from the basic properties of the LQG area measure that $\mu_h(A) = 0$ a.s. Indeed, the analogous statement with $h$ replaced by a zero-boundary GFF an any open domain in $\BB C$ follows from, e.g.,~\cite[Proposition 1.2]{shef-kpz} or \cite[Theorem 1.1]{berestycki-gmt-elementary}. The statement that $\mu_h(A) = 0$ a.s.\ follows from this and straightforward absolute continuity considerations. }

To get the second assertion in~\eqref{eqn-zero-mass}, fix $\delta > 0$ (which we will eventually send to zero). By Lemma~\ref{lem-ssl-reversible0} applied with $V = B_\delta(A)$, a.s.\
\allb \label{eqn-use-reversible0}
\int_U P_z\left[\wh X_t \in A \right] \,d\mu_h(z)
&\leq \int_{ U} P_z\left[  \wh X_t  \in  B_\delta(A)   \right]  d\mu_h(z) \notag\\
&\leq     \int_{\ol{B_\delta(A)}} P_z\left[ \wh X_t  \in \ol U   \right]  d\mu_h(z) \notag\\
&\leq    \mu_h(\ol{B_\delta(A)})  .
\alle
Since $\mu_h(A) = 0$ and $\ol{B_\delta(A)}$ decreases to $A$ as $\delta \rta 0$ (since $A$ is closed) we can send $\delta \rta 0$ in~\eqref{eqn-use-reversible0} to conclude the proof.
\end{proof}

\begin{proof}[Proof of Lemma~\ref{lem-ssl-reversible}]
Fix $t > 0$. We need to show that a.s.\ for any two Borel sets $A,B\subset\BB C$,
\eqb \label{eqn-ssl-reversible}
\int_{ A} P_z\left[  \wh X_t  \in  B  \right]  d\mu_h(z) =    \int_{B} P_z\left[ \wh X_t  \in  A  \right]  d\mu_h(z)  .
\eqe
By Lemmas~\ref{lem-ssl-reversible0} and~\ref{lem-zero-mass}, \eqref{eqn-ssl-reversible} holds a.s.\ if $A,B$ are fixed open sets whose boundaries have zero Lebesgue measure.
In fact, these lemmas (applied with $U$ and $V$ equal to the interiors of $A$ and $B$) imply that~\eqref{eqn-ssl-reversible} holds a.s.\ if $A,B$ are fixed Borel sets such that $\bdy A$ and $\bdy B$ have zero Lebesgue measure, $\ol A$ is equal to closure of the interior of $A$, and $\ol B$ is equal to closure of the interior of $B$.

We now fix $R >0$ and show that a.s.~\eqref{eqn-ssl-reversible} holds for all Borel sets $A,B\subset B_R(0)$.
We do this using the monotone class theorem.
Let $\mcl A$ be the collection of subsets of $B_R(0)$ which can be expressed in terms of finite unions and intersections of sets of the form $B_r(z)$ or $B_R(0) \setminus B_r(z)$ for $r\in \BB Q\cap (0,\infty)$ and $z\in\BB Q^2$.
Then $\mcl A$ has the following properties. It is countable; it is closed under finite unions and intersections; it generates the Borel $\sigma$-algebra of $B_R(0)$; we have $B_R(0) \setminus A \in \mcl A$ for each $A \in \mcl A$; the boundary of each $A\in\mcl A$ has zero Lebesgue measure; and for each $A\in\mcl A$ the closure $\ol A$ is also equal to the closure of the interior of $A$.
From this, we infer that a.s.~\eqref{eqn-ssl-reversible} holds simultaneously for each $A,B\in \mcl A$.

Due to the monotone convergence theorem, for each fixed $A$ the set of Borel sets $B$ for which~\eqref{eqn-ssl-reversible} holds is closed under increasing unions and intersections.
By two applications of the monotone class theorem (one for each of $A$ and $B$), we get that a.s.~\eqref{eqn-ssl-reversible} holds for all Borel sets $A,B\subset B_R(0)$. Sending $R\rta\infty$ now concludes the proof.
\end{proof}

\subsection{Time-changed Brownian motions are determined by invariant measures}
\label{sec-bm-unique}

We now have two random functions $\BB C\rta \op{Prob}\left(C([0,\infty),\BB C) \right)$ which each define a Markovian time-changed Brownian motion, namely the subsequential limit $P$ above and the function which takes $z\in\BB C$ to the law of Liouville Brownian motion w.r.t.\ $h$ started from $z$. These functions have a common reversible measure, namely the $\gamma$-LQG measure.
We want to show that these two functions are a.s.\ the same up to a global linear time change.

\begin{prop} \label{prop-bm-unique}
Let $P^1 , P^2$ be two continuous mappings $\BB C \rta \op{Prob}\left(C([0,\infty),\BB C) \right)$ which each define a Markovian time changed Brownian motion.
Assume that there is a non-zero measure $\mu$ which is invariant for both $P^1$ and $P^2$.
There is a determinsitic constant $c>0$ such that the following is true.
For each $z\in\BB C$, a process with the law $P_z^1$ can be obtained by composing a process with the law $P_z^2$ by the linear time change $t\mapsto c t$.
\end{prop}

We note that Proposition~\ref{prop-bm-unique}, applied under the conditional law given $h$, immediately implies Theorem~\ref{T:LBMchar_intro} since the conditional law of LBM given $h$ is a Markovian time-changed Brownian motion which leaves the LQG measure $\mu_h$ invariant~\cite{grv-lbm}.

Proposition~\ref{prop-bm-unique} is a consequence of results in the general theory of positive continuous additive functionals of Brownian motion.
To explain how the proposition is proven, we will need to review some of this theory.

In what follows, for $z\in\BB C$ we let $\BB P_z$ be the law of standard planar Brownian motion $\mcl B $ started from $z$.

\begin{defn} \label{def-pcaf}
A \emph{positive continuous additive functional (PCAF)} of $\mcl B$ is a process $t\mapsto \mcl A_t$, defined for all possible choices of starting points for $\mcl B$, which is adapted to the filtration generated by $\mcl B$ (completed by null events), such that the following is true.
Almost surely, $\mcl A_0 = 0$ and $\mcl A_t$ is continuous and non-decreasing in $t$.
Furthermore, for each $s,t\geq 0$, a.s.\
\eqb
\mcl A_{s+t} = \mcl A_t + \mcl A_s\circ\theta_t
\eqe
where $\theta_t$ is the shift operator on the state space for $\mcl B$ defined by $\theta_t(\mcl B) = \mcl B_{t+\cdot}$.
\end{defn}

Definition~\ref{def-pcaf} is related to Proposition~\ref{prop-bm-unique} by the following lemma.

\begin{lem} \label{lem-time-change-pcaf}
Suppose that $P : \BB C\rta \op{Prob}(C([0,\infty),\BB C))$ is continuous and defines a Markovian time changed Brownian motion as in Definition~\ref{def-markovian}.
Write $\wh{\mcl B}$ for a process with the law $P_z$ for some starting point $z \in \BB C$.
Let $u\mapsto\tau_u$ be the time change used to get from standard Brownian motion to $\wh{\mcl B}$, as in Definition~\ref{def-time-change}, and for $t\geq 0$ let $\mcl A_t = \tau^{-1}(t) := \inf\left\{u \geq 0 : \tau_u = t \right\}$.
Then $\mcl A$ is a PCAF for the Brownian motion $\mcl B_t = \wh{\mcl B}_{\mcl A_t}$.
\end{lem}

For the proof of Lemma~\ref{lem-time-change-pcaf}, we need the following elementary fact.

\begin{lem} \label{lem-strong-markov}
Suppose that $P : \BB C\rta \op{Prob}(C([0,\infty),\BB C))$ is continuous and defines a Markovian time changed Brownian motion.
Then $P$ is strongly Markovian, i.e., if $\wh{\mcl B}^z$ is sampled from $P_z$ then for each stopping time $\sigma$ for $\wh{\mcl B}^z$, the conditional law of $\wh{\mcl B}_{\sigma +\cdot}^z$ given $\wh{\mcl B}^z|_{[0,\sigma]}$ is $P_{\wh{\mcl B}^z_\sigma}$.
\end{lem}
\begin{proof}
If $\sigma$ takes on only countably many possible values, the Markov property at time $\sigma$ is immediate from the case of a deterministic time.
Due to the continuity of $z\mapsto P_z$ and the a.s.\ continuity of a process with the law $P$, we can approximate a general stopping time by a sequence of stopping times which take on only countably many possible values to get the result in general.
\end{proof}

\begin{proof}[Proof of Lemma~\ref{lem-time-change-pcaf}]
Since $u\mapsto \tau_u$ is continuous, strictly increasing, and satisfies $\tau_0 = 0$, the same is also true for $\mcl A$, and furthermore $\tau_{\mcl A_t} = \mcl A_{\tau_t} = t$, whence $\mcl B_t = \wh{\mcl B}_{\mcl A_t}$ for any $t \ge 0$.
The time $\mcl A_t $ is the smallest $u \geq 0$ for which the quadratic variation of $\wh{\mcl B}|_{[0,u]}$ is $t$.
This implies in particular that $\mcl A_t$ is a stopping time for $\wh{\mcl B}$.

Before arguing that $\mcl A$ is adapted, we will first check that for $0\leq a < b \leq\infty$,
\eqb \label{eqn-time-change-sigma}
\sigma\left(  \{ \wh{\mcl B}_{\mcl A_a + u} \}_{u\in [0 ,  \mcl A_b - \mcl A_a]} \right)
= \sigma\left( \mcl B|_{[a,b]} , \{\mcl A_t -\mcl A_a\}_{t\in [a,b]} \right)  .
\eqe
Indeed, for any $t\in [a,b]$ we have that $\mcl A_t - \mcl A_a$ is the first time that the quadratic variation of the process $\wh{\mcl B}_{\mcl A_a + \cdot} $ reaches $t -a$, and $\mcl B_t = \wh{\mcl B}_{\mcl A_t}$.
This shows that the left $\sigma$-algebra in~\eqref{eqn-time-change-sigma} contains the right $\sigma$-algebra.
For the reverse inclusion, we now observe that for each $u \in [0,\mcl A_b - \mcl A_a]$ we have $\wh{\mcl B}_{\mcl A_a + u} = \mcl B_{\tau_{\mcl A_a + u}}$ and $\tau_{\mcl A_a + u} = \inf\{t \geq a : \mcl A_t - \mcl A_a = u\}$.

By~\eqref{eqn-time-change-sigma} (applied with $(a,b) = (0,t)$ and with $(a,b) =(t,\infty)$) and Lemma~\ref{lem-strong-markov} (applied with $\sigma = \mcl A_t$), we get that the pairs $(\mcl B_{t+\cdot} , \mcl A_{t+\cdot} - \mcl A_t)$ and $(\mcl B|_{[0,t]} ,  \mcl A|_{[0,t]})$ are conditionally independent given $\wh{\mcl B}_{\mcl A_t}$, equivalently, given $\mcl B_t$.

We will now argue that $\mcl A$ is a.s.\ determined by $\mcl B$.
Indeed, it follows from the above that if we condition on $\mcl B$ then the process $\mcl A$ is continuous with conditionally independent increments.
It follows that the conditional law of $\mcl A$ given $\mcl B$ is that of a Gaussian process: indeed, it is easy to see from the independence of the increments that the centered process is a continuous martingale with deterministic quadratic variation, conditionally on $\mcl B$.
But, $\mcl A$ is non-negative, so the conditional law of $\mcl A$ given $\mcl B$ must be a point mass at some $\mcl B$-measurable continuous increasing function.
That is, $\mcl A$ is a.s.\ determined by $\mcl B$  (which implies that also $\tau = \mcl A^{-1}$ and $\wh{\mcl B}_\cdot = \mcl B_{\tau(\cdot)}$ are a.s.\ determined by $\mcl B$) .

Since $(\mcl B_{t+\cdot} , \mcl A_{t+\cdot} - \mcl A_t)$ and $(\mcl B|_{[0,t]} ,  \mcl A|_{[0,t]})$ are conditionally independent given $\mcl B_t$, it follows that $\mcl A|_{[0,t]}$ is a.s.\ determined by $\mcl B|_{[0,t]}$ and that $\mcl A_{t+\cdot} - \mcl A_t$ is a.s.\ determined by $\mcl B|_{[t,\infty)}$. In particular, $\mcl A$ is adapted to the natural filtration of $\mcl B$ completed by null events and $\mcl A_{t+\cdot} -\mcl A_t$ is a.s.\ given by a functional of the shifted process $\mcl B_{t+\cdot}$.
Hence for each $t,s \geq 0$ there is a measurable function $F_{t,s}$ from continuous paths in $\BB C$ to $[0,\infty)$ such that a.s.\ $F_{t,s}(\mcl B_{t+\cdot}) = \mcl A_{t+s} - \mcl A_t$, i.e.,
\eqb \label{eqn-additive}
\mcl A_{t+s} = \mcl A_t + F_{t,s} \circ \theta_t .
\eqe

We need to check that the functional $F_{t,s}$ does not depend on $t$.
The process $\mcl B_{t+\cdot}$ is obtained by parameterizing $\wh{\mcl B}_{\mcl A_t+\cdot}$ by its quadratic variation, so the functional which takes in $\wh{\mcl B}_{\mcl A_t+\cdot}$ and outputs $\mcl B_{t+\cdot}$ does not depend on $t$.
Since $\mcl B_{t+\cdot}$ a.s.\ determines $\mcl A_{t+\cdot} - \mcl A_t$, it follows that $\mcl B_{t+\cdot}$ a.s.\ determines $\wh{\mcl B}_{\mcl A_t+\cdot}$, i.e., there is a measurable functional which takes in $\mcl B_{t+\cdot}$ and a.s.\ outputs $\wh{\mcl B}_{\mcl A_t+\cdot}$.
Since the functional which takes in $\wh{\mcl B}_{\mcl A_t+\cdot}$ and outputs $\mcl B_{t+\cdot}$ does not depend on $t$, it follows that the functional which goes in the opposite direction can be taken to not depend on $t$ either.
Since $\wh{\mcl B}_{\mcl A_t+\cdot}$ determines $\mcl A_{t+\cdot} - \mcl A_t$ in a manner which does not depend on $t$, we get that $\mcl B_{t+\cdot}$ a.s.\ determines $\mcl A_{t+\cdot} - \mcl A_t$ in a manner which does not depend on $t$, i.e., $F_{t,s}$ does not depend on $t$. In particular, $F_{t,s} = F_{0,s}$ is the functional which a.s.\ takes $\mcl B$ to $\mcl A_s$. By~\eqref{eqn-additive}, $\mcl A$ is a PCAF.
\end{proof}


PCAFs are in bijective correspondence with so-called Revuz measures, defined as follows.

\begin{defn} \label{def-revuz}
A \emph{Revuz measure} is a measure $\mu$ on $\BB C$ such that $\mu(A) = 0$ for each set $A$ which is polar for the Brownian motion $\mcl B$ (i.e., it is a.s.\ never hit by $\mcl B$).
\end{defn}

From~\cite[Theorem 5.1.3]{fot-dirichlet-forms} (here the Hunt process in the theorem statement is planar Brownian motion), we know that there is a bijective correspondence between PCAFs and Revuz measures.
For each PCAF $\mcl A$ there exists a unique Revuz measure $\mu$ such that for any two bounded continuous functions $f,g : \BB C\rta [0,\infty)$,
\eqb \label{eqn-revuz}
\int_{\BB C} \BB E_z\left[ \int_0^t f(B_s)\,d\mcl A_s  \right]  g(z) \,dz
= \int_0^t \int_{\BB C} \int_{\BB C}  p_s(z,w) f(w) g(z) \,d\mu(w)\,dz \,ds
\eqe
and for every Revuz measure $\mu$ there exists a unique PCAF $\mcl A$ for which~\eqref{eqn-revuz} holds.
Here, $d\mcl A_s$ is interpreted as a Lebesgue-Stieljes measure.
We call $\mu$ the \emph{Revuz measure of $\mcl A$}.

For a PCAF $\mcl A$, we can define the time-changed process
\eqb \label{eqn-time-change}
\wh{\mcl B}_u := \mcl B_{\tau_u} ,\quad \text{where} \quad \tau_u := \inf\left\{ t \geq 0 : \mcl A_t = u \right\} .
\eqe
Then the law of $\wh{\mcl B}$ for varying choices of starting points defines a Markovian time changed Brownian motion in the sense of Definition~\ref{def-markovian}.
By~\cite[Theorem 6.2.1(i)]{fot-dirichlet-forms}, the Revuz measure of $\mcl A$ is reversible (hence invariant) for $\wh{\mcl B}$.
In fact, the Revuz measure of $\mcl A$ is the \emph{only} invariant measure for $\wh{\mcl B}$, up to multiplication by a deterministic constant, as shown by the following lemma.

\begin{lem} \label{lem-measure-unique}
Let $P$ define a Markovian time changed Brownian motion and assume that $z\mapsto P_z$ is continuous.
Let $\mcl A$ be the corresponding PCAF as in Lemma~\ref{lem-time-change-pcaf} and let $\mu$ be the Revuz measure of $\mcl A$.
Then any other invariant measure for $P$ is a deterministic constant multiple of $\mu$.
\end{lem}

To prove Lemma~\ref{lem-measure-unique}, we will use a general criterion for the uniqueness of invariant measures for Markov processes; see, e.g.,~\cite[Chapter X, Section 3]{revuz-yor}. One of the hypotheses of this criterion involves the \emph{resolvent} associated with $P$, which is the operator $R^\lambda : L^\infty(\BB C) \rta L^\infty(\BB C)$ defined by
\eqb \label{eqn-resolvent}
\mcl R^\lambda f(z) = \int_0^\infty e^{-\lambda t} E_z\left[ f(\wh{\mcl B}_t) \right] \,dt  ,
\eqe
where $E_z$ denotes expectation w.r.t.\ $P_z$.

\begin{lem} \label{lem-resolvent}
Let $P$ define a Markovian time changed Brownian motion and assume that $z\mapsto P_z$ is continuous. If $f \in L^\infty(\BB C)$, then $\mcl R^\lambda f$ is continuous.
That is, $P$ is \emph{strongly Feller}.
\end{lem}
\begin{proof}
The proof is based on the fact that hitting distributions of Brownian motion (i.e., harmonic measure) depend continuously on the starting point in the total variation sense. This allows us to show that if $|z-w|$ is small then we can couple processes sampled from $P_z$ and $P_w$ so that they agree after leaving a small ball centered at $z$. When the processes are coupled together in this way, it is not hard to check that for a fixed choice of $f \in L^\infty(\BB C)$, it holds that $|\mcl R^\lambda f(z) - \mcl R^\lambda f(w)|$ is small. A similar argument is used to prove the strong Feller property for Liouville Brownian motion in~\cite[Theorem 2.4]{grv-heat-kernel}.
\medskip

\noindent\textit{Step 1: defining the coupling.}
Fix $z \in \BB C$ and $\ep > 0$.
A path sampled from the law $P_z$ a.s.\ does not stay constant for any non-trivial interval of time, so we can find $\delta_0 \in (0,1)$ such that
\eqb
P_z\left[ \text{$\wh{\mcl B}$ exits $B_{\delta_0}(z)$ before time $\ep$} \right] \geq 1 - \ep .
\eqe
Since $P_w$ depends continuously on $w$ in the Prokhorov distance, there exists $\delta_1 \in (0,\delta_0]$ such that
\eqb \label{eqn-resolvent-exit}
P_w\left[ \text{$\wh{\mcl B}$ exits $B_{\delta_0}(z)$ before time $2\ep$} \right] \geq 1 - 2\ep ,\quad\forall w \in B_{\delta_1}(z) .
\eqe

The harmonic measure on $\bdy B_{\delta_0}(z)$ as viewed from $w \in B_{\delta_0}(z)$ (which is the same as that of Brownian motion and is therefore known explicitly) depends continuously on $w$ in the total variation sense.
Therefore, we can find $\delta_2 \in (0,\delta_0]$ such that for any $w \in B_{\delta_2}(z)$, we can find a coupling of $\wh{\mcl B}^z \sim P_z$ and $\wh{\mcl B}^w\sim P_w$ such that with probability at least $1-\ep$, $\wh{\mcl B}^z$ and $\wh{\mcl B}^w$ exit $B_{\delta_0}(z)$ at the same point.
Let $T_z$ and $T_w$ denote the exit times of $\wh{\mcl B}^z $ and $\wh{\mcl B}^w $ from $B_{\delta_0}(z)$, respectively.
Then our coupling is such that $\BB P[\wh{\mcl B}^z_{T_z} = \wh{\mcl B}^w_{T_w}] \geq 1-\ep$. By the strong Markov property (Lemma~\ref{lem-strong-markov}) we can re-couple so that
\eqb \label{eqn-resolvent-coupling}
\BB P\left[ \wh{\mcl B}^z_{T_z+s} = \wh{\mcl B}^w_{T_w+s} ,\:\forall s\geq 0 \right] \geq 1-\ep .
\eqe
Henceforth fix such a coupling and let
\eqb
E := \left\{ \wh{\mcl B}^z_{T_z+s} = \wh{\mcl B}^w_{T_w+s} ,\:\forall s\geq 0 \right\} \cap \left\{ T_z \vee T_w \leq 2\ep \right\},
\eqe
so that by~\eqref{eqn-resolvent-exit} and~\eqref{eqn-resolvent-coupling} we have $\BB P[E] \geq 1- 5\ep$.
\medskip

\noindent\textit{Step 2: bounding the resolvent.}
Now fix $f\in L^\infty(\BB C)$.
By considering the positive and negative parts of $f$ separately, we can assume without loss of generality that $f\geq 0$.
We have
\alb
\mcl R^\lambda f(z)
&= \BB E\left[   \int_0^\infty e^{-\lambda t} f(\wh{\mcl B}^z_t) \,dt  \right] \notag\\
&=\BB E\left[  \BB 1_E \int_0^\infty e^{-\lambda t} f(\wh{\mcl B}^z_t) \,dt \right] + \BB E\left[ \BB 1_{E^c} \int_0^\infty e^{-\lambda t} f(\wh{\mcl B}^z_t) \,dt   \right] \notag\\
&= \BB E\left[\BB 1_E \int_0^{T_z} e^{-\lambda t} f(\wh{\mcl B}^z_t) \, dt \right] + \BB E\left[\BB 1_E \int_{0}^\infty  e^{-\lambda (s + T_z) } f(\wh{\mcl B}^z_{s + T_z} ) \, ds \right]  +  \BB E\left[ \BB 1_{E^c} \int_0^\infty e^{-\lambda t} f(\wh{\mcl B}^z_t) \,dt   \right]
\ale
and similarly with $w$ in place of $z$.
By subtracting the formulas for $z$ and $w$ and using the triangle inequality, we get
\allb \label{eqn-resolvent-diff}
| \mcl R^\lambda f(z)  - \mcl R^\lambda f(w) |
&\leq \BB E\left[ \BB 1_E \left( \int_0^{T_z} e^{-\lambda t} f(\wh{\mcl B}^z_t) \, dt  +  \int_0^{T_w} e^{-\lambda t} f(\wh{\mcl B}^w_t) \, dt \right) \right] \notag\\
&\qquad\qquad + \BB E\left[ \BB 1_E \int_0^\infty \left|    e^{-\lambda (s + T_z) } f(\wh{\mcl B}^z_{s + T_z}) -   e^{-\lambda (s + T_w) }  f(\wh{\mcl B}^w_{s + T_w} )    \right| \,ds \right] \notag\\
&\qquad\qquad+ \BB E\left[ \BB 1_{E^c} \int_0^\infty e^{-\lambda t} \left( f(\wh{\mcl B}^z_t) + f(\wh{\mcl B}^w_t)\right) \,dt  \right]
\alle

On $E$, we have $T_z \vee T_w \leq 2\ep$, which implies that the first expectation on the right side of~\eqref{eqn-resolvent-diff} is at most $4\ep \|f\|_\infty$.
Furthermore, on $E$ we have $f(\wh{\mcl B}^z_{s+T_z}) =  f(\wh{\mcl B}^w_{s + T_w} )$ for each $s\geq 0$, so the second expectation on the right side of~\eqref{eqn-resolvent-diff} is equal to
\eqb
 \BB E\left[ \BB 1_E \left|    e^{-\lambda  T_z  }  -   e^{-\lambda   T_w  }  \right|  \int_0^\infty e^{-\lambda s} f(\wh{\mcl B}^z_{s + T_z} )   \,ds \right] ,
\eqe
which in turn is at most $(1-e^{-2 \lambda \ep}) \lambda^{-1} \|f\|_\infty$ since $T_z\vee T_w\leq 2\ep$ on $E$.
Finally, since $\BB P[E^c] \leq 4\ep$ the last term on the right side of~\eqref{eqn-resolvent-diff} is at most $8\ep \lambda^{-1} \|f\|_\infty$.
Plugging the above estimates into~\eqref{eqn-resolvent-diff} gives
\eqb
| \mcl R^\lambda f(z)  - \mcl R^\lambda f(w) | \leq o_\ep(1) \|f\|_\infty
\eqe
where the $o_\ep(1)$ tends to zero as $\ep\rta 0$ at a rate depending only on $\lambda$. Since $\ep > 0$ was arbitrary we infer that $\mcl R^\lambda f$ is continuous at $z$.
\end{proof}

\begin{proof}[Proof of Lemma~\ref{lem-measure-unique}]
As discussed just after~\eqref{eqn-time-change}, the measure $\mu$ is reversible, hence invariant, for $P$.
If $\wh{\mcl B}$ is sampled from $P_z$ for some $z\in\BB C$, then $\wh{\mcl B}$ is a time change of Brownian motion, so every open subset of $\BB C$ is recurrent for $\wh{\mcl B}$. Moreover, Lemma~\ref{lem-resolvent} says that $P$ is strongly Feller.
By~\cite[Chapter X, Proposition 3.9]{revuz-yor}, the process $\wh{\mcl B}$ is Harris recurrent w.r.t.\ the measure $\mu$, i.e., each open set $A\subset\BB C$ with $\mu(A) >0$ is recurrent for $\wh{\mcl B}$.
From this and the ergodic theorem for additive functionals of Brownian motion~\cite[Chapter X, Theorem 3.12]{revuz-yor} (see also~\cite[Chapter X, Exercise 3.14]{revuz-yor}) one obtains the uniqueness of the invariant measure for $P$ up to multiplication by a deterministic constant factor.
\end{proof}

\begin{proof}[Proof of Proposition~\ref{prop-bm-unique}]
Let $\mcl A^1$ and $\mcl A^2$ be the PCAFs associated with $P^1$ and $P^2$, as in Lemma~\ref{lem-time-change-pcaf}.
Let $\mu^1$ and $\mu^2$ be the Revuz measures associated with these PCAFs.
By the discussion just after~\eqref{eqn-time-change}, we know that $\mu^1$ (resp.\ $\mu^2$) is invariant for $P^1$ (resp.\ $P^2$).
By Lemma~\ref{lem-measure-unique}, we know that each of $P^1$ and $P^2$ has a unique invariant measure, up to multiplication by a deterministic constant factor.
Therefore, there are deterministic constants $c_1,c_2 >0$ such that a.s.\ $\mu^1 = c_1 \mu$ and $\mu^2 = c_2 \mu$.
Since the Revuz measure uniquely determines the PCAF, we get that $\mcl A^1 = (c_1/c_2) \mcl A^2$.
By the definitions of $\mcl A^1$ and $\mcl A^2$, this implies Proposition~\ref{prop-bm-unique}.
\end{proof}

\subsection{Proof of Theorem~\ref{thm-lbm-conv}}
\label{sec-lbm-conv}

As above, let $(h,\eta,P)$ be a subsequential limit of the joint laws of $(h,\eta,P^\ep)$ as $\ep\rta 0$.
Let us first record what we get from Proposition~\ref{prop-bm-unique}.

\begin{lem} \label{lem-random-constant}
There is a random variable $c > 0$, which is measurable with respect to the $\sigma$-algebra generated by $(h,\eta,P)$ (completed by null sets), such that a.s.\ for each $z\in \BB C$ the law $P_z$ is the same as the conditional law given $h$ of the Liouville Brownian motion started from $z$ pre-composed with the time change $t\mapsto c t$.
\end{lem}
\begin{proof}
By Lemma~\ref{lem-ssl-markov}, $P$ defines a Markovian time-changed Brownian motion and by Lemma~\ref{lem-ssl-reversible} the $\gamma$-LQG measure $\mu_h$ is reversible, hence invariant, for $P$.
It is shown in~\cite{grv-lbm} that the same properties are a.s.\ true for the function $P^{\op{LBM}} : \BB C\rta \op{Prob}(C([0,\infty), \BB C))$ which maps $z$ to the conditional law of LBM started from $z$ given $h$.
The lemma therefore follows from Proposition~\ref{prop-bm-unique} applied under the conditional law given $(h,\eta,P)$.
\end{proof}

We claim the following.

\begin{lem} \label{lem-deterministic-constant}
The random variable $c >0$ from Lemma~\ref{lem-random-constant} is a.s.\ equal to a deterministic constant.
\end{lem}

We will prove Lemma~\ref{lem-deterministic-constant} just below. We first assume the lemma and conclude the proof of Theorem~\ref{thm-lbm-conv}.

\begin{proof}[Proof of Theorem~\ref{thm-lbm-conv}, assuming Lemma~\ref{lem-deterministic-constant}]
Let $\mcl E$ be a sequence of $\ep$-values tending to zero along which $(h,\eta,P^\ep) \rta (h,\eta,P)$ in law.
Lemmas~\ref{lem-random-constant} and~\ref{lem-deterministic-constant} together imply that there is a deterministic constant $c > 0$ such that a.s.\ $P_z$ for each $z\in\BB C$ is the law of Liouville Brownian motion started from $z$ pre-composed with the time change $t\mapsto c t$.
The constant $c$ may depend on the choice of subsequence $\mcl E$ (equivalently, $c$ may depend on the law of $(h,\eta,P)$).
We do not know how to rule this out so we get around the issue as follows.

Recall that, in the notation used in Section~\ref{sec-tight} and in the theorem statement, $P^\ep$ is the conditional law of $(\wh X^{z,\ep}_t)_{t\geq 0} =( \eta(X^{z,\ep}_{\cdot /\ep}) )_{t\geq 0}$ given $(h,\eta)$.
On the other hand, the theorem asserts that the conditional law of $(\eta(X^{z,\ep}_{ m_\ep t}))_{t\geq 0}$ given $(h,\eta)$ converges to the conditional law of LBM given $h$, where $m_\ep$ is as in~\eqref{eqn-median-def}.
Equivalently, $\ep m_\ep  $ is the annealed median exit time of $\wh X^{0,\ep}$ from $B_{1/2}$.
Since $\wh X^{0,\ep}$ converges in law as $\ep\rta 0$ along $\mcl E$, it follows that $\ep m_\ep$ converges to a constant $\frk m > 0$ as $\ep\rta 0$ along $\mcl E$.

Let $\wh P^\ep$ be the conditional law of $ ( \eta(X^{z,\ep}_{m_\ep t}))_{t\geq 0}  =(\wh X^{z,\ep}_{\ep m_\ep t })_{t\geq 0}$ given $(h,\eta)$.
From the convergence $\ep m_\ep \rta \frk m$ along $\mcl E$, we infer that along $\mcl E$ the joint law of $(h,\eta,\wh P^\ep)$ converges to the joint law of $(h,\eta,\wh P)$, where $\wh P$ is defined as follows.
For $z\in\BB C$, $\wh P_z$ is the law of the path $t\mapsto \wh X^z_{\frk m t}$ where $\wh X^z$ is sampled from $P_z$.
Equivalently, $\wh P_z$ is the law of Liouville Brownian motion started from $z$ pre-composed with the time change $t\mapsto c \frk m t$.

By the definition of $m_\ep$, the annealed median exit time of $\wh X^{0,\ep}_{\ep m_\ep t}$ from $B_{1/2}$ is 1, hence the annealed median exit time of a path with the law $\wh P_0$ from $B_{1/2}$ is 1.
This fixes the value of the constant $c\frk m$ independently of the choice of subsequence.
We therefore get that $(h,\eta,\wh P^\ep)$ converges in law to $(h,\eta,P^{\op{LBM}})$ as $\ep\rta 0$, not just subsequentially, where here $P^{\op{LBM}}_z$ for $z\in\BB C$ is the conditional law given $h$ of \eqref{D:LBMmed}, i.e., LBM with respect to\ $h$ started from $z$ pre-composed with the time-change $ t \mapsto m_0 t$, where $m_0$ is the median of the exit time of the $B_{1/2}$ of this LBM.
Since $P^{\op{LBM}}$ is a.s.\ determined by $h$, it follows from an elementary fact in probability theory (analogous to the fact that convergence in distribution to a constant implies convergence in probability, see, e.g.,~\cite[Lemma 4.5]{ss-contour}) that in fact $P^\ep \rta P^{\op{LBM}}$ in probability.

The bounds~\eqref{eqn-scaling-constants} for $m_\ep$ follow since we have show that for every sequence of $\ep$'s tending to zero, there is a further subsequence along which $\ep m_\ep$ converges to a positive number.
\end{proof}

It remains to prove Lemma~\ref{lem-deterministic-constant}.
Our argument will use the following minor generalization of Definition~\ref{def-lqg-surface}.

\begin{defn} \label{def-curve-decorated}
	A \emph{curve-decorated quantum surface} is an equivalence class of triples $(D,h,\eta)$ where $D\subset\BB C$ is open, $h$ is a distribution on $D$, and $\eta$ is a curve in $\ol D$, under the equivalence relation whereby two such triples $(D,h,\eta)$ and $(\wt D, \wt h , \wt\eta)$ are equivalent if there is a conformal map $\phi : \wt D \rta \wt h$ which extends to a homeomorphism $\wt D\cup \bdy \wt D \rta D\cup \bdy D$ {(with $\bdy \wt D$ and $\bdy D$ viewed as collections of prime ends)} such that $h$ and $\wt h$ are related as in~\eqref{eqn-lqg-coord} and $\phi\circ\wt\eta = \eta$.
	We similarly define a curve-decorated quantum surface with more than one curve.
\end{defn}

The idea of the proof of Lemma~\ref{lem-deterministic-constant} is as follows.
Let $a < u < b$ and let $\wh X$ be a random curve sampled from the law $P_{\eta(u)}$ and stopped upon exiting $\eta([a,b])$. It is easily seen that the constant $c$ is a.s.\ determined by the triple $( \eta([a,b]) , h|_{\eta([a,b])}  , \wh X)$.
Indeed, since the time parametrization of LBM is locally determined by the GFF, from this data we can see how we need to scale the time parametrization of $\wh X$ in order to get LBM w.r.t.\ $h$  (see Lemma~\ref{lem-lbm-mut-sing}).
In fact, $c$ is a.s.\ determined by $( \eta([a,b]) , h|_{\eta([a,b])}  , \wh X)$ viewed as a curve-decorated quantum surface since both LBM and the adjacency graph of cells obey the LQG coordinate change formula (by~\cite[Theorem 1.4]{berestycki-lbm}).

If $a < b < \wt a < \wt b$, then~\cite[Theorem 1.9]{wedges} tells us that the curve-decorated quantum surfaces corresponding to $( \eta([a,b]) , h|_{\eta([a,b])}  )$ and $( \eta([\wt a,\wt b]) , h|_{\eta([\wt a,\wt b])}  )$ are independent.
If we knew that $P$ were in some sense locally determined by $(h,\eta)$, then we would be done: indeed, if we knew this then the preceding paragraph would tell us that $c$ is a.s.\ determined by each of two different independent random variables, which means that $c$ is independent from itself and hence a.s.\ constant.
However, we do \emph{not} know that $P$ is locally determined or even determined at all by $(h,\eta)$, even though $P^\ep$ is a.s.\ determined by $(h,\eta)$.
This is because the property of one random variable being a.s.\ determined by another is not in general preserved under convergence in law.
However, we know that $P^\ep$ is a.s.\ determined by $(h,\eta)$ and moreover independence \emph{is} preserved under convergence in law.
So, we can make the above argument precise by appealing to the subsequential convergence in law $(h,\eta,P^\ep) \rta (h,\eta,P)$.

\begin{figure}
\begin{center}
\includegraphics[scale=1]{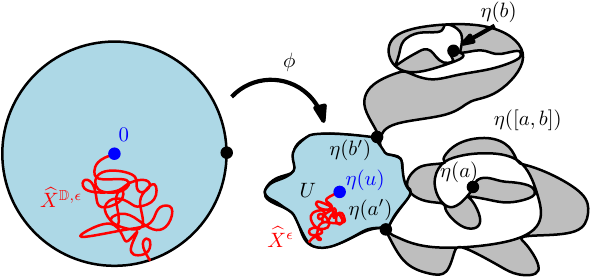}
\caption{\label{fig-ssl-ind}
Illustration of the setup used to prove Lemma~\ref{lem-deterministic-constant}.
The figure shows the case when $\gamma \in (\sqrt 2 ,2)$. For $\gamma \in (0,\sqrt 2)$, the interior of $\eta([a,b])$ is connected and is equal to $U$.
}
\end{center}
\end{figure}

We now describe the framework in which we will prove the independence property needed to get Lemma~\ref{lem-deterministic-constant}.
See Figure~\ref{fig-ssl-ind} for an illustration.
Fix finite times $a < u < b$.
Let $U$ be the connected component of the interior of $\eta([a,b])$ which contains $\eta(u)$ (such a component exists a.s.\ since a.s.\ $\eta(u) \notin \bdy\eta([a,b])$).
Due to the difference in topology for space-filling SLE$_\kappa$ in the phases $\kappa \geq 8$ and  $\kappa \in (4,8)$,
when $\gamma \in (0,\sqrt 2]$  the set $U$ is the entire interior of $\eta([a,b])$ but for $\gamma \in (\sqrt 2, 2)$ a.s.\ $U$ is a proper subset of the interior of $\eta([a,b])$.
The domain $U$ is homeomorphic to the unit disk and there is a unique time interval $[a' ,b']$ for which $\eta([a',b']) = \ol U$.
Again, for $\gamma \in (0,\sqrt 2]$ we have $[a,b] = [a' ,b']$ but this is a.s.\ not the case for $\gamma \in (\sqrt 2 , 2)$.
Also let $\wh X$ be sampled from the law $P_{\eta(u)}$ and let $T$ be the first exit time of $\wh X$ from $U$.

We are interested in the curve(s)-decorated quantum surfaces $(U , h , \eta|_{[a',b']} , \wh X|_{[0,T]})$. We could just work directly with this curve-decorated quantum surface, but we find it more clear to instead specify a particular parametrization which depends only on the curve-decorated quantum surface (not on the original choice of equivalence class representative).
Let $\phi : \BB D \rta U$ be the unique conformal map such that $\phi(0) = \eta(u)$ and $\phi(a') = 1$.
Recalling~\eqref{eqn-lqg-coord}, we define
\eqbn
h^{\BB D} := h\circ \phi + Q\log |\phi'|  , \quad  \eta^{\BB D} := \phi^{-1} \circ \eta |_{[a',b']} , \quad \text{and} \quad \wh X^{\BB D} := \phi^{-1} \circ \wh X |_{[0,T]} .
\eqen
Then $(\BB D , h^{\BB D} , \eta^{\BB D}  , \wh X^{\BB D}) = (U , h , \eta|_{[a',b']} , \wh X|_{[0,T]})$ as curve-decorated quantum surfaces.
We define
\eqb
P^{\BB D}_0 := \left( \text{$  P_{\eta(u)}$-law of $\wh X^{\BB D}$} \right) ,
\eqe
so $P^{\BB D}_0$ is a random probability measure on curves in $\BB D$ started from 0 and stopped upon hitting $\bdy\BB D$.
Our proof of Lemma~\ref{lem-deterministic-constant} is based on the following three lemmas.

\begin{lem} \label{lem-lbm-coord}
Let $c$ be the random variable from Lemma~\ref{lem-random-constant}.
In the setting described just above, a.s.\ $P^{\BB D}_0$ is the law of LBM with respect to $h$ started from 0 and stopped upon hitting $\bdy\BB D$, pre-composed with the time change $t\mapsto c t$.
\end{lem}
\begin{proof}
This is immediate from the $\gamma$-LQG coordinate change formula for LBM~\cite[Theorem 1.3]{berestycki-lbm} and for the adjacency graph of cells of $\mcl G^\ep$ (which follows from the coordinate change formula for the LQG area measure plus the conformal invariance of space-filling SLE).
\end{proof}

\begin{lem} \label{lem-ssl-ind}
Let $a < u < b$ and $\wt a < \wt u < \wt b$ be two triples of times and assume that $b  <  \wt a$.
Define $h^{\BB D} , \eta^{\BB D} ,$ and $P_0^{\BB D}$ as above for $a,u,b$ and define $\wt h^{\BB D} , \wt\eta^{\BB D},$ and $\wt P^{\BB D}_0$ as above with $\wt a , \wt u , \wt b$ in place of $a,u,b$.
Then the triples $(h^{\BB D}  , \eta^{\BB D} , P_0^{\BB D})$ and $(\wt h^{\BB D} , \wt\eta^{\BB D} , \wt P_0^{\BB D})$ are independent.
\end{lem}
\begin{proof}
Fix a time $\BB s \in (b ,\wt a)$.
By~\cite[Theorem 1.9]{wedges}, the past / future curve-decorated quantum surfaces
\eqb \label{eqn-no-bm-surfaces}
\left( \eta((-\infty,\BB s]) , h|_{\eta((-\infty, \BB s ])} ,\eta|_{(-\infty, \BB s ]}  \right)
\quad \text{and} \quad
\left( \eta([\BB s,\infty) ) , h|_{\eta([\BB s,\infty))} ,\eta|_{[\BB s,\infty)} \right)
\eqe
are independent.
By definition, $(h^{\BB D}  ,\eta^{\BB D})$ (resp.\ $(\wt h^{\BB D} , \wt\eta^{\BB D})$ is determined by the first (resp.\ second) quantum surface in~\eqref{eqn-no-bm-surfaces}.
Hence $(h^{\BB D} , \eta^{\BB D})$ and $(\wt h^{\BB D} , \wt\eta^{\BB D})$ are independent.
We need to show that the independence continues to hold if we also include the random laws $P_0^{\BB D}$ and $\wt P_0^{\BB D}$.

The basic idea of the proof is that for each $\ep > 0$, the discrete analogs of the laws $P_0^{\BB D}$ and $\wt P_0^{\BB D}$ in the $\ep$-mated-CRT map setting are determined by the first and second curve-decorated quantum surfaces in~\eqref{eqn-no-bm-surfaces}, respectively, so one has the desired independence statement with these discrete laws used in place of $P_0^{\BB D}$ and $\wt P_0^{\BB D}$.
This will allow us to later use the fact that independence (unlike the condition that one random variable determines another) is preserved under taking limits in law.

Let us now define the discrete analogs of the laws $P_0^{\BB D}$ and $\wt P_0^{\BB D}$ which we will work with. Recalling the times $a',b' \in [a,b]$ defined above, for $\ep >0$ let $\mcl T^\ep$ be the subgraph of $\mcl G^\ep$ induced by the vertex set $[a' ,b'] \cap (\ep\BB Z)$.
Let $\bdy\mcl T^\ep$ be the set of vertices $x\in  \mcl V\mcl T^\ep$ which are joined by an edge of $\mcl G^\ep$ to a vertex which is not in $\mcl T^\ep$.
By the Brownian motion definition of the mated-CRT map~\eqref{edgemap}, for small enough $\ep > 0$ both $\mcl T^\ep$ and $\bdy\mcl T^\ep$ are determined by the left/right boundary length process $(L - L_{\BB s } ,R - R_{\BB s } ) |_{(-\infty,\BB s]} $, which in turn is a.s.\ determined by the first curve-decorated quantum surface~\eqref{eqn-no-bm-surfaces}.

Let $u^\ep := \ep \lceil u / \ep \rceil$, so that $u^\ep \in \ep\BB Z$.
Conditional on $(h,\eta)$, let $X^\ep$ be a simple random walk on $\mcl T^\ep$ stopped at the first time $J^\ep$ that it hits $\bdy \mcl T^\ep$.
Recalling the conformal map $\phi : \BB D\rta U$ above, for $t\in [0,\ep J^\ep] \cap (\ep\BB Z)$, let $\wh X^{\BB D,\ep}_t := \phi^{-1}(\eta(t/\ep))$.
We extend the definition of $\wh X^{\BB D,\ep}$ to all of $[0,\ep J^\ep]$ by piecewise linear interpolation at constant speed.
Let $P_0^{\BB D,\ep}$ be the conditional law of $\wh X^{\BB D,\ep}_t$ given $(h,\eta)$.
By the preceding paragraph and the definition of $\phi$, the law $P_0^{\BB D,\ep}$ is a.s.\ determined by the first quantum surface in~\eqref{eqn-no-bm-surfaces}.

If we define $\wt P_0^{\BB D,\ep}$ similarly but with $\wt a , \wt u , \wt b$ in place of $a,u,b$, then $\wt P_0^{\BB D,\ep}$ is a.s.\ determined by the second quantum surface in~\eqref{eqn-no-bm-surfaces}.
As a consequence, the triples $(h^{\BB D}  , \eta^{\BB D} , P_0^{\BB D,\ep})$ and $(\wt h^{\BB D} , \wt\eta^{\BB D} , \wt P_0^{\BB D,\ep})$ are independent.

We now argue that if $\mcl E$ is a subsequence along which $(h,\eta,P^\ep) \rta (h,\eta,P)$ in law, then we also have
\eqb \label{eqn-ssl-ind-conv}
 (h^{\BB D}  , \eta^{\BB D} , P_0^{\BB D,\ep}  , \wt h^{\BB D}  , \wt \eta^{\BB D} , \wt P_0^{\BB D,\ep})
 \rta (h^{\BB D}  , \eta^{\BB D} , P_0^{\BB D} , \wt h^{\BB D}  , \wt \eta^{\BB D} , \wt P_0^{\BB D})
\eqe
in law as $\mcl E\ni\ep \rta 0$, which concludes the proof since indepenendence is preserved under convergence of joint laws.

To get~\eqref{eqn-ssl-ind-conv}, we note that a process $\wh X^{\BB D,\ep}$ sampled from $P_0^{\BB D , \ep}$ is almost the same as the process $\phi^{-1}(\wh X^\ep)$, where $\wh X^\ep$ is sampled from $P_{\eta(u^\ep)}^\ep$ and stopped at the first time it hits a point in $\eta(\ep\BB Z\setminus [a',b'])$, except for the following minor difference.
The process $\wh X^{\BB D,\ep}$ travels along straight line segments between the times when it hits points of $\phi^{-1}(\eta(\ep\BB Z))$, whereas the path $\phi^{-1}(\wh X^\ep)$ instead travels along paths whose images under $\phi$ are straight line segments.

It is easily seen that the conformal map $\phi$ can be extended to a homeomorphism $\ol U \rta \ol{\BB D}$: indeed, by the flow line construction of space-filling SLE~\cite[Section 1.2.3]{ig4}, $\bdy U$ is a finite union of SLE$_\kappa$-type curves which intersect only at their endpoints.
In particular, $\bdy U$ is a Jordan curve so the desired extension of $\phi$ follows from~\cite[Theorem 2.6]{pom-book}.
By the above continuity and the preceding paragraph (note that $\phi$ does not depend on $\ep$), the Prokhorov distance between the law $P_0^{\BB D,\ep}$ and the $P_{\eta(h^\ep)}^\ep$-law of  $\phi^{-1}(\wh X^\ep)$ tends to zero in probability as $\ep\rta 0$.
On the other hand, it is easily seen from the convergence $(h,\eta,P^\ep) \rta (h,\eta,P)$ in law that~\eqref{eqn-ssl-ind-conv} holds if we replace $P_0^{\BB D,\ep}$ by the $P_{\eta(u^\ep)}^\ep$-law of $\phi^{-1}(\wh X^\ep)$, and we make the analogous replacement made for $\wt P^{\BB D,\ep}_0$. Thus~\eqref{eqn-ssl-ind-conv} holds.
\end{proof}

\begin{lem} \label{lem-lbm-mut-sing}
Let $h$ be a random distribution on a domain $D\subset\BB C$ whose law is locally absolutely continuous w.r.t.\ the law of the GFF on $D$.
Let $z\in D$ and let $X$ be the Liouville Brownian motion w.r.t.\ $h$, stopped at the first time that it exists $D$.
For any $c_1,c_2 >0$ with $c_1\not=c_2$, the laws of $(h , (X_{c_1 t})_{t \geq 0})$ and $(h , (X_{c_2 t})_{t \geq 0})$ are mutually singular.
\end{lem}
\begin{proof}
Let $B$ be a standard planar Brownian motion started from $z$ and stopped when it exits $D$, sampled independently from $h$.
By the definition of Liouville Brownian motion, there is a random $\sigma(h,B)$-measurable function $\phi : [0,\infty) \rta [0,\infty)$ such that $(h , (X_t)_{t\geq 0}) \eqD (h , B_{\phi^{-1}(t)})_{t\geq 0}$.
We have the following trivial elementary probability fact: if $Y$ is a random variable and $f$ and $g$ are measurable functions such that $\BB P[f(Y) \not= g(Y)]  = 1$, then the laws of $(Y , f(Y))$ and $(Y, g(Y))$ are mutually singular.
To get the lemma, we apply this fact with $Y = (h , B)$ and $f$ and $g$ given by the functionals which take $(h,B)$ to $ B_{\phi^{-1}(c_1 t)}$ and $ B_{\phi^{-1}(c_2 t)}$, respectively.
\end{proof}

\begin{remark}
The proof of Lemma~\ref{lem-lbm-mut-sing} relies crucially on the fact that we are looking at the \emph{joint} law of $h$ and the Liouville Brownian motion.
At least for some choices of $h$, the \emph{marginal} laws of $(X_{c_1 t})_{t \geq 0} $ and $(X_{c_2 t})_{t \geq 0}$ are mutually absolutely continuous.
Indeed, if $h$ is the restriction to $D$ of a zero-boundary GFF on a larger domain $D'$ with $\ol D\subset D'$, then the laws of $h+\log c_1$ and $h+ \log c_2$, hence also the laws of  $(h + \log c_1 , (X_{c_1 t})_{t \geq 0})$ and $(h + \log c_2 , (X_{c_2 t})_{t \geq 0})$ are mutually absolutely continuous (see, e.g.,~\cite[Proposition 3.4]{ig1}).
\end{remark}

\begin{proof}[Proof of Lemma~\ref{lem-deterministic-constant}]
Let $c> 0$ be as in Lemma~\ref{lem-random-constant} and assume the setup of Lemma~\ref{lem-ssl-ind}. Due to Lemma~\ref{lem-lbm-coord} (applied to each of the triples $a,u,b$ and $\wt a , \wt u ,\wt b$) followed by Lemma~\ref{lem-lbm-mut-sing}, we know that $c$ can a.s.\ be expressed as a measurable function of either of the two triples $(h^{\BB D}  , \eta^{\BB D} , P_0^{\BB D})$ or $(\wt h^{\BB D} , \wt\eta^{\BB D} , \wt P_0^{\BB D})$.
However, by Lemma~\ref{lem-ssl-ind} these two triples are independent from each other.
Hence $c$ is independent from itself so $c$ must be a.s.\ equal to a deterministic constant.
\end{proof}

\subsection{From whole plane to disk}
\label{sec-plane-to-disk}

In this short section we explain why Theorem \ref{thm-lbm-conv} (convergence to Liouville Brownian motion in the whole plane case with the LQG embedding) implies Theorem \ref{thm-lbm-conv_disk} (convergence to Liouville Brownian motion in the disk case, under the Tutte embedding). The arguments in this section are quite close to (but more detailed than) \cite[Section 3.3]{gms-tutte}, see in particular \cite[Lemma 3.6]{gms-tutte}.

As recalled in Section~\ref{sec-sle-lqg}, let $(\BB{C}, h, 0, \infty)$ denote a $\gamma$-quantum cone (corresponding to a weight $W = 4- \gamma^2$, in the terminology of \cite{wedges}), and let $\eta$ be an independent space-filling SLE$_\kappa$, with $\kappa = 16/\gamma^2$, reparametrised by $\gamma$-quantum area, i.e., $\mu_h (\eta([s,t])) = t-s$ for any $s<t$. Then by \cite[Theorem 1.9]{wedges},
the quantum surface $\mathcal{W}$ obtained by restricting $h$ to {the interior $D$ of} $\eta([0, \infty))$, that is, $\mathcal{W} = (D,h|_D, 0, \infty)$ has the distribution of a quantum wedge with weight $W = 2- \gamma^2/2$. Note that this corresponds to an $\alpha$-quantum wedge where $\alpha = 3\gamma/2$ (see Table 1.1 in \cite{wedges}), where $\alpha$ denotes the strength of the logarithmic singularity near 0 in a half-plane embedding. In particular, since $\alpha$-quantum wedges are thick (have the topology of a half-plane) for $\alpha \le Q$, and are thin (given by a Poissonian chain of so-called ``beads") if $\alpha \in (Q, Q+ \gamma/2)$, this means that $\mathcal{W}$ is thick for $\gamma \le \sqrt{2}$ and thin if $\gamma \in (\sqrt{2}, 2)$. Furthermore, given $D$, in the thick case $\eta$ is simply an independent chordal space-filling SLE$_\kappa$ in $D$ from 0 to $\infty$, reparametrised by its area, while in the thin case it is a concatenation of space-filling chordal SLE$_\kappa$ in each bead of $\mathcal{W}$.
Basically, the idea of the proof is to use the absolute continuity of the law of the quantum disk w.r.t.\ the law of $\mcl W$ (or a single bead of $\mcl W$ if $\gamma  \in (\sqrt 2 , 2)$) away from their respective marked points.

Consider first the case $\gamma \le \sqrt{2}$ so that $\mathcal{W}$ is a thick quantum wedge.
Fix a large constant $C>1$.
We apply a conformal map $D\rta\BB H$ to get a new parametrization $(\BB H , h^{\mcl W}_C  , -C,\infty)$ of the quantum surface $\mcl W$, chosen so that the two marked points are $-C$ and $\infty$ and the $\nu_{h^{\mcl W}_C}$-length of $[-C,0]$ is 1 (this last choice of normalization is somewhat arbitrary --- 1 can be replaced by any number larger than $1/2$).
Let $\eta^{\mcl W}_C$ be the image of $\eta|_{[0,\infty)}$ under this conformal map, so that $\eta^{\mcl W}_C$ is a chordal SLE$_\kappa$ from $-C$ to $\infty$ in $\BB H$ parameterized by $\gamma$-LQG mass w.r.t.\ $h^{\mcl W}_C$.
Recall that Liouville Brownian motion satisfies the quantum surface change of coordinates \eqref{eqn-lqg-coord} by Theorem 1.4 in \cite{berestycki-lbm}.
From this and Theorem \ref{thm-lbm-conv}, we deduce the analog of Theorem~\ref{thm-lbm-conv} for $(h^{\mcl W}_C, \eta^{\mcl W}_C)$: if we condition on $(h^{\mcl W}_C , \eta^{\mcl W}_C)$, then the conditional law of the random walk on the adjacency graph of cells $\eta^{\mcl W}_C([x-\ep,x])$ for $x\in \ep \BB N$, stopped at the first time it hits a cell which intersects $\bdy\BB H$, linearly interpolated and time changed appropriately, converges as $\ep\rta 0$ to the law of Liouville Brownian motion w.r.t\ $h_C^{\mcl W}$ stopped upon hitting $\bdy\BB H$. Moreover, the convergence is uniform over the choice of starting point in any compact subset of $\ol{\BB H}$.

We now use an argument of absolute continuity between the above $\frac{3\gamma}{2}$-quantum wedge and a singly marked quantum disk $\mcl D$ with unit area and unit boundary length.
Let us choose a parametrization $(\BB H  , h^{\mcl D} , \infty)$ of $\mcl D$ so that the marked point is at $\infty$ and the $\nu_h$-lengths of $(-\infty,0]$ and $[0,\infty)$ are each equal to $1/2$.
Using the definitions of the quantum wedge and the quantum disk from~\cite[Sections 4.2 and 4.5]{wedges}, it is easily verified that the law of the restriction of $h^{\mcl D}$ to any fixed ball $B = \BB{H} \cap B(0, R)$ centered at 0 with $R<C $ say, is absolutely continuous with respect to the law of the field $h^{\mcl W}_C$ (here it is important that we are considering the fields away from the marked points, since they have different logarithmic singularities at their respective marked points).
Note, however, that the reverse absolute continuity statement is not true since, e.g., $\mu_{h_C^{\mcl W}}(B)$ can be greater than 1.

Let $\eta^{\mcl D}$ be a space-filling SLE$_\kappa$ loop in $\BB D$ based at $\infty$, parametrized by $\mu_{h^{\mcl D}}$-mass.
Let $\preceq_C^{\mcl W}$ (resp.\ $\preceq^{\mcl D}$) be the total ordering on $\BB Q^2\cap B$ induced by $\eta_C^{\mcl W}$ (resp.\ $\eta^{\mcl D}$).
By Proposition~\ref{prop-sle-loop}, we know that $\preceq_C^{\mcl W} \rta \preceq^{\mcl D}$ in the total variation sense as $C\rta \infty$.
Note that $\preceq_C^{\mcl W}$ determines each segment of $\eta_C^{\mcl W}$ which is contained in $B$ viewed modulo time parameterization, as well as the order in which $\eta_C^{\mcl W}$ traces these segments; a similar statement holds for $\eta^{\mcl D}$.

Let $\mcl F_C^{\mcl W,\ep}$ be the $\sigma$-algebra generated by the space-filling SLE cells $\eta([x-\ep,x])$ for $x\in\ep\BB N$ which intersect $B$. Note that $\mcl F_C^{\mcl W,\ep} \subset \sigma(h_C^{\mcl W} ,\eta_C^{\mcl W})$.
We now combine the total variation convergence in the preceding paragraph with the absolute continuity in the paragraph before that, each applied with $B$ replaced by a slightly larger ball $B'\supset B$ (which contains all of the cells which intersect $B$ with probability tending to 1 as $\ep\rta 0$).
We infer that there is an event $E_C^\ep$ such that $\BB P[E_C^\ep] \rta 1$ as $\ep\rta 0$ and then $C\rta \infty$, such that the following is true.
For any $p \in (0,1)$, there exists $q = q(p) \in (0,1)$ such that if $F \in \mcl F_C^{\mcl W,\ep}$ has probability at most $q$ and $\wt F$ denotes the corresponding event defined with $(h^{\mcl D} , \eta^{\mcl D})$ in place of $(h_C^{\mcl W}, \eta_C^{\mcl W})$, then $\BB P[\wt F\cap E_C^\ep] \leq p$.

Liouville Brownian motion started from $z \in \BB{H}$ leaves the upper half plane almost surely in finite time and its law depends continuously on its starting point.
From this, we infer that for any fixed compact set $K\subset \ol{\BB H}$ we can choose the ball $B$ to be sufficiently large so that the infimum over all $z\in K$ of the probability that Liouville Brownian motion w.r.t.\ $h^{\mcl D}$ started from $z$ hits $\bdy\BB H$ before leaving $\ol B$ is as close to 1 as we like.
Due to this fact and the preceding paragraph, we see that the analog of Theorem~\ref{thm-lbm-conv} for $(h_C^{\mcl W} , \eta_C^{\mcl W})$ discussed above implies the corresponding analog of Theorem~\ref{thm-lbm-conv} for $(h^{\mcl D} , \eta^{\mcl D})$.

We now change coordinates back to $\BB D$. By a slight abuse of notation we do not change our notation for fields / curves, so that now $h^{\mcl D}$ is a field on $\BB D$ and $\eta^{\mcl D} : [0,1]\rta \ol{\BB D}$. From the preceding paragraph and the invariance of LBM under LQG coordinate change~\cite[Theorem 1.4]{berestycki-lbm}, we infer that the statement of Theorem \ref{thm-lbm-conv_disk} holds (at least in the case $\gamma \le \sqrt{2}$) but with the SLE/LQG embedding $x\mapsto \eta^{\mcl D}(x)$ instead of the Tutte embedding. As we recalled in \eqref{E:TutteLQG}, the Tutte embedding $\psi_\ep$ of the mated-CRT map associated with $(h^{\mcl D} , \eta^{\mcl D})$ satisfies
$$
\max_{x \in \mcl{VG}^\eps} | \psi_\eps(x) - \eta^{\mcl D}(x)| \to 0
$$
in distribution; see (3.3) in Section 3.4 of \cite{gms-tutte}. We thus obtain Theorem~\ref{thm-lbm-conv_disk}.

We now discuss the thin case where $\gamma \in (\sqrt{2}, 2)$. The idea is to apply a similar argument except we restrict ourselves to a given bead of $\mathcal{W}$. One issue with this is that it is not obvious how to describe the law of the bead containing a given point $z$. However, since we assume the statement of Theorem \ref{thm-lbm-conv} (which holds uniformly over an arbitrary compact set $K$ of the wedge embedded into $\BB{H}$), the idea is to consider the first time that some bead of $\mathcal{W}$ has area greater than one, say. We then have a quenched invariance principle uniformly over points of the intersection of $K$ with this bead, and by choosing $K$ sufficiently large, we can ensure that with probability as close to one as desired, this intersection contains some given arbitrary compact set of the bead.
Moreover, since the beads of $\mathcal{W}$ have a Poissonian structure, the first bead with area greater than one has a law which is mutually absolutely continuous w.r.t.\ the law of a doubly marked quantum disk conditioned to have area at least one away from their respective marked points~\cite[Section 4.4]{wedges} (for $\gamma=\sqrt{8/3}$, the beads are exactly quantum disks but for other values of $\gamma$ they have different log singularities at the marked points).
The proof in the case when $\gamma \in (\sqrt 2 , 2)$ then follows from the same local absolute continuity argument as in the case when $\gamma \leq \sqrt 2$, except that we use this single bead instead of the whole of $\mcl W$.

\appendix

\section{Appendix}
\label{sec-appendix}

In this appendix we prove some basic facts about the $\gamma$-LQG measure.
The proofs of these facts are standard and do not use any of the other results described in the paper, so are collected here to avoid distracting from the main argument.

\subsection{Basic estimates for the LQG area measure}

Let $\gamma\in(0,2)$ and let $h$ be the circle average embedding of a $\gamma$-quantum cone, as in Definition~\ref{def-quantum-cone}.
In this subsection we record some basic estimates for the $\gamma$-LQG measure $\mu_h$ associated with $h$.

\begin{lem}[LQG mass of Euclidean balls] \label{lem-ball-mass}
Fix a small parameter $\zeta\in (0,1)$.
With polynomially high probability as $\delta \rta 0$,
\eqb \label{eqn-ball-mass}
\delta^{(2+\gamma)^2/2 +\zeta} \leq \mu_h(B_\delta(z)) \leq \delta^{(2-\gamma)^2/2 - \zeta} , \quad\forall z \in B_{1-\delta}  .
\eqe
\end{lem}
\begin{proof}
Let $h^0$ be a whole-plane GFF with the additive constant chosen so that $h^0_1(0) = 0$. It is shown in various places in the literature, e.g.,~\cite[Lemma 3.8]{dg-lqg-dim}, that with polynomially high probability as $\delta\rta 0$, one has
\eqb \label{eqn-ball-mass0}
\delta^{(2+\gamma)^2/2 +\zeta} \leq \mu_{h^0}(B_\delta(z)) \leq \delta^{(2-\gamma)^2/2 - \zeta} , \quad\forall z \in \BB D .
\eqe

Now let $\wt h := h^0 -\gamma\log|\cdot|$.
Then $\wt h|_{\BB D} \eqD h|_{\BB D}$. Hence it suffices to prove that with polynomially high probability as $\delta \rta 0$, the bound~\eqref{eqn-ball-mass0} holds with $\wt h$ in place of $h$.
The lower bound for $\mu_{h^0}(B_\delta(z))$ in~\eqref{eqn-ball-mass0} with $\wt h$ in place of $h$ is immediate from~\eqref{eqn-ball-mass0} since $(\wt h - h^0)|_{\BB D}$ is a non-negative function.
So, we only need to prove the upper bound.
We do this using~\eqref{eqn-ball-mass0} and a union bound over exponential scales. The argument is entirely standard but we give the details for the sake of completeness.

For $k\in\BB N_0$, let
\eqb
\wt h_k := \wt h(e^{-k}\cdot) - \wt h_{e^{-k}}(0) \quad \text{and} \quad A_k  := B_{e^{-k}  }(0) \setminus B_{e^{-k-1} }(0) .
\eqe
By the scale invariance of the law of $h^0$ modulo additive constant, $\wt h_k \eqD \wt h$.
The restriction of $\wt h - h^0 = - \gamma \log|\cdot|$ to the annulus $A_0 $ is bounded above by a finite $\gamma$-dependent constant.
We can therefore apply~\eqref{eqn-ball-mass0} with $e^k\delta$ in place of $\delta$, followed by a union bound over all $k\in [0, (1-\zeta) \log\delta^{-1}]_{\BB Z}$, to get that with polynomially high probability as $\delta\rta 0$,
\eqb \label{eqn-ball-mass-k}
 \mu_{\wt h_k}(B_{e^k\delta}(z)) \leq (e^k \delta)^{(2-\gamma)^2/2 - \zeta} , \quad\forall z \in A_0  ,\quad\forall k\in [0, (1-\zeta) \log\delta^{-1}]_{\BB Z} .
\eqe

By the coordinate change formula for the LQG area measure~\cite[Proposition 2.1]{shef-kpz} and the definition of $\wt h_k$, a.s.\
\eqbn
 \mu_{\wt h}(B_\delta(z))
 = e^{  \gamma \wt h_{e^{-k}}(0)} e^{- \gamma Q k}  \mu_{\wt h_k}(B_{e^k\delta}(e^k z))
 = e^{ \gamma h^0_{e^{-k}}(0)} e^{ - \gamma (Q-\gamma) k} \mu_{\wt h_k}(B_{e^k\delta}(e^k z))  ,\quad\forall \delta >0, \quad\forall k\in\BB N.
\eqen
Therefore,~\eqref{eqn-ball-mass-k} (note that $z\in A_k$ if and only if $e^k z \in A_0$)  implies that
\eqb \label{eqn-ball-mass-scaled}
 \mu_{\wt h}(B_\delta(z))
 \leq  e^{ \gamma h^0_{e^{-k}}(0) +  [  (2-\gamma)^2/2 - \zeta - \gamma(Q-\gamma) ] k} \delta^{(2-\gamma)^2/2 - \zeta} ,
 \quad\forall z\in A_k ,
 \quad\forall k\in [0, (1-\zeta) \log\delta^{-1}]_{\BB Z} .
\eqe

Since $\gamma \in (0,2)$ we have $(2-\gamma)^2/2   - \gamma(Q-\gamma)  = \gamma(2-\gamma) < 0$, so the exponential term in~\eqref{eqn-ball-mass-scaled} is at most $e^{\gamma h^0_{e^{-k}}(0)}$ for a small enough choice of $\zeta \in (0,1)$.
Furthermore, since $t\mapsto h^0_{e^{-t}}(0)$ is a standard linear Brownian motion it holds with polynomially high probability as $\delta \rta 0$
that for any fixed $\zeta>0$, $h^0_{e^{-k}}(0) \leq (\zeta/\gamma)\log\delta^{-1}$ for each $k\in [0, (1-\zeta) \log\delta^{-1}]_{\BB Z}$.
From this and~\eqref{eqn-ball-mass-scaled}, we get that with polynomially high probability as $\delta\rta 0$,
\eqb \label{eqn-ball-mass-final}
 \mu_{\wt h}(B_\delta(z))
 \leq   \delta^{(2-\gamma)^2/2 - 2 \zeta} ,
 \quad\forall z\in \BB D\setminus B_{\delta^{1-\zeta}}(0) .
\eqe

It is easily seen from a scaling argument that with polynomially high probability as $\delta \rta 0$, one has $\mu_h(B_{\delta^{1-\zeta}}(0)) \leq \delta^{(1-\zeta) \gamma(Q-\gamma) - \zeta}$, which is smaller than $\delta^{(2-\gamma)^2/2 - \zeta}$ for a small enough choice of $\zeta \in (0,1)$.
By combining this with~\eqref{eqn-ball-mass-final}, we get the upper bound in~\eqref{eqn-ball-mass0} with $\wt h$ in place of $h^0$ and with $2\zeta$ in place of $\zeta$.
This is sufficient since $\zeta\in (0,1)$ is arbitrary.
\end{proof}

\begin{lem} \label{lem-max-ball-ratio}
Fix $c\in (0,1)$.
For each $\zeta > 0$, it holds with polynomially high probability as $\delta \rta 0$ that
\eqb \label{eqn-max-ball-ratio}
\sup_{z\in B_{1-\delta}} \sup_{w\in B_\delta(z)} \frac{\mu_h(B_\delta(z))}{\mu_h(B_{c\delta}(w))} \leq \delta^{- \gamma^2/2 -\zeta }.
\eqe
\end{lem}

See also \cite[Theorem A.1]{grv-kpz} for a related (and in some ways more precise) statement.

\begin{proof}[Proof of Lemma~\ref{lem-max-ball-ratio}]
We first prove the lemma with $h$ replaced by a whole-plane GFF $h^0$ with the additive constant chosen so that $h^0_1(0) =0$ and with $B_{1-\delta}$ replaced by $\BB D$.
For each $w\in \BB D$, the random variable $\mu_{h^0}(B_c(w))$ has negative moments of all orders which are bounded above by constants which depend on $c$, but not on $w$~\cite[Theorem 2.12]{rhodes-vargas-review}. A union bound therefore shows that
\eqb
\BB P\left[ \inf_{z\in\BB D} \mu_h(B_c(w)) \leq \delta^{\zeta/2} \right] = O_\delta(\delta^p) ,\quad\forall p > 0.
\eqe
Since $\mu_{h^0}(\BB D)$ has positive moments up to order $4/\gamma^2$~\cite[Theorem 2.11]{rhodes-vargas-review},
\eqb
\BB P\left[ \mu_{h^0}(\BB D) > \delta^{- \gamma^2 / 2 -\zeta/2} \right] \leq \delta^{2 + 2\zeta/\gamma^2 + o_\delta(1) } .
\eqe
Therefore,
\eqb \label{eqn-sup-ratio0}
\BB P\left[  \sup_{w\in \BB D}  \frac{\mu_{h^0}(\BB D)}{\mu_{h^0}(B_{c}(w))}   > \delta^{- \gamma^2 / 2 -\zeta } \right] \leq \delta^{2 + 2\zeta/\gamma^2 + o_\delta(1) } .
\eqe

Due to the translation and scale invariance of the law of $h^0$, modulo additive constant, the law of $\sup_{w\in B_\delta(z)} \frac{\mu_{h^0}(B_\delta(z))}{\mu_{h^0}(B_{c\delta}(w))}$ does not depend on $z$ or $\delta$. We may therefore apply~\eqref{eqn-sup-ratio0} and take a union bound over $O_\delta(\delta^{-2})$ points in $\BB D$ to obtain that with polynomially high probability as $\delta \rta 0$, the bound~\eqref{eqn-max-ball-ratio} with $h^0$ in place of $h$ and $\BB D$ in place of $B_{1-\delta}$.
The lemma statement statement for $h$ follows from the preceding estimate for $h^0$ via a multiscale argument as in the proof of Lemma~\ref{lem-ball-mass}.
\end{proof}

\begin{lem} \label{lem-log-int}
Fix $\rho \in (0,1)$.
There are constants $c_0 ,c_1 >0$ depending only on $\rho$ and $\gamma$ such that for each $A>1$, it holds with probability at least $1-c_0 e^{-c_1 A}$ that
\eqb \label{eqn-log-int}
  \sup_{z\in S} \int_{S} \log\left(\frac{1}{|z-w|} \right) \,d\mu_h(w) \leq A \left( \mu_h(S) +   e^{-c_1 A} \right) ,
  \quad \text{$\forall$ Borel sets $S\subset B_\rho$}.
\eqe
\end{lem}
\begin{proof}
For a Borel set $S\subset B_\rho$,
\eqb \label{eqn-log-int-decomp}
\int_{S} \log\left(\frac{1}{|z-w|} \right) \,d\mu_h(w)
\leq \sum_{k=0}^\infty k \mu_h\left(\left( B_{e^{-k+1}}(z)  \setminus B_{e^{-k}}(z)  \right) \cap S \right).
\eqe
By Lemma~\ref{lem-ball-mass} and a union bound over dyadic values of the radius $\delta$, we can find constants $c_0,c_1 > 0$ depending only on $\rho$ and $\gamma$ such that with probability at least $1 - c_0 e^{-c_1 A}$,
\eqb \label{eqn-use-max-ball-mass}
  \sup_{z\in B_\rho} \mu_h(B_\delta(z))  \leq \delta^{c_1} ,\quad \forall \delta \in (0,e^{-A} ]  .
\eqe
If~\eqref{eqn-use-max-ball-mass} holds, then simultaneously for each $S\subset B_\rho$ and each $z\in S$, the right side of~\eqref{eqn-log-int-decomp} is bounded above by a constant (depending only on $\rho, \gamma$) times
\eqb  \label{eqn-sum-to-mass}
A \mu_h(S) + \sum_{k= \lfloor A \rfloor}^\infty k e^{-c_1 k}
\preceq A\left(  \mu_h(S) + e^{- c_1 A} \right)  ,
\eqe
where here the implicit constant in $\preceq$ depends only on $\rho,\gamma$.
Plugging this into~\eqref{eqn-log-int-decomp} gives~\eqref{eqn-log-int}.
\end{proof}

\subsection{Comparing sums over cells to integrals against LQG measure}
\label{sec-cell-sum}

Assume that we are in the setting of Section~\ref{sec-mated-crt-setup}, so that $h$ is the circle average embedding of a $\gamma$-quantum cone, $\eta$ is a space-filling SLE$_{16/\gamma^2}$ sampled independently from $h$ and then parametrized by $\mu_h$-mass, $\mcl G^\ep$ is the associated $\ep$-mated CRT map, and $H_x^\ep := \eta([x-\ep,x])$.
We also recall the constant $\rho \in (0,1)$.
In this subsection we will give a quantitative version of the intuitively obvious statement that the counting measure on cells of $\mcl G^\ep$ approximates the $\gamma$-LQG measure when $\ep$ is small.

\begin{lem} \label{lem-cell-sum}
There exists $\alpha  =\alpha(\gamma ) > 0$ and $\beta = \beta( \gamma) > 0$ such that with polynomially high probability as $\ep\rta 0$, the following is true. Let $D\subset B_\rho$ be a Borel set and let $f : B_{\ep^\alpha}(D)  \rta [0,\infty)$ be a non-negative function which is $ \ep^{-\beta}$-Lipschitz continuous and satisfies $\|f\|_\infty\leq   \ep^{-\beta}$.
For $x\in \mcl V\mcl G^\ep(D)$, let $w_x^\ep$ be an arbitrary point of $H_x^\ep \cap \ol D$ (e.g., we could take $w_x^\ep = \eta(x)$ or take $w_x^\ep$ to be the point where $f$ attains its maximum or minimum value on $H_x^\ep$).
If we let $\mu_h$ be the $\gamma$-LQG area measure induced by $h$, then
\eqb \label{eqn-cell-sum}
 \ep^{-1} \int_D f( z)  \,d\mu_h (  z) - \ep^{-1+\alpha}
 \leq \sum_{  x\in \mcl V\mcl G^\ep(D) } f(w_x^\ep)
 \leq  \ep^{-1} \int_{B_{\ep^\alpha}(D)} f( z)  \,d\mu_h (  z)  + \ep^{-1+\alpha}
\eqe
simultaneously for every choice of $D$ and $f$ as above.
\end{lem}
\begin{proof}
Fix $q\in \left( 0 , \tfrac{2}{(2+\gamma)^2} \right)$ chosen in a manner depending only on $\gamma$ and let $\beta \in (0,q/2)$.
By Lemma~\ref{lem-cell-diam}, it holds with polynomially high probability as $\ep\rta 0$ that
\eqb \label{eqn-cell-sum-diam}
\op{diam}\left( H_x^\ep \right) \leq \ep^q ,\quad \forall x\in \mcl V\mcl G^\ep(B_\rho) .
\eqe
By a standard tail estimate for the $\gamma$-LQG measure (e.g.,~\cite[Lemma A.3]{ghs-dist-exponent} together with the Chebyshev inequality), it holds with polynomially high probability as $\ep\rta 0$ that
\eqb \label{eqn-cell-sum-mass}
\mu_h\left( B_{\rho + \ep^q} \right) \leq \ep^{-q/2} .
\eqe
Henceforth assume that~\eqref{eqn-cell-sum-diam} and~\eqref{eqn-cell-sum-mass} hold.
Let $\beta \in (0, q/2)$. We will show that~\eqref{eqn-cell-sum} holds simultaneously for any choice of $D$ and $f$ as in the lemma statement.

Since $f$ is $ \ep^{-\beta}$-Lipschitz continuous and $\mu_h(H_x^\ep ) = \ep$ for each $x\in \mcl V\mcl G^\ep $, the bound~\eqref{eqn-cell-sum-diam} implies that for each $x\in \mcl V\mcl G^\ep(D) $,
\allb \label{eqn-cell-approx-upper}
f(w_x^\ep)
 = \ep^{-1} \int_{H_x^\ep } f(w_x^\ep) \, d\mu_h ( z)
&\leq \ep^{-1} \int_{H_x^\ep } f(  z) \, d\mu_h ( z) +     \ep^{ -\beta} \op{diam}(H_x^\ep) \notag \\
&\leq \ep^{-1} \int_{H_x^\ep } f( z) \, d\mu_h ( z) +    \ep^{q-\beta}  \quad \text{by~\eqref{eqn-cell-sum-diam}} .
\alle
Similarly, we also have
\eqb  \label{eqn-cell-approx-lower}
f(w_x^\ep)  \geq  \ep^{-1} \int_{H_x^\ep } f( z) \, d\mu_h ( z) -   \ep^{q-\beta}  .
\eqe

By~\eqref{eqn-cell-sum-diam}, for each $x\in \mcl V\mcl G^\ep(D)$ we have $H_x^\ep \subset B_{\ep^q}(D)$ and the intersection of any two of these cells has zero $\mu_h$-mass.
We can therefore sum~\eqref{eqn-cell-approx-upper} over all $x\in \mcl V\mcl G^\ep(D) $ to get
\allb \label{eqn-cell-sum-interior}
\sum_{x \in \mcl V\mcl G^\ep(D) } f(w_x^\ep)
&\leq \ep^{-1} \int_{B_{\ep^q}(D) } f(z) \, d\mu_h ( z) +    \ep^{q-\beta} \# \mcl V\mcl G^\ep(D)   \notag \\
&\leq \ep^{-1} \int_{B_{\ep^q}(D) } f(z) \, d\mu_h ( z) +   \ep^{ - 1  + q/2 -\beta}  ,
\alle
where in the last line we used~\eqref{eqn-cell-sum-mass} and the fact that each cell $H_x^\ep$ for $x\in \mcl V\mcl G^\ep(D)$ is contained in $B_\rho$ and has $\mu_h$-mass $\ep$.
This gives the upper bound in~\eqref{eqn-cell-sum} with $\alpha = \min\{q,  q/2 -\beta\}$.

Since $D$ is contained in the union of the cells $H_x^\ep$ for $x\in \mcl V\mcl G^\ep(D)$, we can similarly obtain the lower bound in~\eqref{eqn-cell-sum-interior} by summing~\eqref{eqn-cell-approx-lower} of all $x\in \mcl V\mcl G^\ep(D) $.
\end{proof}

We will also need a variant of Lemma~\ref{lem-cell-sum} where we weight each cell by its degree, whose proof is somewhat more involved.
The extra factor of 6 in this case comes from the fact that $\mcl G^\ep$ is a planar triangulation, so it has average degree 6 (see e.g. e.g. \cite{angel-hyperbolic} where the same factor 6 occurs for similar reasons).

\begin{lem} \label{lem-cell-sum-deg}
There exists $\alpha  =\alpha(\gamma ) > 0$ and $\beta = \beta( \gamma) > 0$ such that with polynomially high probability as $\ep\rta 0$, the following is true.
Let $D \subset B_\rho$, $f : B_{\ep^\alpha}(D) \rta \BB R$, and $w_x^\ep$ for $x\in \mcl G^\ep(D)$ be as in Lemma~\ref{lem-cell-sum}.
Then
\eqb \label{eqn-cell-sum-deg}
 6 \ep^{-1} \int_{D\setminus B_{\ep^\alpha}(\bdy D)} f( z)  \,d\mu_h (  z) -  \ep^{-1+\alpha}
\leq \sum_{  x\in \mcl V\mcl G^\ep(D) } f^\ep( w_x^\ep) \op{deg}^\ep(x)
  \leq  6 \ep^{-1} \int_{B_{\ep^\alpha}(D)} f( z)  \,d\mu_h (  z) +  \ep^{-1+\alpha}    ,
\eqe
simultaneously for every choice of $D$ and $f$ as above.
\end{lem}

We note that Lemma~\ref{lem-cell-sum-deg} implies the following corollary.

\begin{cor} \label{cor-deg-conv}
For $\ep>0$, let $\mu_h^\ep$ be the measure whose restriction to each cell $H_x^\ep$ for $x\in\ep\BB Z$ is equal to $ \op{deg}^\ep(x) \mu_h|_{H_x^\ep}$.
Then as $\ep\rta 0$, we have $\mu_h^\ep \rta 6 \mu_h$ in probability w.r.t.\ the vague topology.
\end{cor}
\begin{proof}
It is immediate from Lemma~\ref{lem-cell-sum-deg} that $\mu_h^\ep|_{B_\rho} \rta 6\mu_h|_{B_\rho}$ weakly.
The corollary follows from this and the scale invariance property of the $\gamma$-quantum cone (Lemma~\ref{lem-cone-scale}).
\end{proof}

The idea of the proof of Lemma~\ref{lem-cell-sum-deg} is as follows.
Recall that a \emph{triangulation with boundary} is a planar map $\mcl T$ whose faces all have three edges, except for one special face called the \emph{external face}.
The \emph{perimeter} $\op{Perim}(\mcl T)$ is the degree of the external face.
Due to the Euler characteristic formula, if $\mcl T$ is a triangulation with boundary whose perimeter is much smaller than its total number of vertices $\#\mcl V(\mcl T)$, then $\#\mcl E(\mcl T)$ is close to $3\#\mcl E(\mcl T)$ (Lemma~\ref{lem-edge-count}), which implies that the sum of the degrees of the vertices of $\mcl T$ is close to $6\#\mcl V(\mcl T)$.

Now assume for simplicity that $\ep^{-1/2}$ is an integer.
We consider for each $y\in \ep^{1/2}\BB Z$ the subgraph $\mcl T_y^\ep$ of $\mcl G^\ep$ induced by the set of vertices $x \in (y-\ep^{1/2}, y] \cap (\ep\BB Z)$. Equivalently, this is the subgraph of $\mcl G^\ep$ induces by the vertices of $\mcl G^\ep$ whose corresponding cells are contained in the cell $H_y^{\ep^{1/2}}$ of $\mcl G^{\ep^{1/2}}$).
Then $\mcl T_y^\ep$ is a triangulation with boundary, having $\ep^{-1/2}$ vertices in total (not just on its boundary).
Furthermore, a basic Brownian motion estimates shows that the perimeter of $\mcl T_y^\ep$ is extremely unlikely to be much larger than $\ep^{-1/4}$ (Lemma~\ref{lem-cell-bdy-count}).
Consequently, the sum of the degrees of the vetices of $\mcl T_y^\ep$ is typically of order $6\ep^{-1/2}$.
Combining this with Lemma~\ref{lem-cell-sum}, applied with $\ep^{1/2}$ in place of $\ep$, then leads to Lemma~\ref{lem-cell-sum-deg}.

\begin{lem} \label{lem-edge-count}
Let $\mcl T$ be a triangulation with boundary, so that the degree of each face of $\mcl T$ except for the external face is equal to 3.
Let $\op{Perim}(\mcl T)$ be the degree of the external face.
Then
\eqb  \label{eqn-edge-count}
\# \mcl E(\mcl T) = 3 \# \mcl V(\mcl T) + 3 - \op{Perim}(\mcl T) .
\eqe
\end{lem}
\begin{proof}
Let $\mcl F(\mcl T)$ be the set of faces of $\mcl T$, including the external face.
By the Euler characteristic formula,
\eqb \label{eqn-map-euler}
\#\mcl E(\mcl T) = \#\mcl V(\mcl T) + \#\mcl F(\mcl T) - 2.
\eqe
Each edge of $\mcl T$ either lies on the boundary of exactly two faces in $\mcl F(\mcl T) $ or lies on the boundary of a single face but has multiplicity 2 there.
Therefore,
\eqb \label{eqn-map-face-count}
\#\mcl E(\mcl T) = \frac12 \sum_{f\in \mcl F(\mcl T)  } \op{deg}(f) = \frac{3}{2} \left( \#\mcl F(\mcl T) - 1 \right)  + \frac12 \op{Perim}(\mcl T) .
\eqe
Re-arranging~\eqref{eqn-map-face-count} shows that
\eqb
\#\mcl F(\mcl T) =  \frac{2}{3} \#\mcl E(\mcl T) + 1 - \frac{1}{3} \op{Perim}(\mcl T) .
\eqe
Plugging this into~\eqref{eqn-map-euler} gives
\eqb \label{eqn-edge-count0}
\#\mcl E(\mcl T) = \#\mcl V(\mcl T) + \frac{2}{3} \#\mcl E(\mcl T) + 1 - \frac{1}{3} \op{Perim}(\mcl T)  ,
\eqe
which gives~\eqref{eqn-edge-count} upon re-arranging.
\end{proof}

We now define the triangulations with boundary to which we will apply Lemma~\ref{lem-edge-count}.
For $\ep > 0$, $n\in\BB N$, and $y\in n\ep\BB Z $, we let $\mcl T_y^{n,\ep}$ be the subgraph of $\mcl G^\ep$ induced by $(y -n\ep , y] \cap (\ep\BB Z)$.
Then $x\in \mcl V\mcl T_y^{n,\ep}$ if and only if the cell $H_x^\ep$ of $\mcl G^\ep$ is contained in the cell $H_y^{n\ep}$ of $\mcl G^{n\ep}$.
We also define
\eqb \label{eqn-two-cell-bdy}
\bdy \mcl T_y^{n,\ep}  := \left\{x\in \mcl V \mcl T_y^{n,\ep}  : \text{$x\sim x'$ in $\mcl G^\ep$ for some $x' \notin \mcl V\mcl T_y^{n,\ep}$}   \right\} .
\eqe
Then $ \mcl T_y^{n,\ep}$ is a planar triangulation with boundary.

\begin{lem} \label{lem-cell-bdy-count}
There are universal constants $c_1,c_2 > 0$ such that the following is true.
Let $\ep > 0$, let $n\in\BB N$, and define $\bdy  \mcl T_y^{n,\ep}$ for $y\in n\ep\BB Z$ as in~\eqref{eqn-two-cell-bdy}.
Then
\eqb \label{eqn-cell-bdy-count}
\BB P\left[ \# \bdy  \mcl T_y^{n,\ep}   > m \right] \leq c_1 e^{-c_2 m / n^{1/2}} , \quad\forall m \in \BB N .
\eqe
\end{lem}

We will prove Lemma~\ref{lem-cell-bdy-count} using the Brownian motion representation of the mated-CRT map.
For this purpose we need the following lemma.

\begin{lem} \label{lem-bm-time}
Let $\mcl B$ be a standard linear Brownian motion.
For $n \in \BB N$, let $\mcl K_n$ be the number of intervals $[k-1,k]$ for $k\in [1,n]_{\BB Z}$ which contain a running minimum of $\mcl B$ (relative to time 0).
Then for $m\in\BB N$,
\eqb \label{eqn-bm-time}
\BB P\left[ \mcl K_n > m    \right] \leq c_1 e^{-c_2 m / n^{ 1/2}}
\eqe
for constants $c_1,c_2 >0$.
\end{lem}
\begin{proof}
Let $\tau_0 = 0$ and inductively let $\tau_m$ for $m\in\BB N$ be the first time after $\tau_{m-1}  +1$ at which $\mcl B$ attains a running minimum.
We observe that
\eqb \label{eqn-bm-time-compare}
\mcl K_n \leq 2 \max\left\{m\in \BB N : \tau_m < n \right\} .
\eqe
By the strong Markov property, the increments $\tau_m - \tau_{m-1}$ are i.i.d.
By a standard Brownian motion estimate,
\eqb
\BB P\left[ \tau_m  -\tau_{m-1} > t \right] \asymp t^{-1/2} , \quad\forall t > 1.
\eqe
Therefore, for $m , n \in\BB N$,
\eqb
\BB P\left[ \tau_m < n \right]
\leq \BB P\left[ \sup_{j \in [1,m]_{\BB Z}} (\tau_j - \tau_{j-1}) < n \right]
\leq \left(1 - c_0 n^{-1/2} \right)^m
\leq c_1 e^{- c_2 m /n^{1/2}} ,
\eqe
for universal constants $c_0,c_1,c_2$. Combining this with~\eqref{eqn-bm-time-compare} gives~\eqref{eqn-bm-time} (with $c_2/2$ in place of $c_2$).
\end{proof}

\begin{proof}[Proof of Lemma~\ref{lem-cell-bdy-count}]
Recall the Brownian motion $  (L,R)$ used to define $\mcl G^\ep$ in~\eqref{edgemap}.
From~\eqref{edgemap}, we get that $\bdy  \mcl T_y^{n,\ep}$ is the same as the set of $x\in (y-n\ep , y] \cap (\ep\BB Z)$ such that either
\eqb  \label{eqn-bdy-adjacency}
\inf_{t\in [x-\ep ,x]} L_t = \inf_{t\in [y-n\ep ,x]} L_t \quad \text{or} \quad \inf_{t\in [x-\ep , x]} L_t = \inf_{t\in [x-\ep ,y]} L_t
\eqe
or the same is true with $R$ in place of $L$.

By Brownian scaling and the translation invariance of the law of $(L,R)$, it suffices to prove~\eqref{eqn-cell-bdy-count} in the case when $\ep = 1$ and $y = n \ep$.
By~\eqref{eqn-bdy-adjacency} and the symmetry of the law of $(L,R)$ under reversing time and/or swapping $L$ and $R$, it suffices to prove that
\eqb
\BB P\left[ \#\left\{x\in [1,n]_{\BB Z} :  \inf_{t\in [x -1 ,x]} L_t = \inf_{t\in [0,x]} L_t \right\} > m \right] \leq c_1 e^{-c_2 m / n^{1/2}}
\eqe
for universal constants $c_1,c_2 > 0$. This is precisely the content of Lemma~\ref{lem-bm-time}.
\end{proof}

\begin{proof}[Proof of Lemma~\ref{lem-cell-sum-deg}]
\noindent\textit{Step 1: regularity event.}
Let $n_\ep := \lfloor \ep^{-1/2} \rfloor$, so that $n_\ep\in\BB N$ and $n_\ep \ep \approx \ep^{1/2}$.
For $y\in n_\ep \ep\BB Z$, we define the triangulation with boundary $ \mcl T_y^\ep := \mcl T^{n_\ep,\ep}$ as in the discussion above Lemma~\ref{lem-cell-bdy-count}.

For $D$ and $f$ as in the lemma statement and $y\in \mcl V\mcl G^{n_\ep\ep}(D)$, let $u_y^{n_\ep \ep} \in H_y^{n_\ep\ep} \cap \ol D$ be chosen in an arbitrary manner (as in Lemma~\ref{lem-cell-sum} with $n_\ep\ep$ in place of $\ep$).
By Lemma~\ref{lem-cell-sum} applied with $n_\ep \ep$ in place of $\ep$, if $\beta > 0$ is chosen to be sufficiently small (depending only on $\gamma$) then there exists $\alpha_0 = \alpha_0(\gamma ) > 0$ such that with polynomially high probability as $\ep\rta 0$, it holds simultaneously for every choice of $D$ and $f$ as in the lemma statement that
\allb \label{eqn-cell-sum-scaled}
 (n_\ep \ep)^{-1} \int_D f( z)  \,d\mu_h (  z) +  (n_\ep \ep)^{-1+\alpha_0}
&\leq \sum_{  y\in \mcl V\mcl G^{n_\ep \ep}(D) } f(u_y^{n_\ep\ep})  \notag\\
&\qquad  \leq  (n_\ep \ep)^{-1} \int_{B_{(n_\ep\ep)^{\alpha_0}}(D)} f( z)  \,d\mu_h (  z) +  (n_\ep \ep)^{-1+\alpha_0}    .
\alle

Fix $q\in \left( 0 , \tfrac{2}{(2+\gamma)^2} \right)$ chosen in a manner depending only on $\gamma$.
By Lemma~\ref{lem-cell-diam} (applied with $\ep^{1/2}$ in place of $\ep$) and our choice of $n_\ep$, it holds with polynomially high probability as $\ep\rta 0$ that
\eqb \label{eqn-cell-sum-diam-scaled}
\op{diam}\left( H_y^{n_\ep \ep} \right) \leq \ep^{q/2} ,\quad \forall y\in \mcl V\mcl G^{n_\ep \ep}(B_\rho) .
\eqe
By Lemma~\ref{lem-cell-bdy-count} (applied with $m = \lfloor  (\log\ep^{-1})^2 \ep^{-1/4} \rfloor$, say) and a union bound, it holds with superpolynomially high probability as $\ep\rta 0$ that, in the notation of that lemma,
\eqb \label{eqn-cell-sum-max-bdy}
\op{Perim}\left( \mcl T_y^\ep \right)  \leq  (\log\ep^{-1})^2 \ep^{-1/4} ,\quad\forall y \in \mcl V\mcl G^{n_\ep\ep}(B_\rho) .
\eqe
By~\cite[Lemma 2.6]{gms-harmonic}, it holds with superpolynomially high probability as $\ep\rta 0$ that
\eqb \label{eqn-deg-sum-max}
\op{deg}^\ep(x) \leq (\log\ep^{-1})^2 ,\quad\forall x\in \mcl V\mcl G^\ep(B_\rho) .
\eqe

Fix $\zeta \in (0,1)$ to be chosen later, in a manner depending only on $q,\gamma$.
By Lemma~\ref{lem-cell-diam} (applied with $n_\ep\ep$ in place of $\ep$), the fact that the cells of $\mcl G^{n_\ep \ep}$ have LQG mass $n_\ep \ep \approx \ep^{1/2}$, and a standard estimate for the $\gamma$-LQG area measure (e.g.,~\cite[Lemma A.3]{ghs-dist-exponent} together with the Chebyshev inequality) it holds with polynomially high probability as $\ep\rta 0$ that
\eqb \label{eqn-cell-count-scaled}
\#\mcl V\mcl G^{n_\ep \ep}(B_{\rho + \ep^{q/2}} ) \leq \ep^{-1/2 - \zeta}
\eqe
Henceforth assume that~\eqref{eqn-cell-sum-scaled} through~\eqref{eqn-cell-count-scaled} hold, which happens with polynomially high probability as $\ep\rta 0$.
We will check that~\eqref{eqn-cell-sum} holds simultaneously for every choice of $D,f$ as in the lemma statement provided $\zeta$ and $\beta$ are chosen to be sufficiently small (depending only on $q,\gamma$).
\medskip

\noindent\textit{Step 2: summing the degrees of vertices of $ \mcl T_y^\ep$.}
For a vertex $x  \in   \mcl T_y^\ep \setminus \bdy\mcl T_y^\ep$, the degree $\op{deg}(x ; \mcl T_y^\ep)$ of $x$ as a vertex of $\mcl T_y^\ep$ is the same as $\op{deg}^\ep(x)$. If $x\in \bdy \mcl T_y^\ep$ we still have $\op{deg}(x ; \mcl T_y^\ep) \leq \deg^\ep(x)$ but the inequality might be strict since $x$ can have edges to vertices in $\mcl V\mcl G^\ep \setminus (y-n_\ep \ep , y]$.
Therefore,
\allb \label{eqn-total-deg0}
\sum_{x \in \mcl V \mcl T_y^\ep  } \op{deg}^\ep(x)
&=   \sum_{x \in \mcl V \mcl T_y^\ep \setminus \mcl V( \bdy \mcl T_y^\ep )  } \op{deg}^\ep(x)
+  \sum_{x \in \mcl V ( \bdy \mcl T_y^\ep )  } \op{deg}^\ep(x) \notag\\
&\in \left[ \sum_{x \in \mcl V \mcl T_y^\ep } \op{deg}(x ; \mcl T_y^\ep) ,
\sum_{x \in \mcl V \mcl T_y^\ep  } \op{deg}(x ; \mcl T_y^\ep) + \sum_{x \in \mcl V( \bdy \mcl T_y^\ep ) } \op{deg}^\ep(x)\right] .
\alle

The sum of the degrees $\op{deg}(x ; \mcl T_y^\ep)$ over all $x\in \mcl V \mcl T_y^\ep $ is precisely twice the number of edges of $\mcl T_y^\ep $. Recall that $ \mcl T_y^\ep$ has $n_\ep$ vertices and is a triangulation with boundary.
By Lemma~\ref{lem-edge-count} followed by~\eqref{eqn-cell-sum-max-bdy},
\eqb \label{eqn-rel-deg-sum}
\sum_{x \in \mcl V \mcl T_y^\ep  } \op{deg}(x ; \mcl T_y^\ep)
= 6 n_\ep  + 6 - 2 \op{Perim}(\mcl T_y^\ep)
\in \left[ 6 n_\ep - 2(\log\ep^{-1})^2 \ep^{-1/4} , 6 n_\ep \right] .
\eqe
By~\eqref{eqn-cell-sum-max-bdy} and~\eqref{eqn-deg-sum-max}, we also have
\eqb \label{eqn-rel-deg-error}
\sum_{x \in \mcl V( \bdy \mcl T_y^\ep ) } \op{deg}^\ep(x)
\leq  (\log\ep^{-1})^4 \ep^{-1/4}  .
\eqe
By plugging~\eqref{eqn-rel-deg-sum} and~\eqref{eqn-rel-deg-error} into~\eqref{eqn-total-deg0}, we get
\eqb \label{eqn-total-deg}
\sum_{x \in \mcl V \mcl T_y^\ep  } \op{deg}^\ep(x)
\in \left[ 6 n_\ep - 2(\log\ep^{-1})^2 \ep^{-1/4} , 6 n_\ep + (\log\ep^{-1})^4 \ep^{-1/4} \right] .
\eqe
\medskip

\noindent\textit{Step 3: summing $f(w_x^\ep)$ over vertices of $\mcl T_y^\ep$.}
If $x  \in \mcl V\mcl T_y^\ep  $, then both $w_x^\ep$ and $u_y^{n_\ep\ep}$ lie in $H_y^{n_\ep\ep}$ so by~\eqref{eqn-cell-sum-diam-scaled} and the $\ep^{-\beta}$-Lipschitz continuity of $f$,
\eqb \label{eqn-big-cell-compare}
| f(w_x^\ep) - f(u_y^{n_\ep\ep}) |  \leq   \ep^{q/2-\beta} .
\eqe
By~\eqref{eqn-big-cell-compare}, then~\eqref{eqn-total-deg} and the fact that $\# \mcl V \mcl T_y^\ep = n_\ep \approx \ep^{-1/2}$,
\allb \label{eqn-deg-sum-chunk}
&\sum_{x \in \mcl V \mcl T_y^\ep  } f(w_x^\ep)  \op{deg}^\ep(x) \notag\\
&\qquad \in \left[  \left( f(u_y^{n_\ep\ep}) - \ep^{q/2-\beta} \right) \sum_{x \in \mcl V \mcl T_y^\ep  } \op{deg}^\ep(x)    ,  \left( f(u_y^{n_\ep\ep}) + \ep^{q/2-\beta} \right) \sum_{x \in \mcl V \mcl T_y^\ep  } \op{deg}^\ep(x)  \right] \notag\\
&\qquad \subset \left[ 6 n_\ep f(u_y^{n_\ep\ep}) - 2 \ep^{q/2 -1/2 -\beta} , 6 n_\ep f(u_y^{n_\ep\ep})  +  2 \ep^{q/2 -1/2 - \beta} \right]  ,
\alle
where in the last line we used that $\|f\|_\infty \leq \ep^{-\beta}$ and we absorbed the polylogarithmic factors from the error terms in~\eqref{eqn-total-deg} into powers of $\ep$ (which we can do since $q < 1/2$).
\medskip

\noindent\textit{Step 4: conclusion.}
Each $x \in \mcl V\mcl G^\ep(D)$ is contained in $\mcl V \mcl T_y^\ep$ for a unique choice of $y\in \mcl V\mcl G^{n_\ep \ep}(D)$.
By summing the upper bound in~\eqref{eqn-deg-sum-chunk} over all $y\in \mcl V\mcl G^{n_\ep \ep}(D)$, we get
\eqb \label{eqn-cell-sum-deg-rescale}
 \sum_{  x\in \mcl V\mcl G^\ep(D) } f(w_x^\ep)  \op{deg}^\ep(x)
 \leq 6 n_\ep \sum_{  y\in \mcl V\mcl G^\ep(D) } f(u_y^{n_\ep\ep})    +    2 \ep^{q/2 -1/2 - \beta} \#\mcl V\mcl G^{n_\ep \ep}(D).
\eqe
From~\eqref{eqn-cell-count-scaled} we get $ \#\mcl V\mcl G^{n_\ep \ep}(D) \leq \ep^{-1/2-\zeta}$.
By applying this and the upper bound in~\eqref{eqn-cell-sum-scaled} to estimate the right side of~\eqref{eqn-cell-sum-deg-rescale}, we obtain
\alb
 \sum_{  x\in \mcl V\mcl G^\ep(D) } f(w_x^\ep)  \op{deg}^\ep(x)
&\leq 6 \ep^{-1} \int_D f( z)  \,d\mu_h (  z)  +   n_\ep^{\alpha_0} \ep^{-1+\alpha_0}  +  O_\ep\left( \ep^{q/2 -1 - \beta - \zeta} \right)     \notag \\
&\leq 6 \ep^{-1} \int_D f( z)  \,d\mu_h (  z)  + O_\ep\left(  \ep^{-1+\alpha}  \right)    \notag \\
\ale
for any choice of $\alpha  < \min\{\alpha_0  ,\, q/2 - \beta - \zeta \}$.
If we choose $\zeta$ and $\beta$ sufficiently small, in a manner depending only on $q,\gamma$, then we can arrange that this minimum is positive.
Thus, after possibly slightly shrinking $\alpha$ we get that the upper bound in~\eqref{eqn-cell-sum-deg} holds for small enough $\ep  >0$.

We similarly obtain the lower bound in~\eqref{eqn-cell-sum-deg} using the lower bounds in each of~\eqref{eqn-cell-sum-scaled} and~\eqref{eqn-total-deg}.
\end{proof}

\subsection{Space-filling SLE loops}
\label{sec-space-filling-sle}

Let $\kappa > 4$, let $D\subset\BB C$ be a simply connected domain, and let $a \in\bdy D$ be a prime end. The purpose of this section is to explain one possible definition of the counterclockwise space-filling SLE loop $\eta$ in $D$ based at $a$, and to show that this space-filling SLE loop can be obtained as the limit (for a rather strong notion of convergence) of space-filling SLE from $a$ to $b$ in $D$ as $b\rta a$ from the clockwise direction.
The proofs in this section are based on the theory of imaginary geometry from~\cite{ig1,ig2,ig3,ig4}.
We will mostly focus on the case when $(D,a) = (\BB H , \infty)$; the space-filling SLE loop for other choices of $(D,a)$ will be defined via conformal mapping.

Fix $\kappa > 4$. As in~\cite{ig1,ig2,ig3,ig4}, we define the constant
\eqb
\lambda' := \frac{\pi}{\sqrt\kappa} .
\eqe
Note that our $\kappa$ is called $\kappa'$ in~\cite{ig1,ig2,ig3,ig4}.

Let $h$ be a zero-boundary GFF on $\BB H$ plus the constant $-\lambda'$. For $z\in\BB Q^2\cap\BB H$, let $\eta^{z,L}$ (resp.\ $\eta^{z,R}$) be the flow line of $h$ started from $z$ with angle $\pi/2$ (resp.\ $-\pi/2$) and stopped at the first time it hits $\BB R$, as defined in~\cite[Theorem 1.1]{ig4}. The law of $\eta^{z,L}$ is locally absolutely continuous w.r.t.\ that of a SLE$_{16/\kappa}(16/\kappa-2)$ curve in any interval of time before it hits $\BB R$. By~\cite[Theorem 1.9]{ig4}, for distinct $z,w\in \BB Q^2 \cap \BB H$, a.s.\ the curves $\eta^{z,L}$ and $\eta^{w,L}$ (resp.\ $\eta^{z,L}$ and $\eta^{w,R}$) a.s.\ merge into each other after some finite time.
We define a total order $\preceq$ on $\BB Q^2\cap \BB H$ by declaring that $z\preceq w$ if and only if $z$ lies in a connected component of $\BB H\setminus (\eta^{w,L} \cup \eta^{w,R})$ whose boundary is traced by the left side of $\eta^{w,L}$ and the right side of $\eta^{w,R}$.

For $a  < 0$, let $\eta_a$ be a chordal space-filling SLE$_\kappa$ from $a$ to $\infty$.
Let $\preceq_a$ be the total order on $\BB Q^2 \cap \BB H$ induced by $\eta_a$, i.e., $z\preceq_a w$ if and only if $\eta_a$ hits $z$ before $w$.
Due to the construction of space-filling SLE$_\kappa$ in~\cite[Section 1.2.3]{ig4}, the ordering $\preceq_a$ admits the following equivalent description.
Let $h_a$ be a GFF on $\BB H$ with boundary data $ \lambda'$ on $(-\infty,a]$ and $-\lambda'$ on $[a,\infty)$.
Then we can realize $\eta_a$ as the space-filling SLE$_\kappa$ counterflow line of $h_a$ from $a$ to $\infty$.
Furthermore, if we define the flow lines $\eta_a^{z,L}$ and $\eta_a^{z,R}$ for $z\in\BB Q^2\cap\BB H$ in the same manner as $\eta^{z,L}$ and $\eta^{z,R}$ above but with $h_a$ in place of $h$, then for $z\in \BB Q^2\cap \BB H$ the curves $\eta_a^{z,L}$ and $\eta_a^{z,R}$ are the left and right outer boundaries of $\eta_a$ stopped at the first time it hits $z$. We note that $\eta_a^{z,L}$ (resp.\ $\eta_a^{z,R}$) stops upon hitting $(-\infty,a]$ (resp.\ $[a,\infty)$).
Consequently, the ordering $\preceq_a$ can be defined in the same manner as the ordering $\preceq$ above but with $h_a$ in place of $h$.

\begin{prop} \label{prop-sle-loop}
There is a unique continuous curve $\eta : \BB R\rta \ol{\BB H}$ from $\infty$ to $\infty$ which satisfies $\eta(0) = 0$ and $\op{area} \eta([s,t]) =  t-s$ for each $s<t$, hits the points of $\BB Q^2\cap\BB H$ in the order specified by $\preceq$, and fills in all of $\BB H$.
The law of $\eta$ is scale and translation invariant in the sense that for $C >0$ and $b\in\BB R$, we have $C \eta  +b \eqD \eta$ viewed as curves modulo time parametrization.
Moreover, for each fixed $r>0$ it holds that $\preceq_a|_{B_r(0) \cap \BB Q^2\cap \BB H}$ converges to $\preceq|_{B_r(0) \cap \BB Q^2\cap \BB H}$ in the total variation sense as $a \rta - \infty$.
\end{prop}

\begin{defn} \label{def-sle-loop}
The curve $\eta$ of Proposition~\ref{prop-sle-loop} is defined to be the \emph{counterclockwise space-filling SLE$_\kappa$ loop based at $\infty$ in $\BB H$}.
The \emph{clockwise space-filling SLE$_\kappa$ loop based at $\infty$ in $\BB H$} is the curve $\wt\eta$ which is the image of $\eta$ under the anticonformal map $x+iy \mapsto -x + i y$.
For a domain $D$ and a prime end $b\in\bdy D$, we define the \emph{counterclockwise (resp.\ clockwise) space-filling SLE$_\kappa$ loop based at $b$ in $D$} to be the image of $\eta$ (resp.\ $\wt\eta$) under a conformal map $\BB H\rta D$ which takes $\infty$ to $b$, viewed as a curve modulo time parametrization (the law of this curve does not depend on the choice of conformal map due to the scale / translation invariance property of $\eta$).
\end{defn}

It is possible to establish the existence of the curve $\eta$ of Proposition~\ref{prop-sle-loop} by repeating essentially the same arguments used in~\cite{ig4} to construct space-filling SLE with distinct starting and ending points. However, in order to make our proof easier to follow and to get the total variation convergence statement, we will instead deduce Proposition~\ref{prop-sle-loop} from existing results in the literature. The key idea of the proof is that when $a$ is very negative, the restrictions of $h$ and $h_a$ to a fixed neighborhood of 0 in $\BB H$ are close in the total variation sense, and moreover the orderings $\preceq$ and $\preceq_a$ are in a certain sense locally determined by $h$ and $h_a$, respectively.

Let us now specify a precise version of the locality property of $\preceq_a$.
For a radius $\rho >0$, let $T_a^\rho$ (resp.\ $S_a^\rho$) be the first time that $\eta_a$ enters $B_\rho(0)\cap \BB H$ before hitting $0$ (resp.\ the last time that $\eta_a$ enters $B_\rho(0)\cap\BB H$ after hitting 0).
We define
\eqb
 \eta_a^\rho := \eta_a |_{[ T_a^\rho , S_a^\rho]} \quad \text{and} \quad K_a^\rho := \eta_a([T_a^\rho ,S_a^\rho]) .
\eqe
We now describe $K_a^\rho$ in terms of $\preceq_a$.

\begin{lem} \label{lem-sle-segment}
Let $\eta_a^{0,L}$ be the flow line of $h_a$ started from 0 with angle $\pi/2$.
If $z\in B_\rho(0)\cap \BB Q^2\cap \BB H$, then $z\in K_a^\rho$ if and only if $\eta_a^{z,L}$ merges into $\eta_a^{0,L}$ before leaving $B_\rho(0)$ and $\eta_a^{z,R}$ hits $[a,\infty)$ before leaving $B_\rho(0)$.
\end{lem}
\begin{proof}
This follows since $\eta_a^{z,L}$ and $\eta_a^{z,R}$ are the left/right outer boundaries of $\eta_a$ at the (a.s.\ unique) time when it hits $a$ and $\eta_a^{0,L}$ is the outer boundary of $\eta_a$ at the (a.s.\ unique) time when it hits 0.
\end{proof}

As a consequence of Lemma~\ref{lem-sle-segment} and the fact that the flow lines $\eta_a^{z,L}$ and $\eta_a^{z,R}$ are locally determined by $h$ (c.f.\ the proof of~\cite[Lemma 2.4]{gms-harmonic}), it follows that $\preceq_a|_{K_a^\rho\cap \BB Q^2}$ is a.s.\ determined by $h_a|_{B_\rho(0)\cap\BB H}$.

Let $\eta^{0,L}$ be the flow line of $h$ started from 0 with angle $\pi/2$.
Let $K^\rho$ be the closure of the set of $z\in B_\rho(0)\cap \BB Q^2\cap \BB H$ such that  $z\in K_a^\rho$ if and only if $\eta^{z,L}$ merges into $\eta^{0,L}$ before exiting $B_\rho(0)$ and $\eta^{z,R}$ hits $\BB R$ before leaving $B_\rho(0)$.
That is, $K^\rho$ is defined in the same way as $K_a^\rho$ but with $h$ in place of $h_a$.
By definition, $K^\rho$ is a.s.\ determined by $h|_{B_\rho(0)\cap\BB H}$.

A basic Radon-Nikodym estimate for the GFF (see, e.g., the proof of~\cite[Proposition 3.4]{ig1} shows that the following is true for each $\rho>0$.
\begin{itemize}
\item If $a  < -\rho$, then the laws of $h_a|_{B_\rho(0)\cap \BB H}$ and $h |_{B_\rho(0)\cap \BB H}$ are mutually absolutely continuous.
\item The total variation distance between the laws of $h_a|_{B_\rho(0)\cap \BB H}$ and $h |_{B_\rho(0)\cap \BB H}$ tends to zero as $a\rta - \infty$.
\end{itemize}
In particular,
\eqb \label{eqn-total-variation}
\left(K_a^\rho , \preceq_a|_{K_a^\rho\cap \BB Q^2} \right) \rta \left(K^\rho , \preceq|_{K^\rho\cap \BB Q^2}\right) \quad \text{in total variation as $a\rta-\infty$}.
\eqe

\begin{lem} \label{lem-curve-swallow}
Let $r > 0$ and let $p\in (0,1)$. There exists $\rho = \rho(r,p)  >r$ such that
\eqb \label{eqn-end-curve-contain}
\BB P\left[ B_r(0)\cap \BB H\subset K^\rho \right] \geq p \quad \text{and}
\eqe
\eqb \label{eqn-curve-swallow}
\liminf_{a\rta -\infty} \BB P\left[ B_r(0)\cap \BB H\subset K_a^\rho \right] \geq p.
\eqe
\end{lem}
\begin{proof}
Due to the total variation convergence $K_a^\rho\rta K^\rho$, it suffices to show that there exists $\rho = \rho(r,p) > r$ such that~\eqref{eqn-end-curve-contain} holds.
To prove this, we will first show that $K^1$ contains some neighborhood of 0 with probability at least $p$ then use a scaling argument.
Choose $a_0 \in \BB N$ with $a_0  < -1$. Since the laws of $h_{a_0}|_{B_\rho(0)\cap\BB H}$ and $h |_{B_\rho(0)\cap\BB H}$ are mutually absolutely continuous, also the laws of $K_{a_0}^1$ and $K^1$ are mutually absolutely continuous.
Since $\eta_{a_0}$ is a chordal space-filling SLE$_\kappa$, a.s.\ $K_{a_0}^1$ contains some neighborhood of 0 in $\BB H$.
Hence for any $q\in (0,1)$ there exists $\ep > 0$ such that $\BB P[B_\ep(0)\cap\BB H\subset \eta_{a_0}^1 ] \geq q$.
By taking $q$ to be sufficiently close to 1 (depending on $p$) we infer that there exists $\ep > 0$ such that
\eqb \label{eqn-end-curve-contain-ep}
 \BB P[B_\ep(0)\cap\BB H\subset K^1 ] \geq  p .
\eqe
Due to the scale invariance of the law of $h $, we have $\rho K^1 \eqD K^\rho$ for each rational $\rho >0$.
Therefore~\eqref{eqn-end-curve-contain} for $\rho $ a rational number slightly larger than $ r/\ep$ follows from~\eqref{eqn-end-curve-contain-ep}.
\end{proof}

\begin{proof}[Proof of Proposition~\ref{prop-sle-loop}]
Each $\eta_a$ is a continuous curve which satisfies the conditions of the proposition statement with $h_a$ in place of $h$.
Due to the total variation convergence~\eqref{eqn-total-variation}, we infer that for each $\rho>0$, there is a unique curve $\ol\eta^\rho$ mapping some time interval $[T^\rho , S^\rho]$ into $K^\rho$ which satisfies $\eta(0) = 0$ and $\op{area} \eta([s,t]) = t-s$ for each $s<t$ and such that $\ol\eta^\rho$ hits the points of $ K^\rho \cap \BB Q^2$ in the order specified by $\preceq$.
By~\eqref{eqn-end-curve-contain} of Lemma~\ref{lem-curve-swallow} and since $r$ can be made arbitrarily large, it follows that the sets $K^\rho$ a.s.\ increase to all of $\BB C$ as $\rho\rta\infty$. In particular, a.s.\ $S^\rho \rta -\infty$ and $T^\rho \rta \infty$ as $\rho\rta \infty$.
We therefore get a random continuous curve $\eta$ from $-\infty$ to $\infty$ which satisfies the conditions in the first assertion of the lemma.

Furthermore, by the scale and translation invariance of the law of $h$ we get that the law of the ordering $\preceq$ is invariant under scaling by rational scaling factors and translation by rational elements of $\BB R$ (we need to restrict to rationals at this point since we defined our ordering using $\BB Q^2$).
This shows that the law of $\eta$, viewed modulo time parametrization, is invariant under rational scalings and translations.

By~\eqref{eqn-curve-swallow} of Lemma~\ref{lem-curve-swallow} and~\eqref{eqn-total-variation}, we get that $\preceq_a|_{B_r(0) \cap \BB Q^2\cap \BB H}$ converges to $\preceq|_{B_r(0) \cap \BB Q^2\cap \BB H}$ in the total variation sense as $a \rta - \infty$.

It remains to show that the law of $\eta$, viewed modulo time parametrization, is in fact invariant under \emph{all} scalings and translations, not just rational ones. To prove this, we observe that if $C > 0$, then the argument above works equally well with $C \BB Q^2 $ in place of $\BB Q^2$. Let $\eta^C$ be the curve obtained in the same way as $\eta$ but with $C\BB Q^2 $ in place of $\BB Q^2$.
Then the law of $\eta^C$, viewed modulo time parametrization, is invariant under scalings and translations by rational multiples of $C$.
On the other hand, replacing $\BB Q^2$ by $C\BB Q^2  $ does not change the curves $\eta_a$. From the total variation convergence $\preceq_a^C|_{  C\BB Q^2 \cap B_r(0) \cap \BB H} \rta \preceq^C|_{  C\BB Q^2 \cap B_r(0) \cap \BB H}$ (in obvious notation) for each $r>0$ we therefore get that $\eta^C \eqD \eta$. Since $C$ is arbitrary we now conclude the proof.
\end{proof}

\bibliography{cibib,addbib}

\newcommand{\etalchar}[1]{$^{#1}$}
\def\cprime{$'$}
\begin{thebibliography}{KMSW19}

\bibitem[AAJ{\etalchar{+}}99]{Ambjorn_diffusion}
J.~Ambj{\o}rn, K.~N. Anagnostopoulos, L.~Jensen, T.~Ichihara, and Y.~Watabiki.
\newblock Quantum geometry and diffusion.
\newblock {\em Journal of High Energy Physics}, 1998(11):022, 1999.

\bibitem[AB99]{ab-random-curves}
M.~Aizenman and A.~Burchard.
\newblock H\"older regularity and dimension bounds for random curves.
\newblock {\em Duke Math. J.}, 99(3):419--453, 1999, \arxiv{math/9801027}.
  \MR{1712629}

\bibitem[AG21]{ag-disk}
M.~Ang and E.~Gwynne.
\newblock Liouville quantum gravity surfaces with boundary as matings of trees.
\newblock {\em Ann. Inst. Henri Poincar\'{e} Probab. Stat.}, 57(1):1--53, 2021,
  \arxiv{1903.09120}. \MR{4255166}

\bibitem[AHNR16]{angel-hyperbolic}
O.~Angel, T.~Hutchcroft, A.~Nachmias, and G.~Ray.
\newblock Unimodular hyperbolic triangulations: Circle packing and random walk.
\newblock {\em Inventiones mathematicae}, 206(1):229--268, 2016,
  \arxiv{1501.04677}.

\bibitem[AK16]{andres-heat-kernel}
S.~Andres and N.~Kajino.
\newblock Continuity and estimates of the {L}iouville heat kernel with
  applications to spectral dimensions.
\newblock {\em Probab. Theory Related Fields}, 166(3-4):713--752, 2016,
  \arxiv{1407.3240}. \MR{3568038}

\bibitem[ANR{\etalchar{+}}98]{Ambjorn_spectraldimension}
J.~Ambj{\o}rn, J.~L. Nielsen, J.~Rolf, D.~Boulatov, and Y.~Watabiki.
\newblock The spectral dimension of 2d quantum gravity.
\newblock {\em Journal of High Energy Physics}, 1998(02):010, 1998.

\bibitem[AS03]{angel-schramm-uipt}
O.~Angel and O.~Schramm.
\newblock Uniform infinite planar triangulations.
\newblock {\em Comm. Math. Phys.}, 241(2-3):191--213, 2003. \MR{2013797
  (2005b:60021)}

\bibitem[BA{\v{C}}07]{BC07}
G.~Ben~Arous and J.~{\v{C}}ern{\`y}.
\newblock Scaling limit for trap models on $\mathbb{Z}^d$.
\newblock {\em The Annals of Probability}, 35(6):2356--2384, 2007.

\bibitem[BC13]{benjamini-curien-uipq-walk}
I.~Benjamini and N.~Curien.
\newblock Simple random walk on the uniform infinite planar quadrangulation:
  subdiffusivity via pioneer points.
\newblock {\em Geom. Funct. Anal.}, 23(2):501--531, 2013, \arxiv{1202.5454}.
  \MR{3053754}

\bibitem[Ber07a]{bernardi-dfs-bijection}
O.~Bernardi.
\newblock Bijective counting of {K}reweras walks and loopless triangulations.
\newblock {\em J. Combin. Theory Ser. A}, 114(5):931--956, 2007.

\bibitem[Ber07b]{bernardi-maps}
O.~Bernardi.
\newblock Bijective counting of tree-rooted maps and shuffles of parenthesis
  systems.
\newblock {\em Electron. J. Combin.}, 14(1):Research Paper 9, 36 pp.
  (electronic), 2007, \arxiv{math/0601684}. \MR{2285813 (2007m:05125)}

\bibitem[Ber15]{berestycki-lbm}
N.~Berestycki.
\newblock Diffusion in planar {L}iouville quantum gravity.
\newblock {\em Ann. Inst. Henri Poincar\'e Probab. Stat.}, 51(3):947--964,
  2015, \arxiv{1301.3356}. \MR{3365969}

\bibitem[Ber17]{berestycki-gmt-elementary}
N.~Berestycki.
\newblock An elementary approach to {G}aussian multiplicative chaos.
\newblock {\em Electron. Commun. Probab.}, 22:Paper No. 27, 12, 2017,
  \arxiv{1506.09113}. \MR{3652040}

\bibitem[BGRV16]{grv-kpz}
N.~Berestycki, C.~Garban, R.~Rhodes, and V.~Vargas.
\newblock K{PZ} formula derived from {L}iouville heat kernel.
\newblock {\em J. Lond. Math. Soc. (2)}, 94(1):186--208, 2016,
  \arxiv{1406.7280}. \MR{3532169}

\bibitem[BHS18]{bhs-site-perc}
O.~{Bernardi}, N.~{Holden}, and X.~{Sun}.
\newblock {Percolation on triangulations: a bijective path to Liouville quantum
  gravity}.
\newblock {\em ArXiv e-prints}, July 2018, \arxiv{1807.01684}.

\bibitem[Bil13]{Billingsley}
P.~Billingsley.
\newblock {\em Convergence of probability measures}.
\newblock John Wiley \& Sons, 2013.

\bibitem[BP]{bp-lqg-notes}
N.~{Berestycki} and E.~{Powell}.
\newblock {G}aussian {f}ree {f}ield, {L}iouville {q}uantum {g}ravity, and
  {G}aussian multiplicative chaos.
\newblock {A}vailable at
  \url{https://homepage.univie.ac.at/nathanael.berestycki/Articles/master.pdf}.

\bibitem[BS01]{benjamini-schramm-topology}
I.~Benjamini and O.~Schramm.
\newblock Recurrence of distributional limits of finite planar graphs.
\newblock {\em Electron. J. Probab.}, 6:no. 23, 13 pp. (electronic), 2001,
  \arxiv{0011019}. \MR{1873300 (2002m:82025)}

\bibitem[CHK17]{chk-time-changes}
D.~Croydon, B.~Hambly, and T.~Kumagai.
\newblock Time-changes of stochastic processes associated with resistance
  forms.
\newblock {\em Electron. J. Probab.}, 22:Paper No. 82, 41, 2017,
  \arxiv{1609.02120}. \MR{3718710}

\bibitem[Cur16]{curien-peeling-notes}
N.~Curien.
\newblock Peeling random planar maps. {N}otes du cours {P}eccot.
\newblock Available at
  \url{https://www.math.u-psud.fr/~curien/cours/peccot.pdf}, 2016.

\bibitem[Dav88]{david-conformal-gauge}
F.~David.
\newblock Conformal field theories coupled to {2-D} gravity in the conformal
  gauge.
\newblock {\em {M}od. {P}hys. {L}ett. {A}}, 3(17), 1988.

\bibitem[DG18]{dg-lqg-dim}
J.~{Ding} and E.~{Gwynne}.
\newblock {The fractal dimension of {L}iouville quantum gravity: universality,
  monotonicity, and bounds}.
\newblock {\em {C}ommunications in {M}athematical {P}hysics}, 374:1877--1934,
  2018, \arxiv{1807.01072}.

\bibitem[DK89]{dk-qg}
J.~Distler and H.~Kawai.
\newblock Conformal field theory and {2D} quantum gravity.
\newblock {\em {N}ucl.{P}hys. {B}}, 321(2), 1989.

\bibitem[DMS21]{wedges}
B.~Duplantier, J.~Miller, and S.~Sheffield.
\newblock Liouville quantum gravity as a mating of trees.
\newblock {\em Ast\'{e}risque}, (427):viii+257, 2021, \arxiv{1409.7055}.
  \MR{4340069}

\bibitem[DS11]{shef-kpz}
B.~Duplantier and S.~Sheffield.
\newblock Liouville quantum gravity and {KPZ}.
\newblock {\em Invent. Math.}, 185(2):333--393, 2011, \arxiv{1206.0212}.
  \MR{2819163 (2012f:81251)}

\bibitem[DZZ19]{dzz-heat-kernel}
J.~Ding, O.~Zeitouni, and F.~Zhang.
\newblock Heat kernel for {L}iouville {B}rownian motion and {L}iouville graph
  distance.
\newblock {\em Comm. Math. Phys.}, 371(2):561--618, 2019, \arxiv{1807.00422}.
  \MR{4019914}

\bibitem[FIN02]{FIN}
L.~R.~G. Fontes, M.~Isopi, and C.~M. Newman.
\newblock Random walks with strongly inhomogeneous rates and singular
  diffusions: convergence, localization and aging in one dimension.
\newblock {\em The Annals of Probability}, 30(2):579--604, 2002.

\bibitem[FOT11]{fot-dirichlet-forms}
M.~Fukushima, Y.~Oshima, and M.~Takeda.
\newblock {\em Dirichlet forms and symmetric {M}arkov processes}, volume~19 of
  {\em De Gruyter Studies in Mathematics}.
\newblock Walter de Gruyter \& Co., Berlin, extended edition, 2011.
  \MR{2778606}

\bibitem[GGN13]{gn-recurrence}
O.~Gurel-Gurevich and A.~Nachmias.
\newblock Recurrence of planar graph limits.
\newblock {\em Ann. of Math. (2)}, 177(2):761--781, 2013, \arxiv{1206.0707}.
  \MR{3010812}

\bibitem[GH20]{gh-displacement}
E.~Gwynne and T.~Hutchcroft.
\newblock Anomalous diffusion of random walk on random planar maps.
\newblock {\em Probab. Theory Related Fields}, 178(1-2):567--611, 2020,
  \arxiv{1807.01512}. \MR{4146545}

\bibitem[GHM20]{ghm-kpz}
E.~Gwynne, N.~Holden, and J.~Miller.
\newblock An almost sure {KPZ} relation for {SLE} and {B}rownian motion.
\newblock {\em Ann. Probab.}, 48(2):527--573, 2020, \arxiv{1512.01223}.
  \MR{4089487}

\bibitem[GHS19a]{ghs-dist-exponent}
E.~{Gwynne}, N.~{Holden}, and X.~{Sun}.
\newblock {A distance exponent for Liouville quantum gravity}.
\newblock {\em {Probability Theory and Related Fields}}, 173(3):931--997, 2019,
  \arxiv{1606.01214}.

\bibitem[GHS19b]{ghs-mating-survey}
E.~{Gwynne}, N.~{Holden}, and X.~{Sun}.
\newblock {Mating of trees for random planar maps and Liouville quantum
  gravity: a survey}.
\newblock {\em {Panoramas et Syntheses}}, to appear, 2019, \arxiv{1910.04713}.

\bibitem[GHS20]{ghs-map-dist}
E.~Gwynne, N.~Holden, and X.~Sun.
\newblock A mating-of-trees approach for graph distances in random planar maps.
\newblock {\em Probab. Theory Related Fields}, 177(3-4):1043--1102, 2020,
  \arxiv{1711.00723}. \MR{4126936}

\bibitem[GM21]{gm-spec-dim}
E.~Gwynne and J.~Miller.
\newblock Random walk on random planar maps: Spectral dimension, resistance and
  displacement.
\newblock {\em Ann. Probab.}, 49(3):1097--1128, 2021, \arxiv{1711.00836}.

\bibitem[GMS18]{gms-random-walk}
E.~{Gwynne}, J.~{Miller}, and S.~{Sheffield}.
\newblock {An invariance principle for ergodic scale-free random environments}.
\newblock {\em Acta Mathematica}, to appear, 2018, \arxiv{1807.07515}.

\bibitem[GMS19]{gms-harmonic}
E.~{Gwynne}, J.~{Miller}, and S.~{Sheffield}.
\newblock {Harmonic functions on mated-{CRT} maps}.
\newblock {\em Electron. J. Probab.}, 24:no. 58, 55, 2019, \arxiv{1807.07511}.

\bibitem[GMS21]{gms-tutte}
E.~Gwynne, J.~Miller, and S.~Sheffield.
\newblock The {T}utte embedding of the mated-crt map converges to {L}iouville
  quantum gravity.
\newblock {\em Ann. Probab.}, 49(4):1677--1717, 2021, \arxiv{1705.11161}.

\bibitem[GP20]{gp-sle-bubbles}
E.~Gwynne and J.~Pfeffer.
\newblock Connectivity properties of the adjacency graph of
  $\mathrm{SLE}_{\kappa}$ bubbles for $\kappa\in(4,8)$.
\newblock {\em Ann. Probab.}, 48(3):1495--1519, 2020, \arxiv{1803.04923}.

\bibitem[Gri99]{grigoryan-background}
A.~Grigor{'}yan.
\newblock Analytic and geometric background of recurrence and non-explosion of
  the {B}rownian motion on {R}iemannian manifolds.
\newblock {\em Bull. Amer. Math. Soc. (N.S.)}, 36(2):135--249, 1999.
  \MR{1659871}

\bibitem[GRV14]{grv-heat-kernel}
C.~Garban, R.~Rhodes, and V.~Vargas.
\newblock On the heat kernel and the {D}irichlet form of {L}iouville {B}rownian
  motion.
\newblock {\em Electron. J. Probab.}, 19:no. 96, 25, 2014, \arxiv{1302.6050}.
  \MR{3272329}

\bibitem[GRV16]{grv-lbm}
C.~Garban, R.~Rhodes, and V.~Vargas.
\newblock Liouville {B}rownian motion.
\newblock {\em Ann. Probab.}, 44(4):3076--3110, 2016, \arxiv{1301.2876}.
  \MR{3531686}

\bibitem[Jac18]{jackson-lbm-thick-pts}
H.~Jackson.
\newblock Liouville {B}rownian motion and thick points of the {G}aussian free
  field.
\newblock {\em Ann. Inst. Henri Poincar\'{e} Probab. Stat.}, 54(1):249--279,
  2018, \arxiv{1412.1705}. \MR{3765889}

\bibitem[Kah85]{kahane}
J.-P. Kahane.
\newblock Sur le chaos multiplicatif.
\newblock {\em Ann. Sci. Math. Qu\'ebec}, 9(2):105--150, 1985. \MR{829798
  (88h:60099a)}

\bibitem[Kal02]{kallenberg}
O.~Kallenberg.
\newblock {\em Foundations of modern probability}.
\newblock Probability and its Applications (New York). Springer-Verlag, New
  York, second edition, 2002. \MR{1876169}

\bibitem[KMSW19]{kmsw-bipolar}
R.~Kenyon, J.~Miller, S.~Sheffield, and D.~B. Wilson.
\newblock Bipolar orientations on planar maps and {${\mathrm SLE}_{12}$}.
\newblock {\em Ann. Probab.}, 47(3):1240--1269, 2019, \arxiv{1511.04068}.
  \MR{3945746}

\bibitem[Law05]{lawler-book}
G.~F. Lawler.
\newblock {\em Conformally invariant processes in the plane}, volume 114 of
  {\em Mathematical Surveys and Monographs}.
\newblock American Mathematical Society, Providence, RI, 2005. \MR{2129588
  (2006i:60003)}

\bibitem[{Le }13]{legall-uniqueness}
J.-F. {Le Gall}.
\newblock Uniqueness and universality of the {B}rownian map.
\newblock {\em Ann. Probab.}, 41(4):2880--2960, 2013, \arxiv{1105.4842}.
  \MR{3112934}

\bibitem[LG14]{legall-sphere-survey}
J.-F. Le~Gall.
\newblock Random geometry on the sphere.
\newblock In {\em Proceedings of the {I}nternational {C}ongress of
  {M}athematicians---{S}eoul 2014. {V}ol. 1}, pages 421--442. Kyung Moon Sa,
  Seoul, 2014, \arxiv{1403.7943}. \MR{3728478}

\bibitem[LP16]{lyons-peres}
R.~Lyons and Y.~Peres.
\newblock {\em Probability on Trees and Networks}, volume~42 of {\em Cambridge
  Series in Statistical and Probabilistic Mathematics}.
\newblock Cambridge University Press, New York, 2016.
\newblock Available at \url{http://pages.iu.edu/~rdlyons/}. \MR{3616205}

\bibitem[LPW09]{markov-mixing}
D.~A. Levin, Y.~Peres, and E.~L. Wilmer.
\newblock {\em Markov chains and mixing times}.
\newblock American Mathematical Society, Providence, RI, 2009.
\newblock With a chapter by James G. Propp and David B. Wilson. \MR{2466937
  (2010c:60209)}

\bibitem[Mie13]{miermont-brownian-map}
G.~Miermont.
\newblock The {B}rownian map is the scaling limit of uniform random plane
  quadrangulations.
\newblock {\em Acta Math.}, 210(2):319--401, 2013, \arxiv{1104.1606}.
  \MR{3070569}

\bibitem[MRVZ16]{mrvz-heat-kernel}
P.~Maillard, R.~Rhodes, V.~Vargas, and O.~Zeitouni.
\newblock Liouville heat kernel: regularity and bounds.
\newblock {\em Ann. Inst. Henri Poincar\'e Probab. Stat.}, 52(3):1281--1320,
  2016, \arxiv{1406.0491}. \MR{3531710}

\bibitem[MS16a]{ig1}
J.~Miller and S.~Sheffield.
\newblock Imaginary geometry {I}: interacting {SLE}s.
\newblock {\em Probab. Theory Related Fields}, 164(3-4):553--705, 2016,
  \arxiv{1201.1496}. \MR{3477777}

\bibitem[MS16b]{ig2}
J.~Miller and S.~Sheffield.
\newblock Imaginary geometry {II}: {R}eversibility of
  {$\operatorname{SLE}_\kappa(\rho_1;\rho_2)$} for {$\kappa\in(0,4)$}.
\newblock {\em Ann. Probab.}, 44(3):1647--1722, 2016, \arxiv{1201.1497}.
  \MR{3502592}

\bibitem[MS16c]{ig3}
J.~Miller and S.~Sheffield.
\newblock Imaginary geometry {III}: reversibility of {$\mathrm{SLE}_\kappa$}
  for {$\kappa\in(4,8)$}.
\newblock {\em Ann. of Math. (2)}, 184(2):455--486, 2016, \arxiv{1201.1498}.
  \MR{3548530}

\bibitem[MS17]{ig4}
J.~Miller and S.~Sheffield.
\newblock Imaginary geometry {IV}: interior rays, whole-plane reversibility,
  and space-filling trees.
\newblock {\em Probab. Theory Related Fields}, 169(3-4):729--869, 2017,
  \arxiv{1302.4738}. \MR{3719057}

\bibitem[MS19]{sphere-constructions}
J.~Miller and S.~Sheffield.
\newblock Liouville quantum gravity spheres as matings of finite-diameter
  trees.
\newblock {\em Ann. Inst. Henri Poincar\'{e} Probab. Stat.}, 55(3):1712--1750,
  2019, \arxiv{1506.03804}. \MR{4010949}

\bibitem[Mul67]{mullin-maps}
R.~C. Mullin.
\newblock On the enumeration of tree-rooted maps.
\newblock {\em Canad. J. Math.}, 19:174--183, 1967. \MR{0205882 (34 \#5708)}

\bibitem[Nac20]{Nachmias}
A.~Nachmias.
\newblock {\em Planar Maps, Random Walks and Circle Packing: {\'E}cole
  D'{\'E}t{\'e} de Probabilit{\'e}s de Saint-Flour XLVIII-2018}.
\newblock Springer Nature, 2020.

\bibitem[Pom92]{pom-book}
C.~Pommerenke.
\newblock {\em Boundary behaviour of conformal maps}, volume 299 of {\em
  Grundlehren der Mathematischen Wissenschaften [Fundamental Principles of
  Mathematical Sciences]}.
\newblock Springer-Verlag, Berlin, 1992. \MR{1217706 (95b:30008)}

\bibitem[RS05]{schramm-sle}
S.~Rohde and O.~Schramm.
\newblock Basic properties of {SLE}.
\newblock {\em Ann. of Math. (2)}, 161(2):883--924, 2005, \arxiv{math/0106036}.
  \MR{2153402 (2006f:60093)}

\bibitem[RV11]{rhodes-vargas-log-kpz}
R.~Rhodes and V.~Vargas.
\newblock K{PZ} formula for log-infinitely divisible multifractal random
  measures.
\newblock {\em ESAIM Probab. Stat.}, 15:358--371, 2011, \arxiv{0807.1036}.
  \MR{2870520}

\bibitem[RV14a]{rhodes-vargas-review}
R.~Rhodes and V.~Vargas.
\newblock Gaussian multiplicative chaos and applications: {A} review.
\newblock {\em Probab. Surv.}, 11:315--392, 2014, \arxiv{1305.6221}.
  \MR{3274356}

\bibitem[RV14b]{rhodes-vargas-spec-dim}
R.~Rhodes and V.~Vargas.
\newblock Spectral {D}imension of {L}iouville {Q}uantum {G}ravity.
\newblock {\em Ann. Henri Poincar\'e}, 15(12):2281--2298, 2014,
  \arxiv{1305.0154}. \MR{3272822}

\bibitem[RY99]{revuz-yor}
D.~Revuz and M.~Yor.
\newblock {\em Continuous martingales and {B}rownian motion}, volume 293 of
  {\em Grundlehren der Mathematischen Wissenschaften [Fundamental Principles of
  Mathematical Sciences]}.
\newblock Springer-Verlag, Berlin, third edition, 1999. \MR{1725357
  (2000h:60050)}

\bibitem[Sch00]{schramm0}
O.~Schramm.
\newblock Scaling limits of loop-erased random walks and uniform spanning
  trees.
\newblock {\em Israel J. Math.}, 118:221--288, 2000, \arxiv{math/9904022}.
  \MR{1776084 (2001m:60227)}

\bibitem[She16]{shef-zipper}
S.~Sheffield.
\newblock Conformal weldings of random surfaces: {SLE} and the quantum gravity
  zipper.
\newblock {\em Ann. Probab.}, 44(5):3474--3545, 2016, \arxiv{1012.4797}.
  \MR{3551203}

\bibitem[Shi85]{shimura-cone}
M.~Shimura.
\newblock Excursions in a cone for two-dimensional {B}rownian motion.
\newblock {\em J. Math. Kyoto Univ.}, 25(3):433--443, 1985. \MR{807490
  (87a:60095)}

\bibitem[SS13]{ss-contour}
O.~Schramm and S.~Sheffield.
\newblock A contour line of the continuum {G}aussian free field.
\newblock {\em Probab. Theory Related Fields}, 157(1-2):47--80, 2013,
  \arxiv{1008.2447}. \MR{3101840}

\bibitem[SW16]{shef-wang-lqg-coord}
S.~{Sheffield} and M.~{Wang}.
\newblock {Field-measure correspondence in Liouville quantum gravity almost
  surely commutes with all conformal maps simultaneously}.
\newblock {\em ArXiv e-prints}, May 2016, \arxiv{1605.06171}.

\bibitem[Tut63]{tutte-embedding}
W.~T. Tutte.
\newblock How to draw a graph.
\newblock {\em Proc. London Math. Soc. (3)}, 13:743--767, 1963. \MR{0158387}

\bibitem[Wer04]{werner-notes}
W.~Werner.
\newblock Random planar curves and {S}chramm-{L}oewner evolutions.
\newblock In {\em Lectures on probability theory and statistics}, volume 1840
  of {\em Lecture Notes in Math.}, pages 107--195. Springer, Berlin, 2004,
  \arxiv{math/030335}. \MR{2079672 (2005m:60020)}

\bibitem[Wil96]{wilson-algorithm}
D.~B. Wilson.
\newblock Generating random spanning trees more quickly than the cover time.
\newblock In {\em Proceedings of the {T}wenty-eighth {A}nnual {ACM} {S}ymposium
  on the {T}heory of {C}omputing ({P}hiladelphia, {PA}, 1996)}, pages 296--303,
  New York, 1996. ACM. \MR{1427525}

\end{thebibliography}
\bibliographystyle{hmralphaabbrv}

\end{document}